\documentclass[10pt]{amsart}

\usepackage{amsmath}
\usepackage{amsfonts}
\usepackage{amssymb}
\usepackage{mathtools}

\usepackage[all, cmtip]{xy} 
\usepackage{bbm}
\usepackage{bbding}
\usepackage{txfonts}
\usepackage[shortlabels]{enumitem}
\usepackage{ifpdf}
\ifpdf
\usepackage[colorlinks, final, backref=page, hyperindex]{hyperref}
\else
\usepackage[colorlinks, final, backref=page, hyperindex, hypertex]{hyperref}
\fi
\usepackage{mathrsfs}
\usepackage{amscd}
\usepackage{tikz-cd}
\usepackage[active]{srcltx}
\usepackage{tikz}
\usepackage{ulem}
\usetikzlibrary{calc}
\usetikzlibrary{arrows, shapes, chains}
\usetikzlibrary{arrows.meta} 
\usetikzlibrary{decorations.markings} 

\makeatletter

\theoremstyle{definition}
\newtheorem{thm}{Theorem}[section]
\newtheorem{prop}[thm]{Proposition}
\newtheorem{lem}[thm]{Lemma}
\newtheorem{cor}[thm]{Corollary}
\newtheorem{prop-def}[thm]{Proposition-Definition}
\newtheorem{defn}[thm]{Definition}

\newtheorem{remark}[thm]{Remark}
\newtheorem{exam}[thm]{Example}

\newcommand{\nc}{\newcommand}

\nc{\delete}[1]{{}}
\nc{\mmargin}[1]{}
\nc{\Alg}{\mathrm{Alg}}
\nc{\NjO}{\mathrm{NjO}}
\nc{\NjA}{\mathrm{NjA}}
\nc{\NjL}{\mathrm{NjL}}
\nc{\NjLD}{\mathrm{NjLD}}
\nc{\FN}{\mathrm{FN}}
\nc{\FNA}{\widetilde{\FN}}
\nc{\RN}{\mathrm{RN}}
\nc{\RNA}{\widetilde{\RN}}
\nc{\G}{\mathrm{G}}
\nc{\rmH}{\mathrm{H}}
\nc{\DT}{\mathrm{DT}}
\nc{\ac}{\mathrm{\textup{!`}}}
	\nc{\mlabel}[1]{\label{#1}} 
	\nc{\mcite}[1]{\cite{#1}} 
	\nc{\mref}[1]{\ref{#1}} 
	\nc{\mbibitem}[1]{\bibitem{#1}} 

\delete{
	\nc{\mlabel}[1]{\label{#1} 
		{\hfill \hspace{1cm}{\bf{{\ }\hfill(#1)}}}}
	\nc{\mcite}[1]{\cite{#1}{{\bf{{\ }(#1)}}}} 
	\nc{\mref}[1]{\ref{#1}{{\bf{{\ }(#1)}}}} 
	\nc{\mbibitem}[1]{\bibitem[\bf #1]{#1}} 
	
}

 \font\cyrs=wncyr7

\nc{\vep}{\varepsilon}
\nc{\bin}[2]{ (_{\stackrel{\scs{#1}}{\scs{#2}}})} 
\nc{\binc}[2]{(\!\! \begin{array}{c} \scs{#1}\\
		\scs{#2} \end{array}\!\!)} 
\nc{\bincc}[2]{ ( {\scs{#1} \atop
		\vspace{-1cm}\scs{#2}} )} 
\nc{\oline}[1]{\overline{#1}}
\nc{\mapm}[1]{\lfloor\!|{#1}|\!\rfloor}
\nc{\bs}{\bar{S}}
\nc{\la}{\longrightarrow}
\nc{\ot}{\otimes}
\nc{\rar}{\rightarrow}
\nc{\lon }{\, \rightarrow\, }
\nc{\dar}{\downarrow}
\nc{\dap}[1]{\downarrow \rlap{$\scriptstyle{#1}$}}
\nc{\defeq}{\stackrel{\rm def}{=}}
\nc{\dis}[1]{\displaystyle{#1}}
\nc{\dotcup}{\ \displaystyle{\bigcup^\bullet}\ }
\nc{\hcm}{\ \hat{, }\ }
\nc{\hts}{\hat{\otimes}}
\nc{\hcirc}{\hat{\circ}}
\nc{\lleft}{[}
\nc{\lright}{]}
\nc{\curlyl}{\left \{ \begin{array}{c} {} \\ {} \end{array}
	\right. \!\!\!\!\!\!\!}
\nc{\curlyr}{ \!\!\!\!\!\!\!
	\left. \begin{array}{c} {} \\ {} \end{array}
	\right \} }
\nc{\longmid}{\left | \begin{array}{c} {} \\ {} \end{array}
	\right. \!\!\!\!\!\!\!}
\nc{\ora}[1]{\stackrel{#1}{\rar}}
\nc{\ola}[1]{\stackrel{#1}{\la}}
\nc{\scs}[1]{\scriptstyle{#1}} \nc{\mrm}[1]{{\rm #1}}
\nc{\dirlim}{\displaystyle{\lim_{\longrightarrow}}\, }
\nc{\invlim}{\displaystyle{\lim_{\longleftarrow}}\, }
\nc{\dislim}[1]{\displaystyle{\lim_{#1}}} \nc{\colim}{\mrm{colim}}
\nc{\mvp}{\vspace{0.3cm}} \nc{\tk}{^{(k)}} \nc{\tp}{^\prime}
\nc{\ttp}{^{\prime\prime}} \nc{\svp}{\vspace{2cm}}
\nc{\vp}{\vspace{8cm}}
\nc{\modg}[1]{\!<\!\!{#1}\!\!>}
\nc{\intg}[1]{F_C(#1)}
\nc{\lmodg}{\!<\!\!}
\nc{\rmodg}{\!\!>\!}
\nc{\cpi}{\widehat{\Pi}}
\nc{\ssha}{{\mbox{\cyrs X}}} 
\nc{\tsha}{{\mbox{\cyrt X}}}
\nc{\shpr}{\diamond} 
\nc{\labs}{\mid\!}
\nc{\rabs}{\!\mid}

\nc{\RBO}{{\mathrm{RBO}_\lambda}}
\nc{\Sh}{\mathrm{Sh}}
\nc{\sh}{\mathrm{\overline{Sh}}}
\nc{\RBA}{{\mathrm{RBA}_\lambda}}
\nc{\sgn}{\mathrm{sgn}}
\nc{\rH}{\mathrm{H}}

\newcommand{\ci}{C^{\infty}}
\newcommand{\hs}{\widehat{S}}

\nc{\ad}{\mrm{ad}}
\nc{\ann}{\mrm{ann}}
\nc{\Aut}{\mrm{Aut}}
\nc{\bim}{\mbox{-}\mathsf{Bimod}}
\nc{\br}{\mrm{bre}}
\nc{\can}{\mrm{can}}
\nc{\Cont}{\mrm{Cont}}
\nc{\rchar}{\mrm{char}}
\nc{\cok}{\mrm{coker}}
\nc{\de}{\mrm{dep}}
\nc{\dtf}{{R-{\rm tf}}}
\nc{\dtor}{{R-{\rm tor}}}
\renewcommand{\det}{\mrm{det}}
\newcommand{\Der}{\mathfrak{Der}}
\nc{\Div}{{\mrm Div}}
\nc{\Diff}{\mrm{DA}}
\nc{\Diffl}{\mathsf{DA}_\lambda}
\nc{\diffo}{{\mathsf{DO}_\lambda}}
\nc{\alg}{\mathsf{Alg}}
\nc{\End}{\mrm{End}}
\nc{\Alt}{\mrm{Alt}}
\nc{\Ext}{\mrm{Ext}}
\nc{\Fil}{\mrm{Fil}}
\nc{\Fr}{\mrm{Fr}}
\nc{\Frob}{\mrm{Frob}}
\nc{\Gal}{\mrm{Gal}}
\nc{\GL}{\mrm{GL}}
\nc{\Hom}{\mrm{Hom}}
\nc{\Hoch}{\mrm{Hoch}}
\nc{\C}{\mrm{C}}
\nc{\hsr}{\mrm{H}}
\nc{\hpol}{\mrm{HP}}
\nc{\id}{\mrm{Id}}
\nc{\im}{\mrm{im}}
\nc{\Id}{\mrm{Id}}
\nc{\ID}{\mrm{ID}}
\nc{\Irr}{\mrm{Irr}}
\nc{\incl}{\mrm{incl}}
\nc{\length}{\mrm{length}}
\nc{\NLSW}{\mrm{NLSW}}
\nc{\Lie}{\mrm{Lie}}
\nc{\mchar}{\rm char}
\nc{\mpart}{\mrm{part}}
\nc{\ql}{{\QQ_\ell}}
\nc{\qp}{{\QQ_p}}
\nc{\rank}{\mrm{rank}}
\nc{\rcot}{\mrm{cot}}
\nc{\rdef}{\mrm{def}}
\nc{\rdiv}{{\rm div}}
\nc{\rtf}{{\rm tf}}
\nc{\rtor}{{\rm tor}}
\nc{\res}{\mrm{res}}
\nc{\SL}{\mrm{SL}}
\nc{\Spec}{\mrm{Spec}}
\nc{\tor}{\mrm{tor}}
\nc{\Tr}{\mrm{Tr}}
\nc{\tr}{\mrm{tr}}
\nc{\wt}{\mrm{wt}}
\def\ot{\otimes}
\def\wg{\wedge}

\newcommand\dd{\mathrm{d}}

\nc{\bfk}{{\bf k}}
\nc{\bfone}{{\bf 1}}
\nc{\bfzero}{{\bf 0}}

\nc{\detail}{\marginpar{\bf More detail}
	\noindent{\bf Need more detail!}
	\svp}
\nc{\gap}{\marginpar{\bf Incomplete}\noindent{\bf Incomplete!!}
	\svp}
\nc{\FMod}{\mathbf{FMod}}
\nc{\Int}{\mathbf{Int}}
\nc{\Mon}{\mathbf{Mon}}
\nc{\sproof}{\noindent{ \textit{Sketch of Proof:} }}
\nc{\remarks}{\noindent{\bf Remarks: }}
\nc{\Rep}{\mathbf{Rep}}
\nc{\Rings}{\mathbf{Rings}}
\nc{\Sets}{\mathbf{Sets}}
\nc{\ob}{\mathsf{Ob}}

\nc{\BA}{{\mathbb A}} \nc{\CC}{{\mathbb C}}
\nc{\DD}{{\mathbb D}} \nc{\EE}{{\mathbb E}}
\nc{\FF}{{\mathbb F}} \nc{\GG}{{\mathbb G}}
\nc{\LL}{{\mathbb L}}
\nc{\NN}{{\mathbb N}} \nc{\PP}{{\mathbb P}}
\nc{\QQ}{{\mathbb Q}} \nc{\RR}{{\mathbb R}}
\nc{\TT}{{\mathbb T}} \nc{\VV}{{\mathbb V}}
\nc{\ZZ}{{\mathbb Z}} \nc{\TP}{\widetilde{P}}
\nc{\m}{{\mathbbm m}}

\newcommand{\bbR}{\mathbb R}
\newcommand{\bbS}{\mathbb S}
\newcommand{\bbZ}{\mathbb Z}

\nc{\cala}{{\mathcal A}} \nc{\calc}{{\mathcal C}}
\nc{\cald}{\mathcal{D}} \nc{\cale}{{\mathcal E}}
\nc{\calf}{{\mathcal F}} \nc{\calg}{{\mathcal G}}
\nc{\calh}{{\mathcal H}} \nc{\cali}{{\mathcal I}}
\nc{\call}{{\mathcal L}} \nc{\calm}{{\mathcal M}}
\nc{\caln}{{\mathcal N}} \nc{\calo}{{\mathcal O}}
\nc{\calp}{{\mathcal P}} \nc{\calr}{{\mathcal R}}
\nc{\cals}{{\mathcal S}} \nc{\calt}{{\Omega}}
\nc{\calv}{{\mathcal V}} \nc{\calw}{{\mathcal W}}
\nc{\calx}{{\mathcal X}}
\nc{\calLr}{{\mathcal L}^{\rho}}

\newcommand{\calA}{\mathcal A} 
\newcommand{\calB}{\mathcal B}

\newcommand{\calI}{\mathcal I}

\newcommand{\calL}{\mathcal L} 
\newcommand{\calM}{\mathcal M}
\newcommand{\calN}{\mathcal N} 
\newcommand{\calO}{\mathcal O}

\nc{\fraka}{{\mathfrak a}}
\nc{\frakb}{\mathfrak{b}}
\nc{\frakC}{\mathfrak{C}}
\nc{\frakg}{{\frak g}}
\nc{\frakl}{{\frak l}}
\nc{\frakh}{{\frak h}}
\nc{\fraks}{{\frak s}}
\nc{\frakB}{{\frak B}}
\nc{\frakm}{{\frak m}}
\nc{\frakM}{{\frak M}}
\nc{\frakp}{{\frak p}}
\nc{\frakW}{{\frak W}}
\nc{\frakX}{{\frak X}}
\nc{\frakS}{{\frak S}}
\nc{\frakT}{{\frak T}}
\nc{\frakA}{{\frak A}}
\nc{\frakL}{{\frak L}}
\nc{\frakx}{{\frakx}}
\nc{\red}{\color{red}}
\nc{\RB}{\mathfrak{RB}}
\nc{\NA}{\mathfrak{NjA}}
\nc{\NL}{\mathfrak{NjL}}

\nc{\lbar}[1]{\overline{#1}}
\nc{\ra}{\rightarrow}
\nc{\tred}[1]{\textcolor{red}{#1}} \nc{\tgreen}[1]{\textcolor{green}{#1}}
\nc{\tblue}[1]{\textcolor{blue}{#1}} \nc{\tpurple}[1]{\textcolor{purple}{#1}}
\nc{\chao}[1]{\tred{\underline{Chao:}#1 }}
\nc{\yuan}[1]{\tred{\underline{Yuan:}#1 }}
\nc{\kai}[1]{\red \tred{\red{Kai:}#1}}

\begin{document}
	
\title[Homotopy Nijenhuis algebras]{Deformations of  Nijenhuis Lie algebras and  Nijenhuis Lie algebroids}

\author{Chao Song, Kai Wang, Yuanyuan Zhang, and Guodong Zhou}

\date{\today}

\address{Chao Song, School of Mathematical Sciences, Key Laboratory of Mathematics and Engineering Applications (Ministry of Education), Shanghai Key Laboratory of PMMP, East China Normal University, Shanghai 200241, China}
\email{52265500011@stu.ecnu.edu.cn}

\address{Guodong Zhou, School of Mathematical Sciences, Key Laboratory of Mathematics and Engineering Applications (Ministry of Education), Shanghai Key Laboratory of PMMP, East China Normal University, Shanghai 200241, China}
\email{gdzhou@math.ecnu.edu.cn}

\address{Kai Wang, School of Mathematical Sciences, University of Science and Technology of China, Hefei, Anhui Provience 230026, China}
\email{wangkai17@ustc.edu.cn}

\address{Yuanyuan Zhang, School of Mathematics and Statistics, Henan University, Henan, Kaifeng 475004, China}
\email{zhangyy17@henu.edu.cn}

\begin{abstract} 
This paper is the second in a series dedicated to the operadic study of Nijenhuis structures, focusing on Nijenhuis Lie algebras and Nijenhuis geometry. We introduce the concept of homotopy Nijenhuis Lie algebras and establish that the differential graded (=dg) operad $\NL_{\infty}$ governing these structures serves as the minimal model of the operad $\NL$ for Nijenhuis Lie algebras. We construct an $L_\infty$-algebra that encodes the simultaneous deformations of Lie brackets and Nijenhuis operators, leading to the deformation cochain complex and an associated cohomology theory for Nijenhuis Lie algebras. Extending these ideas to geometry, we investigate the deformations of geometric Nijenhuis structures. We introduce the notion of a Nijenhuis Lie algebroid—a Lie algebroid equipped with a Nijenhuis structure, which generalizes the classical Nijenhuis structure on vector fields of manifolds. Using the framework of dg manifolds, we construct an $L_\infty$-algebra that governs the simultaneous deformations of Lie algebroid structures and Nijenhuis operators. As a computational application, we prove that a certain class of Nijenhuis operators satisfies the Poincar\'e Lemma, meaning its cohomology vanishes, which confirms a conjecture by Bolsinov and Konyaev.


\end{abstract}

\subjclass[2020]{
	17B10 
	17B40 
	17B66 
	18N40 
	18M60 
	18M70 
	58A50 
}

\keywords{cohomology, deformations, Nijenhuis Lie algebra, homotopy Nijenhuis Lie algebra, homotopy (co)operad, minimal model, Koszul dual homotopy cooperad, Nijenhuis Lie algebroid}

\maketitle

\tableofcontents

\allowdisplaybreaks
	
\section*{Introduction}

This work is a sequel to \cite{SWZZ2024a}. In this paper, we  study   deformations of  Nijenhuis Lie algebras and   Nijenhuis Lie algebroids.

 A Nijenhuis operator on a Lie algebra $(\frakg, [-, -])$ is a linear endomorphism $P$ satisfying the Nijenhuis relation:
\begin{align} \label{Nijenhuis relation}
	[P(a), P(b)] = P( [P(a), b] + [a, P(b)] - P[a, b]), \ \text{for } a, b \in \frakg.
\end{align}
This concept originates from the study of pseudo-complex manifolds, where Nijenhuis \cite{Nij51} introduced the notion of a Nijenhuis tensor.
Nijenhuis operators, also referred to as Nijenhuis tensors, act as unary operations on the Lie algebra of vector fields, satisfying the relation \eqref{Nijenhuis relation}. They represent a fundamental geometric condition on $(1,1)$-tensors, as discussed in \cite{BKM22}. A fundamental example of a Nijenhuis operator arises from the work of Newlander and Nirenberg \cite{NN57}, who proved that an almost complex manifold $(M, J)$ admits a unique complex structure compatible with the almost complex structure $J$ if and only if $J$ is a Nijenhuis operator. This result provides a classical example of Nijenhuis operators (see also \cite[Theorem1.2, Page356]{Hel78}). Nijenhuis operators on Lie algebras later appeared in \cite{MM84, Kos90} in the broader study of Poisson-Nijenhuis manifolds.

On a smooth manifold $M$, the space of multivector fields $\Gamma(\Lambda^{\bullet} T_M)$ is equipped with a graded Lie bracket of degree $-1$, known as the Schouten-Nijenhuis bracket. This structure generalizes the usual Lie bracket on vector fields \cite{Vai94, KS08}, and its Maurer-Cartan elements correspond precisely to Poisson bivectors. Nijenhuis structures also have a similar characterization. Specifically, Nijenhuis tensors are described in terms of the space of vector-valued differential forms $ \Omega^{\bullet}(M, T_M) $, which is equipped with a graded Lie bracket of degree $ 0 $, known as the Fr\"olicher-Nijenhuis bracket \cite{Mic87, BK21-16, BKM22, DVM95}. The Maurer-Cartan elements of this graded Lie algebra are precisely Nijenhuis tensors on $M$. Fixing a Nijenhuis tensor $P$ defines the Fr\"olicher-Nijenhuis cochain complex $(\Omega^{\bullet}(M, T_M), [-, -]_{\FN}, \dd_{\FN}=[P, -]_{\FN})$, which leads to the cohomology of Nijenhuis operators. Furthermore, Kosmann-Schwarzbach \cite{KS04} showed that the Schouten-Nijenhuis and Fr\"olicher-Nijenhuis brackets are the only two satisfying a certain equation \cite[Equation~(3.4)]{KS04} related to the Cartan formula.

Notably, the Vinogradov bracket, defined on graded differential operators on the algebra of differential forms $\Omega^\bullet(M)$, generalizes both the Schouten-Nijenhuis and Fr\"olicher-Nijenhuis brackets \cite{BM94, Vin90, CV92}. An alternative generalization was provided by Dubois-Violette and Michor \cite{DVM95}, who defined a graded Lie bracket on $\Omega^\bullet(T^{\vee}_M, T_{T^{\vee}_M})$. These results highlight the deep connections between Poisson and Nijenhuis geometry. Recently, in a series of works \cite{BKM22,BKM21,BKM23,BKM23-4,BKM24}, Bolsinov, Konyaev, and Matveev systematically studied Nijenhuis geometry. Their research covers various aspects, including the linearization problem of Nijenhuis operators, the singularities of Nijenhuis operators, the classification of Nijenhuis operators based on local normal forms, global analysis, and applications of Nijenhuis geometry.



Research on the homotopy theory of Nijenhuis Lie algebras remains relatively limited. 
In \cite{Mer05}, Merkulov introduced an operadic approach to Nijenhuis geometry, known as operad profiles. He studied linear Nijenhuis structures on formal manifolds and proved that a linear Nijenhuis structure on a vector space is equivalent to a left-symmetric algebra (also known as a pre-Lie algebra) structure on the space. To explore higher-order Nijenhuis structures, he investigated linear Nijenhuis structures on dg manifolds. In this context, he introduced the quadratic operad $ \mathbf{pre{-}Lie}^{2} $ and showed that a Nijenhuis structure on a formal dg manifold is equivalent to a $ \mathbf{pre{-}Lie}^{2} $-algebra. Strohmayer \cite{Str09} later proved that the operad $ \mathbf{pre{-}Lie}^{2} $ is Koszul, implying that its minimal model is given by the cobar construction of its Koszul dual cooperad, as introduced by Merkulov \cite{Mer05}. Furthermore, Merkulov demonstrated that a $ \mathbf{pre{-}Lie}^{2}_\infty $-algebra, when interpreted as a formal dg manifold, naturally carries a homotopy Nijenhuis structure—specifically, a homotopy Nijenhuis operator on the dg Lie algebra of vector fields. Recently, Das also introduced the cohomologies of Nijenhuis Lie algebras and Nijenhuis Lie bialgebras \cite{Das25}. 

In this paper, we primarily investigate homotopy Nijenhuis structures on $ L_\infty $-algebras from the perspective of operad theory. We present the deformation complex of a Nijenhuis Lie algebra along with its governing $ L_\infty $-structure within the framework of deformation theory. Furthermore, we establish the minimal model and the Koszul dual homotopy cooperad of the operad for Nijenhuis Lie algebras, leading to a natural extension of Nijenhuis structures to $ L_\infty $-algebras and thereby generalizing Merkulov's results. 

We also introduce the concept of a Nijenhuis Lie algebroid, which extends the notion of a Nijenhuis structure on the tangent bundle to arbitrary Lie algebroids. Building on the deformation theory of Nijenhuis Lie algebras, we develop a corresponding framework for Nijenhuis Lie algebroids. Using the language of dg manifolds, we construct an $ L_\infty $-algebra that unifies the Lie bracket of vector fields with the Fr\"olicher-Nijenhuis bracket on dg manifolds, thereby governing the simultaneous deformations of Lie algebroids and Nijenhuis operators. Additionally, we establish the cohomology theory of Nijenhuis operators on Lie algebroids and further develop the cohomology theory of Nijenhuis Lie algebroids. As an application, we prove that an analog of the Poincar\'e lemma holds for a special class of Nijenhuis operators, implying the vanishing of their Fr\"olicher-Nijenhuis cohomology. This result confirms a conjecture by Bolsinov and Konyaev.

This paper is organized as follows.

Part~\ref{Part1} (Sections~\ref{Notations}-\ref{Section: Cohomology and homotopy theory of Lie algebras}) collects relevant concepts and results on graded spaces, $L_\infty$-algebras, symmetric homotopy (co)operads, and the cohomology of Lie algebras. These are gathered from the literature to establish notation and provide a foundation for the subsequent study.

Part~\ref{Part2} (Sections~\ref{Section: Nijenhuis Lie algebras and their representations}-\ref{Section: From minimal model to Linifnity algebras NjL}) focuses on the deformation theory of Nijenhuis Lie algebras. Section~\ref{Section: Nijenhuis Lie algebras and their representations} reviews the fundamental definitions and properties of Nijenhuis Lie algebras. In Section~\ref{Sect: Cohomology theory of Nijenhuis Lie algebras}, we introduce a cochain complex for Nijenhuis operators, which serves as an algebraic analog of $\Omega^{\bullet}(M, T_M)$ in geometric setting. This cochain complex is then combined with the Chevalley-Eilenberg complex to define a deformation complex for Nijenhuis Lie algebras. Section~\ref{Section: Koszul dual cooperad NjL} introduces a symmetric homotopy cooperad, and in Section~\ref{Section: minimal model NjL}, we show that its cobar construction yields the minimal model of the operad for Nijenhuis Lie algebras, which can be viewed as its Koszul dual. This also leads to the definition of homotopy Nijenhuis Lie algebras. Section~\ref{Section: From minimal model to Linifnity algebras NjL} then constructs an $L_\infty$-algebra structure on the deformation complex from Section~\ref{Sect: Cohomology theory of Nijenhuis Lie algebras}, identifying Nijenhuis Lie algebra structures as Maurer-Cartan elements of this $L_\infty$-algebra.

Part~\ref{Part3} (Sections~\ref{Section: concepts of Lie algebroids}-\ref{Section: An Example of Calculating Frolicher-Nijenhuis Cohomology}) extends the deformation theory to geometric Nijenhuis structures. Section~\ref{Section: concepts of Lie algebroids} provides a brief review of Lie algebroids. Section~\ref{Section: Nijenhuis Lie algebroids} introduces Nijenhuis Lie algebroids, a generalization of Nijenhuis structures on vector fields to arbitrary Lie algebroids, and develops the cohomology theory of Nijenhuis operators on Lie algebroids. Section~\ref{The relation of cohomology of Nijenhuis operators between algebraic and geometric versions} compares algebraic and geometric versions of cohomology of Nijenhuis operators. To further refine this theory, Section~\ref{Section. Basic concepts of dg manifolds} reviews necessary background on dg manifolds, which leads to the construction of a deformation complex for Nijenhuis Lie algebroids in Section~\ref{Section: Cohomology theory of Nijenhuis Lie algebroids}. Section~\ref{Section: L infty algebras related with Lie algebroids} establishes an $L_\infty$-algebra structure on this complex, which governs the simultaneous deformations of the Lie algebroid and Nijenhuis structures. Finally, in Section~\ref{Section: An Example of Calculating Frolicher-Nijenhuis Cohomology}, we construct an explicit homotopy to show that a certain class of Nijenhuis operators has vanishing cohomology, thus confirming a conjecture by Bolsinov and Konyaev.

\bigskip	

\part{Preliminaries}\label{Part1}

\medskip

\section{Notations} \label{Notations}

Throughout this paper, let $\bfk$ be a field of characteristic $0$. Unless otherwise specified, all vector spaces are assumed to be $\bfk$-vector spaces, and all tensor products and $\Hom$-spaces are taken over $\bfk$.

A (homologically) graded space is a $\ZZ$-indexed family of vector spaces $V=\{V_n\}_{n \in \ZZ}$. Elements of $\cup_{n \in \ZZ} V_n$ are homogeneous with degree, written as $|v|=n$ for $v\in V_n$. 
Similarly, a cohomologically graded space is a $\ZZ$-indexed family of vector spaces $V=\{V^n\}_{n \in \ZZ}$ with indices written as superscripts.
Writing $V^n=V_{-n}$ turns homological grading into cohomological grading and vice versa.

Given two graded spaces $V$ and $W$, a linear map $f: V \rightarrow W$ such that $f(V_n) \subseteq W_{n+r}$ for all $n$ is a graded map of degree $r$ and denote $|f|=r$.
The graded space $\Hom(V, W)$ is defined to be $\{\Hom(V, W)_r\}_{r\in \ZZ}$, where
$$\Hom(V, W)_r = \prod_{p \in \ZZ} \Hom_\bfk(V_p, W_{p+r})$$ 
is the space of graded maps of degree $r$.

The tensor product $V \otimes W$ of two graded spaces $V$ and $W$
is given by
$$(V\otimes W)_n =\bigoplus_{p+q=n} V_p \otimes W_q, \ n \in \ZZ.$$

Let $V$ be a (homologically) graded space.
The suspension of $V$ is the graded space $sV$ with $(sV)_n=V_{n-1}$ for all $n\in \ZZ$. Write $sv\in (sV)_n$ for $v\in V_{n-1}$. The map $s: V \rightarrow sV$, $v\mapsto sv$ is a graded map of degree $1$.
One can also define the desuspension of $V$ which is denoted by $s^{-1}V$ with $(s^{-1}V)_n=V_{n+1}$. Write $s^{-1}v \in (s^{-1}V)_n$ for $v \in V_{n+1}$ and $s^{-1}: V \rightarrow s^{-1}V$, $v\mapsto s^{-1}v$ is a graded map of degree $-1$.
However, in the context of cohomologically graded spaces, the suspension $s$ and desuspension $s^{-1}$ have opposite degrees, i.e., $(sH)^{n} = H^{n+1}$, $(s^{-1}H)^{n} = H^{n-1}$ for a cohomologically graded space $H$. 
Sometimes, for example in geometry, the suspension of homologically or cohomologically graded space $A$ is denoted as $A[1]$, and the desuspension of $A$ is denoted as $A[-1]$.

For later use, we need to fix two isomorphisms
\begin{eqnarray} \label{Eq: first can isom} 
	\Hom((sV)^{\ot n}, sV) \simeq \Hom(V^{\ot n}, V), \ f\mapsto \widetilde{f}:= s^{-1}\circ f \circ s^{\ot n}
\end{eqnarray}
for $f\in \Hom((sV)^{\ot n}, sV)$ and
\begin{eqnarray} \label{Eq: second can isom}
	\Hom((sV)^{\ot n}, V) \simeq \Hom(V^{\ot n}, V), \ g\mapsto \widehat{g}:= g \circ s^{\ot n} 
\end{eqnarray} 
for $g\in \Hom((sV)^{\ot n}, V)$.

We will employ the Koszul sign rule to determine signs, that is, when exchanging the positions of two graded objects in an expression, we have to multiply the expression by a power of $-1$ whose exponent is the product of their degrees. 
More precisely, let $V$ be a graded space. The graded symmetric algebra $S(V)$ of $V$ is defined to be the quotient of the tensor algebra $T(V)$ by the two-sided ideal $I$ generated by $x\ot y -(-1)^{|x||y|}y\ot x$ for all homogeneous elements $x, y\in V$. For $x_1\ot\cdots\ot x_n\in V^{\ot n}\subseteq T(V)$, write $ x_1\odot x_2\odot\dots\odot x_n$ its image in $S(V)$.
Denote the symmetric group in $n$ variables by $\mathbb{S}_n$ for $n\geqslant 1$.
For homogeneous elements $x_1, \dots, x_n \in V$ and $\sigma\in \mathbb{S}_n$, the Koszul sign $\epsilon(\sigma; x_1, \dots, x_n)$ is defined by 
\begin{equation*} \label{Eq: epsilon sign}  
	x_1\odot x_2\odot\dots\odot x_n = \epsilon(\sigma; x_1, \dots, x_n)x_{\sigma(1)}\odot x_{\sigma(2)}\odot\dots\odot x_{\sigma(n)}.
\end{equation*}
We also define the alternate Koszul sign by
\begin{equation*} \label{Eq: chi sign} 
	\chi(\sigma; x_1, \dots, x_n) = \sgn (\sigma) \epsilon(\sigma; x_1, \dots, x_n), 
\end{equation*}
where $\sgn (\sigma)$ is the sign of the permutation $\sigma$.

For $1\leqslant i_1, \dots, i_r\leqslant n$ with $i_1+\cdots+i_r=n$, define $\Sh(i_1, i_2, \dots, i_r)$ as the set of $(i_1, i_2, \dots, i_r)$ shuffles, i.e., those permutations $\sigma\in \mathbb{S}_n$ such that
$$\sigma(1)< \dots< \sigma(i_1), \sigma(i_1+1)< \dots< \sigma(i_1+i_2), \dots, \sigma(i_1+\cdots+i_{r-1}+1)< \cdots< \sigma(n).$$
Moreover, define $\sh(i_1, i_2, \dots, i_r)$ as the set of $(i_1, i_2, \dots, i_r)$ local shuffles, i.e., those permutations $\sigma \in \Sh(i_1, i_2, \dots, i_r)$ such that 
$$\sigma(1)<\sigma(i_{1}+1)<\cdots<\sigma(i_1+\cdots+i_{r-1}+1).$$
For simplicity of presentation, we usually write $ \sigma \in \sh $ instead of $ \sigma \in \sh(i_1, i_2, \dots, i_r) $ whose parameters $i_1, i_2, \dots, i_r$ depend on the actual context. 


\medskip

\section{$L_\infty$-algebras and Maurer-Cartan elements} \label{section L infty} 

In this section, we recall some basics on $L_\infty$-algebras from \cite{Sta92, Get09, LS93, LM95}.

\begin{defn} \label{Def: L-infty}
	Let $L$ be a graded space. It is called an $L_\infty$-algebra if $L$ is endowed with a family of linear operators $l_n: L^{\ot n}\rightarrow L, n\geqslant 1$ with $|l_n|=n-2$ subject to the following conditions. For $n \geqslant 1$ and homogeneous elements $x_1, \dots, x_n\in L$, 
	\begin{enumerate}
		\item (generalized anti-symmetry) 
		$$l_n(x_{\sigma(1)}\ot \cdots \ot x_{\sigma(n)})=\chi(\sigma; x_1, \dots, x_n)
		l_n(x_1 \ot \cdots \ot x_n), 
		\ \text{ for any } \sigma\in \mathbb{S}_n ;$$
		\item (generalized Jacobi identity)
		$$\sum\limits_{i=1}^n\sum\limits_{\sigma\in \Sh(i, n-i)}\chi(\sigma; x_1, \dots, x_n)(-1)^{i(n-i)}
		l_{n-i+1}\big(l_i(x_{\sigma(1)}\ot \cdots \ot x_{\sigma(i)})\ot x_{\sigma(i+1)}\ot \cdots \ot x_{\sigma(n)}\big)=0. $$
	\end{enumerate}
\end{defn}

\begin{remark} \label{Rem: L-infinity for small n}
	
	In particular, if all $l_n=0$ with $n\geqslant 3$, then $(L, l_1, l_2)$ is just a dg Lie algebra; if all $l_n=0$ except $n=2$, then $(L, l_2)$ is a graded Lie algebra.	
\end{remark}

One can define Maurer-Cartan elements in $L_\infty$-algebras.

\begin{defn} \label{Def: MC element in L infinity algebra}
	An element $\alpha$ of degree $-1$ in $L_\infty$-algebra $(L, \{l_n\}_{n\geqslant1})$ is called a Maurer-Cartan element if it satisfies the Maurer-Cartan equation:
	\begin{eqnarray} \label{Eq: mc-equation}
		\sum_{n=1}^\infty
		(-1)^{\frac{n(n-1)}{2}}
		\frac{1}{n!} 
		l_n(\alpha^{\ot n})=0, 
	\end{eqnarray}
	whenever this infinite sum exists.
\end{defn}

\begin{remark} 
	For a dg Lie algebra $(L, l_1, l_2)$, Equation~\eqref{Eq: mc-equation} reduces to $l_1(\alpha)-\dfrac{1}{2}l_2(\alpha\ot \alpha)=0$.
\end{remark}

\begin{prop}[Twisting procedure] \cite{Get09, CGWZ24} \label{Prop: deformed-L-infty}
	Given a Maurer-Cartan element $\alpha$ in $L_\infty$-algebra $(L, \{l_n\}_{n\geqslant1})$, the twisting procedure gives a new $L_\infty$-structure $\{l_n^\alpha\}_{n\geqslant 1}$ on graded space $L$, by defining $l_n^{\alpha}: L^{\ot n}\rightarrow L$ as:
	\begin{eqnarray} \label{Eq: twisted L infinity algebra} 
		l^\alpha_n(x_1\ot \cdots\ot x_n)
		=\sum_{i=0}^\infty
		(-1)^{in+\frac{i(i-1)}{2}}
		\frac{1}{i!}
		l_{n+i}(\alpha^{\ot i}\ot x_1\ot \cdots \ot x_n) 
	\end{eqnarray}
	for
	$n \geqslant 1$ and homogeneous elements $x_1, \dots, x_n\in L$, whenever these infinite sums exist. The new $L_\infty$-algebra $(L, \{l_n^\alpha\}_{n\geqslant 1})$ is called the twisted $L_\infty$-algebra (by the Maurer-Cartan element $\alpha$).
\end{prop}

\begin{remark} \label{Lem: twist dgla}
	For a dg Lie algebra $(L, l_1, l_2)$ and a Maurer-Cartan element $\alpha\in L_{-1}$, Equation~\eqref{Eq: twisted L infinity algebra} reduces to $l_1^\alpha(x) = l_1(x)-l_2(\alpha\ot x)$ for $x\in L $ and $l_2^\alpha = l_2$.
\end{remark}

\medskip

\section{Symmetric homotopy (co)operads} \label{Subsection: Symmetric homotopy (co)operads} \

In this subsection, we collect some basics on symmetric homotopy (co)operads from \cite{Mar96, Mar08, MV09a, MV09b, DCV13, VdLa},
we also explain how to obtain $L_\infty$-structures from symmetric homotopy operads, in particular, from convolution symmetric homotopy operads.
For the basic theory of symmetric operads, we refer the reader to the textbooks \cite{LV12, BD16}.

Recall that a graded $ \mathbb{S} $-module $\calo=\{\calo(n)\}_{n\geqslant 1}$ is a graded collection, i.e., a family of graded spaces indexed by positive integers, such that each $\calo(n)$ is also a right $ \bfk[\mathbb{S}_{n}] $-module. 
For any $v\in \calo(n)$, the index $n$ is called the arity of $v$. The suspension of $\calo$, denoted by $s\calo$, is defined to be the graded $ \mathbb{S} $-module $\{s\calo(n)\}_{n\geqslant 1}$. In the same way, one has the desuspension $s^{-1}\calo$ of the graded $ \mathbb{S} $-module $\calo$.

We need to recall some preliminaries about planar labeled (rooted) trees. A planar labeled tree $ T $ with arity $ n $ is a planar tree with leaves labeled by a permutation of the set $ \{1, \dots, n\} $ (from left to right), i.e., a bijection from the set $ \{1, \dots, n\} $ to the set of leaves of $ T $.
A planar labeled tree is said to be reduced if each of its vertices has at least one leaf. As we only consider planar labeled reduced trees in this subsection, we will simply call them trees.
For any tree $T$, denote the weight (=number of vertices) and arity (=number of leaves) of $T$ to be $\omega(T)$ and $\alpha(T)$ respectively. 
Denote by $\mathfrak{T}$ the set of all trees of weight $\geqslant 1$. For $n, m \geqslant 1$, denote by $\mathfrak{T}^{(n)}$ the set of trees of weight $n$, by $\mathfrak{T}{(m)}$ the set of trees of arity $m$, and by $\mathfrak{T}^{(n)}(m)$ the set of trees of weight $n$ and arity $m$.
Since trees are planar, each vertex in a tree has a total order on its inputs which will be drawn clockwisely.
By the existence of the root, there is a natural induced total order on the set of all vertices of a given tree $T\in \mathfrak{T}$, which is given by counting the vertices starting from the root clockwisely along the tree. We call this order the planar order. 
For example, leaves of the following tree, of weight $4$ and arity $7$, are labeled by $\{2,5,4,3,1,7,6\}$ and the planar ordered set of all vertices of it is $\{v_1 < v_2 < v_3 < v_4\}$.
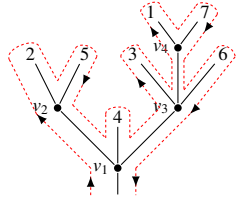
$$
\begin{tikzpicture}[scale=0.8]
	\tikzstyle{every node}=[thick, minimum size=2pt, inner sep=1pt]
		\node (t) at (0,-0.5){};
		\node[circle, fill=black, label=left:\scriptsize$v_1$] (a) at (0,0) {};
		\node[circle, fill=black, label=left:\scriptsize$v_2$] (b1) at (-1,1) {};
		\node (b2) at (0,0.85) {\scriptsize 4};
		\node[circle, fill=black, label=left:\scriptsize$v_3$] (b3) at (1,1) {};
		\node (c1) at (-1.45,1.9){\scriptsize 2};
		\node (c2) at (-0.55,1.9){\scriptsize 5};
		\node (c3) at (0.25,1.9){\scriptsize 3};
		\node[circle, fill=black, label=left:\scriptsize$v_4$] (c4) at (1,2){};
		\node (c5) at (1.75,1.9){\scriptsize 6};
		\node (d1) at (0.55,2.6){\scriptsize 1};
		\node (d2) at (1.45,2.6){\scriptsize 7};
		\draw (t)--(a);
			\draw (a)--(b1);
				\draw (b1)--(c1);
				\draw (b1)--(c2);
			\draw (a)--(b2);
			\draw (a)--(b3);
				\draw (b3)--(c3);
				\draw (b3)--(c4);
					\draw (c4)--(d1);
					\draw (c4)--(d2);
				\draw (b3)--(c5);
		\path[draw=red, dashed, dash pattern=on 1pt off 1pt,postaction={decorate}, decoration={markings, mark=at position 0.7 with {\arrow{latex}}}] 
			(-0.45,-0.45) to (-0.45,0);
		\path[draw=red, dashed, dash pattern=on 1pt off 1pt,postaction={decorate}, decoration={markings, mark=at position 1 with {\arrow{latex}}}] 
			(-0.45,0) to (-1.35,0.9);
		\path[draw=red, dashed, dash pattern=on 1pt off 1pt] 
			(-1.35,0.9) to (-1.8,1.8);
		\path[draw=red, dashed, dash pattern=on 1pt off 1pt]  
			(-1.8,1.8) to [out=120, in=100] (-1.3,2);
		\path[draw=red, dashed, dash pattern=on 1pt off 1pt]  
			(-1.3,2) to (-1,1.45);
		\path[draw=red, dashed, dash pattern=on 1pt off 1pt]  
			(-1,1.45) to (-0.7,2);
		\path[draw=red, dashed, dash pattern=on 1pt off 1pt]  
			(-0.7,2) to [out=50, in=80] (-0.35,1.85);
		\path[draw=red, dashed, dash pattern=on 1pt off 1pt,postaction={decorate}, decoration={markings, mark=at position 0.6 with {\arrow{latex}}}]  
			(-0.35,1.85) to (-0.75,1.05);
		\path[draw=red, dashed, dash pattern=on 1pt off 1pt]  
			(-0.75,1.05) to (-0.2,0.5);
		\path[draw=red, dashed, dash pattern=on 1pt off 1pt]  
			(-0.2,0.5) to (-0.2,0.9);
		\path[draw=red, dashed, dash pattern=on 1pt off 1pt]  
			(-0.2,0.9) to [out=90, in=80] (0.2,0.9);
		\path[draw=red, dashed, dash pattern=on 1pt off 1pt]  
			(0.2,0.9) to (0.2,0.5);
		\path[draw=red, dashed, dash pattern=on 1pt off 1pt]  
			(0.2,0.5) to (0.7,1);
		\path[draw=red, dashed, dash pattern=on 1pt off 1pt,postaction={decorate}, decoration={markings, mark=at position 0.6 with {\arrow{latex}}}]  
			(0.7,1) to (0.05,1.85);
		\path[draw=red, dashed, dash pattern=on 1pt off 1pt]  
			(0.05,1.85) to [out=90, in=80] (0.35,2);
		\path[draw=red, dashed, dash pattern=on 1pt off 1pt]  
			(0.35,2) to (0.9,1.3);
		\path[draw=red, dashed, dash pattern=on 1pt off 1pt]  
			(0.9,1.3) to (0.9,1.8);
		\path[draw=red, dashed, dash pattern=on 1pt off 1pt,postaction={decorate}, decoration={markings, mark=at position 0.7 with {\arrow{latex}}}]  
			(0.9,1.8) to (0.38,2.6);
		\path[draw=red, dashed, dash pattern=on 1pt off 1pt]  
			(0.38,2.6) to [out=90, in=80] (0.7,2.7);
		\path[draw=red, dashed, dash pattern=on 1pt off 1pt]  
			(0.7,2.7) to (1,2.3);
		\path[draw=red, dashed, dash pattern=on 1pt off 1pt]  
			(1,2.3) to (1.3,2.7);
		\path[draw=red, dashed, dash pattern=on 1pt off 1pt]  
			(1.3,2.7) to [out=60, in=70] (1.65,2.55);
		\path[draw=red, dashed, dash pattern=on 1pt off 1pt,postaction={decorate}, decoration={markings, mark=at position 0.4 with {\arrow{latex}}}]  
			(1.65,2.55) to (1.1,1.8);
		\path[draw=red, dashed, dash pattern=on 1pt off 1pt]  
			(1.1,1.8) to (1.1,1.3);
		\path[draw=red, dashed, dash pattern=on 1pt off 1pt]  
			(1.1,1.3) to (1.65,2);
		\path[draw=red, dashed, dash pattern=on 1pt off 1pt]  
			(1.65,2) to [out=50, in=90] (1.95,1.85);
		\path[draw=red, dashed, dash pattern=on 1pt off 1pt]  
			(1.95,1.85) to (1.2,0.9);
		\path[draw=red, dashed, dash pattern=on 1pt off 1pt,postaction={decorate}, decoration={markings, mark=at position 0 with {\arrow{latex}}}]  
			(1.2,0.9) to (0.3,0);
		\path[draw=red, dashed, dash pattern=on 1pt off 1pt,postaction={decorate}, decoration={markings, mark=at position 0.7 with {\arrow{latex}}}]  
			(0.3,0) to (0.3,-0.45);
\end{tikzpicture}
$$

There is a right $\mathbb{S}_{n}$ action on $\mathfrak{T}{(n)}$ for all $n \geqslant 1$. Let $ T \in \mathfrak{T}{(n)} $ be a tree, whose leaves are labeled by $ \{i_{1}, \dots, i_{n}\} $. For any $ \sigma \in \mathbb{S}_{n} $, define $ T \sigma \in \mathfrak{T}{(n)} $ to be the tree that shares the same underlying planar tree with $ T $ but with its leaves labeled by $ \{ i_{\sigma^{-1}(1)}, \dots, i_{\sigma^{-1}(n)}\} $. 
Naturally, we have the notion of isomorphisms of trees, which are isomorphisms of rooted trees, i.e., commutes with the parent functions, preserving all labels.

Let $ \phi: T_{1} \rightarrow T_{2}$ be an isomorphism of trees, $ \{v_{1} < \dots < v_{n}\} $ be vertices of $ T_{1} $ and $ \{v'_{1} < \dots < v'_{n}\} $ be vertices of $ T_{2} $ in the planar order respectively.
Define $ v'_{1} = \phi(v_{\tau(1)}), \dots, v'_{n} = \phi(v_{\tau(n)}) $, this defines a permutation $ \tau \in \mathbb{S}_{n} $. 
Moreover, assume leaves of $ v_{i} $ in $ T_{1} $ are labeled by $ \{l_{1}, \dots, l_{j}\} $ and leaves of $ \phi(v_{i}) $ in $ T_{2} $ are labeled by $ \{l_{\sigma_{i}^{-1}(1)}, \dots, l_{\sigma_{i}^{-1}(j)}\} $, then $ \phi $ induces a permutation $ \sigma_{i} \in \mathbb{S}_{i} $ of labels of leaves at vertex $ v_{i} $ in $ T_{1} $. 
For any $T \in \mathfrak{T}$, each vertex inherits a label by taking the minimal label of the leaves above it. We say $T \in \mathfrak{T}$ is a shuffle tree if the induced labels of leaves of $ v $ increase from left to right for each vertex $ v $ of $ T $. 
Denote by $\mathrm{S} \mathfrak{T}$ the set of all shuffle trees. For $n, m \geqslant 1$, denote by $\mathrm{S} \mathfrak{T}^{(n)}$ the set of shuffle trees of weight $n$, by $\mathrm{S} \mathfrak{T}{(m)}$ the set of shuffle trees of arity $m$, and by $\mathrm{S}\mathfrak{T}^{(n)}(m)$ the set of shuffle trees of weight $n$ and arity $m$.
Note that there is a one-to-one correspondence between the isomorphism class of trees and $\mathrm{S} \mathfrak{T}$. 
For example, consider the following trees $T_{1}$, $T_{2}$ and $T_{3}$,
$$
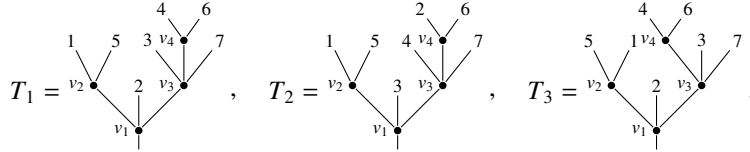
\begin{tikzpicture}[scale=0.6]
	\tikzstyle{every node}=[thick, minimum size=2pt, inner sep=1pt]
		\node (t1) at (-2.3,0.85){$T_{1} = $};
		\node (td) at (2.1,0.75){,};
		\node (t) at (0,-0.5){};
		\node[circle, fill=black, label=left:\scriptsize$v_1$] (a) at (0,0) {};
		\node[circle, fill=black, label=left:\scriptsize$v_2$] (b1) at (-1,1) {};
		\node (b2) at (0,1) {\scriptsize 2}; 
		\node[circle, fill=black, label=left:\scriptsize$v_3$] (b3) at (1,1) {};
		\node (c1) at (-1.5,2){\scriptsize 1}; 
		\node (c2) at (-0.5,2){\scriptsize 5}; 
		\node (c3) at (0.2,2){\scriptsize 3}; 
		\node[circle, fill=black, label=left:\scriptsize$v_4$] (c4) at (1,2){};
		\node (c5) at (1.8,2){\scriptsize 7}; 
		\node (d1) at (0.5,2.7){\scriptsize 4}; 
		\node (d2) at (1.5,2.7){\scriptsize 6}; 
		\draw (t)--(a);
			\draw (a)--(b1);
				\draw (b1)--(c1);
				\draw (b1)--(c2);
			\draw (a)--(b2);
			\draw (a)--(b3);
				\draw (b3)--(c3);
				\draw (b3)--(c4);
					\draw (c4)--(d1);
					\draw (c4)--(d2);
				\draw (b3)--(c5);
\end{tikzpicture}
\quad
\begin{tikzpicture}[scale=0.6]
	\tikzstyle{every node}=[thick, minimum size=2pt, inner sep=1pt]
		\node (t1) at (-2.3,0.85){$T_{2} = $};
		\node (td) at (2.1,0.75){,};
		\node (t) at (0,-0.5){};
		\node[circle, fill=black, label=left:\scriptsize$v_1$] (a) at (0,0) {};
		\node[circle, fill=black, label=left:\scriptsize$v_2$] (b1) at (-1,1) {};
		\node (b2) at (0,1) {\scriptsize 3}; 
		\node[circle, fill=black, label=left:\scriptsize$v_3$] (b3) at (1,1) {};
		\node (c1) at (-1.5,2){\scriptsize 1}; 
		\node (c2) at (-0.5,2){\scriptsize 5}; 
		\node (c3) at (0.2,2){\scriptsize 4}; 
		\node[circle, fill=black, label=left:\scriptsize$v_4$] (c4) at (1,2){};
		\node (c5) at (1.8,2){\scriptsize 7}; 
		\node (d1) at (0.5,2.7){\scriptsize 2}; 
		\node (d2) at (1.5,2.7){\scriptsize 6}; 
		\draw (t)--(a);
			\draw (a)--(b1);
				\draw (b1)--(c1);
				\draw (b1)--(c2);
			\draw (a)--(b2);
			\draw (a)--(b3);
				\draw (b3)--(c3);
				\draw (b3)--(c4);
					\draw (c4)--(d1);
					\draw (c4)--(d2);
				\draw (b3)--(c5);
\end{tikzpicture}
\quad
\begin{tikzpicture}[scale=0.6]
	\tikzstyle{every node}=[thick, minimum size=2pt, inner sep=1pt]
		\node (t1) at (-2.3,0.85){$T_{3} = $};
		\node (td) at (2.1,0.75){.};
		\node (t) at (0,-0.5){};
		\node[circle, fill=black, label=left:\scriptsize$v_1$] (a) at (0,0) {};
		\node[circle, fill=black, label=left:\scriptsize$v_2$] (b1) at (-1,1) {};
		\node (b2) at (0,1) {\scriptsize 2}; 
		\node[circle, fill=black, label=left:\scriptsize$v_3$] (b3) at (1,1) {};
		\node (c1) at (-1.5,2){\scriptsize 5}; 
		\node (c2) at (-0.5,2){\scriptsize 1}; 
		\node[circle, fill=black, label=left:\scriptsize$v_4$] (c3) at (0.2,2){}; 
		\node (c4) at (1,2){\scriptsize 3};
		\node (c5) at (1.8,2){\scriptsize 7}; 
		\node (d1) at (-0.3,2.7){\scriptsize 4}; 
		\node (d2) at (0.7,2.7){\scriptsize 6}; 
		\draw (t)--(a);
			\draw (a)--(b1);
				\draw (b1)--(c1);
				\draw (b1)--(c2);
			\draw (a)--(b2);
			\draw (a)--(b3);
				\draw (b3)--(c3);
					\draw (c3)--(d1);
					\draw (c3)--(d2);
				\draw (b3)--(c4);
				\draw (b3)--(c5);
\end{tikzpicture}
$$
Observe that $T_{1} \in \mathrm{S} \mathfrak{T}^{(4)}(7)$ is a shuffle tree, but $T_{2} \in \mathfrak{T}^{(4)}(7)$ is not a shuffle tree because the labels of leaves of $v_{3}$ and $v_{1}$ are $\{4,2,7\}$ and $\{1,3,2\}$ respectively, which are not increase from left to right. Likewise, $T_{3} \in \mathfrak{T}^{(4)}(7)$ is also not a shuffle tree but isomorphic to $T_{1}$.

Let $T, T^{\prime} \in \mathfrak{T}$, we say $T^{\prime}$ is a divisor (or subtree) of $T$, denoted by $T^{\prime} \subset T$, if $T^{\prime}$ is a divisor of $T$ as planar trees and the labels of $T^{\prime}$ increases from left to right. Assume that $T^{\prime}$ is a divisor of $T \in \mathfrak{T}$, then the quotient tree $T / T^{\prime} \in \mathfrak{T}$ is obtained from $T$ by replacing the underlying planar tree of $ T^{\prime} $ by a corolla of the same arity.

There is a natural permutation $\sigma=\sigma(T, T')\in \mathbb{S}_{\omega(T)}$ associated with the pair $(T, T')$ defined as follows. Assume the ordered set $\{v_1<\dots<v_n\}$ to be the sequence of all vertices of $T$ in the planar order and $\omega(T')=j$. Let $v'$ be the vertex in $T/T'$ corresponding to the divisor $T'$ in $T$ and the serial number of $v'$ in $T/T'$ is $i$ in the planar order (so there are $i-1$ vertices ``before" $T'$). Then define $\sigma=\sigma(T, T')\in \mathbb{S}_{n}$ to be the unique permutation which does not permute the vertices $v_1, \dots, v_{i-1}$, and such that the ordered set $\{v_{\sigma(i)}<\dots<v_{\sigma(i+j-1)}\}$ is exactly the planar ordered set of all vertices of $T'$ and the ordered set $\{v_{1}<\dots<v_{i-1}<v'<v_{\sigma(i+j)}<\dots< v_{\sigma(n)}\}$ is exactly the planar ordered set of all vertices of $T/T'$. 

Let $\calp=\{\calp(n)\}_{n\geqslant 1}$ be a graded $ \mathbb{S} $-module and $T\in \mathfrak{T}^{(n)}$ with vertices $\{v_1<\dots<v_n\}$ in the planar order. Define $\calp^{\ot T}$ to be the graded space $\calp(\alpha(v_1))\ot \cdots \ot \calp(\alpha(v_n))$. 
Each element of $\calp^{\ot T}$ can be pictured as a linear combination of the tree $T$ with each vertex $v_i$ decorated by an element of $\calp(\alpha(v_i))$.
For $T \in \mathfrak{T}(n)$ and $\sigma \in \mathbb{S}_n$, it is easy to see that $\calp^{\ot T} \cong \calp^{\ot T \sigma}$ as graded spaces.

\begin{defn} \label{Def: homotopy operad S}
	A symmetric homotopy operad structure on a graded $ \mathbb{S} $-module $\calp=\{\calp(n)\}_{n\geqslant 1}$ consists of a family of operations $$\{m_T: \calp^{\ot T}\rightarrow \calp(\alpha(T))\}_{T \in \mathfrak{T}}$$ with $|m_T|=\omega(T)-2$ satisfying the following:
	\begin{enumerate}
		\item for any $T \in \mathfrak{T}^{(n)}(m) $, $\sigma \in \mathbb{S}_{m}$ and $ x_{1} \ot \cdots \ot x_{n} \in \calp^{\ot T} \cong \calp^{\ot T \sigma } $, 
		$$ m_{T \sigma}(x_{1} \ot \cdots \ot x_{n}) = m_{T}(x_{1} \ot \cdots \ot x_{n}) \sigma \in \calp^{\ot T \sigma };$$
		
		\item let $ \phi :T_{1} \rightarrow T_{2}$ be an isomorphism of trees with $ \omega(T_{1}) = n $ and $ x_{1} \ot \cdots \ot x_{n} \in \calp^{\ot T_{1}} $, then
		$$ m_{T_{1}}(x_{1} \sigma_{1} \ot \cdots \ot x_{n} \sigma_{n}) = \chi(\tau;x_{1}, \dots, x_{n}) m_{T_{2}} (x_{\tau(1)}\ot \cdots \ot x_{\tau(n)}) $$
		with the induced permutations $ \tau, \sigma_{1}, \dots, \sigma_{n} $;
	
		\item the equation	
		$$\sum_{T'\subset{T}} \sgn(\sigma(T, T')) (-1)^{i-1+jk} m_{T/T'}\circ(\id^{\ot {i-1}}\ot m_{T'}\ot \id^{\ot k}) \circ r_{\sigma(T, T^{\prime})} = 0 $$
		holds for any $T\in \mathfrak{T}$, where $T'$ runs through the set of all divisors of $T$, $i$ is the serial number of the vertex $v'$ in $T/T'$, $j=\omega(T')$, $k=\omega(T)-i-j + 1$ and
		$r_{\sigma(T, T')}$ denotes the right action by $\sigma=\sigma(T, T')$, that is, $$r_\sigma(x_1\ot \cdots\ot x_n):=\epsilon(\sigma; x_1, \dots, x_{n}) x_{\sigma(1)}\ot \cdots \ot x_{\sigma(n)}.$$
	\end{enumerate}
\end{defn}

\begin{remark}
	If condition (i) holds, one can see that operations $ \{m_{T}\}_{T \in \mathfrak{T}} $ are determined by the trees labeled by $ \{1, \dots, n\} $ from left to right, i.e., trivial labeled trees. For example, for the unique trivial labeled tree $T \in \mathfrak{T}^{(1)}(n)$, $ \sigma \in \mathbb{S}_{n} $ and $ f \in \calp(n) $, we have
		$$ \begin{tikzpicture}[scale=0.9]
			\tikzstyle{every node}=[thick, minimum size=3pt, inner sep=1pt]
			\node (eq1) at (-3.2, 0.2){$T \sigma =$};
			\node (e-1) at (0, -0.5){};
			\node (e0) at (0, 0)[fill=black, circle]{};
			\draw (e-1)--(e0);
			\node (e0-1) at (-2, 1){\tiny $ \sigma^{-1}(1) $};
			\node (e0-2) at (-0.7, 1){\tiny $ \sigma^{-1}(2) $};
			\node (e0-3) at (0.7, 1){\tiny $ \sigma^{-1}(n-1) $};
			\node (e0-4) at (2, 1){\tiny $ \sigma^{-1}(n) $};
			\draw (e0)--(e0-1);
			\draw (e0)--(e0-2);
			\draw (e0)--(e0-3);
			\draw (e0)--(e0-4);
			\draw [dotted, line width=1pt] (-0.3, 0.6)--(0.3, 0.6);
		\end{tikzpicture} $$			
		and $ m_{T \sigma }(f) = m_{T}(f) \sigma $; 
		for $ \sigma=(2 \, \, 3) \in \mathbb{S}_{5}$, $ f \in \calp(2) $, $ g \in \calp(2) $, $ h \in \calp(3) $ and 
		$$ \begin{tikzpicture}[scale=0.9]
			\tikzstyle{every node}=[thick, minimum size=3pt, inner sep=1pt]
			\node (eq1) at (-1.5, 0.3){$T_{1}=$};
			\node (eq1-1) at (1.3, 0.2){,};
			\node (e-1) at (0, -0.5){};
			\node (e0) at (0, 0)[fill=black, circle]{};
			\draw (e-1)--(e0);
			\node (e0-1) at (-0.7, 0.5)[fill=black, circle]{};
			\node (e0-2) at (0.7, 0.5)[fill=black, circle]{};
			\node (e1-1) at (-0.9, 1){\tiny $ 1 $};
			\node (e1-2) at (-0.5, 1){\tiny $ 2 $};
			\node (e2-1) at (0.4, 1){\tiny $ 3 $};
			\node (e2-2) at (0.7, 1){\tiny $ 4 $};
			\node (e2-3) at (1, 1){\tiny $ 5 $};
			\draw (e0)--(e0-1);
			\draw (e0)--(e0-2);
			\draw (e0-1)--(e1-1);
			\draw (e0-1)--(e1-2);
			\draw (e0-2)--(e2-1);
			\draw (e0-2)--(e2-2);
			\draw (e0-2)--(e2-3);
		\end{tikzpicture}
		\quad
		\begin{tikzpicture}[scale=0.9]
			\tikzstyle{every node}=[thick, minimum size=3pt, inner sep=1pt]
			\node (eq2) at (3.8, 0.3){$T_{2}=T_{1} \sigma =$};
			\node (eq2-1) at (7.3, 0.2){,};
			\node ('e-1) at (6, -0.5){};
			\node ('e0) at (6, 0)[fill=black, circle]{};
			\draw ('e-1)--('e0);
			\node ('e0-1) at (5.3, 0.5)[fill=black, circle]{};
			\node ('e0-2) at (6.7, 0.5)[fill=black, circle]{};
			\node ('e1-1) at (5.1, 1){\tiny $ 1 $};
			\node ('e1-2) at (5.5, 1){\tiny $ 3 $};
			\node ('e2-1) at (6.4, 1){\tiny $ 2 $};
			\node ('e2-2) at (6.7, 1){\tiny $ 4 $};
			\node ('e2-3) at (7, 1){\tiny $ 5 $};
			\draw ('e0)--('e0-1);
			\draw ('e0)--('e0-2);
			\draw ('e0-1)--('e1-1);
			\draw ('e0-1)--('e1-2);
			\draw ('e0-2)--('e2-1);
			\draw ('e0-2)--('e2-2);
			\draw ('e0-2)--('e2-3);
		\end{tikzpicture}
		$$
		we have $ m_{T_{2}}(f \ot g \ot h) = m_{T_{1}}(f \ot g \ot h) \sigma$.
\end{remark}

A strict morphism between two symmetric homotopy operads is a graded $ \mathbb{S} $-module morphism compatible with all operations $m_T, T \in \mathfrak{T}$.

Let $\cali$ be the graded $\mathbb{S} $-module with $\cali(1)=\bfk \id$ and $\cali(n)=0$ for $n\ne1$. The graded $\mathbb{S} $-module $\cali$ can be endowed with a symmetric homotopy operad structure in a natural way, that is, $m_T: \cali(1)\ot \cali(1) \rightarrow \cali(1)$ is given by the identity when $T$ is the unique tree of weight $2$ and arity $1$, and $m_T$ vanishes otherwise.

A symmetric homotopy operad $\calp$ is called strictly unital if there exists a strict morphism of symmetric homotopy operads $\eta: \cali\rightarrow \calp$ such that for each $n\geqslant 1$, the compositions 
$$\calp(n)\cong \calp(n)\ot \cali(1)\xrightarrow{\id\ot \eta}\calp(n)\ot \calp(1)\xrightarrow{m_{T_{1, i}}}\calp(n)$$ 
and
$$\calp(n)\cong \cali(1)\ot \calp(n)\xrightarrow{\eta\ot \id}\calp(1)\ot \calp(n)\xrightarrow{m_{T_2}} \calp(n)$$
are identity maps on $\calp(n)$, where $T_{1, i}$ with $1\leqslant i\leqslant n$ is the trivial labeled tree of weight $2$ and arity $n$ with its second vertex having arity $1$ and connecting to the first vertex on its $i$-th leaf, and $T_2$ is the trivial labeled tree of weight $2$ and arity $n$ whose first vertex has arity $1$. Intuitively, one can draw trivial labeled trees $T_{1, i}$ and $T_2$ as follows (omit their labels $\{1,\dots,n\}$):
\begin{align*}
	\begin{tikzpicture}[scale=0.8, descr/.style={fill=white}]
		\tikzstyle{every node}=[thick, minimum size=2pt, inner sep=1pt]
		\node (T1) at (-1.8,0.5){$T_{1,i}=$};
		\node (t) at (0,-0.5){};
		\node[circle, fill=black] (a) at (0,0) {};
		\node (b1) at (-1,1) {};
		\node[circle, fill=black] (b2) at (0,1) {};
		\node (b3) at (1,1) {};
		\node (c) at (0,1.5) {};
		\draw (t)--(a);
		\draw (a)--(b1);
		\draw (a)--(b2);
		\draw (a)--(b3);
		\draw (b2)--(c);
		\draw [dotted, line width=1pt] (-0.5, 0.6)--(0.5, 0.6);
		\path[-,font=\scriptsize] (a) edge node[descr]{$i$} (b2);
		\node (d) at (1.3,0.4) {,};
	\end{tikzpicture}
	\quad
	\begin{tikzpicture}[scale=0.8]
		\tikzstyle{every node}=[thick, minimum size=2pt, inner sep=1pt]
		\node (T2) at (-1.7,0.4){$T_{2}=$};
		\node[circle, fill=black] (t) at (0,-0.3){};
		\node[circle, fill=black] (a) at (0,0) {};
		\node (b1) at (-1,1) {};
		\node (b2) at (0,1) {};
		\node (b3) at (1,1) {};
		\node (c) at (0,-0.6) {};
		\draw (a)--(c);
		\draw (a)--(b1);
		\draw (a)--(b2);
		\draw (a)--(b3);
		\draw [dotted, line width=1pt] (-0.5, 0.6)--(0.5, 0.6);
		\node (d) at (1.3,0.4) {.};
	\end{tikzpicture}
\end{align*}
Furthermore, for any tree $T$ of weight at least three, $m_T \circ (\id^{\ot i-1} \ot \eta \ot \id^{\ot \omega(T)-i})$ is required to be 0 for all $1 \leqslant i \leqslant \omega(T)$.

A strictly unital symmetric homotopy operad $\calp$ is called augmented if there exists a strict morphism of symmetric homotopy operads $\varepsilon: \calp\rightarrow \cali$ such that $\varepsilon \circ \eta=\id_{\cali}$.

If a symmetric homotopy operad $\calp$ satisfies $m_T=0$ for all $T\in \frakT$ with $\omega(T)\geqslant 3$, then $\calp$ is just a nonunital symmetric dg operad in the sense of Markl \cite{Mar08}, with $\{m_T \mid T \in \mathrm{S} \mathfrak{T}^{(1)}\}$ and $\{m_T \mid T \in \mathrm{S} \mathfrak{T}^{(2)}\}$ defining its differentials and partial compositions respectively.

Given a (nonunital) symmetric dg operad $\calp$, for each tree $T \in \mathfrak{T}$, one can define the composition $m_\calp^T:\calp^{\ot T} \rightarrow \calp(\alpha(T))$ in $\calp$ along $T$ as follows:
\begin{enumerate}
	\item for the trivial labeled tree $T \in \mathfrak{T}^{(1)}(m)$ with $m \geqslant 1$, $\sigma \in \mathbb{S}_{m}$ and $x \in \calp(m)$, define $m_\calp^{T \sigma}(x) = x \sigma$;
	\item for $T \in \mathfrak{T}^{(2)}$, define $m_\calp^T=m_T$;
	\item for $T \in \mathfrak{T}^{(n)}$ with $n \geqslant 3$, write $T$ as the grafting of a divisor $T'$, whose vertex set is that of $T$ except the last one, with the corolla whose unique vertex is exactly the last vertex of $T$ in the planar order, then define $m_\calp^T=m_{T/T'}\circ (m_\calp^{T'}\ot \Id) $, where $m_\calp^{T'}$ is obtained by induction.
\end{enumerate}

Dualizing the definition of symmetric homotopy operads, one has the notion of symmetric homotopy cooperads.

\begin{defn}
	Let $\calc=\{\calc(n)\}_{n\geqslant 1}$ be a graded $ \mathbb{S} $-module. A symmetric homotopy cooperad structure on $\calc$ consists of a family of operations 
	$$\{\Delta_T: \calc(\alpha(T))\rightarrow \calc^{\ot T }\}_{T\in\mathfrak{T}}$$ 
	with $|\Delta_T|=\omega(T)-2$ such that for any $c\in \calc$, $\Delta_T(c)=0$ for almost all but finitely many $T\in\frakT$, and the family of operations $\{\Delta_T\}_{T\in\frakT}$ satisfy
	\begin{enumerate}
		\item let $T \in \mathfrak{T}(n)$ and $\sigma \in \mathbb{S}_{n}$, $ \Delta_{T \sigma}(c \sigma) = \Delta_{T}(c) \sigma \in \calc^{\ot T \sigma }$;

		\item let $ \phi :T_{1} \rightarrow T_{2}$ be an isomorphism of trees with $T_{1} \in \mathfrak{T}^{(n)}(m)$, $c \in \calc(m)$ and $ \Delta_{T_{1}}(c) = x_{1} \ot \cdots \ot x_{n} $, then
		$$ \Delta_{T_{2}}(c) = \chi(\tau; x_{1}, \dots, x_{n}) \
		x_{\tau(1) } \sigma_{\tau(1)}^{-1} \ot \cdots \ot x_{\tau(n) } \sigma_{\tau(n)}^{-1}$$	
		with the induced permutations $ \tau, \sigma_{1}, \dots, \sigma_{n} $;	

		\item the equation	
		$$\sum_{T'\subset{T}} \sgn(\sigma(T, T')^{-1}) (-1)^{i-1+jk}
		r_{\sigma(T, T^{\prime})^{-1}} \circ (\id^{\ot {i-1}}\ot \Delta_{T'}\ot \id^{\ot k}) \circ \Delta_{T/T'} =0$$
		holds for any $T \in \frakT$, where $T'$ runs through the set of all divisors of $T$, and $i, j, k$ have the same meaning as for symmetric homotopy operads in Definition~\ref{Def: homotopy operad S}.	
	\end{enumerate}
\end{defn}

The graded $ \mathbb{S} $-module $\cali$ has a natural symmetric homotopy cooperad structure, that is, $\Delta_T: \cali(1) \rightarrow \cali(1)\ot \cali(1)$ is given by the identity, when $T$ is the unique tree of weight $2$ and arity $1$, and $\Delta_T$ vanishes otherwise.

A symmetric homotopy cooperad $\calc$ is called strictly counital if there exists a strict morphism of symmetric homotopy cooperads $\varepsilon: \calc\rightarrow \cali$ such that for each $n \geqslant 1$, the compositions
$$\calc(n)\xrightarrow{\Delta_{T_{1, i}}}\calc(n)\ot \calc(1)\xrightarrow{\id\ot \varepsilon}\calc(n)\ot \cali(1)\cong \calc(n)$$
and 
$$\calc(n)\xrightarrow{T_2}\calc(1)\ot \calc(n)\xrightarrow{\varepsilon\ot \id}\cali(1)\ot \calc(n)\cong \calc(n)$$ 
are identity maps on $\calc(n)$, where $T_{1, i}$ with $1\leqslant i\leqslant n$ and $T_2$ are the notations used before. Furthermore, for any tree $T$ of weight at least three, $(\id^{\ot i-1} \ot \varepsilon \ot \id^{\ot \omega(T)-i}) \circ \Delta_T$ is required to be 0 for all $1 \leqslant i \leqslant \omega(T)$.

A strictly counital symmetric homotopy cooperad $\calc$ is called coaugmented if there exists a strict morphism of symmetric homotopy cooperads $\eta:\cali\rightarrow \calc$ such that $\varepsilon\circ \eta=\id_{\cali}$. For a coaugmented symmetric homotopy cooperad $\calc$, the graded $ \mathbb{S} $-module $\overline{\calc}=\ker(\varepsilon)$ endowed with operations $\{\overline{\Delta}_T\}_{T\in\frakT}$ is naturally a symmetric homotopy cooperad, where $\overline{\Delta}_T$ is the the restriction of operation $\Delta_T$ on $\overline{\calc}$.

A symmetric homotopy cooperad $\cale=\{\cale(n)\}_{n\geqslant 1}$ such that $\Delta_T = 0$ for all $T\in \frakT$ with $\omega(T)\geqslant 3$ is exactly a noncounital symmetric dg cooperad in the sense of Markl \cite{Mar08}, with $\{\Delta_T \mid T \in \mathrm{S} \mathfrak{T}^{(1)}\}$ and $\{\Delta_T \mid T \in \mathrm{S} \mathfrak{T}^{(2)}\}$ defining its differentials and partial cocompositions respectively.

For a (noncounital) symmetric dg cooperad $\cale$, on can define the cocomposition $\Delta^{T}_\cale:\cale(\alpha(T)) \rightarrow \cale^{\ot T}$ along a tree $T$ in the dual way as the composition $m^T_\calp$ along $T$ for a symmetric dg operad $\calp$.

\begin{prop-def} \cite{MV09a}
	Let $\calc$ be a symmetric homotopy cooperad and $\cale$ be a symmetric dg cooperad. Then the graded collection $\calc\ot \cale$ with $(\calc\ot\cale)(n):=\calc(n)\ot \cale(n)$ for $ n\geqslant 1$ has a natural symmetric homotopy cooperad structure as follows:
	\begin{enumerate}
		\item for any $c \in \calc(m), e \in \cale(m)$ and $\sigma \in \mathbb{S}_m$, the right $\mathbb{S}_m$ action on $(\calc \ot_{\rH} \cale)(m)$ is defined as
		$$ (c \ot e)\sigma := c \sigma \ot e \sigma ;$$
		\item for any tree $T \sigma \in \mathfrak{T}^{(1)}(m)$ with trivial labeled tree $T$, $c \in \calc(m), e \in \mathcal{E}(m)$, define
		$$
		\Delta_{T \sigma}^{\mathrm{H}}(c \otimes e) :=\Delta_{T \sigma}^{\calc}(c) \otimes e +(-1)^{|c|} c \otimes \Delta^{\cale}_{T \sigma}(e) ;
		$$
		\item for any tree $T \in \mathfrak{T}^{(n)}(m) $ with $n\geqslant 2$, $ c \in \calc(m), e \in \mathcal{E}(m) $, define 
		$$\Delta_T^{\mathrm{H}}(c\ot e) := (-1)^{\sum\limits_{k=1}^{n-1}\sum\limits_{j=k+1}^n|e_k||c_j|}(c_1\ot e_1)\ot \cdots \ot (c_n\ot e_n)\in (\calc\ot \cale )^{\ot T}, $$ 
		with $c_1\ot \cdots\ot c_n = \Delta_T^\calc(c) \in \calc^{\ot T}$ and $e_1\ot \cdots \ot e_n = \Delta^T_\cale(e) \in \cale^{\ot T}$, where $\Delta^T_\cale$ is the cocomposition in $\cale$ along $T$.
	\end{enumerate}
	The new symmetric homotopy cooperad is called the Hadamard product of $\calc$ and $\cale$, and denoted by $\calc\ot_{\rH}\cale$.
\end{prop-def}

Define $\cals=\mathrm{End}_{\bfk s}^c$ to be the symmetric graded cooperad whose underlying graded collection is given by $\cals(n):=\Hom((\bfk s)^{\ot n}, \bfk s) \cong \bfk \delta_{n}$ for $ n\geqslant 1$, where $\delta_n \in \cals(n)$ is the map which takes $s^{\ot n}$ to $s$. 
The $\mathbb{S}$-module structure of $\cals$ is defined as
$$ \delta_{n} \sigma := \sgn (\sigma) \delta_{n} $$
for $ n \geqslant 1 $ and $ \sigma \in \mathbb{S}_{n}$.
The cooperad structure of $\cals$ is given by 
$$ \Delta_{T}(\delta_n) := \sgn (\sigma) (-1)^{(j-1)(i-1)}\delta_{n-i+1}\ot \delta_i\in \cals^{\ot T} $$
for the tree $T$ of the form
$$
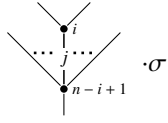
\begin{tikzpicture}[scale=0.8, descr/.style={fill=white}]
	\tikzstyle{every node}=[thick, minimum size=2pt, inner sep=1pt]
	\node (T) at (-1.8,0.5){};
	\node (t) at (0,-0.5){};
	\node[circle, fill=black, label = right : \tiny $n-i+1$] (a) at (0,0) {};
	\node (b1) at (-1,1) {};
	\node[circle, fill=black, label = right : \tiny $i$] (b2) at (0,1) {};
	\node (b3) at (1,1) {};
	\node (c1) at (-0.5,1.5) {};
	\node (c2) at (0.5,1.5) {};
	\draw (t)--(a);
	\draw (a)--(b1);
	\draw (a)--(b2);
	\draw (a)--(b3);
	\draw (b2)--(c1);
	\draw (b2)--(c2);
	\draw [dotted, line width=1pt] (-0.5, 0.6)--(0.5, 0.6);
	\path[-,font=\scriptsize] (a) edge node[descr]{$j$} (b2);
	\node (d) at (1.5,0.4) {$\cdot \sigma$};
\end{tikzpicture}
$$
with $\sigma \in \mathbb{S}_{n}$.
We also define $\cals^{-1}$ to be the symmetric graded cooperad whose underlying graded collection is given by $\cals^{-1}(n):=\Hom((\bfk s^{-1})^{\ot n}, s^{-1}) \cong \bfk \varepsilon_{n} $ for $n\geqslant 1$, where $\varepsilon_n\in \cals^{-1}(n)$ is the map which takes $(s^{-1})^{\ot n}$ to $s^{-1}$. 
The $\mathbb{S}$-module structure of $\cals^{-1}$ is defined as
$$ \varepsilon_{n} \sigma := \sgn (\sigma) \varepsilon_{n} $$
for $ n \geqslant 1 $ and $ \sigma \in \mathbb{S}_{n}$.
The cooperad structure of $\cals^{-1}$ is given by 
$$ \Delta_{T}(\varepsilon_n) :=(-1)^{(i-1)(n-i+1-j)}\varepsilon_{n-i+1}\ot \varepsilon_i\in (\cals^{-1})^{\ot T} $$ 
for the tree $T$ that is the same as before.
It is easy to see that $\cals\ot_{\mathrm{H}}\cals^{-1}\cong \cals^{-1}\ot_{\mathrm{H}}\cals=:\mathbf{Com}^\vee$. Notice that for any symmetric homotopy cooperad $\calc$, we have $ \calc\ot_\rH \mathbf{Com}^\vee\cong\calc\cong \mathbf{Com}^\vee\ot_\rH \calc.$

\begin{defn}
	Let $\calc$ be a symmetric homotopy cooperad. Define the operadic suspension (resp. desuspension) of $\calc$ to be the symmetric homotopy cooperad $\calc\ot_{\mathrm{H}} \cals$ (resp. $\calc\ot_{\mathrm{H}}\cals^{-1}$), denoted as $\mathscr{S}\calc$ (resp. $\mathscr{S}^{-1}\calc$).
\end{defn}

\begin{defn} \label{defn: cobar construction symmetric}
	Let $\calc=\{\calc(n)\}_{n\geqslant 1}$ be a coaugmented symmetric homotopy cooperad. The cobar construction of $\calc$, denoted by $\Omega\calc$, is the free symmetric graded operad generated by the graded $\mathbb{S}$-module $s^{-1}\overline{\calc}$, endowed with the differential $\partial$ which is lifted from $\partial: s^{-1}\overline{\calc} \rightarrow \Omega\calc$ given by 
	$$\partial(s^{-1}f)=-\sum_{T\in \mathrm{S}\mathfrak{T}(n)} (s^{-1})^{\ot \omega(T)} \circ \overline\Delta_T(f)$$
	for any $f\in \overline{\calc}(n)$.
	More precisely, let $s^{-1} f_1 \otimes \cdots \otimes s^{-1} f_n \in(s^{-1} \overline{\mathcal{C}})^{\otimes T_{1}}$ with $T_{1} \in \mathrm{S} \mathfrak{T}^{(n)}$
	\begin{align*}
		\partial(s^{-1} f_1 \otimes \cdots \otimes s^{-1} f_n) = -\sum_{i=1}^n \sum_{T_2 \in \mathrm{S} \mathfrak{T}{(\alpha(f_i))}} r_{\sigma(T_1 \times_i T_2, T_2)^{-1}}\big(\mathrm{Id}^{\otimes i-1} \otimes((s^{-1})^{\otimes \omega(T_2)} \circ \overline{\Delta}_{T_2} \circ s) \otimes \mathrm{Id}^{\otimes n-i}\big)(s^{-1} f_1 \otimes \cdots \otimes s^{-1} f_n), 
	\end{align*}
	where $T_1 \times_i T_2$ is obtained from $T_1$ by rearranging the order of the leaves of $i$-th vertex $v_i$ according to the labeling of $T_2$ and then replacing $v_i$ with the underlying planar tree of $T_2$.
\end{defn}

This provides an alternative definition for symmetric homotopy cooperads. In fact, a graded $ \mathbb{S} $-module $\overline{\calc}=\{\overline{\calc}(n)\}_{n\geqslant 1}$ is a symmetric homotopy cooperad if and only if the free symmetric graded operad generated by $s^{-1}\overline{\calc}$ (also called the cobar construction of $\calc=\overline{C}\oplus \cali$) is endowed with a differential such that it becomes a symmetric dg operad.

There is a natural $L_\infty$-algebra associated with a symmetric homotopy operad $\calp=\{\calp(n)\}_{n\geqslant 1}$.
Denote $\calp^{\prod}:= \prod\limits_{n=1}^\infty\calp(n)$. For each $n\geqslant 1$, define operations $m_n :=\sum\limits_{T\in \mathrm{S}\mathfrak{T}^{(n)}}m_T: (\calp^{\prod})^{\ot n} \rightarrow \calp^{\prod} $ and $l_n$ to be the anti-symmetrization of $m_n$, i.e., 
$$l_n(x_1\ot \cdots \ot x_n) := \sum\limits_{\sigma\in \mathbb{S}_n}\chi(\sigma; x_1, \dots, x_n) \ m_n(x_{\sigma(1)}\ot \cdots \ot x_{\sigma(n)}).$$

Then we have the following result.

\begin{prop} \cite{MV09a, VdLa}
	Let $\calp$ be a symmetric homotopy operad. Then $(\calp^{\prod}, \{l_n\}_{n\geqslant 1})$ is an $L_\infty$-algebra.
\end{prop}

Denote $\calp(n)^{\mathbb{S}_n} \subset \calp(n)$ be the invariant subspace under the right $\mathbb{S}_n$ action. Define $\calp_{\mathbb{S}}^{\prod}:=$ $\prod\limits_{n=1}^{\infty} \calp(n)^{\mathbb{S}_n} \subset \calp ^{\prod}$. By restricting $\{l_n\}_{n \geqslant 1}$ of $\calp ^{\prod}$ to $\calp_{\mathbb{S}} ^{\prod}$, we have

\begin{prop} \cite{MV09a, VdLa} \label{prop: S homotopy operad L infty}
	Let $\calp$ be a symmetric homotopy operad. Then $(\calp_{\mathbb{S}} ^{\prod}, \{l_n\}_{n \geqslant 1})$ is an $L_{\infty}$-algebra. In particular, when $\calp$ is a symmetric dg operad, $\calp_{\mathbb{S}}^{\prod}$ is just a dg Lie algebra.
\end{prop}

\begin{defn} \label{Def: brace operation of S operad}
	Let $\calp$ be a (nonunital) symmetric dg operad. For any $f\in \calp(m)$ and $g_1\in\calp(k_1), \dots, g_n\in \calp(k_n) $ with $1\leqslant n\leqslant m$, define 
	$$ f\{g_1, \dots, g_n\}:=\sum_{T} m^T_\calp (f \ot g_1 \ot \cdots \ot g_n), $$
	where the sum is taken over $T \in \mathrm{S}\mathfrak{T}^{(n+1)}$ of height $2$, i.e., whose underlying planar trees form
	$$ 
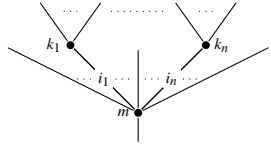
\begin{tikzpicture}[scale=0.9, descr/.style={fill=white}]
	\tikzstyle{every node}=[thick, minimum size=3pt, inner sep=1pt]
	\node(r) at (0, -0.5)[minimum size=0pt, circle]{};
	\node(v0) at (0, 0)[fill=black, circle, label=left: \tiny $m$]{};
	\node(v1-1) at (-2, 1){};
	\node(v1-2) at (-1, 1)[fill=black, circle, label=left: \tiny $k_{1}$]{};
	\node(v1-3) at(0, 1){};
	\node(v1-4) at (1, 1)[fill=black, circle, label=right: \tiny $k_{n}$]{};
	\node(v1-5) at(2, 1){};
	\node(v1-6) at(2.5, 0.7){};
	\node(v2-1)at (-1.5, 1.7){};
	\node(v2-2)at (-0.5, 1.7){};
	\node(v2-3)at (0.5, 1.7){};
	\node(v2-4)at (1.5, 1.7){};
	\draw(r)--(v0);
	\draw(v0)--(v1-1);
	\draw(v0)--(v1-2);
	\draw(v0)--(v1-3);
	\draw(v0)--(v1-4);
	\draw(v0)--(v1-5);
	\draw(v1-2)--(v2-1);
	\draw(v1-2)--(v2-2);
	\draw(v1-4)--(v2-3);
	\draw(v1-4)--(v2-4);
	\draw[dotted](-1.1, 1.5)--(-0.9, 1.5);
	\draw[dotted](0.9, 1.5)--(1.1, 1.5);
	\draw[dotted](-0.4, 1.5)--(0.4, 1.5);
	\draw[dotted](-0.9, 0.5)--(-0.1, 0.5);
	\draw[dotted](0.1, 0.5)--(0.9, 0.5);
	\path[-, font=\scriptsize] (v0) edge node[descr]{{\tiny$i_{1}$}} (v1-2);
	\path[-, font=\scriptsize] (v0) edge node[descr]{{\tiny$i_{n}$}} (v1-4);
	\end{tikzpicture} $$
	with $1 \leqslant i_1 < \dots < i_n \leqslant m$.
	It is called the shuffle brace operation on $\calp^{\prod}$. For $f\in \calp(m), g\in \calp(n)$, define a bracket
	$$ [f, g] := f\{g\}-(-1)^{|f||g|}g\{f\}\in \calp(m+n-1). $$
\end{defn}

When $\calp$ is a symmetric dg operad, the operation $l_2$ on the dg Lie algebra $\calp^{\prod}_{\mathbb{S}}$ is exactly the bracket defined above.

\begin{remark} \label{Rem. relation between two braces}
	Let $\calp$ be a symmetric homotopy operad. One may not expect the shuffle brace operation to satisfy the pre-Jacobi identity as the nonsymmetric case, see~\cite[Equation (8)]{SWZZ2024a}, \cite[Proposition 1.10]{Wang2024} or \cite[Proposition 3.7]{CGWZ24}. 
	Denote
	\begin{equation} \label{Eq. symmetric brace definition}
		f [g_1, \dots, g_n ] := \sum_{\sigma \in \mathbb{S}_{n}} \epsilon(\sigma; g_{1}, \dots, g_{n}) f\{g_{\sigma(1)}, \cdots, g_{\sigma(n)}\}
	\end{equation}
	for any homogeneous elements in $\calp_{\mathbb{S}}^{\prod}$ with $f\in \calp(m)$ and $g_1\in\calp(k_1), \dots, g_n\in \calp(k_n) $ with $1\leqslant n\leqslant m$, this define a symmetric brace operation $-[-]$ in sense of Lada and Markl \cite{LM05}. The symmetric brace operation satisfies the following identity:
	\begin{equation}\label{Eq. symmetric brace pre-jacobi}
		\begin{aligned} 
			& \ \Big(f [g_1, \dots, g_m ]\Big) [ h_1, \dots, h_n ]  = \\
			  & \sum\limits_{\sigma \in \Sh(t_{1},\dots,t_{m}, t_{m+1})} \epsilon \
			f \big[ g_{1} [ h_{\sigma(1)}, \dots, h_{\sigma(t_{1})} ], \dots, g_{m} [ h_{\sigma(t_{1}+\dots+t_{m-1}+1)}, \dots, h_{\sigma(t_{1}+\dots+t_{m})} ], h_{\sigma(t_{1}+\dots+t_{m}+1)}, \dots, h_{\sigma(n)} \big],
		\end{aligned}
	\end{equation}
	where $t_{m+1} = n-t_{1}-\dots-t_{m}$, $\epsilon$ is the corresponding Koszul sign. In particular, we have $f\{g\} = f [g]$ and the following identity
	\begin{align} \label{Eq: pre-jacobi1 S}
		(f \{g\})\{h\} = f\{g\{h\}\}+f\{g, h\}+(-1)^{|g||h|}f\{h, g\}.
	\end{align}
\end{remark}

Next, we introduce the notion of convolution symmetric homotopy operads.

\begin{prop-def} \cite{MV09a, VdLa} \label{prop-def. convolution symmetric homotopy operad}
	Let $\calc$ be a symmetric homotopy cooperad and $\calp$ be a symmetric dg operad. Then the graded collection $\mathbf{Hom}(\calc, \calp)$ with $\mathbf{Hom}(\calc, \calp)(n):= \Hom(\calc(n), \calp(n))$ for $n\geqslant 1$ has a natural symmetric homotopy operad structure, called convolution symmetric homotopy operad, which is defined as follows:
	\begin{enumerate}
		\item for any $ f\in\mathbf{Hom}(\calc, \calp)(n) $, $ \sigma \in \mathbb{S}_{n} $ and $ x \in \calc (n) $, the right $\mathbb{S}_n$ action on $\mathbf{Hom}(\calc, \calp)(n)$ is defined as
		$$ (f \sigma)(x) := f(x\sigma^{-1}) \sigma; $$		
		
		\item for any $T \sigma \in \mathfrak{T}^{(1)}(n)$ with trivial labeled tree $T$, $f \in \mathbf{Hom}(\calc, \calp)(n)$ and $ c \in \calc(n)$, 
		$$
		\big(m_{T \sigma}(f)\big)(c) := m^{\calp}_{T \sigma} \big(f(c \sigma^{-1})\big)-(-1)^{|f|}(f\sigma)\big( \Delta_{T \sigma}^{\calc}(c)\big);
		$$
				
		\item for any $T \in \mathfrak{T}^{(n)}$ with $n \geqslant 2$ and $ f_1 \otimes \cdots \otimes f_n \in \mathbf{Hom}(\calc, \calp)^{\otimes T}$
		$$ m_T(f_1\ot \cdots \ot f_n) := (-1)^{\frac{n(n-1)}{2}+1+n(\sum\limits_{i=1}^n |f_i|)}m_{\calp}^T\circ(f_1\ot \cdots\ot f_n)\circ \Delta_T^{\calc}, $$ 
		where $m_\calp^T$ is the composition in $\calp$ along $T$.
	\end{enumerate}
\end{prop-def}


\begin{prop} \cite{MV09a, VdLa} \label{Prop: Linfinity give MC S} 
	Let $\calc$ be a coaugmented symmetric homotopy cooperad and $\calp$ be an unital symmetric dg operad. Then there is a natural bijection:
	$$\Hom_{udgOp}(\Omega\calc, \calp)\cong \calm\calc\big(\mathbf{Hom}(\overline{\calc}, \calp)_\mathbb{S}^{\prod}\big), $$
	where the left-handed side is the set of morphisms of unital symmetric dg operads from $\Omega C$ to $\calp$ and the right-handed side is the set of Maurer-Cartan elements in the $L_\infty$-algebra $\mathbf{Hom}(\overline{\calc}, \calp)_{\mathbb{S}}^{\prod}$.
\end{prop}

At last, we recall the notions of minimal models and Koszul dual symmetric homotopy cooperads of symmetric operads and explain how they are related to deformation complexes and $L_\infty$-algebra structures on deformation complexes.

For an $\mathbb{S}$-module $M=\{M(n)\}_{n\geqslant 1} $, denote by $ \mathcal{F}(M)$ the free symmetric graded operad generated by $M$. Recall that a symmetric dg operad is called {\bf quasi-free} if its underlying symmetric graded operad is free.

\begin{defn}\cite{DCV13} \label{symmetric minimal model} 
	A minimal model for a symmetric dg operad $\calp$ is a quasi-free symmetric dg operad $ (\mathcal{F}(M), d)$ together with a surjective quasi-isomorphism of symmetric dg operads $(\mathcal{F}(M), \partial)\overset{\sim}{\twoheadrightarrow}\calp$, where the symmetric dg operad $(\mathcal{F}(M), \partial)$ satisfies the following conditions:
	\begin{enumerate}
		\item the differential $\partial$ is decomposable, i.e., $\partial$ takes $M$ to $\mathcal{F}(M)^{(\geqslant 2)}$, the subspace of $\mathcal{F}(M)$ consisting of elements with weight $\geqslant 2$; \label{it:min1}
		\item the generating collection $M$ admits a decomposition $M=\bigoplus\limits_{i\geqslant 1}M_{(i)}$ such that $\partial(M_{(k+1)})\subset \mathcal{F}\big(\bigoplus\limits_{i=1}^kM_{(i)}\big)$ for all $k\geqslant 1$. \label{it:min2}
	\end{enumerate}
\end{defn}

\begin{thm} \cite{DCV13}
	When a symmetric dg operad $\calp$ admits a minimal model, it is unique up to isomorphisms.
\end{thm}

For a symmetric operad $\calp$, assume that its minimal model $\calp_\infty$ exists. Since $\calp_\infty$ is quasi-free, 
there exists a coaugmented symmetric homotopy cooperad $\calc$ such that $\Omega\calc\cong \calp_\infty$. So $\calc$ is called the Koszul dual symmetric homotopy cooperad of $\calp$, denoted by $\calp^\ac$.

Let $V$ be a complex. Denote by $\End_V$ the symmetric endomorphism dg operad for $V$. Then the underlying complex of $\mathbf{Hom}(\overline{\calp^\ac}, \End_V)_{\mathbb{S}}^{\prod}$ 
is called the deformation complex of $\calp$ on the complex $V$.
An element of $\Hom_{udgOp}( \calp_\infty, \End_V)$ is exactly a symmetric homotopy $\calp$-structure on $V$. So Proposition~\ref{Prop: Linfinity give MC S} gives a bijection between the set of symmetric homotopy $\calp$-structures on $V$ and that of Maurer-Cartan elements in the $L_\infty$-algebra on the deformation complex.

\medskip

\section{Cohomology and homotopy theory of Lie algebras} \label{Section: Cohomology and homotopy theory of Lie algebras}

In this section, we recall some basics of the cohomology of Lie algebras and Richardson-Nijenhuis brackets. We also define $L_\infty$-algebras via Maurer-Cartan elements.



Let $M$ be a representation over a Lie algebra $ (\frakg, \mu=[-, -]_{\mu}) $. Recall that the Chevalley-Eilenberg cochain complex of Lie algebra $(\frakg, \mu)$ with coefficients in representation $M$ is the cochain complex
$$ \C_{\Lie}^{\bullet}(\frakg, M) := \bigoplus_{n = 0}^{\infty} \C_{\Lie}^{n}(\frakg, M), $$
where, for $n \geqslant 0$, $\C_{\Lie}^{n}(\frakg, M) = \Hom(\frakg^{\wg n}, M)$ and its differential 
$$ \delta_{\Lie, M}: \C_{\Lie}^{n}(\frakg, M) \rightarrow \C_{\Lie}^{n+1}(\frakg, M)$$
is given by
\begin{align*} 
	\delta_{\Lie, M} (f) (a_{1}, \dots, a_{n+1}) 
	=& \sum_{i=1}^{n+1}(-1)^{i-1}a_{i} f(a_{1}, \dots, \widehat{a_{i}}, \dots, a_{n+1})\\
	&+ \sum_{1 \leqslant i<j\leqslant n+1}(-1)^{i+j}f([a_{i}, a_{j}]_{\mu}, a_{1}, \dots, \widehat{a_{i}}, \dots, \widehat{a_{j}}, \dots, a_{n+1}) 
\end{align*}
for any $f \in \C_{\Lie}^{n}(\frakg, M)$ and $a_{1}, \dots, a_{n} \in \frakg$, where $\widehat{a_i}$ means deleting element $a_i$. 

The cohomology of the cochain complex, denoted by $ \rmH^{\bullet}_{\Lie}(\frakg, M) $, is called the Chevalley-Eilenberg cohomology of Lie algebra $(\frakg, \mu)$ with coefficients in representation $M$. 
When $ M $ is the adjoint representation $ \frakg $ itself, we simply write the differential $\delta_{\Lie, \frakg}$ as $\delta_{\Lie}$, and also simply denote $\C^\bullet_\Lie(\frakg, \frakg)$ by $\C^\bullet_\Lie(\frakg) $ and $\rmH^\bullet_\Lie(\frakg, \frakg)$ by $\rmH^\bullet_\Lie(\frakg)$, called the Chevalley-Eilenberg cochain complex of Lie algebra $(\frakg, \mu)$ and the Chevalley-Eilenberg cohomology of Lie algebra $(\frakg, \mu)$, respectively.

Let $V$ be a graded space. The cofree conilpotent symmetric coalgebra generated by $V$ is defined as
$$S^c(sV) :=\bfk\oplus sV\oplus (sV)^{\odot 2}\oplus \cdots \oplus (sV)^{\odot n} \oplus \cdots $$
with the following shuffle coproduct:
$$\Delta(sv_1\odot \cdots \odot sv_n) := \sum_{i=0}^{n}\sum_{\sigma\in \rm{Sh}(i, n-i)}\epsilon(\sigma; sv_1, \dots, sv_{n})sv_{\sigma(1)}\odot \cdots\odot sv_{\sigma(i)}\otimes sv_{\sigma(i+1)}\odot \cdots \odot sv_{\sigma(n)}$$ 
for homogeneous elements $v_{1}, \dots, v_{n} \in V$.
We identify $\wedge^\bullet V$ with $S^{c}(sV)$.
The reduced cofree conilpotent symmetric coalgebra is defined as
$$ \overline{S^c}(sV) := sV\oplus (sV)^{\odot 2}\oplus \cdots \oplus (sV)^{\odot n} \oplus \cdots $$ 
with the following shuffle coproduct:
$$\Delta(sv_1\odot\cdots \odot sv_n) := \sum_{i=1}^{n-1}\sum_{\sigma\in \rm{Sh}(i, n-i)}\epsilon(\sigma; sv_1, \dots, sv_{n})sv_{\sigma(1)}\odot\cdots\odot sv_{\sigma(i)}\otimes sv_{\sigma(i+1)}\odot \cdots \odot sv_{\sigma(n)}$$ 
for homogeneous elements $v_{1}, \dots, v_{n} \in V$.

Consider the symmetric endomorphism graded operad $\End_{sV}$.
By Definition~\ref{Def: brace operation of S operad}, $(\End_{sV}^{\prod})_{\mathbb{S}}$ can be endowed with the shuffle brace operation and a Lie bracket, which in this case is called the Richardson-Nijenhuis bracket.
Let's write down the explicit formulas of the shuffle brace operation on $\mathfrak{C}_{\Lie}(V) := \Hom(\overline{S^c} (sV), sV)$. We always consider $\frak{C}_\Lie(V)$ as a subspace of $\Hom(\overline{T^c}(sV), sV)$ by identifying $\Hom((sV)^{\odot n}, sV)$ with the $\mathbb{S}_n$-invariant subspace of $\Hom((sV)^{\ot n}, sV)$. 
For $m \geqslant 1$ and homogeneous elements $sf\in \Hom((sV)^{\odot m}, sV)$, $sg_i\in \Hom((sV)^{\odot k_i}, sV)$ with $ 1\leqslant i\leqslant n \leqslant m$, $k_{i} \geqslant 1$, the shuffle brace operation 
$$sf \{sg_1, \dots, sg_n\}\in \Hom((sV)^{\odot (m+k_1+\cdots+k_n-n)}, sV)$$ 
is defined as
\begin{align*} \label{Eq. shuffle brace}
	sf \{sg_1, \dots, sg_n\}
	:=
	\sum\limits_{
		\substack{ i_{1}+ i_{2} +\cdots + i_{n+1} + n = m
			\\ i_{p} \geqslant 0, 1 \leqslant p \leqslant n+1}
	}
	\sum_{\sigma \in \sh} 
	\big(
		sf \circ (
	\id^{\ot i_{1}} \ot sg_1 \ot
	\id^{\ot i_{2}} \ot sg_2 \ot
	\cdots \ot
	\id^{\ot i_{n}} \ot sg_n \ot
	\id^{i_{n+1}})
	\big)
	\ \sigma^{-1}, 
\end{align*}
where the abbreviated notation
$$\sh = \sh(\underbrace{1, \dots, 1}_{i_{1}}, k_{1}, \underbrace{1, \dots, 1}_{i_{2}}, k_{2}, \dots, \underbrace{1, \dots, 1}_{i_{n}}, k_{n}, \underbrace{1, \dots, 1}_{i_{n+1}})$$ 
was introduced in Section~\ref{Notations}.
Given homogeneous elements $sf\in \Hom((sV)^{\odot m}, sV)$ and $sg\in \Hom((sV)^{\odot n}, sV)$, $m, n \geqslant 1$, their Lie bracket is defined as
\begin{eqnarray} \label{Eq: Richardson-Nijenhuis}
	[sf, sg]_{\RNA} := sf\{sg\}-(-1)^{(|f|+1)(|g|+1)}sg\{sf\}, 
\end{eqnarray}
called the Richardson-Nijenhuis bracket of $sf$ and $sg$. 
In fact, the shuffle brace operation and the Richardson-Nijenhuis bracket can be extended to the graded space $\Hom({S^c}(sV), sV) $ in the natural way.

At this point, the following results can be obtained directly:
\begin{thm} \cite{NR67} 
	For a graded space $V$, the Richardson-Nijenhuis bracket endows with the graded space $\mathfrak{C}_{\Lie}(V)$ into a graded Lie algebra.
\end{thm}

\begin{prop} \cite{NR67} \label{Prop. relation between differential and MC Lie} 
	Let $(\frakg, \mu)$ be a Lie algebra. For all $n \geqslant 0$ and $f \in \C_{\Lie}^n(\frakg)$, we have $ \delta_{\Lie}(f) = (-1)^{n-1}[\mu, f]_{\RNA} = -[f, \mu]_{\RNA}$. 
\end{prop}



One can give an equivalent definition of $L_{\infty}$-algebras as follows:

\begin{defn} \label{Def: L infinity algebras 2}
	An $L_\infty$-algebra structure on graded space $V$ is defined to be a Maurer-Cartan element in graded Lie algebra $ {\mathfrak{C}_{\Lie}}(V)$.
\end{defn}

By above definition, an $L_\infty$-algebra structure on $V$ is given by a family of operators $b_n:(sV)^{\odot n}\rightarrow sV, n\geqslant1$ with $|b_n|=-1$ satisfying: for any $n \geqslant 1$, 
\begin{equation}\label{Eq: L infinity algebras sV}
	\sum_{j=1}^{n}b_{n-j+1}\{ b_{j}\}=0.
\end{equation}
This is also called an $L_\infty[1]$-algebra structure on $sV$.
If one define operators $l_n := s^{-1}\circ b_n\circ s^{\ot n}: V^{\wg n}\rightarrow V$ for $n \geqslant 1$, then one can find that Equation~\eqref{Eq: L infinity algebras sV} is equivalent to the generalized Jacobi identity in Definition~\ref{Def: L-infty}.

\bigskip

\part{Nijenhuis Lie algebras}\label{Part2}

\medskip

\section{Nijenhuis Lie algebras and Nijenhuis representations} \label{Section: Nijenhuis Lie algebras and their representations}

In this section, we recall some basic concepts and facts about Nijenhuis Lie algebras and their Nijenhuis representations.

\begin{defn} \label{Def:Nijenhuis Lie algebras}
	Let $(\frakg, \mu=[-, -]_{\mu})$ be a Lie algebra. A linear operator $P: \frakg \rightarrow \frakg$ is said to be a Nijenhuis operator if it satisfies
	\begin{equation*}
		[P(a), P(b)]_{\mu} = P  \big( [P(a), b]_{\mu} + [a, P(b)]_{\mu} - P([a, b]_{\mu})  \big)
	\end{equation*}	
	for any $a, b \in \frakg$, or in terms of maps
	\begin{equation} \label{Eq: Nijenhuis relation in terms of maps}
		\mu \circ (P \ot P) = P \circ  \big(\mu \circ( P \ot \Id ) + \mu \circ ( \Id \ot P ) - P \circ \mu  \big).
	\end{equation}
	In this case, $\frakg = (\frakg, \mu, P)$ is called a Nijenhuis Lie algebra. 
\end{defn}

\begin{defn}
	Let $\frakg_{1} = (\frakg_{1}, \mu_{1}, P_{1})$ and $\frakg_{2} = (\frakg_{2}, \mu_{2}, P_{2})$ be two Nijenhuis Lie algebras. A Lie algebra morphism $f: \frakg_{1} \rightarrow \frakg_{2}$ is called a Nijenhuis Lie algebra morphism, or morphism of Nijenhuis Lie algebras, if $ P_{2} \circ f = f \circ P_{1}$.
\end{defn}

\begin{defn} \label{Def: Nijenhuis Lie representation}
	Let $ \frakg = (\frakg, \mu, P) $ be a Nijenhuis Lie algebra and $M$ be a representation over the Lie algebra $ (\frakg, \mu) $. Then $M$ is said to be a Nijenhuis representation over Nijenhuis Lie algebra $ \frakg $, 
	if $ M $ is endowed with a linear operator $ P_{M}: M \rightarrow M $ such that the following equation
	\begin{eqnarray*} \label{Eq: Nijenhuis Lie representation}
		P(a) P_{M}(x) = P_{M} \big( P(a) x + a P_{M}(x) - P_{M}(a x) \big)
	\end{eqnarray*}
	holds for any $ a \in \frakg $ and $ x \in M $. 
\end{defn}

Obviously, a Nijenhuis Lie algebra $ \frakg = (\frakg, \mu, P) $ is a Nijenhuis representation over itself, called the adjoint Nijenhuis representation.

Let $(\frakg, \mu)$ be a Lie algebra and $M$ be a representation over it. It is well known that $ \frakg \oplus M $ becomes a Lie algebra whose Lie bracket $ \overline{\mu} = [-, -]_{\overline{\mu}} $ is given by
\begin{equation*}
	[(a, x), (b, y)]_{\overline{\mu}} := ( [a, b]_{\mu}, ay - bx ),
\end{equation*}
for any $a, b \in \frakg$ and $x, y \in M$. This is called the semi-direct product (or trivial extension) of Lie algebra $(\frakg, \mu)$ by representation $M$, denoted by $\frakg \ltimes M$.
Write $ \iota : \frakg \rightarrow \frakg \oplus M $, $ a \mapsto (a, 0) $ and $ \pi : \frakg \oplus M \rightarrow \frakg $, $ (a, x)\mapsto a $, which are obviously Lie algebra morphisms.

The following results are direct consequences of the definitions, whose proofs are left to the reader. 

\begin{prop} \label{Prop: Nijenhuis Lie algebra trivial extension}
	Let $ \frakg = (\frakg, \mu, P) $ be a Nijenhuis Lie algebra and $M$ be a representation over Lie algebra $ (\frakg, \mu) $. Then the semi-direct product $ \frakg \ltimes M $ of $\frakg$ by $M$ is a Nijenhuis Lie algebra such that $ \iota $ and $ \pi $ are both morphisms of Nijenhuis Lie algebras if and only if $ M $ is a Nijenhuis representation over $ \frakg $.

	This new Nijenhuis Lie algebra will still be denoted by $ \frakg \ltimes M $, called the (Nijenhuis) semi-direct product (or trivial extension) of Nijenhuis Lie algebra $ \frakg $ by Nijenhuis representation $ M $.
\end{prop}

Nijenhuis Lie algebras and Nijenhuis representations have some descendant properties.

\begin{prop} \cite{Kos90} 
	\label{Prop: Star Lie bracket}
	Let $ \frakg = (\frakg, \mu, P) $ be a Nijenhuis Lie algebra. Define a new binary operation $\mu_{P} = [-,-]_{P}$ over $\frakg$ as:
	\begin{equation*}
		[a, b]_{P} := [P(a), b]_{\mu} + [a, P(b)]_{\mu} - P([a, b]_{\mu}),
	\end{equation*}
	for any $a, b\in \frakg$. Then
	\begin{enumerate}		
		\item the pair $(\frakg, \mu_{P})$ is a new Lie algebra;
		\item the triple $(\frakg, \mu_{P}, P)$ also forms a Nijenhuis Lie algebra and denote it by $\frakg_{P}$;
		\item the map $P: \frakg_{P} = (\frakg, \mu_{P}, P) \rightarrow \frakg = (\frakg, \mu, P)$ is a morphism of Nijenhuis Lie algebras.
	\end{enumerate}
\end{prop}

One can also obtain new Nijenhuis representations from old ones.

\begin{prop} \label{Prop: new Lie representation}
	Let $\frakg = (\frakg, \mu, P)$ be a Nijenhuis Lie algebra and $M = (M, P_{M})$ be a Nijenhuis representation over it. We define a left action $``\rhd"$ of $\frakg$ on $M$ as follows:
	\begin{eqnarray*}
		a \rhd x:= P(a) x
	\end{eqnarray*}
	for any $a\in \frakg$ and $x \in M$. Then this action makes $M$ into a Nijenhuis representation over $ \frakg_{P}$, and we denote this new representation by $M_{P}$.
\end{prop}

Let’s describe the symmetric operad for Nijenhuis Lie algebras.
\begin{defn} \label{def: operad NjL}
	The symmetric operad for Nijenhuis Lie algebras, denoted by $\NL$, is generated by a unary operator $P$ and a binary operator $\mu$, which satisfies $\mu (1 \ 2) = - \mu$ for the unique nontrivial permutation $(1 \ 2) \in \bbS_{2}$, with the operadic relations:
	\begin{align} \label{Eq: operadic Lie}
		\sum_{\sigma \in \sh(2, 1)} \sgn (\sigma) \ (\mu \circ_{1} \mu)  \sigma^{-1} - \mu \circ_{2} \mu
	\end{align}
	and
	\begin{align} \label{Eq: operadic Nijenhuis relation}
		(\mu\circ_1 P)\circ_2 P-(P \circ_1\mu)\circ_1 P-(P \circ_1\mu)\circ_2 P + P \circ_1 (P \circ_1 \mu).
	\end{align}
\end{defn}

In other words, $\NL=\mathcal{F}(\bfk \mu \oplus \bfk P)/I$, the quotient operad of the free symmetric operad generated by the $\bbS$-module $\bfk \mu \oplus \bfk P$, by the operadic ideal $I$ which is generated by Relations~\eqref{Eq: operadic Lie} and \eqref{Eq: operadic Nijenhuis relation}.
Here, Relation~\eqref{Eq: operadic Lie} is the Jacobi identity axiom for the bracket $\mu$ and Relation~\eqref{Eq: operadic Nijenhuis relation} corresponds to the defining relation \eqref{Eq: Nijenhuis relation in terms of maps} of the Nijenhuis operator $P$.

\begin{remark}
	Since the defining relation \eqref{Eq: operadic Nijenhuis relation} of the Nijenhuis operator is cubic, it is obvious that the operad $\NL$ is not Koszul.
	So the classical Koszul duality theory of operads \cite{GJ94, GK94} could not be applied directly to develop the operadic cohomology theory, minimal model, controlling $L_\infty$-structures, etc., for Nijenhuis Lie algebras.
\end{remark}

\medskip

\section{Cohomology theory of Nijenhuis Lie algebras} \label{Sect: Cohomology theory of Nijenhuis Lie algebras}

In this section, we will introduce a cohomology theory of Nijenhuis Lie algebras. We will see later that this cohomology theory controls the deformations of Nijenhuis Lie algebras.

\smallskip

\subsection{Cohomology of Nijenhuis operators} \label{Subsect: cohomology NjL operator} \

Firstly, let's introduce the cohomology theory of Nijenhuis operators.

Let $\frakg = (\frakg, \mu, P)$ be a Nijenhuis Lie algebra and $M = (M, P_M)$ be a Nijenhuis representation over it. 
According to Propositions~\ref{Prop: Star Lie bracket} and \ref{Prop: new Lie representation}, one can define a new Nijenhuis Lie algebra $\frakg_{P} = (\frakg, \mu_{P}, P) $ and a new Nijenhuis representation $M_{P} = (M, \rhd, P_{M})$ over $\frakg_{P} $. 
Consider the Chevalley-Eilenberg cochain complex of Lie algebra $(\frakg_{P}, \mu_{P})$ with coefficients in representation $M_{P}$: 
$$ \C_{\Lie}^{\bullet}(\frakg_{P}, M_{P}) = \bigoplus_{n = 0}^{\infty} \C_{\Lie}^{n}(\frakg_{P}, M_{P}), $$
where, for $n \geqslant 0$, $\C^n_{\mathrm{Lie}}(\frakg_{P}, M_{P})=\Hom (\frakg^{\wedge n}, M)$ and its differential 
$$\partial: \C^n_{\mathrm{Lie}}(\frakg_{P}, M_{P})\rightarrow \C^{n+1}_{\mathrm{Lie}}(\frakg_{P}, M_{P}) $$ 
is given by
\begin{align*} \label{Eq: Nijenhuis operator cohomology}
	\partial(f) (a_{1}, \dots, a_{n+1})
	=&\ \sum_{i=1}^{n+1}(-1)^{i-1}a_{i} \rhd f(a_{1}, \dots, \widehat{a_{i}}, \dots, a_{n+1})\\
	&\ + \sum_{1\leqslant i<j\leqslant n+1}(-1)^{i+j}f([a_{i}, a_{j}]_{P}, a_{1}, \dots, \widehat{a_{i}}, \dots, \widehat{a_{j}}, \dots, a_{n+1})\\
	=&\ \sum_{i=1}^{n+1}(-1)^{i-1} P(a_{i}) f(a_{1}, \dots, \widehat{a_{i}}, \dots, a_{n+1})\\
	&\ + \sum_{1\leqslant i<j\leqslant n+1}(-1)^{i+j}f([P(a_{i}), a_{j}]_{\mu}, a_{1}, \dots, \widehat{a_{i}}, \dots, \widehat{a_{j}}, \dots, a_{n+1})\\
	&\ + \sum_{1\leqslant i<j\leqslant n+1}(-1)^{i+j}f([a_{i}, P(a_{j})]_{\mu}, a_{1}, \dots, \widehat{a_{i}}, \dots, \widehat{a_{j}}, \dots, a_{n+1})\\
	&\ - \sum_{1\leqslant i<j\leqslant n+1}(-1)^{i+j}f(P([a_{i}, a_{j}]_{\mu}), a_{1}, \dots, \widehat{a_{i}}, \dots, \widehat{a_{j}}, \dots, a_{n+1})
\end{align*}
for any $f\in \C^n_{\Lie}(\frakg_{P}, M_{P})$ and $a_1, \dots, a_{n+1}\in \frakg$.

\smallskip

\begin{defn} \label{defn: Nijenhuis operator complex Lie}
	Let $\frakg = (\frakg, \mu, P)$ be a Nijenhuis Lie algebra and $M = (M, P_M)$ be a Nijenhuis representation over it. Define the cochain complex of Nijenhuis operator $P$ with coefficients in Nijenhuis representation $M$, denoted by $(\C_{\NjO}^\bullet(\frakg, M), \delta_{\NjO, M})$, 
	as follows: 
	for any $n\geqslant 0 $, 
	$$\C_{\NjO}^n(\frakg, M) :=\Hom(\frakg^{\wedge n}, M)$$ 
	and its differential 
	$$\delta_{\NjO, M}: \C^n_{\NjO}(\frakg, M )\rightarrow \C^{n+1}_{\NjO}(\frakg, M) $$ 
	is defined to be
	$$\delta_{\NjO, M}(f) := - P_M\circ \delta_{\Lie, M}(f)+\partial(f)$$
	for any $f\in \C_{\NjO}^n(\frakg, M)$. 
	
	The cohomology of $\C_{\NjO}^\bullet(\frakg, M)$, denoted by $\mathrm{H}_{\NjO}^\bullet(\frakg, M)$, is called the cohomology of Nijenhuis operator $P$ with coefficients in Nijenhuis representation $M$.	
	When $M$ is the adjoint Nijenhuis representation $\frakg$ itself, we simply denote the differential $\delta_{\NjO, \frakg}$ by $\delta_{\NjO}$. We also simply denote $\C^\bullet_{\NjO}(\frakg, \frakg)$ by $\C^\bullet_{\NjO}(\frakg)$ and $\rmH^\bullet_{\NjO}(\frakg, \frakg)$ by $\rmH^\bullet_{\NjO}(\frakg)$, called the cochain complex of Nijenhuis operator $P$ and the cohomology of Nijenhuis operator $P$, respectively.
\end{defn}

The above definition is justified by the following lemma, which can be verified through direct inspection. However, we will first deduce the specific case when $M$ is the adjoint Nijenhuis representation $\frakg$ itself using a particular $L_\infty$-algebra (see Proposition~\ref{Prop: justifying Cohomology theory of Nijenhuis operator} (ii)). Building on this result, the general cases will naturally follow as a direct corollary from the results presented in Subsection~\ref{Subsec: cohomology NjL}.

\begin{lem} \label{Lem: cochain complex of Lie Nijenhuis operator}
	$(\C_{\NjO}^\bullet(\frakg, M), \delta_{\NjO, M})$ is a cochain complex.
\end{lem}

\smallskip

\subsection{Cohomology of Nijenhuis Lie algebras} \label{Subsec: cohomology NjL} \

Now, we define a cohomology theory of Nijenhuis Lie algebras by combining the Chevalley-Eilenberg cohomology of Lie algebras and the cohomology of Nijenhuis operators. 

Let $M = (M, P_M)$ be a Nijenhuis representation over a Nijenhuis Lie algebra $\frakg = (\frakg, \mu, P)$. Let's construct a chain map 
$$\Psi_{M}:\C^\bullet_{\Lie}(\frakg, M) \rightarrow \C_{\NjO}^\bullet(\frakg, M), $$ 
i.e., the following commutative diagram:
$$\xymatrix{
	\C^0_{\Lie}(\frakg, M)\ar[r]^-{\delta_{\Lie, M}}\ar[d]^-{\Psi_{M}}& \C^1_{\Lie}(\frakg, M)\ar[r]\ar[d]^-{\Psi_{M}} &\cdots \ar[r] &\C^n_{\Lie}(\frakg, M)\ar[r]^-{\delta_{\Lie, M}}\ar[d]^-{\Psi_{M}}&\C^{n+1}_{\Lie}(\frakg, M)\ar[d]^{\Psi_{M}}\ar[r] & \cdots\\
	\C^0_{\NjO}(\frakg, M)\ar[r]^-{\delta_{\NjO, M}}&\C^1_{\NjO}(\frakg, M)\ar[r] &\cdots \ar[r]& \C^n_{\NjO}(\frakg, M)\ar[r]^-{\delta_{\NjO, M}}&\C^{n+1}_{\NjO}(\frakg, M)\ar[r]&\cdots
	,}$$
where $\Psi_{M}: \C^0_{\Lie}(\frakg, M) \to \C^0_{\Lie}(\frakg, M)$ is $\Id_{M}$, and for $n \geqslant 1$, $ f \in \C^n_{\Lie}(\frakg, M)$, 
$\Psi_{M}(f)\in \C^n_{\NjO}(\frakg, M)$ 
is defined as:
\begin{align*}
	\Psi_{M}(f)&(a_1, \dots, a_n) := \sum_{k=0}^n 
	\sum_{1\leqslant i_1<i_2<\cdots<i_k\leqslant n}(-1)^{n-k} 
	P_M^{n-k} \circ f(a_{1, i_1-1}, P(a_{i_1}), a_{i_1+1, i_2-1}, P(a_{i_2}), \dots, P(a_{i_k}), a_{i_k+1, n}),
\end{align*}
where for $1 \leqslant i \leqslant j \leqslant n$, $a_{i, j}=a_i\otimes \cdots \otimes a_j$ but for $i>j$, $a_{i, j}$ is by convention empty.

In particular, when $M$ is the adjoint Nijenhuis representation $\frakg$ itself, we simply denote $\Psi_{\frakg}$ by $\Psi$.

\begin{prop} \label{Prop: Chain map Psi}
	The map $\Psi_{M}: \C^\bullet_\Lie(\frakg, M)\rightarrow \C^\bullet_{\NjO}(\frakg, M)$ is a chain map.
\end{prop}

The proof of Proposition~\ref{Prop: Chain map Psi} can be done by direct inspection, which is not difficult. But we will deduce Proposition~\ref{Prop: Chain map Psi} in the particular case when $M$ is the adjoint Nijenhuis representation $\frakg$ itself, from a specific $L_\infty$-algebra, see Proposition~\ref{Prop: cohomology complex as the underlying complex of L infinity algebra NjL}. 


\begin{defn} \label{Def: definition of Nijenhuis cohomology}
	Let $M = (M, P_M)$ be a Nijenhuis representation over a Nijenhuis Lie algebra $\frakg = (\frakg, \mu, P)$. We define the cochain complex of Nijenhuis Lie algebra $\frakg$ with coefficients in Nijenhuis representation $M$, denoted by $(\C^\bullet_{\NjL}(\frakg, M), \delta_{\NjL, M})$, to be the the mapping cone of 
	$\Psi_{M}$ shifted by $-1$, that is, 
	$$\C^0_{\NjL}(\frakg, M) := \C^0_\Lie(\frakg, M) \quad \mathrm{and} \quad \C^n_{\NjL}(\frakg, M):=\C^n_\Lie(\frakg, M)\oplus \C^{n-1}_{\NjO}(\frakg, M), \text{ for } n\geqslant 1, $$ and 
	the differential $\delta_{\NjL, M}: \C^n_{\NjL}(\frakg, M)\rightarrow \C^{n+1}_{\NjL}(\frakg, M)$ is given by
	\begin{align*} 
		\delta_{\NjL, M}(f, g)
		=&\ (\delta_{\Lie, M}(f), -\Psi_{M}(f) -\delta_{\NjO, M}(g))
	\end{align*}
	for any $f\in \C^n_\Lie(\frakg, M)$ and $g\in \C^{n-1}_{\NjO}(\frakg, M)$.
	
	The cohomology of $\C_{\NjL}^\bullet(\frakg, M)$, denoted by $\mathrm{H}_{\NjL}^\bullet(\frakg, M)$, is called the cohomology of Nijenhuis Lie algebra $\frakg$ with coefficients in Nijenhuis representation $M$.	
	When $M$ is the adjoint Nijenhuis representation $\frakg$ itself, we simply denote the differential $\delta_{\NjL, \frakg}$ by $\delta_{\NjL}$. We also simply denote $\C^\bullet_{\NjL}(\frakg, \frakg)$ by $\C^\bullet_{\NjL}(\frakg)$ and $\rmH^\bullet_{\NjL}(\frakg, \frakg)$ by $\rmH^\bullet_{\NjL}(\frakg)$, called the cochain complex of Nijenhuis Lie algebra $\frakg$ and the cohomology of Nijenhuis Lie algebra $\frakg$, respectively.
\end{defn}

By properties of the mapping cone, there is a short exact sequence of cochain complexes:
\begin{eqnarray*} 
	0 \to s^{-1}\C^\bullet_{\NjO}(\frakg, M) \to \C^\bullet_{\NjL}(\frakg, M) \to \C^\bullet_{\Lie}(\frakg, M) \to 0, 
\end{eqnarray*}
which induces a long exact sequence of cohomology groups
$$0 \to \rmH^{0}_{\NjL}(\frakg, M) \to \mathrm{H}_{\Lie}^0(\frakg, M) \to \rmH^0_{\NjO}(\frakg, M) \to \rmH^{1}_{\NjL}(\frakg, M) \to \mathrm{H}_{\Lie}^1(\frakg, M) \to\cdots$$
$$\cdots \to \mathrm{H}_{\Lie}^p(\frakg, M) \to \rmH^p_{\NjO}(\frakg, M) \to \rmH^{p+1}_{\NjL}(\frakg, M) \to \mathrm{H}_{\Lie}^{p+1}(\frakg, M) \to \cdots.$$
Here the symbol $s^{-1}$ represents the desuspension of a cohomologically graded vector space.

\medskip


Let $M=(M, P_M)$ be a Nijenhuis representation over a Nijenhuis Lie algebra $\frakg=(\frakg, \mu, P)$.
According to Proposition~\ref{Prop: Nijenhuis Lie algebra trivial extension}, the semi-direct product $\frakg \ltimes M$ of Nijenhuis Lie algebra $\frakg$ by its Nijenhuis representation $M$ is also a Nijenhuis Lie algebra. Hence, we get the cochain complex $(\C^{\bullet}_{\NjL}(\frakg \ltimes M), \delta_{\NjL})$ of the Nijenhuis Lie algebra $\frakg \ltimes M$, where
\begin{align*}
	\C^0_{\NjL}(\frakg \ltimes M) &= \C^0_{\Lie}(\frakg \ltimes M) = \frakg \oplus M,\\
	\C^n_{\NjL}(\frakg \ltimes M) &= \C^n_{\Lie}(\frakg \ltimes M) \oplus \C^{n-1}_{\NjO}(\frakg \ltimes M) \\
	&= \Hom((\frakg \oplus M)^{\wg n}, \frakg \oplus M) \oplus \Hom((\frakg \oplus M)^{\wg n-1}, \frakg \oplus M), \text{ for } n\geqslant 1.
\end{align*}

In order to obtain the cochain complex of Nijenhuis Lie algebra $\frakg$ with coefficients in Nijenhuis representation $M$, one also needs to consider the obvious embedding
$$\iota: \C^{\bullet}_{\NjL}(\frakg, M) \hookrightarrow \C^{\bullet}_{\NjL}(\frakg \ltimes M), \ f \mapsto \overline{f}.$$

\begin{prop}
	With above notations, the image $\mathrm{Im}(\iota)$ is a subcomplex of $(\C^{\bullet}_{\NjL}(\frakg \ltimes M), \delta_{\NjL})$.
\end{prop}
\begin{proof}
	Let $(f,g) \in \C^n_{\NjL}(\frakg, M)$ for $n \geqslant 0$. Direct calculation can obtain:
	\begin{align*}
		\delta_{\NjL}(\iota(f,g))
		=&\ \delta_{\NjL}(\overline{f}, \overline{g})\\
		=&\ (\delta_{\Lie}(\overline{f}), -\Psi(\overline{f})-\delta_{\NjO}(\overline{g})) \\
		=&\ \iota(\delta_{\Lie, M}(f), -\Psi_{M}(f)-\delta_{\NjO, M}(g))\\
		=&\ \iota(\delta_{\NjL, M}(f,g)).
	\end{align*}
\end{proof}

Let $p: \mathrm{Im}(\iota) \twoheadrightarrow \C^{\bullet}_{\NjL}(\frakg, M), \overline{f} \mapsto f$, be the obvious projection. Using the fact $\iota \circ p = \id_{\mathrm{Im}(\iota)}$, we find that $\delta_{\NjL, M} = p \circ \delta_{\NjL} \circ \iota$ is indeed a differential on the space $\C^{\bullet}_{\NjL}(\frakg, M)$, which justify the Definition~\ref{Def: definition of Nijenhuis cohomology}. Moreover, Lemma~\ref{Lem: cochain complex of Lie Nijenhuis operator} and Proposition~\ref{Prop: Chain map Psi} are direct corollaries as well.

\medskip

\section{The Koszul dual homotopy cooperad of the operad for Nijenhuis Lie algebras} \label{Section: Koszul dual cooperad NjL}

In this section, we construct a symmetric homotopy cooperad $\NL^\ac$, which serves as the Koszul dual of $\NL$, the symmetric operad for Nijenhuis Lie algebras. This is justified by the fact that the cobar construction of $\NL^\ac$ is exactly $\NL_\infty$, the minimal model of $\NL$, as will be proved in Section~\ref{Section: minimal model NjL}. 

Define a graded $\mathbb{S}$-module $\mathscr{S}(\NL^\ac)$ by $$\mathscr{S}(\NL^\ac)(n):=\bfk u_n\oplus \bfk v_n$$ with $|u_n|=0, |v_n|=1$ for $n\geqslant 1$, and the action of $\mathbb{S}_{n}$ on $u_{n}$ and $v_{n}$ are trivial. Now, we put a coaugmented symmetric homotopy cooperad structure on $\mathscr{S}(\NL^\ac)$. Firstly, consider shuffle trees of arity $n\geqslant 1$ in the following list:
\begin{enumerate}
	\item Shuffle trees of type (I) which are of weight $2$: for each $1\leqslant j\leqslant n$, $1\leqslant i\leqslant n-j+1$, there exists such shuffle trees, which can be visualized as
	\begin{eqnarray*}
		\begin{tikzpicture}[scale=0.75, descr/.style={fill=white}]
			\tikzstyle{every node}=[thick, minimum size=3pt, inner sep=1pt]
			
			\node(sigma-1) at (2.2, 1){$\cdot  \sigma^{-1} $};
			
			\node(r) at (0, -0.5)[minimum size=0pt, circle]{};
			\node(v0) at (0, 0)[fill=black, circle, label=right:{\tiny $ n-j+1$}]{};
			\node(v1-1) at (-1.5, 1){\tiny$ 1 $};
			\node(v1-2) at(0, 1)[fill=black, circle, label=right:{\tiny $\tiny j$}]{};
			\node(v1-3) at(1.5, 1){\tiny$ n $};
			\node(v2-1)at (-1, 2){\tiny$ i $};
			\node(v2-2) at(1, 2){\tiny$ i+j-1 $};
			\draw(v0)--(v1-1);
			\draw(v0)--(v1-3);
			\draw(v1-2)--(v2-1);
			\draw(v1-2)--(v2-2);
			\draw[dotted](-0.4, 1.5)--(0.4, 1.5);
			\draw[dotted](-0.5, 0.5)--(-0.1, 0.5);
			\draw[dotted](0.1, 0.5)--(0.5, 0.5);
			\path[-, font=\scriptsize]
			(v0) edge node[descr]{{\tiny$i$}} (v1-2);
			
			\node(vroot) at (0, -0.4){};
			\draw(v0)--(vroot);
		\end{tikzpicture}
	\end{eqnarray*}
	for all $ \sigma \in \sh $.
	
	\item Shuffle trees of type (II) which are of weight $p+1\geqslant 3$ and height $ t+2 $ with $0 \leqslant t \leqslant p \leqslant n$, and there exists a unique vertex in the first $ t+1 $ levels: in these trees, there are numbers $1\leqslant k_t<\cdots<k_{p-1}\leqslant p$, 
	$r_1, \dots, r_p\geqslant 1$ and $1\leqslant i_{h} \leqslant r_{h} $ for all $1\leqslant h \leqslant t $, such that
	the $ q $-th vertex has arity $ r_{q} $ for all $1 \leqslant q \leqslant t $, the $ m $-th vertex is connected to the $ i_{m-1} $-th leaf of the $ (m-1) $-th vertex for all $ 2 \leqslant m \leqslant t $, the $ (t+1) $-th vertex has arity $ p $ and is connected to the $ i_{t} $-th leaf of the $ t $-th vertex, and the other vertices connected to the $k_t$-th (resp. $k_{t+1}$-th, $\dots, k_{p-1}$-th) leaf of the $ (t+1) $-th vertex and with arity $r_{t+1}$ (resp. $r_{t+2}, \dots, r_p$), so $r_1+\cdots+r_p=n$. These  shuffle trees can be visualized as:
	\begin{eqnarray*}
		\begin{tikzpicture}[scale=1.1, descr/.style={fill=white}]
			\tikzstyle{every node}=[thick, minimum size=3pt, inner sep=1pt]
			
			\node(sigma-1) at (2.6, 1.2){$\cdot  \sigma^{-1} $};
			
			\begin{scope}[shift={(0,0.8)}, scale=1]  
				\node(v-1) at (0, -2.2)[circle, fill=black, label=right:{\tiny $r_1$}]{}; 
				\node(v0-1) at (-2, -1){\tiny $ 1 $}; 
				\node(v0-3) at(2, -1){\tiny $ n $};	
				
				\node(vroot) at (0, -2.5){};
				\draw(v-1)--(vroot);
				
				\node(v-1-2) at (0, -1.6)[circle, fill=black, label=right:{\tiny $r_2$}]{}; 
				\node(v0-1-2) at (-2, -0.4){\tiny $ \eta_{1} $}; 
				\node(v0-3-2) at(2, -0.4){\tiny $ \xi_{2} $};
				\node(v0-2-1-2) at(0, -0.9)[]{}; 
				\node(v0-2-2) at(0, -0.6)[]{}; 
				\draw(v-1-2)--(v0-1-2);
				\draw(v-1-2)--(v0-2-1-2);
				\draw(v-1-2)--(v0-3-2);
				\draw(v-1)--(v-1-2);
				\draw[dotted](v0-2-1-2)--(v0-2-2);
				\path[-, font=\scriptsize]
				
				(v-1) edge node[descr]{{\tiny$i_{1}$}} (v-1-2);
				\path[-, font=\scriptsize]
				(v-1-2) edge node[descr]{{\tiny$i_{2}$}} (v0-2-1-2);
				
				\draw[dotted](-0.5, -1.8)--(-0.1, -1.8);
				\draw[dotted](0.1, -1.8)--(0.5, -1.8);
				
				\draw[dotted](-0.5, -1.2)--(-0.1, -1.2);
				\draw[dotted](0.1, -1.2)--(0.5, -1.2);
			\end{scope}
			\node(v0) at (0, 0.5)[circle, fill=black, label=right:{\tiny $r_t$}]{};
			\node(v1-1) at (-2, 1.8){\tiny $ \eta_{t-1} $};
			\node(v1-2) at(0, 1.2)[circle, fill=black, label=right:{\tiny $p$}]{};
			\draw[dotted](-0.6, 1)--(-0.1, 1);
			\draw[dotted](0.1, 1)--(0.6, 1);
			
			\node(v1-3) at(2, 1.8){\tiny $ \xi_{t} $};
			\node(v2-1) at(-1.9, 2.6){\tiny $ \eta_{t} $};
			\node(v2-2) at (-0.9, 2.8)[circle, fill=black, label=right:{\tiny $r_{t+1}$}]{};
			\node(v2-3) at (0, 2.9){};
			\node(v2-4) at(0.9, 2.8)[circle, fill=black, label=right:{\tiny $r_{p}$}]{};
			\node(v2-5) at(1.9, 2.6){\tiny $ \xi_{t+1} $};
			\node(v3-1) at (-1.6, 3.5){\tiny $ \eta_{t+1} $};
			\node(v3-2) at (-0.4, 3.5){\tiny $ \eta_{t+1} + r_{t+1} - 1 $};
			\node(v3-3) at (0.4, 3.5){\tiny $ \eta_{p} $};
			\node(v3-4) at(1.6, 3.5){\tiny $ \eta_{p} + r_{p} - 1 $};
			
			\draw(v0-2-2)--(v0);
			\draw(v0)--(v1-1);
			\draw(v0)--(v1-3);
			\draw(v-1)--(v0-1);
			\draw(v-1)--(v0-3);
			
			\path[-, font=\scriptsize]
			(v0) edge node[descr]{{\tiny$i_{t}$}} (v1-2);
			
			\draw(v1-2)--(v2-1);
			\draw(v1-2)--(v2-3);
			\draw(v1-2)--(v2-5);
			
			\path[-, font=\scriptsize]
			(v1-2) 	edge node[descr]{{\tiny$k_t$}} (v2-2)
			edge node[descr]{{\tiny$k_{p-1}$}} (v2-4);
			\draw(v2-2)--(v3-1);
			\draw(v2-2)--(v3-2);
			\draw(v2-4)--(v3-3);
			\draw(v2-4)--(v3-4);
			
			\draw[dotted](-0.5, 2.4)--(-0.1, 2.4);
			\draw[dotted](0.1, 2.4)--(0.5, 2.4);
			\draw[dotted](-1.4, 2.4)--(-0.8, 2.4);
			\draw[dotted](1.4, 2.4)--(0.8, 2.4);
			\draw[dotted](-1.2, 3.2)--(-0.7, 3.2);
			\draw[dotted](1.2, 3.2)--(0.7, 3.2);
		\end{tikzpicture}
	\end{eqnarray*}
	for all $ \sigma \in \sh $, 
	where
	\begin{equation*} \label{minimal module NjL eta}
			\eta_{j}
			= \begin{dcases}
					i_{1} + \cdots + i_{j} - j + 1, & 1 \leqslant j \leqslant t, \\
					k_{j-1} + \sum_{\alpha=t+1}^{j-1}r_{\alpha} -j+t+\eta_{t}, & t+1 \leqslant j \leqslant p, 
				\end{dcases}
		\end{equation*}
	\begin{equation*} \label{minimal module NjL xi}
			\xi_{j} = \eta_{p} + r_{p} + p - k_{p-1} - 1 + \sum_{\alpha=j}^{t}(r_{\alpha}-i_{\alpha}), \ 2 \leqslant j \leqslant t+1.
		\end{equation*}
\end{enumerate}

Now, we define a family of operations $\{\Delta_T: \mathscr{S}(\NL^\ac)\rightarrow \mathscr{S}({\NL^\ac})^{\ot T}\}_{T\in \mathrm{S}\frakT}$ as follows:
\begin{enumerate}
	\item For a shuffle tree $T$ of type $\mathrm{(I)}$,
	define $\Delta_T(u_n)=u_{n-j+1}\ot u_j$, which can be drawn as
	\begin{eqnarray*}
		\begin{tikzpicture}[scale=0.8, descr/.style={fill=white}]
			\tikzstyle{every node}=[thick, minimum size=5pt, inner sep=1pt]
			
			\node(sigma-1) at (2.3, 1){$\cdot  \sigma^{-1}. $};
			
			\node(r) at (0, -0.5)[minimum size=0pt, rectangle]{};
			\node(v-2) at(-2, 1)[minimum size=0pt, label=left:{ $\Delta_T(u_n)=$}]{};
			\node(v0) at (0, 0)[draw, rectangle]{{\small $u_{n-j+1}$}};
			\node(v1-1) at (-1.5, 1){\tiny $ 1 $};
			\node(v1-2) at(0, 1)[draw, rectangle]{\small$u_j$};
			\node(v1-3) at(1.5, 1){\tiny $ n $};
			\node(v2-1)at (-1, 2){\tiny $ i $};
			\node(v2-2) at(1, 2){\tiny $ i+j-1 $};
			\draw(v0)--(v1-1);
			\draw(v0)--(v1-3);
			\draw(v1-2)--(v2-1);
			\draw(v1-2)--(v2-2);
			\draw[dotted](-0.4, 1.5)--(0.4, 1.5);
			\draw[dotted](-0.5, 0.5)--(-0.1, 0.5);
			\draw[dotted](0.1, 0.5)--(0.5, 0.5);
			\path[-, font=\scriptsize]
			(v0) edge node[descr]{{\tiny$i$}} (v1-2);
			
			\node(vroot) at (0, -0.5){};
			\draw(v0)--(vroot);
		\end{tikzpicture}
	\end{eqnarray*}
	Define
	$$\Delta_T(v_n)=
	\left\{\begin{array}{ll} v_n\ot u_1, & j=1, \\
		0, & 2\leqslant j\leqslant n-1, \\
		u_1\ot v_n, & j=n, 
	\end{array}\right.$$
	which can be pictured as
	\begin{enumerate}
		\item when $j=1$,
		$$\begin{tikzpicture}[scale=0.8, descr/.style={fill=white}]
			\tikzstyle{every node}=[thick, minimum size=5pt, inner sep=1pt]
			\node(r) at (0, -0.5)[minimum size=0pt, rectangle]{};
			\node(v-1) at(-2, 0.5)[minimum size=0pt, label=left:{ $\Delta_T(v_n)=$}]{};
			\node(dd) at(1.5, 0.5){;};
			\node(v0) at (0, 0)[draw, rectangle]{\small$v_n$};
			\node(v1-1) at (-1.3, 1){\tiny$ 1 $};
			\node(v1-2) at(0, 1)[draw, rectangle]{\small$u_1$};
			\node(v1-3) at(1.3, 1){\tiny$ n $};
			\node(v2-1)at (0, 1.8){\tiny$ i $};
			\draw(v0)--(v1-1);
			\draw(v0)--(v1-3);
			\draw(v1-2)--(v2-1);
			\draw[dotted](-0.5, 0.5)--(-0.1, 0.5);
			\draw[dotted](0.1, 0.5)--(0.5, 0.5);
			\path[-, font=\scriptsize]
			(v0) edge node[descr]{{\tiny$i$}} (v1-2);
			
			\node(vroot) at (0, -0.5){};
			\draw(v0)--(vroot);
		\end{tikzpicture}$$
		\item when $2\leqslant j\leqslant n-1$, $\Delta_T(v_n)=0$;
		\item when $j=n$,
		$$\begin{tikzpicture}[scale=0.8, descr/.style={fill=white}]
			\tikzstyle{every node}=[thick, minimum size=5pt, inner sep=1pt]
			\node(r) at (0, -0.5)[minimum size=0pt, rectangle]{};
			\node(va) at(2, 0.7)[minimum size=0pt, label=left:{ $\Delta_T(v_n)=$}]{};
			\node(dd) at(4.7, 0.6){.};
			\node(ve) at (3.5, 0)[draw, rectangle]{\small $u_1$};
			\node(ve1) at (3.5, 1)[draw, rectangle]{\small $v_n$};
			\node(ve2-1) at(2.5, 2){\tiny$ 1 $};
			\node(ve2-2) at(4.5, 2){\tiny$ n $};
			\draw(ve)--(ve1);
			\draw(ve1)--(ve2-1);
			\draw(ve1)--(ve2-2);
			\draw[dotted](3.1, 1.5)--(3.9, 1.5);
			
			\node(vroot) at (3.5, -0.5){};
			\draw(ve)--(vroot);
		\end{tikzpicture}$$
	\end{enumerate}
	
	\item For a shuffle tree $T$ of type $\mathrm{(II)}$,
	define
	$$\begin{tikzpicture}[scale=1.2, descr/.style={fill=white}]
		\tikzstyle{every node}=[thick, minimum size=3pt, inner sep=1pt]
		
		\node(v-2) at (-2.5, 1.2)[minimum size=0pt, label=left:{$\Delta_T(v_n) = (-1)^\frac{p(p-1)}{2}$}]{};
		
		\node(sigma-1) at (2.8, 1.2){$\cdot  \sigma^{-1}. $};
		
		\begin{scope}[shift={(0,0.8)}, scale=1]  
			\node(v-1) at (0, -2.2)[rectangle, draw]{\tiny $v_{r_{1}}$}; 
			\node(v0-1) at (-2, -1){\tiny $ 1 $}; 
			\node(v0-3) at(2, -1){\tiny $ n $};	
			
			\node(v-1-2) at (0, -1.6)[rectangle, draw]{\tiny $v_{r_{2}}$}; 
			\node(v0-1-2) at (-2, -0.4){\tiny $ \eta_{1} $}; 
			\node(v0-3-2) at(2, -0.4){\tiny $ \xi_{2} $};
			\node(v0-2-1-2) at(0, -0.9)[]{}; 
			\node(v0-2-2) at(0, -0.6)[]{}; 
			\draw(v-1-2)--(v0-1-2);
			\draw(v-1-2)--(v0-2-1-2);
			\draw(v-1-2)--(v0-3-2);
			\draw(v-1)--(v-1-2);
			\draw[dotted](v0-2-1-2)--(v0-2-2);
			\path[-, font=\scriptsize]
			
			(v-1) edge node[descr]{{\tiny$i_{1}$}} (v-1-2);
			\path[-, font=\scriptsize]
			(v-1-2) edge node[descr]{{\tiny$i_{2}$}} (v0-2-1-2);
			
			\draw[dotted](-0.5, -1.8)--(-0.1, -1.8);
			\draw[dotted](0.1, -1.8)--(0.5, -1.8);
			
			\draw[dotted](-0.5, -1.2)--(-0.1, -1.2);
			\draw[dotted](0.1, -1.2)--(0.5, -1.2);
			
			\node(vroot) at (0, -2.5){};
			\draw(v-1)--(vroot);
		\end{scope}
		\node(v0) at (0, 0.5)[rectangle, draw]{\tiny $v_{r_{t}}$};
		\node(v1-1) at (-2, 1.8){\tiny $ \eta_{t-1} $};
		\node(v1-2) at(0, 1.2)[rectangle, draw]{\tiny $u_{p}$};
		\draw[dotted](-0.6, 1)--(-0.1, 1);
		\draw[dotted](0.1, 1)--(0.6, 1);
		
		\node(v1-3) at(2, 1.8){\tiny $ \xi_{t} $};
		\node(v2-1) at(-1.9, 2.6){\tiny $ \eta_{t} $};
		\node(v2-2) at (-0.9, 2.8)[rectangle, draw]{\tiny $v_{r_{t+1}}$};
		\node(v2-3) at (0, 2.9){};
		\node(v2-4) at(0.9, 2.8)[rectangle, draw]{\tiny $v_{r_{p}}$};
		\node(v2-5) at(1.9, 2.6){\tiny $ \xi_{t+1} $};
		\node(v3-1) at (-1.6, 3.5){\tiny $ \eta_{t+1} $};
		\node(v3-2) at (-0.4, 3.5){\tiny $ \eta_{t+1} + r_{t+1} - 1 $};
		\node(v3-3) at (0.4, 3.5){\tiny $ \eta_{p} $};
		\node(v3-4) at(1.6, 3.5){\tiny $ \eta_{p} + r_{p} - 1 $};
		
		\draw(v0-2-2)--(v0);
		\draw(v0)--(v1-1);
		\draw(v0)--(v1-3);
		\draw(v-1)--(v0-1);
		\draw(v-1)--(v0-3);
		
		\path[-, font=\scriptsize]
		(v0) edge node[descr]{{\tiny$i_{t}$}} (v1-2);
		
		\draw(v1-2)--(v2-1);
		\draw(v1-2)--(v2-3);
		\draw(v1-2)--(v2-5);
		
		\path[-, font=\scriptsize]
		(v1-2) 	edge node[descr]{{\tiny$k_t$}} (v2-2)
		edge node[descr]{{\tiny$k_{p-1}$}} (v2-4);
		\draw(v2-2)--(v3-1);
		\draw(v2-2)--(v3-2);
		\draw(v2-4)--(v3-3);
		\draw(v2-4)--(v3-4);
		
		\draw[dotted](-0.5, 2.4)--(-0.1, 2.4);
		\draw[dotted](0.1, 2.4)--(0.5, 2.4);
		\draw[dotted](-1.4, 2.4)--(-0.8, 2.4);
		\draw[dotted](1.4, 2.4)--(0.8, 2.4);
		\draw[dotted](-1.2, 3.2)--(-0.7, 3.2);
		\draw[dotted](1.2, 3.2)--(0.7, 3.2);
	\end{tikzpicture}$$
	
	\item For any other shuffle tree $T$ in $\mathrm{S}\frakT$, define $\Delta_T = 0$.
\end{enumerate}

\begin{prop} \label{prop: homotopy cooperad cobar construction NjL}
	The graded $\mathbb{S}$-module $\mathscr{S}(\NL^\ac)$ endowed with the operations $\{\Delta_T\}_{T\in \mathrm{S}\frakT}$ introduced above forms a coaugmented symmetric homotopy cooperad, whose strict counit is the natural projection $\varepsilon:\mathscr{S}(\NL^\ac)\twoheadrightarrow \bfk u_1\cong \cali$ and the coaugmentation is just the natural embedding $\eta:\cali\cong \bfk u_1\hookrightarrow \mathscr{S}(\NL^\ac)$. 	
\end{prop}

\begin{proof} 
	Equivalently, one needs to show that the induced derivation $\partial$ on the cobar construction of $\mathscr{S}(\NL^\ac)$, i.e., the free symmetric operad generated by $s^{-1}\overline{\mathscr{S}(\NL^\ac)}$, is a differential, that is, $\partial^2=0$.
	
	Denote $s^{-1}u_n, n\geqslant 2$ (resp. $s^{-1}v_n, n\geqslant 1$) by $x_n$ (resp. $y_n$) which are the generators of $s^{-1}\overline{\mathscr{S}(\NL^\ac)}$. Notice that $|x_n|=-1$ and $|y_n|=0$. By the definition of cobar construction of coaugmented symmetric homotopy cooperads, the action of differential $\partial$ on generators $x_n, y_n$ is given by the following formulas:
	\begin{align} \label{Eq: partial xn NjL}
		\partial(x_n) = -\sum\limits_{j=2}^{n-1}x_{n-j+1}\{x_j\}, \ n\geqslant 2,
	\end{align}
	\begin{align} \label{Eq: partial yn NjL}
		\partial(y_n) =
		-\sum_{\substack{ r_1+\cdots+r_p =n\\ r_1, \dots, r_p \geqslant 1 \\	2 \leqslant p \leqslant n}}
		\sum_{t=0}^{p}
		(-1)^{t}
		y_{r_{1}} \bigg\{ \cdots \Big\{ y_{r_{t}} \big \{ x_{p} \{ y_{r_{t+1}}, \dots, y_{r_{p}} \} \big\} \Big\} \cdots \bigg\},\ n \geqslant 1.
	\end{align}
	Note that $\partial(x_{2})=0$, $\partial(y_1)=0$.
	
	We just need to proof that $\partial^2=0$ holds on generators $x_n, n\geqslant 2$ and $y_n, n\geqslant 1$, which can be checked by direct computations.
	As $\partial$ is a derivation with respect to $-\{-\}$, we have
	\begin{eqnarray*}
		\partial^2(x_n) 
		&\stackrel{\eqref{Eq: partial xn NjL}}{=}&\ \partial\Big(-\sum\limits_{j=2}^{n-1}x_{n-j+1}\{x_j\}\Big)\\
		&=&-\sum\limits_{j=2}^{n-1}\partial(x_{n-j+1})\{x_j\}+\sum_{j=2}^{n-1}x_{n-j+1}\{\partial(x_j)\}\\
		&\stackrel{\eqref{Eq: partial xn NjL}}{=}&\ \sum_{i+j+k-2=n \atop
			2\leqslant i, j, k\leqslant n-2}(x_i\{x_j\})\{x_k\}-\sum_{i+j+k-2=n \atop
			2\leqslant i, j, k\leqslant n-2}x_i\{x_j\{x_k\}\}\\
		&\stackrel{\eqref{Eq: pre-jacobi1 S}}{=}&\ \sum_{i+j+k-2=n \atop
			2\leqslant i, j, k\leqslant n-2}x_i\{x_j\{x_k\}\}+\sum_{i+j+k-2=n \atop
			2\leqslant i, j, k\leqslant n-2}x_i\{x_j, x_k\}-\sum_{i+j+k-2=n \atop
			2\leqslant i, j, k\leqslant n-2}x_i\{x_k, x_j\}-\sum_{i+j+k-2=n \atop
			2\leqslant i, j, k\leqslant n-2}x_i\{x_j\{x_k\}\}\\
		&=&\ 0.
	\end{eqnarray*}
	
	 
	Now, we prove that $\partial^{2}(y_{n}) = 0$. This can also be verified directly using the shuffle brace operation. This approach relies on the relationship between the symmetric and shuffle brace operations, as described in Remark~\ref{Rem. relation between two braces}. However, due to the complexity of the computation, we do not pursue this method.
	 
	Instead, we derive $\partial^2(y_n)=0$ from the corresponding result  for Nijenhuis associative algebras established in \cite[Proposition 3.1]{SWZZ2024a}. In that case, $\partial^{2}(y_{n}) = 0$ was proven using the pre-Jacobi identity \cite[Equation (8)]{SWZZ2024a} for the nonsymmetric brace operation. Specifically, in the expansion of $\partial^{2}(y_{n})$ there, each tree appears exactly twice with opposite signs, leading to cancellation.
	  
	Similarly, in the case of Nijenhuis Lie algebras, the expansion of $\partial^2(y_n)$ consists of shuffle trees. This expansion is obtained by applying all possible local shuffles to the planar trees appearing in \cite[Proposition 3.1]{SWZZ2024a}. As a result, in the case of Nijenhuis Lie algebras, each shuffle tree also appears twice with opposite signs, ensuring their cancellation and confirming that $\partial^{2}(y_{n}) = 0$.

\end{proof}

We will justify the following definition by showing its cobar construction is exactly the minimal model of $\NL$, see Definition~\ref{defn: cobar minimal model NjL} and Theorem~\ref{Thm: Minimal model NjL}, hence the name ``Koszul dual symmetric homotopy cooperad".

\begin{defn}
	The symmetric homotopy cooperad $\mathscr{S}(\NL^\ac)\ot_{\mathrm{H}} \cals^{-1}$ is called the Koszul dual symmetric homotopy cooperad of $\NL$, denoted by $\NL^\ac$.
\end{defn}

Precisely, the underlying graded collection of ${\NL^\ac}$ is 
$${\NL^\ac}(n)=\bfk e_n\oplus \bfk o_n$$ 
with $e_n=u_n\ot \varepsilon_n$ and $o_n=v_n\ot \varepsilon_n$ for $n\geqslant 1$, thus $|e_n|=n-1$ and $|o_n|=n$.
The defining operations $\{\Delta_T\}_{T\in \mathrm{S}\frakT}$ is given by the following formulas:
\begin{enumerate}
	\item For a shuffle tree $T$ of type $\mathrm{(I)}$,
	\begin{eqnarray*}
		\begin{tikzpicture}[scale=0.8, descr/.style={fill=white}]
			\tikzstyle{every node}=[thick, minimum size=5pt, inner sep=1pt]
			
			\node(sigma-1) at (2.3, 1){$\cdot  \sigma^{-1}. $};
			
			\node(r) at (0, -0.5)[minimum size=0pt, rectangle]{};
			\node(v-2) at(-1.8, 1)[minimum size=0pt, label=left:{$\Delta_T(e_n)=(-1)^{(j-1)(n-i+1)} \sgn(\sigma)$}]{};
			\node(v0) at (0, 0)[draw, rectangle]{{\small $e_{n-j+1}$}};
			\node(v1-1) at (-1.5, 1){\tiny $ 1 $};
			\node(v1-2) at(0, 1)[draw, rectangle]{\small$e_j$};
			\node(v1-3) at(1.5, 1){\tiny $ n $};
			\node(v2-1)at (-1, 2){\tiny $ i $};
			\node(v2-2) at(1, 2){\tiny $ i+j-1 $};
			\draw(v0)--(v1-1);
			\draw(v0)--(v1-3);
			\draw(v1-2)--(v2-1);
			\draw(v1-2)--(v2-2);
			\draw[dotted](-0.4, 1.5)--(0.4, 1.5);
			\draw[dotted](-0.5, 0.5)--(-0.1, 0.5);
			\draw[dotted](0.1, 0.5)--(0.5, 0.5);
			\path[-, font=\scriptsize]
			(v0) edge node[descr]{{\tiny$i$}} (v1-2);
			
			\node(vroot) at (0, -0.5){};
			\draw(v0)--(vroot);
		\end{tikzpicture}
	\end{eqnarray*}
	We consider the following three distinct cases when introducing $\Delta_T(o_n)$. 
	\begin{enumerate}
		\item when $j=1$,
		$$\begin{tikzpicture}[scale=0.8, descr/.style={fill=white}]
			\tikzstyle{every node}=[thick, minimum size=5pt, inner sep=1pt]
			\node(r) at (0, -0.5)[minimum size=0pt, rectangle]{};
			\node(v-1) at(-2, 0.5)[minimum size=0pt, label=left:{ $\Delta_T(o_n)=$}]{};
			\node(dd) at(1.5, 0.5){;};
			\node(v0) at (0, 0)[draw, rectangle]{\small$o_n$};
			\node(v1-1) at (-1.3, 1){\tiny$ 1 $};
			\node(v1-2) at(0, 1)[draw, rectangle]{\small$e_1$};
			\node(v1-3) at(1.3, 1){\tiny$ n $};
			\node(v2-1)at (0, 1.8){\tiny$ i $};
			\draw(v0)--(v1-1);
			\draw(v0)--(v1-3);
			\draw(v1-2)--(v2-1);
			\draw[dotted](-0.5, 0.5)--(-0.1, 0.5);
			\draw[dotted](0.1, 0.5)--(0.5, 0.5);
			\path[-, font=\scriptsize]
			(v0) edge node[descr]{{\tiny$i$}} (v1-2);
			
			\node(vroot) at (0, -0.5){};
			\draw(v0)--(vroot);
		\end{tikzpicture}$$
		\item when $2\leqslant j\leqslant n-1$, $\Delta_T(o_n)=0$;
		\item when $j=n$,
		$$\begin{tikzpicture}[scale=0.8, descr/.style={fill=white}]
			\tikzstyle{every node}=[thick, minimum size=5pt, inner sep=1pt]
			\node(r) at (0, -0.5)[minimum size=0pt, rectangle]{};
			\node(va) at(2, 0.7)[minimum size=0pt, label=left:{ $\Delta_T(o_n)=$}]{};
			\node(dd) at(4.7, 0.6){.};
			\node(ve) at (3.5, 0)[draw, rectangle]{\small $e_1$};
			\node(ve1) at (3.5, 1)[draw, rectangle]{\small $o_n$};
			\node(ve2-1) at(2.5, 2){\tiny$ 1 $};
			\node(ve2-2) at(4.5, 2){\tiny$ n $};
			\draw(ve)--(ve1);
			\draw(ve1)--(ve2-1);
			\draw(ve1)--(ve2-2);
			\draw[dotted](3.1, 1.5)--(3.9, 1.5);
			
			\node(vroot) at (3.5, -0.5){};
			\draw(ve)--(vroot);
		\end{tikzpicture}$$
	\end{enumerate}
	
	\item For a tree $T$ of type $\mathrm{(II)}$,
	
	\begin{eqnarray*}
		\begin{tikzpicture}[scale=1.2, descr/.style={fill=white}]
			\tikzstyle{every node}=[thick, minimum size=3pt, inner sep=1pt]
			
			\node(v-2) at (-2.5, 1.2)[minimum size=0pt, label=left:{$\Delta_T(o_n)=(-1)^\gamma \sgn (\sigma)$}]{};
			
			\node(sigma-1) at (2.8, 1.2){$\cdot  \sigma^{-1}, $};
			
			\begin{scope}[shift={(0,0.8)}, scale=1]  
				\node(v-1) at (0, -2.2)[rectangle, draw]{\tiny $o_{r_{1}}$}; 
				\node(v0-1) at (-2, -1){\tiny $ 1 $}; 
				\node(v0-3) at(2, -1){\tiny $ n $};	
				
				\node(v-1-2) at (0, -1.6)[rectangle, draw]{\tiny $o_{r_{2}}$}; 
				\node(v0-1-2) at (-2, -0.4){\tiny $ \eta_{1} $}; 
				\node(v0-3-2) at(2, -0.4){\tiny $ \xi_{2} $};
				\node(v0-2-1-2) at(0, -0.9)[]{}; 
				\node(v0-2-2) at(0, -0.6)[]{}; 
				\draw(v-1-2)--(v0-1-2);
				\draw(v-1-2)--(v0-2-1-2);
				\draw(v-1-2)--(v0-3-2);
				\draw(v-1)--(v-1-2);
				\draw[dotted](v0-2-1-2)--(v0-2-2);
				\path[-, font=\scriptsize]
				
				(v-1) edge node[descr]{{\tiny$i_{1}$}} (v-1-2);
				\path[-, font=\scriptsize]
				(v-1-2) edge node[descr]{{\tiny$i_{2}$}} (v0-2-1-2);
				
				\draw[dotted](-0.5, -1.8)--(-0.1, -1.8);
				\draw[dotted](0.1, -1.8)--(0.5, -1.8);
				
				\draw[dotted](-0.5, -1.2)--(-0.1, -1.2);
				\draw[dotted](0.1, -1.2)--(0.5, -1.2);

				\node(vroot) at (0, -2.6){};
				\draw(v-1)--(vroot);
			\end{scope}
			\node(v0) at (0, 0.5)[rectangle, draw]{\tiny $o_{r_{t}}$};
			\node(v1-1) at (-2, 1.8){\tiny $ \eta_{t-1} $};
			\node(v1-2) at(0, 1.2)[rectangle, draw]{\tiny $e_{p}$};
			\draw[dotted](-0.6, 1)--(-0.1, 1);
			\draw[dotted](0.1, 1)--(0.6, 1);
			
			\node(v1-3) at(2, 1.8){\tiny $ \xi_{t} $};
			\node(v2-1) at(-1.9, 2.6){\tiny $ \eta_{t} $};
			\node(v2-2) at (-0.9, 2.8)[rectangle, draw]{\tiny $o_{r_{t+1}}$};
			\node(v2-3) at (0, 2.9){};
			\node(v2-4) at(0.9, 2.8)[rectangle, draw]{\tiny $o_{r_{p}}$};
			\node(v2-5) at(1.9, 2.6){\tiny $ \xi_{t+1} $};
			\node(v3-1) at (-1.6, 3.5){\tiny $ \eta_{t+1} $};
			\node(v3-2) at (-0.4, 3.5){\tiny $ \eta_{t+1} + r_{t+1} - 1 $};
			\node(v3-3) at (0.4, 3.5){\tiny $ \eta_{p} $};
			\node(v3-4) at(1.6, 3.5){\tiny $ \eta_{p} + r_{p} - 1 $};
			
			\draw(v0-2-2)--(v0);
			\draw(v0)--(v1-1);
			\draw(v0)--(v1-3);
			\draw(v-1)--(v0-1);
			\draw(v-1)--(v0-3);
			
			\path[-, font=\scriptsize]
			(v0) edge node[descr]{{\tiny$i_{t}$}} (v1-2);
			
			\draw(v1-2)--(v2-1);
			\draw(v1-2)--(v2-3);
			\draw(v1-2)--(v2-5);
			
			\path[-, font=\scriptsize]
			(v1-2) 	edge node[descr]{{\tiny$k_t$}} (v2-2)
			edge node[descr]{{\tiny$k_{p-1}$}} (v2-4);
			\draw(v2-2)--(v3-1);
			\draw(v2-2)--(v3-2);
			\draw(v2-4)--(v3-3);
			\draw(v2-4)--(v3-4);
			
			\draw[dotted](-0.5, 2.4)--(-0.1, 2.4);
			\draw[dotted](0.1, 2.4)--(0.5, 2.4);
			\draw[dotted](-1.4, 2.4)--(-0.8, 2.4);
			\draw[dotted](1.4, 2.4)--(0.8, 2.4);
			\draw[dotted](-1.2, 3.2)--(-0.7, 3.2);
			\draw[dotted](1.2, 3.2)--(0.7, 3.2);
		\end{tikzpicture}
	\end{eqnarray*}
	where
	\begin{align*}
		\gamma
		= \sum_{h=1}^{t} \Big(1-n-h+ \sum_{j=1}^{h} r_{j}\Big) ( r_{h} - i_{h})
		+ \sum_{j=t+1}^{p}(r_{j}-1) (p-k_{j-1})
		+ \sum_{j=1}^{p-1}(p-j)r_{j}
		+ (p-t)(p-1).
	\end{align*}
	\item For any other shuffle tree $T$ in $\mathrm{S}\frakT$, $\Delta_T = 0$.
\end{enumerate}

\medskip

\section{The Minimal model of the operad for Nijenhuis Lie algebras} \label{Section: minimal model NjL}


In this section, we will prove that the cobar construction of the symmetric homotopy cooperad $\NL^\ac$ is exactly the minimal model for $\NL$, the symmetric operad of Nijenhuis Lie algebras. Then we will introduce the notion of homotopy Nijenhuis Lie algebras. 

\smallskip

\subsection{The symmetric dg operad for homotopy Nijenhuis Lie algebras}\

\begin{defn} \label{defn: cobar minimal model NjL}
	The symmetric dg operad for homotopy Nijenhuis Lie algebras, denoted by $\NL_{\infty}$, is defined to be the cobar construction $\Omega(\NL^\ac)$.
\end{defn}

By Definition~\ref{defn: cobar construction symmetric}, the symmetric dg operad $\NL_{\infty} = \Omega(\NL^\ac)$ is the free symmetric operad generated by the graded $\mathbb{S}$-module $s^{-1} \overline{\NL^\ac}$ endowed with differential induced from the symmetric homotopy cooperad structure on $\NL^\ac$. More precisely,
$$ s^{-1} \overline{\NL^\ac}(1) = \bfk s^{-1}o_{1} \quad \mathrm{and} \quad s^{-1} \overline{\NL^\ac}(n) = \bfk s^{-1}e_{n} \oplus \bfk s^{-1}o_{n},\ n \geqslant 2. $$
Denote $\mu_{n} = s^{-1}e_{n}$ for $n \geqslant 2$ and $P_{n} = s^{-1} o_{n}$ for $n \geqslant 1$ respectively, so $|\mu_{n}| = n-2$ and $|P_{n}| = n-1$. A direct calculation shows that the action of differential on these generators in $\Omega(\NL^\ac)$ is given by, for $ n \geqslant 2 $, 
\begin{equation} \label{Eq: defining HNjL 1}
	\begin{aligned}
		\partial(\mu_n) 
		=& \ \sum_{j=2}^{n-1}\sum_{i=1}^{n-j+1}\sum_{\sigma \in \sh} \sgn (\sigma)
		(-1)^{i+j(n-i)}(\mu_{n-j+1}\circ_i \mu_j) \sigma^{-1}, 
	\end{aligned}
\end{equation}
and for $ n \geqslant 1 $, 
\begin{equation} \label{Eq: defining HNjL 2}
	\begin{aligned}
			\partial(P_{n}) 
			= & \sum_{ \substack{ r_1+\cdots+r_p =n
					\\	r_1, \dots, r_p \geqslant 1
					\\	2 \leqslant p \leqslant n} }
			\sum_{ t=0}^{p}
			\sum_{\substack{ 1 \leqslant i_{1} \leqslant r_{1} \\ \cdots \\ 1 \leqslant i_{t} \leqslant r_{t} }}
			\sum_{1 \leqslant k_{t} < \cdots < k_{p-1} \leqslant p}
			\sum_{\sigma \in \sh} \sgn (\sigma)
			(-1)^{\alpha'} \\
			& \Bigg(P_{r_{1}} \circ_{i_{1}}
			\bigg(
			P_{r_{2}} \circ_{i_{2}}
			\bigg(
			\cdots
			\circ_{i_{t-1}}
			\Big(
			P_{r_{t}} \circ_{i_{t}}
			\big(
			(
			\cdots
			((\mu_{p} \circ_{k_{t}} P_{r_{t+1}}) \circ_{k_{t+1}+r_{t+1}-1} P_{r_{t+2}})
			\cdots
			) \circ_{\beta} P_{r_{p}}
			\big)
			\Big)
			\cdots
			\bigg)
			\bigg)
			\Bigg) \sigma^{-1}, 
	\end{aligned}
\end{equation}
where
\begin{align}
	\beta &= k_{p-1}+r_{t+1} + \cdots + r_{p-1}-(p-1-t), \label{Eq. equation of beta}\\
	\alpha'
	&= 1+ \sum^{t}_{q=1} \big(i_{q} + (\sum_{s=q+1}^{p} r_{s}) (r_{q}-i_{q}) - q(r_{q}-i_{q})\big)
	+ \sum_{i=t+1}^{p}(k_{i-1} -p )(r_{i}-1). \label{Eq. equation of alpha}
\end{align}

Note that, the underlying graded space of $ \NL_{\infty} $ is spanned by the set of all shuffle trees generated by corollas.
We use the corolla with $n$ leaves and a solid vertex to represent generators $\mu_n$ for $ n\geqslant 2$, and the corolla with $n$ leaves and a hollow vertex to represent generators $P_n$ for $n\geqslant 1$:
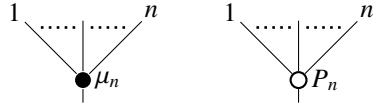
\begin{eqnarray*}
	\begin{tikzpicture}[scale=0.7]
		\tikzstyle{every node}=[thick, minimum size=6pt, inner sep=1pt]
		\node(r)[fill=black, circle, label=right:$\mu_n$] at (0,0){};
		\node (a1) at (-1.3,1.3){$1$};
			\draw (r)--(a1);
		\node (a2) at (0,1.3){};
			\draw (r)--(a2);
		\node (a2) at (1.3,1.3){$n$};
			\draw (r)--(a2);
		\draw [dotted, line width=1pt] (-0.8, 1)--(-0.1, 1);
		\draw [dotted, line width=1pt] (0.8, 1)--(0.1, 1);

		\node(vroot) at (0, -0.6){};
		\draw(r)--(vroot);
	\end{tikzpicture}
	\hspace{8mm}
	\begin{tikzpicture}[scale=0.7]
		\tikzstyle{every node}=[thick, minimum size=6pt, inner sep=1pt]
		\node(r)[draw, circle, label=right: $P_n$] at (0,0){};
		\node (a1) at (-1.3,1.3){$1$};
			\draw (r)--(a1);
		\node (a2) at (0,1.3){};
			\draw (r)--(a2);
		\node (a2) at (1.3,1.3){$n$};
			\draw (r)--(a2);
		\draw [dotted, line width=1pt] (-0.8, 1)--(-0.1, 1);
		\draw [dotted, line width=1pt] (0.8, 1)--(0.1, 1);

		\node(vroot) at (0, -0.6){};
		\draw(r)--(vroot);
	\end{tikzpicture}
\end{eqnarray*}

A (planar labeled) tree with all vertices dyed solid or hollow (such a tree will be called a tree monomial) gives an element in $\NL_\infty$ by composing its vertices in the planar order.
Conversely, any element in $\NL_\infty$ can be represented by such a tree monomial in this way. 
In this means, the action of the differential $\partial$ on generators, i.e., Equations~\eqref{Eq: defining HNjL 1}-\eqref{Eq: defining HNjL 2}, can be expressed by trees as follows:
$$
\begin{tikzpicture}[scale=0.8]
	\tikzstyle{every node}=[thick, minimum size=6pt, inner sep=1pt]
	\node(a) at (-9.7, 0){\large $\partial$};
	\node[circle, fill=black, label=right:$\mu_n$] (b0) at (-8.5, -0.45) {};
	\node (b1) at (-9.3, 0.5) [minimum size=0pt, label=above:$1$]{};
	\node (b2) at (-8.5, 0.5) [minimum size=0pt]{};
	\node (b3) at (-7.7, 0.5) [minimum size=0pt, label=above:$n$]{};
	\draw (b0)--(b1);
	\draw (b0)--(b2);
	\draw (b0)--(b3);
	\draw [dotted, line width=1pt] (-9, 0.3)--(-8.6, 0.3);
	\draw [dotted, line width=1pt] (-8.4, 0.3)--(-8, 0.3);

	\node(vroot) at (-8.5, -0.9){};
	\draw(b0)--(vroot);
	\node(eq) at (-4.5, 0){$=\ \sum\limits_{j=2}^{n-1} \sum\limits_{i=1}^{n-j+1}
		\sum\limits_{\sigma \in \sh} \sgn (\sigma)
		(-1)^{i+j(n-i)} $};
	\begin{scope}[scale=0.9]
		\node(sigma) at (2.5, 0){$\cdot  \sigma^{-1} $};
		\node(e0) at (0, -1.5)[circle, fill=black, label=right:$\mu_{n-j+1}$]{};
		\node(e1) at(-1.5, 0){{\tiny$1$}};
		\node(e2-0) at (0, -0.5){{\tiny$i$}};
		\node(e3) at (1.5, 0){{\tiny{$n$}}};
		\node(e2-1) at (0, 0.5) [circle, fill=black, label=right: $\mu_j$]{};
		\node(e2-1-1) at (-1, 1.5){{\tiny$i$}};
		\node(e2-1-2) at (1, 1.5){{\tiny $i+j-1$}};
		\draw [dotted, line width=1pt] (-0.7, -0.5)--(-0.2, -0.5);
		\draw [dotted, line width=1pt] (0.3, -0.5)--(0.8, -0.5);
		\draw [dotted, line width=1pt] (-0.4, 1.2)--(0.4, 1.2);
		\draw (e0)--(e1);
		\draw (e0)--(e3);
		\draw (e0)--(e2-0);
		\draw (e2-0)--(e2-1);
		\draw (e2-1)--(e2-1-1);
		\draw (e2-1)--(e2-1-2);

		\node(vroot) at (0, -2){};
		\draw(e0)--(vroot);
	\end{scope}
	
\end{tikzpicture} 
$$
$$
\begin{tikzpicture}[scale=0.6]
	\tikzstyle{every node}=[thick, minimum size=6pt, inner sep=1pt]
	\begin{scope}[shift={(1,0)}, scale=0.8]
		\node(a) at (-4, 0){\large$\partial$};
		\node[circle, draw, label=right:$P_n$] (b0) at (-2, -0.7) {};
		\node (b1) at (-3.5, 1) [minimum size=0pt, label=above:$1$]{};
		\node (b2) at (-2, 1) [minimum size=0pt]{};
		\node (b3) at (-0.5, 1) [minimum size=0pt, label=above:$n$]{};
		\draw (b0)--(b1);
		\draw (b0)--(b2);
		\draw (b0)--(b3);
		\draw [dotted, line width=1pt] (-2.9, 0.6)--(-2.2, 0.6);
		\draw [dotted, line width=1pt] (-1.8, 0.6)--(-1.1, 0.6);

		\node(vroot) at (-2, -1.5){};
		\draw(b0)--(vroot);
	\end{scope}
	
	\node (eq2) at (7, -0.5){$=\ \sum\limits_{ \substack{ r_1+\cdots+r_p =n
				\\	r_1, \dots, r_p \geqslant 1
				\\	2 \leqslant p \leqslant n} }
		\sum\limits_{ t=0}^{p}
		\sum\limits_{\substack{ 1 \leqslant i_{1} \leqslant r_{1} \\ \cdots \\ 1 \leqslant i_{t} \leqslant r_{t} }}
		\sum\limits_{1 \leqslant k_{t} < \cdots < k_{p-1} \leqslant p}
		\sum\limits_{\sigma \in \sh} \sgn (\sigma)
		(-1)^{\alpha'} $};
\begin{scope}[shift={(8.5,-12)}, scale=1.2]  
	\tikzstyle{every node}=[thick, minimum size=6pt, inner sep=1pt]
	\node (sigma-1) at (6, 4){$
	\cdot \sigma^{-1}
		$};
	\node(1) at(0, 0)[circle, draw ]{};
	\node(1+) at(0.4, -0.4){\tiny $ P_{r_{1}} $};
	\node(1-1) at (-4, 1.5){\tiny $ 1 $};
	\node(1-2) at (0, 0.9){\tiny $ i_{1} $};
	\node(1-3) at (4, 1.5){\tiny $ n $};
	\draw(1)--(1-1);
	\draw(1)--(1-2);
	\draw(1)--(1-3);
	\draw[dotted, line width=1pt](-1, 0.6)--(1, 0.6);
	
	\node(2) at(0, 1.5)[circle, draw ]{};
	\node(2+) at(0.4, 1.1){\tiny $ P_{r_{2}} $};
	\node(2-1) at (-4, 3){\tiny $ \eta_{1} $};
	\node(2-2) at (0, 2.4){\tiny $ i_{2} $};
	\node(2-3) at (4, 3){\tiny $ \xi_{2} $};
	\draw(2)--(2-1);
	\draw(2)--(2-2);
	\draw(2)--(2-3);
	\draw(1-2)--(2);
	\draw[dotted, line width=1pt](-1, 2.1)--(1, 2.1);
	
	\node(3) at (0, 4)[circle, draw ]{};
	\node(3+) at(0.4, 3.6){\tiny $ P_{r_{t}} $};
	\node(3-) at (0, 3.45){};
	\draw[dotted, line width=1pt](2-2)--(3-);
	\draw(3-)--(3);
	\node(3-1) at (-4, 5.5){\tiny $ \eta_{t-1} $};
	\node(3-2) at (0, 4.9){\tiny $ i_{t} $};
	\node(3-3) at (4, 5.5){\tiny $ \xi_{t} $};
	\draw(3)--(3-1);
	\draw(3)--(3-2);
	\draw(3)--(3-3);
	\draw[dotted, line width=1pt](-1, 4.6)--(1, 4.6);
	
	\node(4) at (0, 5.5)[circle, fill=black]{};
	\draw(3-2)--(4);
	\node(3+) at(0.4, 5.1){\tiny $ \mu_{p} $};
	\node(4-0) at (-4, 7){\tiny $ \eta_{t} $};
	\node(4-1) at (-2.5, 7)[circle, draw, label=left:\tiny$ P_{r_{t+1}} $]{};
	\node(4-2) at (-0.5, 7)[circle, draw, label=left:\tiny$ P_{r_{t+2}} $]{};
	\node(4-3) at (2.5, 7)[circle, draw, label=left:\tiny$ P_{r_{p}} $]{};
	\node(4-4) at (4, 7){\tiny $ \xi_{t+1} $};
	
	\node(4-1-) at (-1.48, 6.4){\tiny$ k_{t} $};
	\node(4-2-) at (-0.3, 6.4){\tiny$ k_{t+1} $};
	\node(4-3-) at (1.5, 6.4){\tiny$ k_{p-1} $};
	
	\draw(4)--(4-0);
	\draw(4)--(4-1-);
	\draw(4)--(4-2-);
	\draw(4)--(4-3-);
	\draw(4)--(4-4);
	
	\draw(4-1-)--(4-1);
	\draw(4-2-)--(4-2);
	\draw(4-3-)--(4-3);
	\draw[dotted, line width=1pt](-1.2, 6)--(1.2, 6);
	
	\node(4-1-1) at (-5, 8){\tiny $ \eta_{t+1} $};
	\node(4-1-2) at (-2.9, 8){\tiny $ \eta_{t+1} + r_{t+1} - 1$};
	\draw(4-1)--(4-1-1);
	\draw(4-1)--(4-1-2);
	\draw[dotted, line width=1pt](4-1-1)--(4-1-2);
	
	\node(4-2-1) at (-0.8, 8){\tiny $ \eta_{t+2} $};
	\node(4-2-2) at (1.3, 8){\tiny $ \eta_{t+2} + r_{t+2} - 1$};
	\draw(4-2)--(4-2-1);
	\draw(4-2)--(4-2-2);
	\draw[dotted, line width=1pt](4-2-1)--(4-2-2);
	
	\node(4-3-1) at (3.1, 8){\tiny $ \eta_{p} $};
	\node(4-3-2) at (5, 8){\tiny $ \eta_{p} + r_{p} - 1 $};
	\draw(4-3)--(4-3-1);
	\draw(4-3)--(4-3-2);
	\draw[dotted, line width=1pt](4-3-1)--(4-3-2);

	\node(vroot) at (0, -0.8){};
	\draw(1)--(vroot);
\end{scope}
\end{tikzpicture}
$$

\smallskip

\subsection{The minimal model}\

The following result is the main result of this section.

\begin{thm} \label{Thm: Minimal model NjL}
	The symmetric dg operad $\NL_\infty$ is the minimal model of the symmetric operad $\NL$.
\end{thm}

\begin{proof}
	We follow the line of proof for the minimal model of Rota-Baxter associative operad \cite{Wang2024} or the minimal model of Nijenhuis associative operad \cite{SWZZ2024a}. However, there are important differences.
	
	It can be easily seen that the differential $\partial$ on $\NL_\infty$ satisfies Conditions (i) and (ii) in Definition~\ref{symmetric minimal model}, the definition of minimal models.
	In order to prove Theorem~\ref{Thm: Minimal model NjL}, we only need to construct a surjective quasi-isomorphism of symmetric dg operads from $\NL_\infty$ to $\NL$, where $\NL$, defined in Definition~\ref{def: operad NjL}, is considered as a symmetric dg operad concentrated in degree $0$.
		
	Introduce a morphism of symmetric dg operads $\phi: \NL_\infty \twoheadrightarrow \NL$, which sends generators $\mu_{2}$ to $\mu$, $P_{1}$ to $P$, and all other generators to $0$.
	The degree zero part of $\NL_\infty$ is the free symmetric graded operad generated by $\mu_{2}$ and $P_{1}$. The image of $\partial$ in this degree zero part is the operadic ideal generated by $\partial (\mu_{3})$ and $\partial (P_{2}) $. By definition, we have:
	\begin{align*}
		\partial(\mu_3) &= \sum_{i=1}^{2} \sum_{\sigma \in \sh} \sgn (\sigma) (-1)^{i} (\mu_{2} \circ_{i} \mu_{2}) \sigma^{-1}, \\
		\partial (P_2)
		&= - (\mu_{2} \circ_{1} P_{1}) \circ_{2} P_{1}
		+ P_{1} \circ_{1} (\mu_{2} \circ_{1} P_{1})
		+ P_{1} \circ_{1} (\mu_{2} \circ_{2} P_{1})
		- (P_{1} \circ_{1} P_{1}) \circ_{1} \mu_{2}.
	\end{align*}
	Thus the map $\phi$ induces the isomorphism $\rmH_0(\NL_\infty, \partial) \cong \NL$.
	
	To prove the map $\phi: \NL_\infty \twoheadrightarrow \NL$ is a quasi-isomorphism, we just need to prove that $\rmH_i(\NL_\infty, \partial)=0$ for all $i\geqslant 1$. This will be achieved by constructing a homotopy map. To this end, we need to establish an (arity-graded) monomial order on $\NL_{\infty}$. The main difference from \cite{Wang2024} is that we need a different monomial order in order that the leading terms are the same as the case of Rota-Baxter associative operad. A more direct approach is to make minor modifications to the order in \cite{SWZZ2024a}, the case of Nijenhuis associative operad, specifically by ordering the labels of the leaves as well.
	
	By \cite[Chapter 5]{BD16}, each tree monomial $\mathcal{T}$ in $ \NL_{\infty} $ with $ n $ leaves can be represented by $\mathcal{T} = (X_{1}, X_{2}, \dots, X_{n};\sigma)$, where $ X_{i} $ is the word formed by generators of $ \NL_{\infty} $ corresponding to the vertices along the unique path from the root of $ \mathcal{T} $ to its leaf labeled by $ i $, and $\sigma \in \sh$ is the label of $\mathcal{T}$, i.e., $ \mathcal{T} $ is labeled by $\{ \sigma(1), \dots, \sigma(n)\}$.

Define a function $ \varphi $ on corollas as follows:
$$\begin{tikzpicture}[scale=0.45]
	\tikzstyle{every node}=[thick, minimum size=6pt, inner sep=1pt]
	\node(a) at (-4, 0.5){\large$\varphi$};
	\node[circle, fill=black, label=right:$\mu_n$] (b0) at (-2, -0.4) {};
	\node (b1) at (-3.5, 1.5) [minimum size=0pt, label=above:$1$]{};
	\node (b2) at (-2, 1.5) [minimum size=0pt]{};
	\node (b3) at (-0.5, 1.5) [minimum size=0pt, label=above:$n$]{};
	\draw (b0)--(b1);
	\draw (b0)--(b2);
	\draw (b0)--(b3);
	\draw [dotted, line width=1pt] (-3, 1)--(-2.2, 1);
	\draw [dotted, line width=1pt] (-1.8, 1)--(-1, 1);
	\node(0) at (3,0.5){$:=n-1(n\geqslant2)$,};

	\node(vroot) at (-2, -1.4){};
	\draw(b0)--(vroot);
\end{tikzpicture}
\quad
\begin{tikzpicture}[scale=0.45]
	\tikzstyle{every node}=[thick, minimum size=6pt, inner sep=1pt]
	\node(a) at (-4, 0.5){\large$\varphi$};
	\node[circle, draw, label=right:$P_n$] (b0) at (-2, -0.4) {};
	\node (b1) at (-3.5, 1.5) [minimum size=0pt, label=above:$1$]{};
	\node (b2) at (-2, 1.5) [minimum size=0pt]{};
	\node (b3) at (-0.5, 1.5) [minimum size=0pt, label=above:$n$]{};
	\draw (b0)--(b1);
	\draw (b0)--(b2);
	\draw (b0)--(b3);
	\draw [dotted, line width=1pt] (-3, 1)--(-2.2, 1);
	\draw [dotted, line width=1pt] (-1.8, 1)--(-1, 1);
	\node(0) at (3,0.5){{$:=2n-1(n\geqslant1)$.}};

	\node(vroot) at (-2, -1.4){};
	\draw(b0)--(vroot);
\end{tikzpicture}
$$

Now, we establish a monomial order $ \Xi $ as we need. For two tree monomials $\mathcal{T} = (X_1, \dots, X_n; \sigma)$, $\mathcal{T}' = (Y_1, \dots, Y_m ; \sigma^{\prime})$, we compare $\mathcal{T}, \mathcal{T}'$ in the following way:
\begin{enumerate}
	\item Compare arity. If $n > m$, then $\mathcal{T}>\mathcal{T}'$; 
	\item Compare the first path. If $n = m$, we compare $\mathcal{T}, \mathcal{T}'$ as follows: 
	Let $ X_{1}=x_{1} x_{2} \cdots x_{r} $ and $ Y_{1}= y_{1} y_{2} \cdots y_{s} $ be the paths from the root of $ \mathcal{T} $ and $ \mathcal{T}' $, respectively, to the leaf labeled by $ 1 $.
	\begin{enumerate}
		\item Compare $ \varphi(X_{1}):= \varphi(x_{1}) + \cdots + \varphi(x_{r}) $ and $ \varphi(Y_{1}):= \varphi(y_{1}) + \cdots + \varphi(y_{s}) $, that is, if $ \varphi(X_{1}) > \varphi(Y_{1}) $, then $ \mathcal{T} > \mathcal{T}' $;
		\item If $ \varphi(X_{1}) = \varphi(Y_{1}) $, then compare the following words:
		\begin{eqnarray*}
			\underbrace{x_{1} \cdots x_{1}}_{\varphi(x_{1})} \underbrace{x_{2} \cdots x_{2}}_{\varphi(x_{2})} \cdots \underbrace{x_{r} \cdots x_{r}}_{\varphi(x_{r})}
			\quad \mathrm{and} \quad
			\underbrace{y_{1} \cdots y_{1}}_{\varphi(y_{1})} \underbrace{y_{2} \cdots y_{2}}_{\varphi(y_{2})} \cdots \underbrace{y_{s} \cdots y_{s}}_{\varphi(y_{s})},
		\end{eqnarray*}
		that is if the former is greater than the latter with respect to the lexicographic order induced by the following total order:
		$$ P_{1} < \mu_{2} < P_{2} < \mu_{3} < \cdots < P_{n} < \mu_{n+1} < P_{n+1} < \cdots, $$
		then $ \mathcal{T} > \mathcal{T}' $.
	\end{enumerate}
	\item Compare the other paths. If we can't compare $ \mathcal{T} $ and $ \mathcal{T}' $ using paths ending at the leaves labeled by $ 1 $, then repeat the process, outlined in step (ii), for their leaves labeled by $ 2, 3, \dots $, in turn until the size is distinguished.
	\item Compare label. If we still can't compare $ \mathcal{T} $ and $ \mathcal{T}' $ according to the steps above, then compare their labels $ \sigma $ and $ \sigma^{\prime} $ by the natural lexicographic order, that is, if $ \sigma > \sigma^{\prime} $, then $ \mathcal{T} < \mathcal{T}' $.
\end{enumerate}
	
	It is ready to see that this is indeed a monomial order.
	Under this order, the leading term in the expansion of $\partial(\mu_n)$, $\partial(P_n)$ are the following tree monomials respectively:
$$
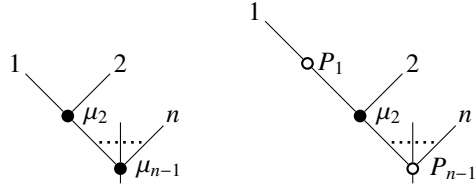
\begin{tikzpicture}[scale=0.7]
		\tikzstyle{every node}=[thick, minimum size=4pt, inner sep=1pt]
		\node(1) at (0, 0) [draw, circle, fill=black, label=right:$\ \mu_{n-1}$]{};
		\node(2-1) at (-1, 1) [draw, circle, fill=black, label=right: $\ \mu_2$]{};
		\node(3-1) at (-2, 2){$ 1 $};
		\node(3-2) at (0, 2){$ 2 $};
		\node(2-2) at (0, 1){};
		\node(2-3) at (1, 1){$ n $};
		\draw (1)--(2-1);
		\draw (1)--(2-2);
		\draw (1)--(2-3);
		\draw (2-1)--(3-1);
		\draw (2-1)--(3-2);
		\draw [dotted, line width=1pt](-0.4, 0.5)--(0.4, 0.5);

		\node(vroot) at (0,-0.4){};
		\draw(1)--(vroot);
	\end{tikzpicture}
	\hspace{8mm}
	\begin{tikzpicture}[scale=0.7]
		\tikzstyle{every node}=[thick, minimum size=4pt, inner sep=1pt]
		\node(1) at (0, 0) [draw, circle, label=right:$\ P_{n-1}$]{};
		\node(2-1) at (-1, 1) [draw, circle, fill=black, label=right: $\ \mu_2$]{};
		\node(3-1) at (-2, 2)[draw, circle, label=right: $P_1$]{};
		\node(3-2) at (0, 2){$ 2 $};
		\node(2-2) at (0, 1){};
		\node(2-3) at (1, 1){$ n $};
		\node(4) at (-3, 3){$ 1 $};
		\draw (1)--(2-1);
		\draw (1)--(2-2);
		\draw (1)--(2-3);
		\draw (2-1)--(3-1);
		\draw (2-1)--(3-2);
		\draw (3-1)--(4);
		\draw [dotted, line width=1pt](-0.4, 0.5)--(0.4, 0.5);

		\node(vroot) at (0,-0.4){};
		\draw(1)--(vroot);
	\end{tikzpicture}$$

	Once the leading terms are seen to be the same as the case of Rota-Baxter associative operad, the remaining part of the proof carries verbatim as that of \cite[Theorem 3.5]{Wang2024}.	
\end{proof}

\smallskip

\subsection{Homotopy Nijenhuis Lie algebras} \label{Homotopy Nijenhuis Lie Algebras}\

We can now introduce the notion of homotopy Nijenhuis Lie algebras.

\begin{defn} \label{defn. homotopy Nijenhuis Lie algebra via dg operad}
	Let $(V,d_V)$ be a complex. Then a homotopy Nijenhuis Lie algebra (or a $\NL_\infty$-algebra) on $V$ is defined to be a morphism of symmetric dg operads from $\NL_\infty$ to the symmetric endomorphism dg operad $\End_V$.
\end{defn}

Let $(V, d_V)$ be an algebra over the symmetric dg operad $\NL_\infty$.
Still denote by
$ \mu_{n}: V^{\ot n} \rightarrow V, n \geqslant 2$ (resp. $ P_{n}: V^{\ot n} \rightarrow V, n \geqslant 1$) the image of $\mu_{n} \in \NL_\infty$ (resp. $P_{n}\in \NL_\infty$).
We also rewrite $\mu_{1} := d_V$.
For each $n \geqslant 1$, Equations~\eqref{Eq: defining HNjL 1}-\eqref{Eq: defining HNjL 2} are respectively equivalent to:
\begin{eqnarray} \label{Eq: jacobi-id NjL}
	\sum_{ i+j+k= n \atop
		i, k\geqslant 0, j\geqslant 1 } \sum_{\sigma \in \sh} \sgn (\sigma)
	(-1)^{i+jk}
	\big( \mu_{i+1+k}\circ (\id^{\ot i}\ot \mu_j\ot \id^{\ot k} ) \big) \sigma^{-1} =0
\end{eqnarray} 
and
\begin{equation} \label{Eq: homotopy NjL-operator-version-2}
	\begin{aligned}
		\sum_{ r_1+\cdots+r_p =n
			\atop
			r_1, \dots, r_p \geqslant 1, p \geqslant 1 }
		\sum_{ t=0}^{p}
		\sum_{\substack{i_{1}+k_{1}+1=r_{1} \\ \dots \\ i_{t}+k_{t}+1=r_{t} \\ i_{1}, \dots, i_{t} \geqslant 0 \\ k_{1}, \dots, k_{t} \geqslant 0 }}
		\sum_{\substack{h_{t}+ \cdots + h_{p} =t \\ h_{t}, \dots, h_{p} \geqslant 0}}
		\sum_{\sigma \in \sh} \sgn (\sigma)
		(-1)^{\alpha} &\\
		\bigg(
		P_{r_{1}} \circ \Big( \id^{\ot i_{1}} \ot P_{r_{2}} \circ \Big( \id^{\ot i_{2}} \ot \cdots \ot P_{r_{t}} \circ \big(\id^{\ot i_{t}} \ot & \\
		\mu_{p} \circ (\id^{h_{t}} \ot P_{r_{t+1}} \ot \id^{\ot h_{t+1}} \ot P_{r_{t+2}} \ot \cdots \ot P_{r_{p-1}} \ot & \id^{h_{p-1}} \ot P_{r_{p}} \ot \id^{\ot h_{p}}) \\
		& \ot \id^{k_{t}}\big) \cdots \ot \id^{\ot k_{2}}\Big)\ot \id^{\ot k_{1}}\Big)
		\bigg) \sigma^{-1} =0, 
	\end{aligned}
\end{equation}
where
\begin{eqnarray*}
	\alpha
	= \sum^{t}_{q=1} \big(r_{q} + k_{q} + (\sum_{s=q+1}^{p} r_{s}) k_{q} - qk_{q}\big)
	+ \sum_{i=t+1}^{p} \big(\sum_{j=t}^{i-1} h_{j} + i - t -p \big)(r_{i}-1).
\end{eqnarray*}

Equation~(\ref{Eq: jacobi-id NjL}) is exactly the generalized Jacobi identity as introduced in Definition~\ref{Def: L-infty}. In particular, the operator $\mu_1$ is a differential on $V$, and the operator $\mu_2$ is a Lie bracket up to homotopy, in other words, $\mu_2$ induces a Lie algebra structure on the homology $\rmH_{\bullet}(V, \mu_{1})$.

For $n=1, 2$, Equation~(\ref{Eq: homotopy NjL-operator-version-2}) gives
\begin{eqnarray} \label{operator-differential NjL} 
	\mu_1\circ P_1=P_1\circ \mu_1
\end{eqnarray}
and 
\begin{eqnarray} \label{NjO-homotopy NjL} 
	& \mu_2\circ(P_1\ot P_1)-P_1\circ \mu_2\circ (\id\ot P_1)-P_1\circ \mu_2\circ (P_1\ot \id) + P_1\circ P_1\circ \mu_2 \\
	\notag &= - (\mu_1\circ P_2+P_2\circ (\id\ot \mu_1)+P_2\circ(\mu_1\ot \id)).
\end{eqnarray}
Equation~(\ref{operator-differential NjL}) implies that $P_1: (V, \mu_1) \rightarrow (V, \mu_1)$ is a chain map, thus $P_1$ is well-defined on $\rmH_\bullet(V, \mu_1)$; Equation~(\ref{NjO-homotopy NjL}) indicates that $P_1$ is a Nijenhuis operator with respect to $\mu_2$ up to homotopy, whose obstruction is exactly the operator $P_2$. As a consequence, $(\rmH_\bullet(V, \mu_1), \mu_2, P_1)$ is a Nijenhuis Lie algebra.

\medskip

\section{From minimal model to $L_\infty$-algebras} \label{Section: From minimal model to Linifnity algebras NjL}

In this section, we will use the minimal model $\NL_\infty$, or more precisely the Koszul dual symmetric homotopy cooperad ${\NL^{\ac}}$, to determine the deformation complex as well as the $L_\infty$-algebra structure on it for Nijenhuis Lie algebras. 

\smallskip

\subsection{The $L_\infty$-algebra structure on deformation complex}\ \label{subSec. The L-algebra structure on deformation complex}

According to Proposition-Definition~\ref{prop-def. convolution symmetric homotopy operad}, the construction of convolution symmetric homotopy operad, and Proposition~\ref{prop: S homotopy operad L infty}, we obtain the following $L_{\infty}$-algebra:

\begin{defn} \label{Defn. L-infty algebra as we need}
	Let $V$ be a graded space. Introduce an $L_{\infty}$-algebra associated to $V$ as $\frakC_{\NjL}(V):=\mathbf{Hom}({\NL^\ac}, \End_V)^{\prod}_{\mathbb{S}}$. 
\end{defn}

Now, let's determine the $L_\infty$-algebra $\frakC_{\NjL}(V)$ explicitly. The sign rules in the symmetric homotopy cooperad ${\NL^\ac}$ are complicated, so we need some transformations. Notice that there is a natural isomorphism of symmetric graded operad
$ \mathbf{Hom}(\cals, \End_{sV}) \cong \End_V. $
Explicitly, any $h \in \End_V(n)$ corresponds to an element $\hbar \in\mathbf{Hom}(\cals, \End_{sV})(n)$ which is defined as 
$$\big(\hbar(\delta_n)\big)(sv_1\ot\cdots\ot sv_n) := (-1)^{\sum\limits_{k=1}^{n-1}\sum\limits_{j=1}^k|v_j|}(-1)^{(n-1)|h|}sh(v_1\ot \cdots \ot v_n)$$ 
for any homogeneous elements $v_1, \dots, v_n\in V$.

Thus we have the following isomorphisms of symmetric homotopy operads:
\begin{align*}
	\mathbf{Hom}\big({\NL^\ac}, \End_V\big)
	\cong&\ \mathbf{Hom}\big({\NL^\ac}, \mathbf{Hom}(\cals, \End_{sV})\big)
	\\
	\cong&\ \mathbf{Hom}\big({\NL^\ac}\ot_{\mathrm{H}}\cals, \End_{sV}\big)
	\\
	=&\ \mathbf{Hom}\big({\mathscr{S}({\NL^\ac})}, \End_{sV}\big).
\end{align*}
We obtain 
$$\frakC_{\NjL}(V)\cong \mathbf{Hom}\big({\mathscr{S}({\NL^\ac})}, \End_{sV}\big)_{\mathbb{S}}^{\prod}.$$
Recall that ${\mathscr{S}({\NL^\ac})}(n)=\bfk u_n\oplus \bfk v_n$ with $|u_n|=0$ and $|v_n|=1$.
By definition 
$$\mathbf{Hom}\big({\mathscr{S}({\NL^\ac})}, \mathrm{End}_{sV}\big)_{\mathbb{S}}(n)
=\Hom(\bfk u_n\oplus \bfk v_n, \Hom((sV)^{\odot n}, sV)).$$
Each $f\in \Hom((sV)^{\odot n}, sV)$ and $g\in \Hom((sV)^{\odot n}, V)$ determines bijectively a map
$\widetilde{f} \in \Hom(\bfk u_n, \Hom((sV)^{\odot n}, sV)) $ and $ \widehat{g} \in \Hom(\bfk v_n, \Hom((sV)^{\odot n}, sV))$, respectively, 
such that $\widetilde{f}(u_n)=f$ and $\widehat{g}(v_n)=(-1)^{|g|}sg$.

Denote
$$\frakC_{\Lie}(V) = \Hom(\overline{S^c} (sV), sV) \quad \mathrm{and} \quad \frakC_{\NjO}(V) = \Hom(\overline{S^c} (sV), V).$$
In this way, we identify $\frakC_{\NjL}(V)$ with $\frakC_{\Lie}(V)\oplus \frakC_{\NjO}(V)$.

Notice that for a Nijenhuis Lie algebra $\frakg=(\frakg, m, P)$, considered as a graded space concentrated in degree zero, $\frakC_{\NjL}(\frakg)$ is just the underlying space of the cochain complex of Nijenhuis Lie algebra $\frakg$ up to shift.

By the general theory recalled in Section~\ref{Subsection: Symmetric homotopy (co)operads}, direct inspections of Proposition~\ref{prop: S homotopy operad L infty} gives the $L_\infty$-algebra structure $\{l_n\}_{n\geqslant 1}$ on $\frakC_{\NjL}(V)$ as follows:
\begin{enumerate}
	
	\item For homogeneous elements $sf, sh\in \mathfrak{C}_{\Lie}(V)$, define 
	$$l_2(sf\ot sh):= [sf, sh]_{\RNA}\in\mathfrak{C}_{\Lie}(V), $$
	where the operation $[-, -]_{\RNA}$ is the Richardson-Nijenhuis bracket defined by Equation~(\ref{Eq: Richardson-Nijenhuis}).
	
	\item \label{Defn. define l b+1}
	Let $n\geqslant 1$. For homogeneous elements $sh\in \Hom((sV)^{\odot n}, sV)\subset \mathfrak{C}_{\Lie}(V)$ and $g_1, \dots, g_n\in \mathfrak{C}_{\NjO}(V)$, 	define 
	$$l_{n+1}(sh\ot g_1\ot \cdots \ot g_n)\in \mathfrak{C}_{\NjO}(V)$$ 
	as :
	\begin{align*}
		\qquad l_{n+1} & (sh\ot g_1\ot \cdots \ot g_n) := 
		\sum_{\sigma\in \mathbb{S}_n}(-1)^{\eta} \Big( \sum_{k=0}^n (-1)^{\xi} s^{-1} \circ (sg_{\sigma (1)})
		\{ \cdots
		\{
		sg_{\sigma(k)}
		\{sh
		\{sg_{\sigma(k+1)}, \dots, sg_{\sigma(n)}\}
		\}
		\}
		\cdots \}
		\Big), 
	\end{align*}
	where 
	$(-1)^{\eta} = \chi(\sigma; g_1, \dots, g_n) (-1)^{ n(|h|+1)+ \sum\limits_{p=1}^{n-1} \sum\limits_{j=1}^{p} |g_{\sigma(j)}|}$,
	$(-1)^{\xi} = (-1)^{(|h|+1) ( \sum\limits_{i=1}^{k}(|g_{\sigma(i)}|+1) ) + k}$,
	and the operation $ -\{-\} $ is the shuffle brace operation introduced in Section~\ref{Section: Cohomology and homotopy theory of Lie algebras}.	
	
	\item Let $n\geqslant 1$. For homogeneous elements $sh\in \Hom((sV)^{\odot n}, sV)\subset \mathfrak{C}_{\Lie}(V)$ and $g_1, \dots, g_n\in \mathfrak{C}_{\NjO}(V)$, for $1\leqslant k\leqslant n$, define $$l_{n+1}(g_1\ot \cdots\ot g_k\ot sh \ot g_{k+1}\ot \cdots\ot g_n)\in \mathfrak{C}_{\NjO}(V)$$ to be
	\begin{align*}
		l_{n+1}(g_1\ot \cdots\ot g_k\ot sh \ot g_{k+1}\ot \cdots\ot g_n):=(-1)^{(|h|+1)(\sum\limits_{j=1}^k|g_j|)+k}l_{n+1}(sh\ot g_1\ot \cdots \ot g_n),
	\end{align*}
	where the RHS has been introduced in \eqref{Defn. define l b+1}.	
	
	\item All other components of operators $\{l_n\}_{n\geqslant 1}$ vanish.
\end{enumerate}

	
	
With this $L_\infty$-algebra, we can use the Maurer-Cartan method to redefine homotopy Nijenhuis Lie algebras, which is equivalent to Definition~\ref{defn. homotopy Nijenhuis Lie algebra via dg operad} according to Proposition~\ref{Prop: Linfinity give MC S}, as follows:

\begin{defn} \label{Def: homotopy NjL algebras}
	Let $V$ be a graded space. A homotopy Nijenhuis Lie algebra structure (or a $\NL_\infty$-algebra structure) on $V$ is defined to be a Maurer-Cartan element in the $L_\infty$-algebra ${\frakC_{\NjL}}(V)$.
\end{defn}

Let's make the definition explicit. Given an element 
$
\alpha=(\{b_i\}_{i\geqslant 1}, \{R_i\}_{i\geqslant 1})\in {\mathfrak{C}_{\NjL}}(V)_{-1}
$
with $b_i:(sV)^{\odot i}\rightarrow sV$ and $R_i:(sV)^{\odot i}\rightarrow V$, then $\alpha$ satisfies the Maurer-Cartan equation if and only if for each $n\geqslant 1$, the following equations hold:
\begin{eqnarray} \label{Eq: L-infinity NjL}
	\sum_{i=1}^{n} b_{n-i+1}\{b_{i}\} 
	=0, 
\end{eqnarray}
\begin{eqnarray} \label{Eq: homotopy-NjL-operator NjL}
	\sum_{ r_1+\cdots+r_p =n \atop
		r_1, \dots, r_p \geqslant 1, p \geqslant 1 } \sum_{ t=0}^{p} (-1)^{t}
	(sR_{r_{1}})
	\{ \cdots
	\{
	sR_{r_{t}}
	\{b_{p}
	\{sR_{r_{t+1}}, \dots, sR_{r_{p}}\}
	\}
	\}
	\cdots \} = 0.
\end{eqnarray}

Note that, two fixed isomorphisms \eqref{Eq: first can isom} and \eqref{Eq: second can isom} give the equivalences between Equation~\eqref{Eq: L-infinity NjL} and Equation~\eqref{Eq: jacobi-id NjL}, and between Equation~\eqref{Eq: homotopy-NjL-operator NjL} and Equation~\eqref{Eq: homotopy NjL-operator-version-2}, when we denote
$$ \mu_{n} = \widetilde{b_{n}} = s^{-1} \circ b_{n} \circ s^{\ot n}: V^{\wg n} \rightarrow V 
\quad \mathrm{and} \quad 
P_{n} = \widehat{R_{n}} = R_{n} \circ s^{\ot n}: V^{\wg n} \rightarrow V.$$


\smallskip

\subsection{Realising Nijenhuis Lie algebra structures as Maurer-Cartan elements}\

In this subsection, we will see that the cohomology of Nijenhuis Lie algebras introduced in Subsection \ref{Subsec: cohomology NjL} can be obtained using the twisting procedure in the $L_\infty$-algebra introduced in previous subsection.

The following result can be directly verified by definition. It can also be deduced from Proposition~\ref{Prop: Linfinity give MC S}.

\begin{prop} \label{Prop: NjL is MC element} 
	Let $V$ be an ungraded space considered as a graded space concentrated in degree 0. Then a Nijenhuis Lie algebra structure on $V$ is equivalent to a Maurer-Cartan element in the $L_\infty$-algebra ${\mathfrak{C}_{\NjL}}(V)$.
\end{prop}

Let us justify the cohomology theory of Nijenhuis Lie algebras introduced in Section~\ref{Sect: Cohomology theory of Nijenhuis Lie algebras}.

\begin{prop} \label{Prop: cohomology complex as the underlying complex of L infinity algebra NjL}
	Let $\frakg = (\frakg, \mu, P)$ be a Nijenhuis Lie algebra. Twist the $L_\infty$-algebra $\mathfrak{C}_{\NjL}(\frakg)$ by the Maurer-Cartan element corresponding to the Nijenhuis Lie algebra structure, then its underlying complex is isomorphic to
	$s\C^\bullet_{\NjL}(\frakg)$, the shift of the cochain complex of Nijenhuis Lie algebra $\frakg$ defined in Subsection \ref{Subsec: cohomology NjL}.
\end{prop}

\begin{proof}
	By Proposition~\ref{Prop: NjL is MC element}, the Nijenhuis Lie algebra structure on $\frakg$ is equivalent to a Maurer-Cartan element $\alpha=(\nu, \tau)$ in the $L_\infty$-algebra $\mathfrak{C}_{\NjL}(\frakg)$ with
	$$\nu=-s\circ\mu\circ (s^{-1})^{\ot 2}: (s\frakg)^{\odot 2} \rightarrow s\frakg \quad \mathrm{and} \quad \tau=P\circ s^{-1}: s\frakg\rightarrow \frakg.$$
	Via twisting procedure Proposition~\ref{Prop: deformed-L-infty}, this Maurer-Cartan element induces a new $L_\infty$-algebra structure $\{l_n^\alpha\}_{n\geqslant 1}$ on the graded space $\mathfrak{C}_{\NjL}(\frakg)$.
	More precisely, for any $(sf, g)\in\Hom((s\frakg)^{\odot n}, s\frakg) \oplus \Hom((s\frakg)^{\odot n-1}, \frakg) \subset \mathfrak{C}_{\NjL}(\frakg)$,
	we have
	\begin{align*}
		l_1^\alpha(sf)
		&=l_{1}(sf) -l_2(\alpha \ot sf) + \sum_{i=2}^n (-1)^{\frac{i(i+1)}{2}} \frac{1}{i!} l_{i+1}(\alpha^{\ot i}\ot sf)\\
		&=-l_2(\nu \ot sf) -l_2(\tau \ot sf) + \sum_{i=2}^n (-1)^{\frac{i(i+1)}{2}} \frac{1}{i!} l_{i+1}(\tau^{\ot i}\ot sf)\\
		&=-[\nu, sf]_{\RNA} + (-1)^{\frac{n(n+1)}{2}} \frac{1}{n!} l_{n+1}(\tau^{\ot n}\ot sf).
	\end{align*}
	Note that $-[\nu, sf]_{\RNA}$ corresponds to
	$-(-1)^{n+1}\delta_{\Lie}(\widetilde{sf})$
	in $s\C^\bullet_\Lie(\frakg)$ under the fixed isomorphism (\ref{Eq: first can isom}). And 
	\begin{align*}
		(-1)^{\frac{n(n+1)}{2}} \frac{1}{n!} l_{n+1}(\tau^{\ot n}\ot sf)
		&= (-1)^{\frac{n(n+1)}{2}} (-1)^{n(|f|+1)+n} \frac{1}{n!} l_{n+1}( sf \ot \tau^{\ot n})\\
		&= \sum_{k=0}^{n} (-1)^{k} s^{-1} \underbrace{(s\tau) \circ \cdots \circ (s\tau)}_{k} \circ sf \{\underbrace{s\tau, \dots, s\tau}_{n-k}\}, 
	\end{align*}
	which corresponds to $\Psi(\widetilde{sf})$
	under the fixed isomorphism (\ref{Eq: second can isom}). We also have
	\begin{align*}
		l_1^{\alpha}(g)
		=&\ l_1(g) - l_2(\alpha \ot g) - \frac{1}{2!} l_3(\alpha \ot \alpha \ot g)\\
		=&\ -l_3(\nu \ot \tau\ot g)\\
		=&\ s^{-1}\nu \{s\tau, sg\} - \tau\{\nu\{sg\}\} + (-1)^{n} \tau \{sg\{\nu\}\} + s^{-1}\nu \{sg, s\tau\} - (-1)^{n}g\{\nu\{s\tau\}\} + (-1)^{n} g \{s\tau\{\nu\}\}, 
	\end{align*}
	which corresponds to $ (-1)^{n}\delta_{\NjO}(\widehat{g})$ 
	under the fixed isomorphism (\ref{Eq: second can isom}).
	
	In conclusion, 	we have that the underlying complex of the twisted $L_\infty$-algebra $\mathfrak{C}_{\NjL}(\frakg)$ by Maurer-Cartan element $\alpha$ is isomorphic to the cochain complex $s\C^\bullet_{\NjL}(\frakg)$. 
\end{proof}

Although $\mathfrak{C}_{\NjL}(\frakg)$ is an $L_\infty$-algebra, the next proposition shows that once the Lie algebra structure $\mu$ over $\frakg$ is fixed, the graded space $\mathfrak{C}_{\NjO}(\frakg)$, which, after twisting procedure, controls deformations of Nijenhuis operators, is a genuine graded Lie algebra. 

\begin{prop} \label{Prop: justifying Cohomology theory of Nijenhuis operator}
	Let $(\frakg, \mu)$ be a Lie algebra. Then we have the following statements:
	\begin{enumerate}
		\item The graded space $\mathfrak{C}_{\NjO}(\frakg)$ can be endowed with a graded Lie algebra structure, and a Nijenhuis operator on $(\frakg, \mu)$ is equivalent to a Maurer-Cartan element in this graded Lie algebra.
		
		\item Given a Nijenhuis operator $P$ on Lie algebra $(\frakg, \mu)$, the underlying complex of the twisted dg Lie algebra $\mathfrak{C}_{\NjO}(\frakg)$ by the corresponding Maurer-Cartan element is isomorphic to the cochain complex $\C_{\NjO}^\bullet(\frakg)$ of Nijenhuis operator $P$ defined in Subsection \ref{Subsect: cohomology NjL operator}.
	\end{enumerate}
\end{prop}

\begin{proof}
	(i) Consider $\frakg$ as a graded space concentrated in degree 0. Define
	$$\nu=- s\circ \mu\circ (s^{-1})^{\ot 2}: (s\frakg)^{\odot 2}\rightarrow s\frakg.$$
	Note that $\alpha=(\nu, 0)$ is naturally a Maurer-Cartan element in $L_\infty$-algebra $\mathfrak{C}_{\NjL}(\frakg)$.
	By the construction of $\{l_n\}_{n\geqslant 1}$ on $\mathfrak{C}_{\NjL}(\frakg)$, the graded subspace $\mathfrak{C}_{\NjO}(\frakg)$ is closed under the action of $\{l_n^\alpha\}_{n\geqslant 1}$. Note that $ l_1^\alpha \equiv 0 $ on $ \mathfrak{C}_{\NjO}(\frakg) $, and the restriction of $l_n^\alpha$ on $\mathfrak{C}_{\NjO}(\frakg)$ is $0$ for $n\geqslant 3$ because the arity of $\nu$ is 2. Thus $(\mathfrak{C}_{\NjO}(\frakg), l_2^\alpha)$ forms a graded Lie algebra.
	More explicitly, for any $f\in \Hom((s\frakg)^{\odot n}, \frakg)$, $g\in \Hom((s\frakg)^{\odot k}, \frakg)$, we have
	\begin{align*}
		l_2^\alpha(f\ot g)
		=&\ l_{2}(f, g)+l_{3}(\alpha, f, g)\\
		=&\ l_{3}(\nu, f, g)\\
		=&\ (-1)^{|f|}\big(s^{-1}\nu\{sf, sg\}-(-1)^{|f|+1} f\{\nu \{sg\} \} + (-1)^{|f|+1+|g|+1}f\{sg\{\nu\}\}\big)\\
		&+(-1)^{|f||g|+1+|g|}\big(s^{-1}\nu\{sg, sf\}-(-1)^{|g|+1} g\{\nu \{sf\} \}+ (-1)^{|g|+1+|f|+1}g\{sf\{\nu\}\}\big)\\
		=&\ (-1)^{n}s^{-1}\nu\{sf, sg\} + f\{\nu \{sg\}\} + (-1)^{k}f\{sg\{\nu\}\}\\
		&+(-1)^{nk+k+1}s^{-1}\nu\{sg, sf\}-(-1)^{nk} g\{\nu \{sf\} \}+ (-1)^{nk+n+1}g\{sf\{\nu\}\}.
	\end{align*}
	
	Since $\frakg$ is concentrated in degree 0, we have $\mathfrak{C}_{\NjO}(\frakg)_{-1}=\Hom(s\frakg, \frakg)$. Take an element $\tau\in \Hom(s\frakg, \frakg) $, then $\tau$
	satisfies the Maurer-Cartan equation
	$$-\frac{1}{2}l_2^{\alpha}(\tau\ot \tau)=0 $$
	if and only if
	$$s^{-1}\nu\{s\tau, s\tau\}-\tau\{\nu\{s\tau\}\}+\tau\{s\tau\{\nu\}\}=0.$$
	Define $P=\tau\circ s: \frakg \rightarrow \frakg$. The above equation can be rewritten as
	$$ \mu \circ (P \ot P) = P \circ \mu \circ (P \ot \id) + P \circ \mu \circ (\id \ot P) - P \circ P \circ \mu, $$
	which means that $P$ is a Nijenhuis operator on Lie algebra $(\frakg, \mu)$.
	
	(ii) Now let $P$ be a Nijenhuis operator on Lie algebra $(\frakg, \mu)$. By the first statement, $P$ corresponds to a Maurer-Cartan element $\beta=P \circ s^{-1}$ in the graded Lie algebra $(\mathfrak{C}_{\NjL}(\frakg), l_2^\alpha)$.
	For $f\in \Hom((s\frakg)^{\odot n}, \frakg) \subset \mathfrak{C}_{\NjO}(\frakg)$, we compute $(l_1^\alpha)^\beta(f)$ as follows:
	\begin{align*}
		(l_1^\alpha)^\beta(f)
		=&\ l_1^\alpha(f)-l_2^\alpha(\beta\ot f)\\
		=&\ -l_2^\alpha(\beta\ot f)\\
		=&\ s^{-1}\nu\{s\beta, sf\} - \beta\{\nu\{sf\}\} - (-1)^{n} \beta \{sf\{\nu\}\} + s^{-1}\nu\{sf, s\beta\} +(-1)^{n}f\{\nu\{s\beta\}\} - (-1)^{n} f \{s\beta\{\nu\}\}, 
	\end{align*}
	which corresponds to $ (-1)^{n+1}\delta_{\NjO}(\widehat{f})$ 
	under the fixed isomorphism (\ref{Eq: second can isom}). So
	the underlying complex of the twisted dg Lie algebra $\mathfrak{C}_{\NjO}(\frakg)$ by the corresponding Maurer-Cartan element $\beta$ is isomorphic to the cochain complex $\C_{\NjO}^\bullet(\frakg)$ of Nijenhuis operator $P$.
\end{proof}

\begin{remark}
	In particular, when one considers the well-known Lie algebra structure on the vector space $\mathfrak{X}(M)$ of vector fields of a manifold $M$ and a Nijenhuis operator $P$ on it, the classical dg Lie algebra structure on the Fr\"olicher-Nijenhuis complex of $P$ endowed with the Fr\"olicher-Nijenhuis bracket is a dg Lie subalgebra of the dg Lie algebra structure on $\mathfrak{C}_{\NjO}(\mathfrak{X}(M))$, introduced as above, which we will discuss in Section~\ref{The relation of cohomology of Nijenhuis operators between algebraic and geometric versions} in detail, see Theorem~\ref{Prop. relation between algebra and geometry Nijenhuis}.
\end{remark}

\bigskip

\part{Deformations of geometric Nijenhuis structures}\label{Part3}

\medskip

\section{Basic concepts of Lie algebroids} \label{Section: concepts of Lie algebroids}

In this section, we briefly review some basic concepts and facts about Lie algebroids from \cite{Marle2008,Xu1999}. In what follows all the vector bundles will be assumed to be locally trivial and of finite rank.

\begin{defn}
	Let $M$ be a smooth manifold. A Lie algebroid is a vector bundle $A$ over $M$ together with a Lie algebra structure $[-,-]_{A}$ on $\Gamma(A)$, the space of smooth sections, and a bundle map $\rho: A \to T_{M}$, called the anchor, extended to a map $\rho: \Gamma(A) \to \frakX(M) := \Gamma(T_{M})$, such that
	$$[X_{1}, fX_{2}]_{A} = f[X_{1}, X_{2}]_{A} + (\rho(X_{1})f)X_{2},$$
	for any $f \in \ci(M)$ and $X_{1}, X_{2} \in \Gamma(A)$. 
\end{defn}

\begin{remark}
	The definition above implies the fact that $\rho: \Gamma(A) \to \frakX(M)$ is a Lie algebra morphism, in other words, for any $X_{1}, X_{2} \in \Gamma(A)$, 
	$$\rho([X_{1}, X_{2}]_{A}) = [\rho(X_{1}), \rho(X_{2})].$$
\end{remark}

Denote the dual bundle of $A$ by $A^{\vee}$. 
We say that elements in $\Gamma(\wg^{\bullet}A^{\vee})$, the exterior algebra of $\Gamma(A^{\vee})$, are forms, and elements in $\Gamma(\wg^{\bullet}A)$, the exterior algebra of $\Gamma(A)$, are multivectors. 

\begin{defn}
	Let $A$ be a vector bundle over $M$. Assume that $X \in \Gamma(A)$ and $\theta \in \Gamma(\wg^{p}A^{\vee})$, $p \in \mathbb{Z}$, there exists $i_{X}\theta \in \Gamma(\wg^{p-1}A^{\vee})$, called interior product of $\theta$ by $X$, defined by the following formulas:
	\begin{enumerate}
		\item for $p \leqslant 0$, $i_{X}\theta = 0$;
		\item for $p = 1$, $i_{X}\theta = \langle \theta, X \rangle = \theta(X) \in \ci(M)$, which is given by duality coupling;
		\item for $p > 1$, $i_{X}\theta$ is the $(p-1)$-form on $A$, such that, for all $X_{1}, \dots, X_{p-1} \in \Gamma(A)$,
		$$(i_{X}\theta)(X_{1}, \dots, X_{p-1}) = \theta(X, X_{1}, \dots, X_{p-1}).$$
	\end{enumerate}
\end{defn}

For any $X \in \Gamma(A)$, the map $\theta \mapsto i_{X}\theta$ extends, by linearity, as a derivation of degree $-1$ of the exterior algebra $\Gamma(\wg^{\bullet}A^{\vee})$, that is, for any $\theta \in \Gamma(\wg^{p}A^{\vee})$, $\eta \in \Gamma(\wg^{q}A^{\vee})$,
$$i_{X}(\theta \wg \eta) = (i_{X}\theta) \wg \eta + (-1)^{p} \theta \wedge (i_{X}\eta) \in \Gamma(\wg^{p+q-1}A^{\vee}).$$

Let $\eta \in \Gamma(\wg^{p}A^{\vee})$ and $X \in \Gamma(\wg^{q}A)$. We set their pairing $\langle \eta, X \rangle$ as follows: if $p \neq q$, or if $p < 0$, or if $q < 0$, we set 
$\langle \eta, X \rangle = 0$; if $p = q = 0$, i.e., $\eta$ and $X$ are smooth functions on $M$, we set $\langle \eta, X \rangle = \eta X$; if $p = q \geqslant 1$ and $\eta = \eta^{1} \wg \cdots \wg \eta^{p}$, $X = X_{1} \wg \cdots \wg X_{p}$, we set $\langle \eta, X \rangle = \det(\langle \eta^{i}, X_{j} \rangle)$. One can see that $\langle \eta, X \rangle$ depends only on $X$ and $\eta$, not on the way in which they are expressed as exterior products of elements of degree $1$. The map $(\eta, X) \mapsto \langle \eta, X \rangle$ extends in a unique way as a $\ci(M)$-linear map $\Gamma(\wg^{\bullet}A^{\vee}) \ot \Gamma(\wg^{\bullet}A) \to \ci(M)$, still denoted by $(\eta, X) \mapsto \langle \eta, X \rangle$. 

In particular, let $\eta \in \Gamma(\wg^{p}A^{\vee})$, $X_{1}, \dots, X_{p} \in \Gamma(A)$. The pairing $\langle \eta, X_{1} \wg \cdots \wg X_{p} \rangle$ is related to the value taken by $\eta$, considered as a $p$-form on $\Gamma(A)$, on the set $(X_{1}, \dots, X_{p})$. We have 
$$\langle \eta, X_{1} \wg \cdots \wg X_{p} \rangle = \eta(X_{1}, \dots, X_{p}).$$

\begin{prop-def} \cite{Marle2008}
	Let $(A, [-,-]_{A}, \rho)$ be a Lie algebroid over $M$. For any $X \in \Gamma(A)$, there is a unique derivation of degree $0$ of $\Gamma(\wg^{\bullet}A^{\vee})$, called the Lie derivative with respect to $X$ and denoted by $\calLr_{X}$, which satisfies the following properties:
	\begin{enumerate}
		\item For any $f \in \ci(M)$, 
		$$\calLr_{X}f = i_{\rho(X)} \dd f = \langle \dd f, \rho(X)\rangle = \calL_{\rho(X)}f,$$ 
		where $\calL_{\rho(X)}$ denotes the classical Lie derivative with respect to vector field $\rho(X)$.
		\item For any form $\eta \in \Gamma(\wg^{p}A^{\vee})$ and $X_{1}, \dots, X_{p} \in \Gamma(A)$, $p \geqslant 1$,
		$$(\calLr_{X}\eta)(X_{1}, \dots, X_{p}) = \calLr_{X}(\eta(X_{1}, \dots, X_{p})) - \sum_{i=1}^{p} \eta(X_{1}, \dots, X_{i-1}, [X, X_{i}]_{A}, X_{i+1}, \dots, X_{p}).$$
	\end{enumerate}
\end{prop-def}

The next result shows that, for any $X \in \Gamma(A)$, the Lie derivative $\calLr_{X}$ can also be extended into a derivation of degree $0$ of $\Gamma(\wg^{\bullet}A)$.

\begin{prop-def} \cite{Marle2008}
	Let $(A, [-,-]_{A}, \rho)$ be a Lie algebroid over $M$. For any $X \in \Gamma(A)$, there is a unique derivation of degree $0$ of $\Gamma(\wg^{\bullet}A)$, called the Lie derivative with respect to $X$ and also denoted by $\calLr_{X}$, which satisfies the following properties:
	\begin{enumerate}
		\item For any $f \in \ci(M)$, 
		$$\calLr_{X}f = i_{\rho(X)} \dd f = \langle \dd f, \rho(X)\rangle = \calL_{\rho(X)}f.$$ 
		\item For any multivector $P \in \Gamma(\wg^{p}A)$ with $p \geqslant 1$, $\calLr_{X}P$ is the unique element in $\Gamma(\wg^{p}A)$ such that, for any $\eta \in \Gamma(\wg^{p}A^{\vee})$,
		$$\langle \eta, \calLr_{X} P\rangle = \calLr_{X} \langle \eta, P \rangle - \langle \calLr_{X}\eta, P\rangle.$$
	\end{enumerate}
\end{prop-def}

Note that, in above result, when $p = 1$, we obtain that $\calLr_{X}V = [X, V]_{A}$ for any $X, V \in \Gamma(A)$.

\begin{prop-def}\cite{Marle2008}
	Let $(A, [-,-]_{A}, \rho)$ be a Lie algebroid over $M$. There is a unique derivation of degree $1$ of $\Gamma(\wg^{\bullet}A^{\vee})$, called the exterior derivative and denoted by $\dd_{\rho}$, which satisfies the following properties:
	\begin{enumerate}
		\item For any $f \in \ci(M)$, $\dd_{\rho}f$ is the unique element in $\Gamma(A^{\vee})$ such that, for each $X \in \Gamma(A)$,
		$$\langle \dd_{\rho}f, X \rangle = \calLr_{X}f = \langle \dd f, \rho(X) \rangle.$$
		\item For any $\eta \in \Gamma(\wg^{p}A^{\vee})$ with $p \geqslant 1$, $\dd_{\rho} \eta$ is the unique element in $\Gamma(\wg^{p+1}A^{\vee})$ such that, for any $X_{1}, \dots, X_{p+1} \in \Gamma(A)$,
		\begin{align*}
			(\dd_{\rho} \eta)(X_{1}, \dots, X_{p+1}) 
			=&\ \sum_{i=1}^{p+1}(-1)^{i} \calLr_{X_{i}}(\eta(X_{1}, \dots, \widehat{X_{i}}, \dots, X_{p+1}))\\
			&+ \sum_{1 \leqslant i < j \leqslant p+1} (-1)^{i+j} \eta([X_{i}, X_{j}]_{A}, X_{1}, \dots, \widehat{X_{i}}, \dots, \widehat{X_{j}}, \dots, X_{p+1}).
		\end{align*}
	\end{enumerate}
\end{prop-def}

Direct verification yields $(\dd_{\rho})^{2} = 0$, so $\dd_{\rho}$ makes $\Gamma(\wg^{\bullet}A^{\vee})$ into a dg algebra, which defines the Lie algebroid cohomology \cite{Marle2008,Xu1999,Mackenzie1987,Mackenzie1994}.

Finally, for subsequent use, we simply list some of the properties of interior product, Lie derivative, and exterior derivative as follows, and more of their properties are in the references \cite{Marle2008}.

\begin{prop}\cite{Marle2008} \label{Prop. Some properties of Lie derivative}
	Let $(A, [-,-]_{A}, \rho)$ be a Lie algebroid over $M$ and $X, Y \in \Gamma(A)$.
	\begin{enumerate}
		\item In both case, we all have $\calLr_{[X,Y]_{A}} = [\calLr_{X}, \calLr_{Y}] = \calLr_{X} \circ \calLr_{Y} - \calLr_{Y} \circ \calLr_{X}$.
		\item For $f \in \ci(M)$ and $\eta \in \Gamma(\wg^{\bullet}A^{\vee})$, $$\calLr_{fX}\eta = f \calLr_{X}\eta + \dd_{\rho}f \wg i_{X}\eta.$$
		\item The interior product $i_{X}$, the exterior derivative $\dd_{\rho}$, and the Lie derivative $\calLr_{X}$ are related by the formula $$\calLr_{X} = [i_{X}, \dd_{\rho}] = i_{X} \dd_{\rho} + \dd_{\rho} i_{X}.$$
		\item The Lie derivative $\calLr_{X}$ and the exterior derivative $\dd_{\rho}$ are commutative, that is $$[\calLr_{X}, \dd_{\rho}] = \calLr_{X} \dd_{\rho} - \dd_{\rho} \calLr_{X} = 0.$$
	\end{enumerate}
\end{prop}

\medskip

\section{Nijenhuis Lie algebroids and cohomology theory of Nijenhuis operators on Lie algebroids} \label{Section: Nijenhuis Lie algebroids}

In classical Nijenhuis geometry, the space of vector-valued differential forms $\Omega^{\bullet}(M, T_M) = \Gamma(\wg^{\bullet}T_{M}^{\vee} \ot T_{M})$ is endowed with a graded Lie bracket of degree $0$, called the Fr\"olicher-Nijenhuis bracket and denoted by $[-, -]_{\FN}$, which is a generalization of the usual Lie bracket on vector fields \cite{Mic87, BK21-16, BKM22, DVM95}. The Maurer-Cartan elements of this graded Lie algebra are exactly Nijenhuis tensors on $M$. When one fixes a Nijenhuis tensor $P$, the underlying cochain complex of the twisted dg Lie algebra $(\Omega^{\bullet}(M, T_M), [-, -]_{\FN}, \dd_{\FN}=[P, -]_{\FN})$, is called the Fr\"olicher-Nijenhuis cochain complex of Nijenhuis tensor $P$, whose cohomology is called the Fr\"olicher-Nijenhuis cohomology of Nijenhuis tensor $P$ \cite{Mic87, BK21-16}. In this section, we develop the corresponding theory for general Lie algebroids, generalizing the classical theory. 

Let $A$ be a vector bundle over $M$. We say that elements in $\psi^{\bullet}(M,A) := \Gamma(\wg^{\bullet}A^{\vee} \ot A)$ are vector-valued forms. Specifically, elements in $\psi^{k}(M,A) := \Gamma(\wg^{k}A^{\vee} \ot A)$ are antisymmetric tensors of type $(1,k)$, which are also called $(1,k)$-forms, for simplicity.

\begin{defn} \label{Defn. Nijenhuis torsion}
	Let $(A, [-,-]_{A}, \rho)$ be a Lie algebroid over $M$, $P \in \psi^{1}(M,A) = \Gamma(A^{\vee} \ot A)$ be an $(1,1)$-form, viewed as a $\ci(M)$-linear map $P: \Gamma(A) \to \Gamma(A)$. Define the Nijenhuis torsion of $P$ as 
	$$\calN_{P}(X,Y) := [P(X),P(Y)]_{A} - P([P(X),Y]_{A}) - P([X,P(Y)]_{A}) + P^{2}([X,Y]_{A}),$$
	for any $X, Y \in \Gamma(A)$.
\end{defn}

\begin{prop} \label{Prop. Nijenhuis torsion is a tensor}
	The Nijenhuis torsion $\calN_{P}$ is a tensor of type $(1,2)$, moreover, is an $(1,2)$-form in $\psi^{2}(M,A)$.
\end{prop}
\begin{proof}
	For any $X, Y \in \Gamma(A)$, it is obvious that $\calN_{P}(X,Y) = -\calN_{P}(Y,X)$. So we only need to show that, for any $f \in \ci(M)$, $\calN_{P}(X,fY) = f\calN_{P}(X,Y)$.
	\begin{align*}
		\calN_{P}(X,fY)
		=&\ [P(X),P(fY)]_{A} - P([P(X),fY]_{A}) - P([X,P(fY)]_{A}) + P^{2}([X,fY]_{A})\\
		=&\ [P(X),fP(Y)]_{A} - P([P(X),fY]_{A}) - P([X,fP(Y)]_{A}) + P^{2}([X,fY]_{A})\\
		=&\ f[P(X),P(Y)]_{A} + (\rho(P(X))f)P(Y) - P(f[P(X),Y]_{A} + (\rho(P(X))f)Y)\\
		&\ - P(f[X,P(Y)]_{A} + (\rho(X)f)P(Y)) + P^{2}(f[X,Y]_{A} + (\rho(X)f)Y)\\
		=&\ f([P(X),P(Y)]_{A} - P([P(X),Y]_{A}) - P([X,P(Y)]_{A}) + P^{2}([X,Y]_{A}))\\
		=&\ f\calN_{P}(X,Y).
	\end{align*}
\end{proof}

In 1997, Grabowski and Urbanski \cite{Grabowski1997} introduced the concept of Nijenhuis operators on a pre-Lie algebroid which satisfies all conditions of a Lie algebroid except Jacobi identity that the bracket $[-,-]_{A}$ satisfies. In this paper, we will not consider such a broad concept for the time being, but only focus on Nijenhuis operators on Lie algebroids, which are defined as follows.

\begin{defn} \label{Defn. Nijenhuis Lie algebroid}
	Let $(A, [-,-]_{A}, \rho)$ be a Lie algebroid over $M$. An $(1,1)$-form $P \in \Gamma(A^{\vee} \ot A)$ is called a Nijenhuis tensor, or Nijenhuis operator, if its Nijenhuis torsion $\calN_{P}$ vanishes, that is,
	$$[P(X),P(Y)]_{A} = P([P(X),Y]_{A}) + P([X,P(Y)]_{A}) - P^{2}([X,Y]_{A}),$$
	for any $X, Y \in \Gamma(A)$.
	In this case, $(A, [-,-]_{A}, \rho, P)$ is called a Nijenhuis Lie algebroid over $M$.
\end{defn}


\begin{defn}
	Let $A_{1} = (A_{1}, [-,-]_{A_{1}}, \rho_{1}, P_{1})$ and $A_{2} = (A_{2}, [-,-]_{A_{2}}, \rho_{2}, P_{2})$ be two Nijenhuis Lie algebroids over same manifold $M$. 
	A Lie algebroid morphism $\phi: A_{1} \to A_{2}$ is call a Nijenhuis Lie algebroid morphism, or morphism of Nijenhuis Lie algebroids, if $ P_{2} \circ \phi = \phi \circ P_{1}$.
\end{defn}

The following result is similar to Proposition~\ref{Prop: Star Lie bracket}, and the proof is a direct verification of the definition, which is left to the reader.

\begin{prop}
	Let $A = (A, [-,-]_{A}, \rho, P)$ be a Nijenhuis Lie algebroid over $M$. Define a new binary operation as
	$$[X,Y]^{P}_{A} := [P(X),Y]_{A} + [X,P(Y)]_{A} - P[X,Y]_{A},$$
	for any $X, Y \in \Gamma(A)$.
	Then 
	\begin{enumerate}
		\item the pair $([-,-]^{P}_{A}, \rho P)$ defines a new Lie algebroid structure on $A$;
		\item the tuple $(A, [-,-]^{P}_{A}, \rho P, P)$ also forms a Nijenhuis Lie algebroids and denote it by $A_{P}$;
		\item the map $P: A_{P} \to A$ is a Nijenhuis Lie algebroid morphism.
	\end{enumerate}
\end{prop}

From now on, we always assume that $(A, [-,-]_{A}, \rho)$ is a Lie algebroid with vector bundle $A$ of rank $n$ over smooth manifold $M$ of dimension $m$.
For any $-1 \leqslant d \leqslant n-1$, denote $\Der^{d}(M,A)$ the space of all derivations of exterior algebra $\Gamma(\wg^{\bullet}A^{\vee})$ of degree $d$, and denote
$$\Der^{\bullet}(M,A) := \bigoplus_{d=-1}^{n-1} \Der^{d}(M,A).$$

Consider any derivation $D$ of degree $-1$. By definition, $D: \ci(M) = \Gamma(\wg^{0}A^{\vee}) \to 0$, so $D(f \alpha) = fD(\alpha)$ for any $\alpha \in \Gamma(\wg^{\bullet}A^{\vee})$ and $f \in \ci(M)$.
This means $D: \Gamma(A^{\vee}) \to \ci(M)$ is a $\ci(M)$-linear map. So $D$ corresponds to a tensor of type $(1,0)$, i.e., a vector in $\Gamma(A)$, says $X$. Thus $D = i_{X}$ on $\ci(M)$ and $\Gamma(A^{\vee})$. By the properties of derivations, this implies that $D = i_{X}$ as operators on the entire $\Gamma(\wg^{\bullet}A^{\vee})$.
Note that, the map $X \mapsto i_{X}$ is injective, so we have $\Der^{-1}(M,A) \cong \Gamma(A)$ as vector spaces.

\begin{defn}
	A derivation $D \in \Der^{d}(M,A)$ is called an algebraic derivation if $D$ takes $\ci(M)$ to $0$, i.e., $D$ is $\ci(M)$-linear.
\end{defn}

Let $D$ be an algebraic derivation of degree $d$. It is not difficult to see that the map $D: \Gamma(A^{\vee}) \to \Gamma(\wg^{d+1}A^{\vee})$ is a tensor map, i.e., is $\ci(M)$-linear. Thus, there exists a unique $(1,d+1)$-form $K \in \psi^{d+1}(M,A) = \Gamma(\wg^{d+1}A^{\vee} \ot A)$ such that
$$(D\alpha)(X_{1}, \dots, X_{d+1}) = \alpha(K(X_{1}, \dots, X_{d+1})),$$
for any $\alpha \in \Gamma(A^{\vee})$ and $X_{1}, \dots, X_{d+1} \in \Gamma(A)$.
Denote the operator $D$ on $\ci(M)$ and $\Gamma(A^{\vee})$ by $i_{K}$, then
$$(i_{K}\alpha)(X_{1}, \dots, X_{d+1}) = \alpha(K(X_{1}, \dots, X_{d+1})).$$
We also denote $D$ on the entire $\Gamma(\wg^{\bullet}A^{\vee})$ by $i_{K}$.
On the other hand, any elements in $\psi^{d+1}(M,A)$ can define an algebraic derivation as above. So the space of algebraic derivations of degree $d$ is isomorphic to $\psi^{d+1}(M,A)$.

We can give more specific characterizations of $i_{K}$, as follows.
\begin{prop}
	Let $K = \omega \ot X \in \psi^{d+1}(M,A)$ and $\alpha \in \Gamma(A^{\vee})$. We have
	$$i_{K}\alpha = \omega \wg i_{X}\alpha = \langle \alpha, X \rangle \omega \in \Gamma(\wg^{d+1}A^{\vee}).$$
\end{prop}
\begin{proof} For any $X_{1}, \dots, X_{d+1} \in \Gamma(A)$,
	\begin{align*}
		(i_{K}\alpha)(X_{1}, \dots, X_{d+1}) 
		=&\ \alpha(K(X_{1}, \dots, X_{d+1})) \\
		=&\ \alpha(\langle \omega, X_{1} \wg \cdots \wg X_{d+1} \rangle X) \\
		=&\ \langle \alpha, X \rangle \langle \omega, X_{1} \wg \cdots \wg X_{d+1} \rangle \\
		=&\ (\langle \alpha, X \rangle\omega)(X_{1}, \dots, X_{d+1}).
	\end{align*}
\end{proof}

\begin{prop} \label{Prop. formula of iK 1}
	Let $K = \omega \ot X \in \psi^{d+1}(M,A)$ and $\beta \in \Gamma(\wg^{k}A^{\vee})$. We have
	$$i_{K}\beta = \omega \wg i_{X}\beta \in \Gamma(\wg^{k+d}A^{\vee}).$$
\end{prop}
\begin{proof}
	By linearity, we may assume $\beta = \beta^{1} \wg \cdots \wg \beta^{k}$. Since $i_{K}$ is a derivation of degree $d$, we have
	\begin{align*}
		i_{K}\beta 
		=&\ \sum_{i=1}^{k}(-1)^{(i-1)d} \beta^{1} \wg \cdots \wg i_{K}\beta^{i} \wg \cdots \wg \beta^{k}\\
		=&\ \sum_{i=1}^{k}(-1)^{(i-1)d} \beta^{1} \wg \cdots \wg \omega \wg i_{X}\beta^{i} \wg \cdots \wg \beta^{k}\\
		=&\ \sum_{i=1}^{k}(-1)^{i-1} \omega \wg \beta^{1} \wg \cdots \wg i_{X}\beta^{i} \wg \cdots \wg \beta^{k}\\
		=&\ \omega \wg i_{X}\beta.
	\end{align*}
\end{proof}

\begin{prop} \label{Prop. formula of iK 2}
	Let $K = \omega \ot X \in \psi^{d+1}(M,A)$ and $\beta \in \Gamma(\wg^{k}A^{\vee})$. For any $X_{1}, \dots, X_{k+d} \in \Gamma(A)$, we have
	$$(i_{K}\beta)(X_{1}, \dots, X_{k+d}) = \sum_{\sigma \in \Sh(d+1,k-1)} \sgn(\sigma) \beta(K(X_{\sigma(1)}, \dots, X_{\sigma(d+1)}), X_{\sigma(d+2)}, \dots, X_{\sigma(d+k)}).$$
\end{prop}
\begin{proof}
	By Proposition~\ref{Prop. formula of iK 1}, we have $i_{K}\beta = \omega \wg i_{X}\beta$. Therefore,
	\begin{align*}
		(i_{K}\beta)(X_{1}, \dots, X_{k+d}) 
		=&\ (\omega \wg i_{X}\beta) (X_{1}, \dots, X_{k+d})\\
		=&\ \sum_{\sigma \in \Sh(d+1,k-1)} \sgn(\sigma) \omega(X_{\sigma(1)}, \dots, X_{\sigma(d+1)}) (i_{X}\beta)(X_{\sigma(d+2)}, \dots, X_{\sigma(d+k)})\\
		=&\ \sum_{\sigma \in \Sh(d+1,k-1)} \sgn(\sigma) \omega(X_{\sigma(1)}, \dots, X_{\sigma(d+1)}) \beta(X, X_{\sigma(d+2)}, \dots, X_{\sigma(d+k)})\\
		=&\ \sum_{\sigma \in \Sh(d+1,k-1)} \sgn(\sigma) \beta(\omega(X_{\sigma(1)}, \dots, X_{\sigma(d+1)}) X, X_{\sigma(d+2)}, \dots, X_{\sigma(d+k)})\\
		=&\ \sum_{\sigma \in \Sh(d+1,k-1)} \sgn(\sigma) \beta(K(X_{\sigma(1)}, \dots, X_{\sigma(d+1)}), X_{\sigma(d+2)}, \dots, X_{\sigma(d+k)}).
	\end{align*}
\end{proof}

The space $\Der^{\bullet}(M,A)$ of all derivations of $\Gamma(\wg^{\bullet}A^{\vee})$ is equipped with a natural Lie algebra structure, which is given by the graded commutator $[-,-]$. Note that, the commutator of any two algebraic derivations is also an algebraic derivation. Therefore, the space of all algebraic derivations $\psi^{\bullet}(M,A) \subset \Der^{\bullet}(M,A)$, endowed with the restriction of the Lie bracket, also forms a Lie algebra. This induces an operation on $\psi^{\bullet}(M,A)$, as follows. 

For any $K \in \psi^{k}(M,A)$ and $L \in \psi^{l}(M,A)$, define the Richardson-Nijenhuis bracket of $K$ and $L$, denoted by $[K,L]_{\RN} \in \psi^{k+l-1}(M,A)$, as
$$i_{[K,L]_{\RN}} := [i_{K},i_{L}] = i_{K}i_{L}-(-1)^{(k-1)(l-1)}i_{L}i_{K}.$$ 
We have a direct formula of Richardson-Nijenhuis bracket.

\begin{prop}
	Let $K = \omega \ot X \in \psi^{k}(M,A)$ and $L = \theta \ot Y \in \psi^{l}(M,A)$. We have
	$$[K,L]_{\RN} = i_{K}L - (-1)^{(k-1)(l-1)}i_{L}K,$$
	where the operation $i_{K}L$ is given by
	$$i_{K}L:= i_{K}\theta \ot Y = \omega \wg i_{X}\theta \ot Y \in \psi^{k+l-1}(M,A).$$
\end{prop}
\begin{proof}
	Since $i_{[K,L]_{\RN}}$ is an algebraic derivation, so we only need to check the formula on $\Gamma(A^{\vee})$. For any $\alpha \in \Gamma(A^{\vee})$, by definition, we have
	$$i_{K}i_{L}\alpha = i_{K}(\theta \wg i_{Y}\alpha) = i_{Y}\alpha i_{K}\theta = i_{i_{K}\theta \ot Y}\alpha = i_{i_{K}L}\alpha.$$
	Therefore,
	$$i_{[K,L]_{\RN}}\alpha 
	= i_{K}i_{L}\alpha -(-1)^{(k-1)(l-1)} i_{L}i_{K}\alpha 
	= i_{i_{K}L}\alpha -(-1)^{(k-1)(l-1)} i_{i_{L}K}\alpha.$$
	This implies the result what we want.
\end{proof}

Here is a direct corollary of Proposition~\ref{Prop. formula of iK 2}.
\begin{cor}
	Let $K \in \psi^{k}(M,A)$, $L \in \psi^{l}(M,A)$ and $X_{1},\dots,X_{k+l-1} \in \Gamma(A)$. We have
	$$i_{K}L(X_{1},\dots,X_{k+l-1}) = \sum_{\sigma \in \Sh(k,l-1)} \sgn(\sigma) L(K(X_{\sigma(1)},\dots,X_{\sigma(k)}), X_{\sigma(k+1)},\dots,X_{\sigma(k+l-1)}),$$
	and
	\begin{align*}
		[K,L]_{\RN} (X_{1},\dots,X_{k+l-1})
		=&\ \sum_{\sigma \in \Sh(k,l-1)} \sgn(\sigma) L(K(X_{\sigma(1)},\dots,X_{\sigma(k)}), X_{\sigma(k+1)},\dots,X_{\sigma(k+l-1)})\\
		&-(-1)^{(k-1)(l-1)} \sum_{\sigma \in \Sh(l,k-1)} \sgn(\sigma) K(L(X_{\sigma(1)},\dots,X_{\sigma(l)}), X_{\sigma(l+1)},\dots,X_{\sigma(k+l-1)}).
	\end{align*}
\end{cor}

\begin{remark}
	View $K \in \Hom_{\ci(M)}(\Gamma(A)^{\wg k}, \Gamma(A))$ and $L \in \Hom_{\ci(M)}(\Gamma(A)^{\wg l}, \Gamma(A))$. When compared with the definition of the algebraic version of the Richardson-Nijenhuis bracket $[-,-]_{\RNA}$, see Equation~\eqref{Eq: Richardson-Nijenhuis}, we have $[K,L]_{\RN} = [L,K]_{\RNA}$.
\end{remark}




\begin{defn}
	Let $K \in \psi^{d+1}(M,A)$. Define the Lie derivative on $\psi^{\bullet}(M,A)$ along $K$ as
	$$\calLr_{K} := [i_{K}, \dd_{\rho}] = i_{K} \dd_{\rho} - (-1)^{d}\dd_{\rho} i_{K},$$
	which is of degree $d+1$. 

	In particular, for $d=-1$, $X \in \psi^{0}(M,A) = \Gamma(A)$, we get the aforementioned definition of Lie derivative alone $X \in \Gamma(A)$.
\end{defn}

Since $(\dd_{\rho})^{2} = 0$, i.e., $[\dd_{\rho}, \dd_{\rho}] = 0$, by the Jacobi identity, we obtain the following corollary.

\begin{cor}
	For any $K \in \psi^{\bullet}(M,A)$, we have $[\calLr_{K},\dd_{\rho}] = 0$.
\end{cor}

The following are more direct characterizations of Lie derivative.

\begin{prop}
	For any $K \in \omega \ot X \in \psi^{k}(M,A)$ and $\alpha \in \Gamma(\wg^{l}A^{\vee})$, we have
	$$\calLr_{K}\alpha = \omega \wg \calLr_{X}\alpha + (-1)^{k} \dd_{\rho}\omega \wg i_{X}\alpha.$$
\end{prop}

\begin{proof}
	By definition,
	\begin{align*}
		\calLr_{K}\alpha 
		=&\ i_{K} \dd_{\rho} \alpha - (-1)^{k-1}\dd_{\rho} i_{K} \alpha\\
		=&\ \omega \wg i_{X}(\dd_{\rho} \alpha) - (-1)^{k-1}\dd_{\rho}(\omega \wg i_{X}\alpha)\\
		=&\ \omega \wg i_{X}(\dd_{\rho} \alpha) - (-1)^{k-1} \dd_{\rho}\omega \wg i_{X}\alpha + \omega \wg \dd_{\rho}(i_{X}\alpha)\\
		=&\ \omega \wg \calLr_{X} \alpha + (-1)^{k} \dd_{\rho} \omega \wg i_{X}\alpha.
	\end{align*}
\end{proof}

\begin{prop}
	For any $K \in \omega \ot X \in \psi^{k}(M,A)$, $\alpha \in \Gamma(\wg^{l}A^{\vee})$, and $X_{1}, \dots, X_{k+l} \in \Gamma(A)$, we have
	\begin{align*}
	(\calLr_{K}\alpha) (X_{1}, \dots, X_{k+l})
	=&\ \sum_{\sigma \in \Sh(k,l)} \sgn(\sigma) \calLr_{K(X_{\sigma(1)}, \dots, X_{\sigma(k)})}(\alpha(X_{\sigma(k+1)},\dots,X_{\sigma(k+l)}))\\
	&\ - \sum_{\sigma \in \Sh(k,1,l-1)} \sgn(\sigma) \alpha([K(X_{\sigma(1)}, \dots, X_{\sigma(k)}),X_{\sigma(k+1)}]_{A},X_{\sigma(k+2)},\dots,X_{\sigma(k+l)})\\
	&\ -(-1)^{k} \sum_{\sigma \in \Sh(2,k-1,l-1)} \sgn(\sigma) \alpha(K([X_{\sigma(1)}, X_{\sigma(2)}]_{A},X_{\sigma(3)},\dots,X_{\sigma(k+1)}),X_{\sigma(k+2)},\dots,X_{\sigma(k+l)}).
	\end{align*}
\end{prop}

\begin{proof}
	Based on the previous proposition, we obtain
	\begin{align*}
		&\ (\calLr_{K}\alpha) (X_{1}, \dots, X_{k+l}) \\
		=&\ (\omega \wg \calLr_{X}\alpha) (X_{1}, \dots, X_{k+l}) + (-1)^{k} (\dd_{\rho}\omega \wg i_{X}\alpha) (X_{1}, \dots, X_{k+l})\\
		=&\ \sum_{\sigma \in \Sh(k,l)} \sgn(\sigma) 
		\omega(X_{\sigma(1)},\dots,X_{\sigma(k)}) (\calLr_{X}\alpha)(X_{\sigma(k+1)},\dots,X_{\sigma(k+l)})\\
		&\ + (-1)^{k} \sum_{\sigma \in \Sh(k+1,l-1)} \sgn(\sigma)
		(\dd_{\rho}\omega)(X_{\sigma(1)},\dots,X_{\sigma(k+1)}) (i_{X}\alpha)(X_{\sigma(k+2)},\dots,X_{\sigma(k+l)})\\
		=&\ \sum_{\sigma \in \Sh(k,l)} \sgn(\sigma) 
		\omega(X_{\sigma(1)},\dots,X_{\sigma(k)}) \calLr_{X}(\alpha(X_{\sigma(k+1)},\dots,X_{\sigma(k+l)}))\\
		&\ - \sum_{\sigma \in \Sh(k,l)} \sgn(\sigma) \sum_{\theta \in \Sh(1,l-1)} \sgn(\theta) 
		\omega(X_{\sigma(1)},\dots,X_{\sigma(k)}) \alpha([X,X_{\sigma(k+\theta(1))}]_{A},X_{\sigma(k+\theta(2))},\dots,X_{\sigma(k+\theta(l))})\\
		&\ + (-1)^{k} \sum_{\sigma \in \Sh(k+1,l-1)} \sgn(\sigma) \sum_{\theta \in \Sh(1,k)} \sgn(\theta)
		\calLr_{X_{\sigma\theta(1)}}(\omega(X_{\sigma\theta(2)},\dots,X_{\sigma\theta(k+1)})) \alpha(X, X_{\sigma(k+2)},\dots,X_{\sigma(k+l)})\\
		&\ - (-1)^{k} \sum_{\sigma \in \Sh(k+1,l-1)} \sgn(\sigma) \sum_{\theta \in \Sh(2,k-1)} \sgn(\theta)
		\omega([X_{\sigma\theta(1)},X_{\sigma\theta(2)}]_{A},X_{\sigma\theta(3)},\dots,X_{\sigma\theta(k+1)}) \alpha(X, X_{\sigma(k+2)},\dots,X_{\sigma(k+l)})\\
		=&\ \sum_{\sigma \in \Sh(k,l)} \sgn(\sigma) 
		\omega(X_{\sigma(1)},\dots,X_{\sigma(k)}) \calLr_{X}(\alpha(X_{\sigma(k+1)},\dots,X_{\sigma(k+l)}))\\
		&\ - \sum_{\sigma \in \Sh(k,1,l-1)} \sgn(\sigma) 
		\omega(X_{\sigma(1)},\dots,X_{\sigma(k)}) \alpha([X,X_{\sigma(k+1)}]_{A},X_{\sigma(k+2)},\dots,X_{\sigma(k+l)})\\
		&\ + (-1)^{k} \sum_{\sigma \in \Sh(1,k,l-1)} \sgn(\sigma) 
		\calLr_{X_{\sigma(1)}}(\omega(X_{\sigma(2)},\dots,X_{\sigma(k+1)})) \alpha(X, X_{\sigma(k+2)},\dots,X_{\sigma(k+l)})\\
		&\ - (-1)^{k} \sum_{\sigma \in \Sh(2,k-1,l-1)} \sgn(\sigma) 
		\omega([X_{\sigma(1)},X_{\sigma(2)}]_{A},X_{\sigma(3)},\dots,X_{\sigma(k+1)}) \alpha(X, X_{\sigma(k+2)},\dots,X_{\sigma(k+l)})\\
		=&\ \sum_{\sigma \in \Sh(k,l)} \sgn(\sigma) 
		\calLr_{\omega(X_{\sigma(1)},\dots,X_{\sigma(k)}) X}(\alpha(X_{\sigma(k+1)},\dots,X_{\sigma(k+l)}))\\
		&\ - \sum_{\sigma \in \Sh(k,1,l-1)} \sgn(\sigma) 
		\alpha(\omega(X_{\sigma(1)},\dots,X_{\sigma(k)})[X,X_{\sigma(k+1)}]_{A},X_{\sigma(k+2)},\dots,X_{\sigma(k+l)})\\
		&\ + \sum_{\sigma \in \Sh(k,1,l-1)} \sgn(\sigma) 
		\alpha(\calLr_{X_{\sigma(k+1)}}(\omega(X_{\sigma(1)},\dots,X_{\sigma(k)})) X, X_{\sigma(k+2)},\dots,X_{\sigma(k+l)})\\
		&\ - (-1)^{k} \sum_{\sigma \in \Sh(2,k-1,l-1)} \sgn(\sigma) 
		\alpha(\omega([X_{\sigma(1)},X_{\sigma(2)}]_{A},X_{\sigma(3)},\dots,X_{\sigma(k+1)}) X, X_{\sigma(k+2)},\dots,X_{\sigma(k+l)})\\
		=&\ \sum_{\sigma \in \Sh(k,l)} \sgn(\sigma) 
		\calLr_{\omega(X_{\sigma(1)},\dots,X_{\sigma(k)}) X}(\alpha(X_{\sigma(k+1)},\dots,X_{\sigma(k+l)}))\\
		&\ - \sum_{\sigma \in \Sh(k,1,l-1)} \sgn(\sigma) 
		\alpha([\omega(X_{\sigma(1)},\dots,X_{\sigma(k)}) X,X_{\sigma(k+1)}]_{A},X_{\sigma(k+2)},\dots,X_{\sigma(k+l)})\\
		&\ - (-1)^{k} \sum_{\sigma \in \Sh(2,k-1,l-1)} \sgn(\sigma) 
		\alpha(\omega([X_{\sigma(1)},X_{\sigma(2)}]_{A},X_{\sigma(3)},\dots,X_{\sigma(k+1)}) X, X_{\sigma(k+2)},\dots,X_{\sigma(k+l)})\\
		=&\ \sum_{\sigma \in \Sh(k,l)} \sgn(\sigma) 
		\calLr_{K(X_{\sigma(1)},\dots,X_{\sigma(k)})}(\alpha(X_{\sigma(k+1)},\dots,X_{\sigma(k+l)}))\\
		&\ - \sum_{\sigma \in \Sh(k,1,l-1)} \sgn(\sigma) 
		\alpha([K(X_{\sigma(1)},\dots,X_{\sigma(k)}),X_{\sigma(k+1)}]_{A},X_{\sigma(k+2)},\dots,X_{\sigma(k+l)})\\
		&\ - (-1)^{k} \sum_{\sigma \in \Sh(2,k-1,l-1)} \sgn(\sigma) 
		\alpha(K([X_{\sigma(1)},X_{\sigma(2)}]_{A},X_{\sigma(3)},\dots,X_{\sigma(k+1)}), X_{\sigma(k+2)},\dots,X_{\sigma(k+l)}).
		\end{align*}
\end{proof}

\begin{remark}
	In classical Nijenhuis geometry, i.e., when considering the trivial Lie algebroid $(T_{M},[-,-],\id)$, any derivation on the space of differential forms $\Omega^{\bullet}(M) = \Gamma(\wg^{\bullet}T_{M}^{\vee})$ can be uniquely decomposed into a Lie derivative and an algebraic derivation, and the map, $K \mapsto \calL^{\id}_{K}$, is injective. Therefore, this induces a Lie bracket on the space of vector-valued differential forms $\Omega^{\bullet}(M,T_{M}) = \Gamma(\wg^{\bullet}T_{M}^{\vee} \ot T_{M})$, called the Fr\"olicher-Nijenhuis bracket \cite{Mic87, BK21-16}. However, in the general case, a derivation may not be able to be decomposed into a Lie derivative and an algebraic derivation, and the map $K \mapsto \calLr_{K}$ may not be an injective. So the classic approach, which induces the Fr\"olicher-Nijenhuis bracket, doesn't work. But, we can define the Fr\"olicher-Nijenhuis bracket directly as follows.
\end{remark}

\begin{defn} \label{Defn. definition of FN bracket}
	Let $K = \alpha \ot X \in \psi^{k}(M,A)$ and $L = \beta \ot Y \in \psi^{l}(M,A)$. Define the Fr\"olicher-Nijenhuis bracket $[K,L]_{\FN} \in \psi^{k+l}(M,A)$ as:
	\begin{align*}
		[K,L]_{\FN} 
		:=&\ \alpha \wg \beta \ot [X,Y]_{A} + \alpha \wg (\calLr_{X} \beta) \ot Y - (\calLr_{Y} \alpha) \wg \beta \ot X  +(-1)^{k} \dd_{\rho}\alpha \wg i_{X}\beta \ot Y + (-1)^{k} i_{Y}\alpha \wg \dd_{\rho}\beta \ot X \\
		=&\ \alpha \wg \beta \ot [X,Y]_{A} + \calLr_{K} \beta \ot Y - (-1)^{kl} \calLr_{L} \alpha \ot X.
	\end{align*}
\end{defn}

It is clear that, for $K \in \psi^{k}(M,A)$ and $L \in \psi^{l}(M,A)$, we have $[K,L]_{\FN} = -(-1)^{kl}[L,K]_{\FN}$. We will later indirectly prove that the Fr\"olicher-Nijenhuis bracket defines a graded Lie algebra structure on the graded space $\psi^{\bullet}(M,A)$. For details, refer to Proposition~\ref{Prop. relation between algebra and geometry Nijenhuis}. Before that, we need to present the following result:

\begin{prop}
	Let $K \in \psi^{k}(M,A)$, $L \in \psi^{l}(M,A)$ and $X_{1}, \dots, X_{k+l} \in \Gamma(A)$. We have the following formula:
	\begin{align*}
		[K,L]_{\FN}(X_{1}, \dots, X_{k+l})
		=&\ \sum_{\sigma \in \Sh(k,l)} \sgn(\sigma) 
		[K(X_{\sigma(1)}, \dots, X_{\sigma(k)}),L(X_{\sigma(k+1)}, \dots, X_{\sigma(k+l)})]_{A}\\
		&\ - \sum_{\sigma \in \Sh(k,1,l-1)} \sgn(\sigma) 
		L([K(X_{\sigma(1)}, \dots, X_{\sigma(k)}),X_{\sigma(k+1)}]_{A}, X_{\sigma(k+2)}, \dots, X_{\sigma(k+l)})\\
		&\ + (-1)^{kl} \sum_{\sigma \in \Sh(l,1,k-1)} \sgn(\sigma)
		K([L(X_{\sigma(1)}, \dots, X_{\sigma(l)}),X_{\sigma(l+1)}]_{A}, X_{\sigma(l+2)}, \dots, X_{\sigma(k+l)})\\
		&\ - (-1)^{k} \sum_{\sigma \in \Sh(2,k-1,l-1)} \sgn(\sigma) 
		L(K([X_{\sigma(1)},X_{\sigma(2)}]_{A},X_{\sigma(3)},\dots,X_{\sigma(k+1)}), X_{\sigma(k+2)}, \dots, X_{\sigma(k+l)})\\
		&\ + (-1)^{(k-1)l} \sum_{\sigma \in \Sh(2,l-1,k-1)} \sgn(\sigma)
		K(L([X_{\sigma(1)},X_{\sigma(2)}]_{A},X_{\sigma(3)},\dots,X_{\sigma(l+1)}),X_{\sigma(l+2)}, \dots, X_{\sigma(k+l)}).
	\end{align*}
\end{prop}

\begin{proof}
	Assume that $K = \alpha \ot X \in \psi^{k}(M,A)$ and $L = \beta \ot Y \in \psi^{l}(M,A)$, we have
	\begin{align*}
		&\ [K,L]_{\FN}(X_{1}, \dots, X_{k+l}) \\
		=&\ \big( \alpha \wg \beta \ot [X,Y]_{A} + \alpha \wg (\calLr_{X} \beta) \ot Y - (\calLr_{Y} \alpha) \wg \beta \ot X 
		+(-1)^{k} \dd_{\rho}\alpha \wg i_{X}\beta \ot Y + (-1)^{k} i_{Y}\alpha \wg \dd_{\rho}\beta \ot X \big) (X_{1}, \dots, X_{k+l})\\
		=&\ \sum_{\sigma \in \Sh(k,l)} \sgn(\sigma) 
		\alpha(X_{\sigma(1)}, \dots, X_{\sigma(k)}) \beta(X_{\sigma(k+1)}, \dots, X_{\sigma(k+l)}) [X,Y]_{A}\\
		&\ + \sum_{\sigma \in \Sh(k,l)} \sgn(\sigma)
		\alpha(X_{\sigma(1)}, \dots, X_{\sigma(k)}) (\calLr_{X}\beta)(X_{\sigma(k+1)}, \dots, X_{\sigma(k+l)}) Y\\
		&\ - \sum_{\sigma \in \Sh(k,l)} \sgn(\sigma)
		(\calLr_{Y}\alpha)(X_{\sigma(1)}, \dots, X_{\sigma(k)}) \beta(X_{\sigma(k+1)}, \dots, X_{\sigma(k+l)}) X\\
		&\ + (-1)^{k} \sum_{\sigma \in \Sh(k+1,l-1)} \sgn(\sigma)
		(\dd_{\rho}\alpha)(X_{\sigma(1)}, \dots, X_{\sigma(k+1)}) (i_{X}\beta)(X_{\sigma(k+2)}, \dots, X_{\sigma(k+l)}) Y\\
		&\ + (-1)^{k} \sum_{\sigma \in \Sh(k-1,l+1)} \sgn(\sigma)
		(i_{Y}\alpha)(X_{\sigma(1)}, \dots, X_{\sigma(k-1)}) (\dd_{\rho}\beta)(X_{\sigma(k)}, \dots, X_{\sigma(k+l)}) X\\
		=&\ \sum_{\sigma \in \Sh(k,l)} \sgn(\sigma) 
		\alpha(X_{\sigma(1)}, \dots, X_{\sigma(k)}) \beta(X_{\sigma(k+1)}, \dots, X_{\sigma(k+l)}) [X,Y]_{A}\\
		&\ + \sum_{\sigma \in \Sh(k,l)} \sgn(\sigma)
		\alpha(X_{\sigma(1)}, \dots, X_{\sigma(k)}) \calLr_{X}(\beta(X_{\sigma(k+1)}, \dots, X_{\sigma(k+l)})) Y\\
		&\ - \sum_{\sigma \in \Sh(k,l)} \sgn(\sigma) \sum_{\theta \in \Sh(1,l-1)} \sgn(\theta)
		\alpha(X_{\sigma(1)}, \dots, X_{\sigma(k)}) \beta([X,X_{\sigma(k+\theta(1))}]_{A}, X_{\sigma(k+\theta(2))}, \dots, X_{\sigma(k+\theta(l))}) Y\\
		&\ - \sum_{\sigma \in \Sh(k,l)} \sgn(\sigma)
		\calLr_{Y}(\alpha(X_{\sigma(1)}, \dots, X_{\sigma(k)})) \beta(X_{\sigma(k+1)}, \dots, X_{\sigma(k+l)}) X\\
		&\ + \sum_{\sigma \in \Sh(k,l)} \sgn(\sigma) \sum_{\theta \in \Sh(1,k-1)} \sgn(\theta)
		\alpha([Y,X_{\sigma\theta(1)}]_{A}, X_{\sigma\theta(2)}, \dots, X_{\sigma\theta(k)}) \beta(X_{\sigma(k+1)}, \dots, X_{\sigma(k+l)}) X\\
		&\ + (-1)^{k} \sum_{\sigma \in \Sh(k+1,l-1)} \sgn(\sigma) \sum_{\theta \in \Sh(1,k)} \sgn(\theta)
		\calLr_{X_{\sigma\theta(1)}}(\alpha(X_{\sigma\theta(2)}, \dots, X_{\sigma\theta(k+1)})) \beta(X, X_{\sigma(k+2)}, \dots, X_{\sigma(k+l)}) Y\\
		&\ - (-1)^{k} \sum_{\sigma \in \Sh(k+1,l-1)} \sgn(\sigma) \sum_{\theta \in \Sh(2,k-1)} \sgn(\theta)
		\alpha([X_{\sigma\theta(1)},X_{\sigma\theta(2)}]_{A},X_{\sigma\theta(3)},\dots,X_{\sigma\theta(k+1)}) \beta(X, X_{\sigma(k+2)}, \dots, X_{\sigma(k+l)}) Y\\
		&\ + (-1)^{k} \sum_{\sigma \in \Sh(k-1,l+1)} \sgn(\sigma) \sum_{\theta \in \Sh(1,l)} \sgn(\theta)
		\alpha(Y,X_{\sigma(1)}, \dots, X_{\sigma(k-1)}) \calLr_{X_{\sigma(k-1+\theta(1))}}(\beta(X_{\sigma(k-1+\theta(2))}, \dots, X_{\sigma(k-1+\theta(l+1))})) X\\
		&\ - (-1)^{k} \sum_{\sigma \in \Sh(k-1,l+1)} \sgn(\sigma) \sum_{\theta \in \Sh(2,l-1)} \sgn(\theta)
		\\
		&\ \quad 
		\alpha(Y,X_{\sigma(1)}, \dots, X_{\sigma(k-1)}) \beta([X_{\sigma(k-1+\theta(1))},X_{\sigma(k-1+\theta(2))}]_{A},X_{\sigma(k-1+\theta(3))},\dots,X_{\sigma(k-1+\theta(l+1))}) X\\
		=&\ \sum_{\sigma \in \Sh(k,l)} \sgn(\sigma) 
		\alpha(X_{\sigma(1)}, \dots, X_{\sigma(k)}) \beta(X_{\sigma(k+1)}, \dots, X_{\sigma(k+l)}) [X,Y]_{A}\\
		&\ + \sum_{\sigma \in \Sh(k,l)} \sgn(\sigma)
		\alpha(X_{\sigma(1)}, \dots, X_{\sigma(k)}) \calLr_{X}(\beta(X_{\sigma(k+1)}, \dots, X_{\sigma(k+l)})) Y\\
		&\ - \sum_{\sigma \in \Sh(k,1,l-1)} \sgn(\sigma) 
		\alpha(X_{\sigma(1)}, \dots, X_{\sigma(k)}) \beta([X,X_{\sigma(k+1)}]_{A}, X_{\sigma(k+2)}, \dots, X_{\sigma(k+l)}) Y\\
		&\ - \sum_{\sigma \in \Sh(k,l)} \sgn(\sigma)
		\calLr_{Y}(\alpha(X_{\sigma(1)}, \dots, X_{\sigma(k)})) \beta(X_{\sigma(k+1)}, \dots, X_{\sigma(k+l)}) X\\
		&\ + \sum_{\sigma \in \Sh(1,k-1,l)} \sgn(\sigma)
		\alpha([Y,X_{\sigma(1)}]_{A}, X_{\sigma(2)}, \dots, X_{\sigma(k)}) \beta(X_{\sigma(k+1)}, \dots, X_{\sigma(k+l)}) X\\
		&\ + (-1)^{k} \sum_{\sigma \in \Sh(1,k,l-1)} \sgn(\sigma) 
		\calLr_{X_{\sigma(1)}}(\alpha(X_{\sigma(2)}, \dots, X_{\sigma(k+1)})) \beta(X, X_{\sigma(k+2)}, \dots, X_{\sigma(k+l)}) Y\\
		&\ - (-1)^{k} \sum_{\sigma \in \Sh(2,k-1,l-1)} \sgn(\sigma) 
		\alpha([X_{\sigma(1)},X_{\sigma(2)}]_{A},X_{\sigma(3)},\dots,X_{\sigma(k+1)}) \beta(X, X_{\sigma(k+2)}, \dots, X_{\sigma(k+l)}) Y\\
		&\ + (-1)^{k} \sum_{\sigma \in \Sh(k-1,1,l)} \sgn(\sigma) 
		\alpha(Y,X_{\sigma(1)}, \dots, X_{\sigma(k-1)}) \calLr_{X_{\sigma(k)}}(\beta(X_{\sigma(k+1)}, \dots, X_{\sigma(k+l)})) X\\
		&\ - (-1)^{k} \sum_{\sigma \in \Sh(k-1,2,l-1)} \sgn(\sigma)
		\alpha(Y,X_{\sigma(1)}, \dots, X_{\sigma(k-1)}) \beta([X_{\sigma(k)},X_{\sigma(k+1)}]_{A},X_{\sigma(k+2)},\dots,X_{\sigma(k-l)}) X\\
		=&\ \sum_{\sigma \in \Sh(k,l)} \sgn(\sigma) 
		\alpha(X_{\sigma(1)}, \dots, X_{\sigma(k)}) \beta(X_{\sigma(k+1)}, \dots, X_{\sigma(k+l)}) [X,Y]_{A}\\
		&\ + \sum_{\sigma \in \Sh(k,l)} \sgn(\sigma)
		\alpha(X_{\sigma(1)}, \dots, X_{\sigma(k)}) \calLr_{X}(\beta(X_{\sigma(k+1)}, \dots, X_{\sigma(k+l)})) Y\\
		&\ - \sum_{\sigma \in \Sh(k,l)} \sgn(\sigma)
		\calLr_{Y}(\alpha(X_{\sigma(1)}, \dots, X_{\sigma(k)})) \beta(X_{\sigma(k+1)}, \dots, X_{\sigma(k+l)}) X\\
		&\ - \sum_{\sigma \in \Sh(k,1,l-1)} \sgn(\sigma) 
		\beta(\alpha(X_{\sigma(1)}, \dots, X_{\sigma(k)}) [X,X_{\sigma(k+1)}]_{A}, X_{\sigma(k+2)}, \dots, X_{\sigma(k+l)}) Y\\
		&\ + (-1)^{k} \sum_{\sigma \in \Sh(1,k,l-1)} \sgn(\sigma) 
		\beta(\calLr_{X_{\sigma(1)}}(\alpha(X_{\sigma(2)}, \dots, X_{\sigma(k+1)})) X, X_{\sigma(k+2)}, \dots, X_{\sigma(k+l)}) Y\\
		&\ + \sum_{\sigma \in \Sh(1,k-1,l)} \sgn(\sigma)
		\alpha(\beta(X_{\sigma(k+1)}, \dots, X_{\sigma(k+l)}) [Y,X_{\sigma(1)}]_{A}, X_{\sigma(2)}, \dots, X_{\sigma(k)})  X\\
		&\ + (-1)^{k} \sum_{\sigma \in \Sh(k-1,1,l)} \sgn(\sigma) 
		\alpha(\calLr_{X_{\sigma(k)}}(\beta(X_{\sigma(k+1)}, \dots, X_{\sigma(k+l)})) Y,X_{\sigma(1)}, \dots, X_{\sigma(k-1)}) X\\
		&\ - (-1)^{k} \sum_{\sigma \in \Sh(2,k-1,l-1)} \sgn(\sigma) 
		\beta(\alpha([X_{\sigma(1)},X_{\sigma(2)}]_{A},X_{\sigma(3)},\dots,X_{\sigma(k+1)}) X, X_{\sigma(k+2)}, \dots, X_{\sigma(k+l)}) Y\\
		&\ - (-1)^{k} \sum_{\sigma \in \Sh(k-1,2,l-1)} \sgn(\sigma)
		\alpha(\beta([X_{\sigma(k)},X_{\sigma(k+1)}]_{A},X_{\sigma(k+2)},\dots,X_{\sigma(k-l)}) Y,X_{\sigma(1)}, \dots, X_{\sigma(k-1)}) X\\
		=&\ \sum_{\sigma \in \Sh(k,l)} \sgn(\sigma) 
		\alpha(X_{\sigma(1)}, \dots, X_{\sigma(k)}) \beta(X_{\sigma(k+1)}, \dots, X_{\sigma(k+l)}) [X,Y]_{A}\\
		&\ + \sum_{\sigma \in \Sh(k,l)} \sgn(\sigma)
		\alpha(X_{\sigma(1)}, \dots, X_{\sigma(k)}) \calLr_{X}(\beta(X_{\sigma(k+1)}, \dots, X_{\sigma(k+l)})) Y\\
		&\ - \sum_{\sigma \in \Sh(k,l)} \sgn(\sigma)
		\calLr_{Y}(\alpha(X_{\sigma(1)}, \dots, X_{\sigma(k)})) \beta(X_{\sigma(k+1)}, \dots, X_{\sigma(k+l)}) X\\
		&\ - \sum_{\sigma \in \Sh(k,1,l-1)} \sgn(\sigma) 
		\beta(\alpha(X_{\sigma(1)}, \dots, X_{\sigma(k)}) [X,X_{\sigma(k+1)}]_{A}, X_{\sigma(k+2)}, \dots, X_{\sigma(k+l)}) Y\\
		&\ + \sum_{\sigma \in \Sh(k,1,l-1)} \sgn(\sigma) 
		\beta(\calLr_{X_{\sigma(k+1)}}(\alpha(X_{\sigma(1)}, \dots, X_{\sigma(k)})) X, X_{\sigma(k+2)}, \dots, X_{\sigma(k+l)}) Y\\
		&\ + (-1)^{kl} \sum_{\sigma \in \Sh(l,1,k-1)} \sgn(\sigma)
		\alpha(\beta(X_{\sigma(1)}, \dots, X_{\sigma(l)}) [Y,X_{\sigma(l+1)}]_{A}, X_{\sigma(l+2)}, \dots, X_{\sigma(k+l)})  X\\
		&\ - (-1)^{kl} \sum_{\sigma \in \Sh(l,1,k-1)} \sgn(\sigma) 
		\alpha(\calLr_{X_{\sigma(l+1)}}(\beta(X_{\sigma(1)}, \dots, X_{\sigma(l)})) Y,X_{\sigma(l+1)}, \dots, X_{\sigma(k+l)}) X\\
		&\ - (-1)^{k} \sum_{\sigma \in \Sh(2,k-1,l-1)} \sgn(\sigma) 
		\beta(\alpha([X_{\sigma(1)},X_{\sigma(2)}]_{A},X_{\sigma(3)},\dots,X_{\sigma(k+1)}) X, X_{\sigma(k+2)}, \dots, X_{\sigma(k+l)}) Y\\
		&\ + (-1)^{(k-1)l} \sum_{\sigma \in \Sh(2,l-1,k-1)} \sgn(\sigma)
		\alpha(\beta([X_{\sigma(1)},X_{\sigma(2)}]_{A},X_{\sigma(3)},\dots,X_{\sigma(l+1)}) Y,X_{\sigma(l+2)}, \dots, X_{\sigma(k+l)}) X\\
		=&\ \sum_{\sigma \in \Sh(k,l)} \sgn(\sigma) 
		[\alpha(X_{\sigma(1)}, \dots, X_{\sigma(k)})X,\beta(X_{\sigma(k+1)}, \dots, X_{\sigma(k+l)})Y]_{A}\\
		&\ - \sum_{\sigma \in \Sh(k,1,l-1)} \sgn(\sigma) 
		\beta([\alpha(X_{\sigma(1)}, \dots, X_{\sigma(k)})X,X_{\sigma(k+1)}]_{A}, X_{\sigma(k+2)}, \dots, X_{\sigma(k+l)}) Y\\
		&\ + (-1)^{kl} \sum_{\sigma \in \Sh(l,1,k-1)} \sgn(\sigma)
		\alpha([\beta(X_{\sigma(1)}, \dots, X_{\sigma(l)})Y,X_{\sigma(l+1)}]_{A}, X_{\sigma(l+2)}, \dots, X_{\sigma(k+l)})  X\\
		&\ - (-1)^{k} \sum_{\sigma \in \Sh(2,k-1,l-1)} \sgn(\sigma) 
		\beta(\alpha([X_{\sigma(1)},X_{\sigma(2)}]_{A},X_{\sigma(3)},\dots,X_{\sigma(k+1)}) X, X_{\sigma(k+2)}, \dots, X_{\sigma(k+l)}) Y\\
		&\ + (-1)^{(k-1)l} \sum_{\sigma \in \Sh(2,l-1,k-1)} \sgn(\sigma)
		\alpha(\beta([X_{\sigma(1)},X_{\sigma(2)}]_{A},X_{\sigma(3)},\dots,X_{\sigma(l+1)}) Y,X_{\sigma(l+2)}, \dots, X_{\sigma(k+l)}) X\\
		=&\ \sum_{\sigma \in \Sh(k,l)} \sgn(\sigma) 
		[K(X_{\sigma(1)}, \dots, X_{\sigma(k)}),L(X_{\sigma(k+1)}, \dots, X_{\sigma(k+l)})]_{A}\\
		&\ - \sum_{\sigma \in \Sh(k,1,l-1)} \sgn(\sigma) 
		L([K(X_{\sigma(1)}, \dots, X_{\sigma(k)}),X_{\sigma(k+1)}]_{A}, X_{\sigma(k+2)}, \dots, X_{\sigma(k+l)})\\
		&\ + (-1)^{kl} \sum_{\sigma \in \Sh(l,1,k-1)} \sgn(\sigma)
		K([L(X_{\sigma(1)}, \dots, X_{\sigma(l)}),X_{\sigma(l+1)}]_{A}, X_{\sigma(l+2)}, \dots, X_{\sigma(k+l)})\\
		&\ - (-1)^{k} \sum_{\sigma \in \Sh(2,k-1,l-1)} \sgn(\sigma) 
		L(K([X_{\sigma(1)},X_{\sigma(2)}]_{A},X_{\sigma(3)},\dots,X_{\sigma(k+1)}), X_{\sigma(k+2)}, \dots, X_{\sigma(k+l)})\\
		&\ + (-1)^{(k-1)l} \sum_{\sigma \in \Sh(2,l-1,k-1)} \sgn(\sigma)
		K(L([X_{\sigma(1)},X_{\sigma(2)}]_{A},X_{\sigma(3)},\dots,X_{\sigma(l+1)}),X_{\sigma(l+2)}, \dots, X_{\sigma(k+l)}).
	\end{align*}
\end{proof}

\begin{exam}
	Direct calculation gives:
	\begin{enumerate}
		\item For $K, L \in \psi^{0}(M,A) = \Gamma(A)$, we have
		$$[K,L]_{\FN} = [K,L]_{A} = [K,L]_{\mathrm{SN}},$$
		where the bracket $[-,-]_{\mathrm{SN}}$ is the Schouten-Nijenhuis bracket \cite{Marle2008} on the space $\Gamma(\wg^{\bullet}A)$ of multivectors associated with the Lie algebroid $(A,[-,-]_{A},\rho)$. This implies that both Schouten-Nijenhuis bracket and Fr\"olicher-Nijenhuis bracket are generalizations of the Lie bracket $[-,-]_{A}$ on $\Gamma(A)$.
		\item For $K \in \psi^{0}(M,A), L \in \psi^{l}(M,A)$, we have $[K,L]_{\FN} = \calLr_{K}L$.
		\item For $K, L \in \psi^{1}(M,A) = \Gamma(A^{\vee} \ot A)$, we have $[K,L]_{\FN} = \calN_{K+L} - \calN_{K} - \calN_{L}$.
	\end{enumerate}
\end{exam}

\begin{prop}
	Let $P \in \psi^{1}(M,A) = \Gamma(A^{\vee} \ot A)$. Then $P$ is a Nijenhuis operator on Lie algebroid $(A,[-,-]_{A},\rho)$ if and only if $[P,P]_{\FN} = 0$.
\end{prop}

\begin{proof}
	On can see that
	$$[P,P]_{\FN} = \calN_{P+P} - \calN_{P} - \calN_{P} = 2\calN_{P}.$$
	So the conclusion is obvious.
\end{proof}

Applying Proposition~\ref{Prop: deformed-L-infty}, which describes the twisting procedure of $L_{\infty}$-algebras, to the category of dg Lie algebras, we obtain the following main result of this section.

\begin{thm} \label{Thm. geometric version of cohomology of Nijenhuis operator}
	Let $(A,[-,-]_{A},\rho,P)$ be a Nijenhuis Lie algebroid. We redenote the graded space $\psi^{\bullet}(M,A)$ by $\C_{\FN}^{\bullet}(A)$, i.e., $\C_{\FN}^{k}(A) = \psi^{k}(M,A) = \Gamma(\wg^{k}A^{\vee} \ot A)$. Then $(\C_{\FN}^{\bullet}(A),\dd_{\FN} = [P,-]_{\FN}, [-,-]_{\FN})$ forms a dg Lie algebra. The underlying cochain complex $(\C_{\FN}^{\bullet}(A),\dd_{\FN})$ is called the Fr\"olicher-Nijenhuis cochain complex of Nijenhuis operator $P$, whose cohomology is called the Fr\"olicher-Nijenhuis cohomology of Nijenhuis operator $P$, denoted by $\rmH_{\FN}^{\bullet}(A)$.
\end{thm}

\medskip

\section{Comparising algebraic and geometric versions of cohomology of Nijenhuis operators} \label{The relation of cohomology of Nijenhuis operators between algebraic and geometric versions}

Let $(A,[-,-]_{A},\rho,P)$ be a Nijenhuis Lie algebroid.
Now, we introduce the relation of two versions, algebraic and geometric, of cochain complexs of Nijenhuis operators on the Lie algebra $(\Gamma(A), [-,-]_{A})$, introduced in Definition~\ref{defn: Nijenhuis operator complex Lie} and Theorem~\ref{Thm. geometric version of cohomology of Nijenhuis operator}, respectively.

According to Proposition~\ref{Prop: justifying Cohomology theory of Nijenhuis operator}, taking the Lie algebra $\frakg$ as $(\Gamma(A), \mu := [-,-]_{A})$, we obtain that $(\mathfrak{C}_{\NjO}(\Gamma(A)), (l_{1}^{\alpha})^{\beta}, l_{2}^{\alpha})$ forms a dg Lie algebra, where $\mathfrak{C}_{\NjO}(\Gamma(A))=\Hom_{\bbR}(S^c(s\Gamma(A)), \Gamma(A))$, and 
\begin{align*}
	\alpha &= (\nu,0), \nu = -s \circ \mu \circ (s^{-1})^{\ot 2}: (s\Gamma(A))^{\odot 2} \to \Gamma(A), \\
	\beta &= P \circ s^{-1}: s\Gamma(A) \to \Gamma(A).
\end{align*}
More explicitly, for any $f\in \Hom_{\bbR}((s\Gamma(A))^{\odot k}, \Gamma(A))$, $g\in \Hom_{\bbR}((s\Gamma(A))^{\odot l}, \Gamma(A))$, we have
	\begin{align*}
		l_2^\alpha(f\ot g)
		=&\ (-1)^{k}s^{-1}\nu\{sf, sg\} + f\{\nu \{sg\}\} + (-1)^{l}f\{sg\{\nu\}\}\\
		&+(-1)^{kl+l+1}s^{-1}\nu\{sg, sf\}-(-1)^{kl} g\{\nu \{sf\} \}+ (-1)^{kl+k+1}g\{sf\{\nu\}\},
	\end{align*}
	which corresponds to
	\begin{align*}
		(-1)^{kl} \Big(
			& \sum_{\sigma \in \Sh(k,l)} \sgn(\sigma) 
			- \sum_{\sigma \in \Sh(k,1,l-1)} \sgn(\sigma) 
			\widehat{g}(\mu \circ (\widehat{f} \ot \id) \ot \underbrace{\id \ot \cdots \ot \id}_{l-1}) \sigma^{-1}\\
			& - (-1)^{k} \sum_{\sigma \in \Sh(2,k-1,l-1)} \sgn(\sigma) 
			\widehat{g} \circ (\widehat{f} \circ (\mu \ot \underbrace{\id \ot \cdots \ot \id}_{k-1}) \ot \underbrace{\id \ot \cdots \ot \id}_{l-1}) \sigma^{-1}\\
			& + (-1)^{kl} \sum_{\sigma \in \Sh(l,1,k-1)} \sgn(\sigma)
			\widehat{f}(\mu \circ (\widehat{g} \ot \id) \ot \underbrace{\id \ot \cdots \ot \id}_{k-1}) \sigma^{-1}\\
			& + (-1)^{(k+1)l} \sum_{\sigma \in \Sh(2,l-1,k-1)} \sgn(\sigma)
			\widehat{f} \circ (\widehat{g} \circ (\mu \ot \underbrace{\id \ot \cdots \ot \id}_{l-1}) \ot \underbrace{\id \ot \cdots \ot \id}_{k-1}) \sigma^{-1} \Big),
	\end{align*}
	under the fixed isomorphism (\ref{Eq: second can isom}), denoted by $(-1)^{kl}[\widehat{f},\widehat{g}]_{\FNA}$, and we also have
	\begin{align*}
		(l_1^\alpha)^\beta(f)
		=&\ -l_2^\alpha(\beta\ot f)\\
		=&\ s^{-1}\nu\{s\beta, sf\} - \beta\{\nu\{sf\}\} - (-1)^{k} \beta \{sf\{\nu\}\}
		+ s^{-1}\nu\{sf, s\beta\} +(-1)^{k}f\{\nu\{s\beta\}\} - (-1)^{k} f \{s\beta\{\nu\}\}, 
	\end{align*}
	which corresponds to $(-1)^{k+1} \delta_{\NjO}(\widehat{f})$. 
	So we have $\delta_{\NjO}(\widehat{f}) = [P,\widehat{f}]_{\FNA}$ and $(\C_{\NjO}^{\bullet}(\Gamma(A)), \delta_{\NjO} = [P,-]_{\FNA}, [-,-]_{\FNA})$ forms a dg Lie algebra  that has been introduced in Proposition~\ref{Prop: justifying Cohomology theory of Nijenhuis operator}.

	Note that the graded space $\C_{\FN}^{\bullet}(A) = \Gamma(\wg^{\bullet}A^{\vee} \ot A)$ is a graded subspace of $\C_{\NjO}^{\bullet}(\Gamma(A)) = \bigoplus\limits_{k=0}^{n} \Hom_{\bbR}(\Gamma(A)^{\wg k}, \Gamma(A))$. Specifically, the former is $\ci(M)$-linear, whereas the latter is $\mathbb{R}$-linear. 

\begin{prop} \label{Prop. relation between algebra and geometry Nijenhuis}
	Let $(A,[-,-]_{A},\rho,P)$ be a Nijenhuis Lie algebroid. Then $(\C_{\FN}^{\bullet}(A),\dd_{\FN} = [P,-]_{\FN}, [-,-]_{\FN})$ is a dg Lie subalgebra of $(\C_{\NjO}^{\bullet}(\Gamma(A)), \delta_{\NjO} = [P,-]_{\FNA}$).
\end{prop}

\begin{proof}
	Based on the previous discussion and Theorem~\ref{Thm. geometric version of cohomology of Nijenhuis operator}, we know that there is an obvious inclusion of graded spaces from $\C_{\FN}^{\bullet}(A)$ to $\C_{\NjO}^{\bullet}(\Gamma(A))$, and $[-,-]_{\FNA} \mid_{\C_{\FN}^{\bullet}(A)} = [-,-]_{\FN}$, $\delta_{\NjO} \mid_{\C_{\FN}^{\bullet}(A)} = \dd_{\FN}$. Therefore, the conclusion of the theorem is valid.
\end{proof}

\medskip

\section{Basic concepts of dg manifolds} \label{Section. Basic concepts of dg manifolds} \

In this section, we briefly review some basic concepts of dg manifolds, and for more background knowledge, please refer to \cite{Seol2022, Stienon2021}.

Let $M$ be a finite-dimensional smooth manifold, and $\calO_{M}$ be the sheaf of smooth functions on $M$.

\begin{defn}
	A graded manifold $\calM$ with support $M$ consists of a sheaf $\calA$ of (cohomological) $\bbZ$-graded commutative $\calO_{M}$-algebras over $M$ such that there exists a finite-dimensional (cohomological) $\bbZ$-graded vector space $V$, for which
	\begin{align*}
		\calA(U) 
		\cong \calO_{M}(U) \ \hat{\ot} \ \Hom_{\bbR}(S(V), \mathbb{R})
		\cong \calO_{M}(U) \ \hat{\ot} \  \hs(V^{\vee}) 
	\end{align*}
	for any sufficiently small open set $U \subset M$, where $\calO_{M}(U) = C^{\infty}(U)$ and $\hs(V^{\vee})$ denotes the completed symmetric algebra generated by $V^{\vee}$, i.e., the algebra of formal power series on $V$.
\end{defn}


We denote by $C^{\infty}(\calM)$ the $\bbZ$-graded algebra of global sections of sheaf $\calA$, 
and denote by $\calI_{\calA}$ the sheaf of ideals of $\calA$ characterized by the property
	$$
	\calI_{\calA}(U) \cong \calO_{M}(U) \ \hat{\ot} \ \hs^{\geqslant 1}(V^{\vee})
	$$
for any sufficiently small open set $U \subset M$, i.e., all functions vanishing at the support $M$.

\begin{defn}
	Let $\calM = (M, \calA)$ and $\calN = (N, \calB)$ be two graded manifolds.
	A morphism $\phi: \calM \to \calN$ of a graded manifolds consists of a pair $(f,\psi)$, where $f: M \to N$ is a morphism of smooth manifolds, and $\psi: f^{*}\calB \to \calA$ is a morphism of sheaves of $\bbZ$-graded algebras such that $\psi(f^{*} \calI_{\calB}) \subset \calI_{\calA}$, i.e., $\psi$ continuous with respect to the $\calI$-adic topology.	
\end{defn}

	Vector bundles in the category of graded manifolds are called graded vector bundles. 

\begin{defn}
	A section of a graded vector bundle $\pi: \mathcal{E} \rightarrow \mathcal{M}$ is a morphism of graded manifolds $s: \mathcal{M} \rightarrow \mathcal{E}$ such that $\pi \circ s=\operatorname{id}_{\mathcal{M}}$. We denote the $C^{\infty}(\mathcal{M})$-module of sections of $\mathcal{E}$ over $\mathcal{M}$ by $\Gamma(\mathcal{E})=\Gamma(\mathcal{M}, \mathcal{E})$.
\end{defn}

\begin{defn}
	Let $\calM$ be a graded manifold. Its tangent bundle $T_{\mathcal{M}}$ is a graded manifold with support $T_M$ and is a graded vector bundle over $\mathcal{M}$.
	Sections of $T_{\mathcal{M}}$ are called vector fields on $\mathcal{M}$ and the space of vector fields $\mathfrak{X}(\mathcal{M}) := \Gamma\left(T_{\mathcal{M}}\right)$ can be identified with that of graded derivations of $C^{\infty}(\mathcal{M})$.
\end{defn}

Observe that $\mathfrak{X}(\mathcal{M})$ admits a graded Lie algebra structure, whose Lie bracket is given by the graded commutator
$$
[X, Y]=X \circ Y-(-1)^{|X| |Y|} Y \circ X
$$
for homogeneous elements $X, Y \in \mathfrak{X}(\mathcal{M})$, regarded as derivations of $C^{\infty}(\mathcal{M})$.

\begin{defn}
	A differential graded manifold (dg manifold) is a graded manifold $\calM$ together with a homological vector field, i.e., a vector field $Q$ in $\frakX(\calM)$ of degree $1$ satisfying $[Q,Q] = 2 Q \circ Q = 0$.
\end{defn}



A graded vector bundle over $M$ naturally induces a graded manifold structure with support $M$. Let
\begin{align} \label{Eq. [-m,n] graded bundle}
	E=E^{-m} \oplus \dots \oplus E^n
\end{align}
be a graded vector bundle over $M$. Then $\calM = (M, \mathcal{A})$ is a graded manifold, where $\mathcal{A}(U) :=\Gamma(U, \widehat{S}(E^{\vee}))$ for all open submanifolds $U \subset M$. 

\begin{defn}
	A dg manifold $(\mathcal{M}, Q)$ is said to be of amplitude $[-m, n]$, or be split, if $\mathcal{M}$ is a graded manifold associated with a graded vector bundle \eqref{Eq. [-m,n] graded bundle} concentrated in degrees $-m$ to $n$ under the following functor \eqref{Eq. Batchelor equivalence}.
\end{defn}


\begin{prop}[Batchelor] \cite{Deligne1999,Cattaneo2006}
	The following functor is an equivalence of categories:
	\begin{align} 
		(\text{graded vector bundles over } M) &\rightarrow (\text{graded manifolds with support } M) \label{Eq. Batchelor equivalence}\\
		(M,E)& \mapsto (M, \calA(U)=\Gamma(U, \widehat{S}(E^{\vee})) ) \notag
	\end{align}
\end{prop}

Batchelor's theorem states that any $\bbZ$-graded manifold is noncanonically diffeomorphic to a split $\bbZ$-graded manifold.

The following proposition provides a very important example of a dg manifold, which will be repeatedly used in subsequent discussions, including the notation introduced in its proof.

\begin{prop} \cite{Vaintrob1997, Kiselev2011} \label{Prop. relation between Lie algebroids and dg manifolds}
	Let $A$ be a vector bundle over $M$. Then $A[1]$ is a graded manifold with the algebra of functions $\ci(A[1]) = \Gamma(\wedge^{\bullet}A^{\vee})$, and $A$ is a Lie algebroid over $M$ if and only if $A[1]$ is a dg manifold with support $M$.
\end{prop}
\begin{proof}
	The algebra of functions on the graded manifold $A[1]$ is 
	$$\ci(A[1]) 
	= \Gamma(\hs((A[1])^{\vee})) 
	= \Gamma(\wedge^{\bullet}A^{\vee}).$$

	Let $(x^{1},\dots,x^{m})$ be a local coordinate system of $M$ and $(\eta^{1},\dots,\eta^{n})$ be a local basis of $A^{\vee}$.
	Denote $(\varepsilon_{1},\dots,\varepsilon_{n})$ the local basis of $A$, dual to $(\eta^{1},\dots,\eta^{n})$.
	Any homological vector field of degree $1$ on graded manifold $A[1]$ has the form
	\begin{align} \label{Eq. Lie algebroid structure A implies Q}
		Q 
		= \rho_{i}^{\alpha} \eta^{i} \partial_{x^{\alpha}} -\frac{1}{2} c_{p q}^{k}\eta^{p} \eta^{q} \partial_{\eta^{k}}
		= \rho_{i}^{\alpha} \eta^{i} \partial_{x^{\alpha}} - \sum_{p < q} c_{p q}^{k}\eta^{p} \eta^{q} \partial_{\eta^{k}},
	\end{align}
	where $\rho_{i}^{\alpha}$ and $c_{p q}^{k}$ are smooth functions on $M$ depending on $x^{i}$, and $c_{p q}^{k} = - c_{q p}^{k}$.

	Define $\rho: \Gamma(A) \rightarrow \frakX(M), \
	\varepsilon_{i} \mapsto \rho_{i}^{\alpha}\partial_{x^{\alpha}}$, $[\varepsilon_{i},\varepsilon_{j}]_{A} := c_{ij}^{k} \varepsilon_{k}$, and for any sections $X = f^{i} \varepsilon_{i}, Y = g^{j} \varepsilon_{j} \in \Gamma(A)$, define
	\begin{align*}
		&\rho: \Gamma(A) \rightarrow \frakX(M), \
		X \mapsto f^{i}\rho_{i}^{\alpha}\partial_{x^{\alpha}}, \\
		& [X,Y]_{A} 
		:= f^{i}\rho_{i}^{\alpha}\frac{\partial g^{j}}{\partial x^{\alpha}}\varepsilon_{j} - g^{j}\rho_{j}^{\alpha}\frac{\partial f^{i}}{\partial x^{\alpha}}\varepsilon_{i} + f^{i}g^{j}c_{ij}^{k}\varepsilon_{k}.
	\end{align*}

	It can be verified directly that $(A,\rho,[-,-]_{A})$ a Lie algebroid if and only if $Q^{2} = 0$. 
\end{proof}

\begin{cor} \label{Lem. Lie algebroid and dg manifold}
	Let $M$ be a manifold. Then the Lie algebroid structures over $M$ are in one-to-one correspondence with the dg manifold structures of amplitude $[-1,-1]$ with support $M$.
\end{cor}

\medskip

\section{Cohomology theory of Nijenhuis Lie algebroids} \label{Section: Cohomology theory of Nijenhuis Lie algebroids}

In this section, we introduce the cohomology theory of Nijenhuis Lie algebroids, using the language of dg manifolds, by combining the cohomology of Lie algebroids and the cohomology of Nijenhuis operators. 

Let $A$ be a vector bundle over a manifold $M$. According to Proposition~\ref{Prop. relation between Lie algebroids and dg manifolds}, we obtain that $A[1]$ is a graded manifold with  $\ci(A[1]) = \Gamma(\wedge^{\bullet}A^{\vee})$. Then any vector field $ X \in \frakX(A[1])$ of degree $b-1$ can be uniquely written as $a_{X} + \partial_{X}$, where
\begin{enumerate}
	\item $a_{X}: \ci(M) \rightarrow \Gamma(\wg^{b-1}A^{\vee})$ is a derivation, i.e., 
	$$a_{X}(fg) = a_{X}(f)g + f a_{X}(g)$$
	for any $f, g \in \ci(M)$.
	\item $\partial_{X}: \Gamma(A^{\vee}) \rightarrow \Gamma(\wg^{b}A^{\vee})$ is a $\mathbb{R}$-linear map such that
	$$\partial_{X}(fE) = a_{X}(f)E + f\partial_{X}(E)$$
	for any $f \in \ci(M)$ and $E \in \Gamma(A^{\vee})$.
\end{enumerate}

Subsequently, we will continue to use the notations and conventions established in Proposition~\ref{Prop. relation between Lie algebroids and dg manifolds} and its proof. 
Locally, $X$ forms
$$X = \sum_{i_{1} < \dots < i_{b-1}} f_{i_{1} \cdots i_{b-1}}^{\alpha} \eta^{i_{1}} \wg \cdots \wg \eta^{i_{b-1}} \partial_{x^{\alpha}} 
+ \sum_{j_{1} < \dots < j_{b}} g_{j_{1} \cdots j_{b}}^{\beta} \eta^{j_{1}} \wg \cdots \wg \eta^{j_{b}} \partial_{\eta^{\beta}},$$
where $f_{i_{1} \cdots i_{b-1}}^{\alpha}, g_{j_{1} \cdots j_{b}}^{\beta} \in \ci(M)$.
In this time, we have 
$$a_{X} = \sum\limits_{i_{1} < \dots < i_{b-1}} f_{i_{1} \cdots i_{b-1}}^{\alpha} \eta^{i_{1}} \wg \cdots \wg \eta^{i_{b-1}} \partial_{x^{\alpha}}$$ 
and 
\begin{align*}
	\partial_{X}(\eta^{k}) 
	= \sum\limits_{j_{1} < \dots < j_{b}} g_{j_{1} \cdots j_{b}}^{k} \eta^{j_{1}} \wg \cdots \wg \eta^{j_{b}}.
\end{align*}

\begin{defn}
	For any vector field $ X \in \frakX(A[1])$ of degree $b-1$, define a map
	$$B_{X}: \underbrace{\Gamma(A) \times \dots \times \Gamma(A)}_{b} \rightarrow \Gamma(A)$$
	by
	\begin{align*}
	\langle W,&\ B_{X}(E_{1}, \dots, E_{b}) \rangle :=
		(-1)^{b-1} \Big( \langle \partial_{X}(W), E_{1} \wg \cdots \wg E_{b} \rangle 
		- \sum_{i=1}^{b}(-1)^{b-i} \langle a_{X}\langle W, E_{i} \rangle, E_{1} \wg \cdots\wg \widehat{E_{i}} \wg \cdots \wg E_{b} \rangle \Big),
	\end{align*}
	for any $W \in \Gamma(A^{\vee})$ and $E_{1}, \dots, E_{b} \in \Gamma(A)$.
\end{defn}

It can be seen that $B_{X}$ is antisymmetric after calculation, which means 
$$B_{X}(E_{1}, \dots, E_{i}, \dots, E_{j}, \dots, E_{b}) = - B_{X}(E_{1}, \dots, E_{j}, \dots, E_{i}, \dots, E_{b})$$
for any $E_{1}, \dots, E_{b} \in \Gamma(A)$ and $1 \leqslant i < j \leqslant b$. Further, $B_{X}$ acts like a derivation on each component, see the following result.

\begin{prop} \label{Prop. B_X as a derivation}
	For any $E_{1}, \dots, E_{b} \in \Gamma(A)$ and $1 \leqslant j \leqslant b$, we have 
	$$B_{X}(E_{1}, \dots, fE_{j}, \dots, E_{b}) = f B_{X}(E_{1}, \dots, E_{j}, \dots, E_{b}) + (-1)^{j} \langle a_{X}(f), E_{1} \wg \cdots \wg \widehat{E_{j}} \wg \cdots \wg E_{b} \rangle E_{j}.$$
\end{prop}

\begin{proof}
	Direct computation yields, for any $W \in \Gamma(A^{\vee})$,
	\begin{align*}
		&\ \langle W, B_{X}(E_{1}, \dots, f E_{j}, \dots, E_{b}) \rangle \\
		=&\ (-1)^{b-1} \Big( \langle \partial_{X}(W), E_{1} \wg \cdots \wg fE_{j} \wg \cdots \wg E_{b} \rangle 
		- \sum_{i \neq j} (-1)^{b-i} \langle a_{X}\langle W, E_{i} \rangle, E_{1} \wg \cdots \wg \widehat{E_{i}} \cdots \wg f E_{j} \wg \cdots \wg E_{b} \rangle\\
		&\ - (-1)^{b-j} \langle a_{X}\langle W, f E_{j} \rangle, E_{1} \wg \cdots \wg E_{i} \cdots \wg \widehat{f E_{j}} \wg \cdots \wg E_{b} \rangle \Big)\\
		=&\ \langle W, f B_{X}(E_{1}, \dots, E_{j}, \dots, E_{b}) \rangle 
		- (-1)^{j-1} \langle W, E_{j}\rangle \langle a_{X}(f), E_{1} \wg \cdots \wg \widehat{f E_{j}} \wg \cdots \wg E_{b} \rangle\\
		=&\ \langle W, f B_{X}(E_{1}, \dots, E_{j}, \dots, E_{b}) + (-1)^{j} \langle a_{X}(f), E_{1} \wg \cdots \wg \widehat{E_{j}} \wg \cdots \wg E_{b} \rangle E_{j} \rangle.
	\end{align*}
	Thus, we have obtained the desired result of the proposition.
\end{proof}

Note that, for any vector field $ X \in \frakX(A[1])$, the map $B_{X}$ is $\mathbb{R}$-linear but not $\ci(M)$-linear. In fact, with the help of a $(1,1)$-form in $\Gamma(A^{\vee} \ot A)$, not necessarily a Nijenhuis operator, we can also construct a $\ci(M)$-linear map $\Phi(X)$, justified by Proposition~\ref{Prop. Phi is a tensor}, based on $B_{X}$ as follows.

\begin{defn} \label{Defn. definition of Phi}
	For any vector field $ X \in \frakX(A[1])$ of degree $b-1$ and a $(1,1)$-form $P$, define a map
	$$\Phi(X): \underbrace{\Gamma(A) \times \dots \times \Gamma(A)}_{b} \rightarrow \Gamma(A)$$
	as
	\begin{align*}
	\Phi(X)(E_{1}, \dots, E_{b}):=
		\sum_{k=0}^{b} \sum_{1 \leqslant i_{1} < \dots < i_{k} \leqslant b} (-1)^{b-k} P^{b-k} B_{X}(E_{1}, \dots, PE_{i_{1}}, \dots, PE_{i_{k}}, \dots, E_{b}),
	\end{align*}
	for any $E_{1}, \dots, E_{b} \in \Gamma(A)$.
\end{defn}

\begin{prop} \label{Prop. Phi is a tensor}
		For any vector field $ X \in \frakX(A[1])$ of degree $b-1$, $\Phi(X) \in \Gamma(\wg^{b}A^{\vee} \ot A)$.
\end{prop}

\begin{proof}
	It can be easily seen that $\Phi(X)$ is antisymmetric, i.e., 
	$$\Phi(X)(E_{1}, \dots, E_{i}, \dots, E_{j}, \dots, E_{b}) = - \Phi(X)(E_{1}, \dots, E_{j}, \dots, E_{i}, \dots, E_{b})$$
	for any $E_{1}, \dots, E_{b} \in \Gamma(A)$ and $1 \leqslant i < j \leqslant b$. So we only need to show that $\Phi(X)$ is $\ci(M)$-linear for the first component.
	For $f \in \ci(M)$, we have
	\begin{align*}
		&\ \Phi(X)(fE_{1}, E_{2}, \dots, E_{b})\\
		=&\ \sum_{k=0}^{b-1} \sum_{1 < i_{1} < i_{2} < \dots < i_{k} \leqslant b} (-1)^{b-k} P^{b-k} B_{X}(fE_{1}, \dots, PE_{i_{1}}, \dots, PE_{i_{k}}, \dots, E_{b})\\
		&\ + \sum_{k=1}^{b} \sum_{1 = i_{1} < i_{2} < \dots < i_{k} \leqslant b} (-1)^{b-k} P^{b-k} B_{X}(fPE_{1}, \dots, PE_{i_{2}}, \dots, PE_{i_{k}}, \dots, E_{b})\\
		=&\ f(\sum_{k=0}^{b-1} \sum_{1 < i_{1} < i_{2} < \dots < i_{k} \leqslant b} (-1)^{b-k} P^{b-k} B_{X}(E_{1}, \dots, PE_{i_{1}}, \dots, PE_{i_{k}}, \dots, E_{b}))\\
		&\ + \sum_{k=0}^{b-1} \sum_{1 < i_{1} < i_{2} < \dots < i_{k} \leqslant b} (-1)^{b-k} P^{b-k} (-1)^{1} \langle a_{X}(f), \widehat{E_{1}} \wg E_{2} \wg \cdots \wg PE_{i_{1}} \wg \cdots \wg PE_{i_{k}}, \dots, E_{b} \rangle E_{1}\\
		&\ + f(\sum_{k=1}^{b} \sum_{1 = i_{1} < i_{2} < \dots < i_{k} \leqslant b} (-1)^{b-k} P^{b-k} B_{X}(PE_{1}, \dots, PE_{i_{2}}, \dots, PE_{i_{k}}, \dots, E_{b}))\\
		&\ + \sum_{k=1}^{b} \sum_{1 = i_{1} < i_{2} < \dots < i_{k} \leqslant b} (-1)^{b-k} P^{b-k} (-1)^{1} \langle a_{X}(f), \widehat{PE_{1}} \wg E_{2} \wg \cdots \wg PE_{i_{2}} \wg \cdots \wg PE_{i_{k}} \wg \cdots \wg E_{b} \rangle PE_{1}\\
		=&\ f\Phi(X)(E_{1}, E_{2}, \dots, E_{b})\\
		&\ + \sum_{k=0}^{b-1} \sum_{1 < i_{1} < i_{2} < \dots < i_{k} \leqslant b} (-1)^{b-k} (-1)^{1} \langle a_{X}(f), \widehat{E_{1}} \wg E_{2} \wg \cdots \wg PE_{i_{1}} \wg \cdots \wg PE_{i_{k}}, \dots, E_{b} \rangle P^{b-k} E_{1}\\
		&\ + \sum_{k=1}^{b} \sum_{1 = i_{1} < i_{2} < \dots < i_{k} \leqslant b} (-1)^{b-k} (-1)^{1} \langle a_{X}(f), \widehat{PE_{1}} \wg E_{2} \wg \cdots \wg PE_{i_{2}} \wg \cdots \wg PE_{i_{k}} \wg \cdots \wg E_{b} \rangle P^{b-k+1}E_{1}\\
		=&\ f\Phi(X)(E_{1}, E_{2}, \dots, E_{b})\\
		&\ + \sum_{k=0}^{b-1} \sum_{1 < i_{1} < i_{2} < \dots < i_{k} \leqslant b} (-1)^{b-k} (-1)^{1} \langle a_{X}(f), \widehat{E_{1}} \wg E_{2} \wg \cdots \wg PE_{i_{1}} \wg \cdots \wg PE_{i_{k}}, \dots, E_{b} \rangle P^{b-k} E_{1}\\
		&\ + \sum_{l=0}^{b-1} \sum_{1 = i_{1} < i_{2} < \dots < i_{l+1} \leqslant b} (-1)^{b-l-1} (-1)^{1} \langle a_{X}(f), \widehat{PE_{1}} \wg E_{2} \wg \cdots \wg PE_{i_{2}} \wg \cdots \wg PE_{i_{l+1}} \wg \cdots \wg E_{b} \rangle P^{b-l}E_{1}\\
		=&\ f\Phi(X)(E_{1}, E_{2}, \dots, E_{b}).
	\end{align*}
\end{proof}

Therefore, we get a well-defined map of graded space:
\begin{align} \label{Eq. definition of chain map Phi}
	\Phi: \frakX(A[1]) \rightarrow \C_{\FN}(A), \
	X \mapsto \Phi(X).
\end{align}

As stated in Proposition~\ref{Prop. relation between Lie algebroids and dg manifolds}, when $(A, [-,-]_{A}, \rho)$ is a Lie algebroid, we have that $A[1]$ is a dg manifold with homological vector field $Q \in \frakX(A[1])$ given by Equation~\eqref{Eq. Lie algebroid structure A implies Q}. Therefore $(\frakX(A[1]), \dd_{Q} = - [Q,-], [-,-])$ is a dg Lie algebra. Moreover, when $(A, [-,-]_{A}, \rho, P)$ is a Nijenhuis Lie algebroid, Theorem~\ref{Thm. geometric version of cohomology of Nijenhuis operator} states that $(\C_{\FN}^{\bullet}(A), \dd_{\FN} = [P,-]_{\FN}, [-,-]_{\FN})$ forms a dg Lie algebra as well. Next, we will show that $\Phi: (\frakX(A[1]), \dd_{Q}) \rightarrow (\C_{\FN}(A), \dd_{\FN})$ is a chain map, see Theorem~\ref{Thm. Phi is a chain map}.

\begin{prop} \label{Prop. Equation of B_X}
	Let $$X = \sum_{i_{1} < \dots < i_{b-1}} f_{i_{1} \cdots i_{b-1}}^{\alpha} \eta^{i_{1}} \wg \cdots \wg \eta^{i_{b-1}} \partial_{x^{\alpha}} 
	+ \sum_{j_{1} < \dots < j_{b}} g_{j_{1} \cdots j_{b}}^{\beta} \eta^{j_{1}} \wg \cdots \wg \eta^{j_{b}} \partial_{\eta^{\beta}}$$
	be a vector field of degree $b-1$ in $\frakX(A[1])$. Then 
	$$B_{X}(\varepsilon_{\varphi_{1}}, \dots, \varepsilon_{\varphi_{b}}) = (-1)^{b-1} g_{\varepsilon_{\varphi_{1}} \cdots \varepsilon_{\varphi_{b}}}^{q} \varepsilon_{q}.$$
\end{prop}

\begin{proof}
	By definition, we have
	\begin{align*}
		\langle \eta^{q}, B_{X}(\varepsilon_{\varphi_{1}}, \dots, \varepsilon_{\varphi_{b}}) \rangle 
		=&\ (-1)^{b-1} \langle \partial_{X}(\eta^{q}), \varepsilon_{\varphi_{1}} \wg \cdots \wg \varepsilon_{\varphi_{b}} \rangle\\
		=&\ (-1)^{b-1} \langle \sum_{j_{1} < \dots < j_{b}} g_{j_{1} \cdots j_{b}}^{q} \eta^{j_{1}} \wg \cdots \wg \eta^{j_{b}}, \varepsilon_{\varphi_{1}} \wg \cdots \wg \varepsilon_{\varphi_{b}} \rangle\\
		=&\ (-1)^{b-1} g_{\varepsilon_{\varphi_{1}} \cdots \varepsilon_{\varphi_{b}}}^{q}.
	\end{align*}
	Hence, $B_{X}(\varepsilon_{\varphi_{1}}, \dots, \varepsilon_{\varphi_{b}}) = (-1)^{b-1} g_{\varepsilon_{\varphi_{1}} \cdots \varepsilon_{\varphi_{b}}}^{q} \varepsilon_{q}$.
\end{proof}

Assume that $(A, [-,-]_{A}, \rho, P)$ is a Nijenhuis Lie algebroid with Lie algebroid structure $\rho(\varepsilon_{i}) = \rho_{i}^{\alpha}\partial_{x^{\alpha}}$ and $[\varepsilon_{i},\varepsilon_{j}]_{A} = c_{ij}^{k} \varepsilon_{k}$, which corresponds to the following homological vector field
\begin{align*} 
	Q 
	= \rho_{i}^{\alpha} \eta^{i} \partial_{x^{\alpha}} - \sum_{p < q} c_{p q}^{k}\eta^{p} \eta^{q} \partial_{\eta^{k}} \in \frakX(A[1]).
\end{align*}

\begin{prop} \label{Prop. equation of B_Q}
	With above notions, we obtain $B_{Q} = [-,-]_{A}: \Gamma(A) \times \Gamma(A) \to \Gamma(A)$, and further derive $\Phi(Q) = \calN_{P} = 0$.
\end{prop}

\begin{proof}
	By Proposition~\ref{Prop. Equation of B_X}, we obtain that $B_{Q}(\varepsilon_{j}, \varepsilon_{k}) = c_{jk}^{i} \varepsilon_{i} = [\varepsilon_{j}, \varepsilon_{k}]_{A}$.
	Moreover, for any $g \in \ci(M)$, by Proposition~\ref{Prop. B_X as a derivation}, we have
	\begin{align*}
		B_{Q}(\varepsilon_{j}, g\varepsilon_{k})
		=&\ g B_{Q}(\varepsilon_{j}, \varepsilon_{k}) + \langle a_{Q}(g), \varepsilon_{j} \rangle \varepsilon_{k}\\
		=&\ g B_{Q}(\varepsilon_{j}, \varepsilon_{k}) + \rho^{\alpha}_{j} \frac{\partial g}{\partial x^{\alpha}} \varepsilon_{k}\\
		=&\ g B_{Q}(\varepsilon_{j}, \varepsilon_{k}) + \rho(\varepsilon_{j})(g) \varepsilon_{k}\\
		=&\ [\varepsilon_{j}, g\varepsilon_{k}]_{A}.
	\end{align*}
	As a conclusion, we obtain $B_{Q} = [-,-]_{A}$, and the result $\Phi(Q) = \calN_{P} = 0$ follows directly from Definitions~\ref{Defn. definition of Phi} and \ref{Defn. Nijenhuis Lie algebroid}.
\end{proof}

\begin{prop} \label{Prop. definition of chi}
	Consider the following map
	\begin{align*}
		\chi: (\frakX(A[1]), [-,-]) \rightarrow (\C_{\Lie}^{\bullet}(\Gamma(A)), [-,-]_{\RNA}), \ X \mapsto B_{X}.
	\end{align*}
	We have 
	$$B_{[X,Y]} = [B_{X}, B_{Y}]_{\RNA}$$
	for any $X, Y \in \frakX(A[1])$.
\end{prop}
\begin{proof}
	Assume that 
	$$X = \sum_{i_{1} < \dots < i_{b-1}} 
	f_{i_{1} \cdots i_{b-1}}^{\alpha} \eta^{i_{1}} \wg \cdots \wg \eta^{i_{b-1}} \partial_{x^{\alpha}} 
	+ \sum_{j_{1} < \dots < j_{b}} 
	g_{j_{1} \cdots j_{b}}^{\beta} \eta^{j_{1}} \wg \cdots \wg \eta^{j_{b}} \partial_{\eta^{\beta}}$$
	is a vector field of degree $b-1$ in $\frakX(A[1])$, and 
	$$Y = \sum_{k_{1} < \dots < k_{c-1}} 
	\varphi_{k_{1} \cdots k_{c-1}}^{\theta} \eta^{k_{1}} \wg \cdots \wg \eta^{k_{c-1}} \partial_{x^{\theta}} 
	+ \sum_{l_{1} < \dots < l_{c}} 
	\psi_{l_{1} \cdots l_{c}}^{\omega} \eta^{l_{1}} \wg \cdots \wg \eta^{l_{c}} \partial_{\eta^{\omega}}$$
	is a vector field of degree $c-1$ in $\frakX(A[1])$. 

	On the one hand,
	\begin{align*}
		[X,Y] 
		=&\ XY - (-1)^{(b-1)(c-1)}YX\\ 
		=&\ \sum_{\substack{i_{1} < \dots < i_{b-1} \\ k_{1} < \dots < k_{c-1}}} 
		f_{i_{1} \cdots i_{b-1}}^{\alpha} \frac{\partial \varphi_{k_{1} \cdots k_{c-1}}^{\theta}}{\partial x^{\alpha}}
		(\eta^{i_{1}} \wg \cdots \wg \eta^{i_{b-1}}) \wg (\eta^{k_{1}} \wg \cdots \wg \eta^{k_{c-1}}) \partial_{x^{\theta}} \\
		&\ + \sum_{\substack{j_{1} < \dots < j_{b} \\ k_{1} < \dots < k_{c-2}}}
		g_{j_{1} \cdots j_{b}}^{\beta} \varphi_{\beta k_{1} \cdots k_{c-2}}^{\theta}
		(\eta^{j_{1}} \wg \cdots \wg \eta^{j_{b}}) \wg (\eta^{k_{1}} \wg \cdots \wg \eta^{k_{c-2}}) \partial_{x^{\theta}}\\
		&\ + \sum_{\substack{i_{1} < \dots < i_{b-1} \\ l_{1} < \dots < l_{c}}} 
		f_{i_{1} \cdots i_{b-1}}^{\alpha} \frac{\partial \psi_{l_{1} \cdots l_{c}}^{\omega}}{\partial x^{\alpha}}
		(\eta^{i_{1}} \wg \cdots \wg \eta^{i_{b-1}}) \wg (\eta^{l_{1}} \wg \cdots \wg \eta^{l_{c}}) \partial_{\eta^{\omega}}\\
		&\ + \sum_{\substack{j_{1} < \dots < j_{b} \\ l_{1} < \dots < l_{c-1}}}
		g_{j_{1} \cdots j_{b}}^{\beta} \psi_{\beta l_{1} \cdots l_{c-1}}^{\omega} 
		(\eta^{j_{1}} \wg \cdots \wg \eta^{j_{b}}) \wg (\eta^{l_{1}} \wg \cdots \wg \eta^{l_{c-1}}) \partial_{\eta^{\omega}}\\ 
		&\ - (-1)^{(b-1)(c-1)} \sum_{\substack{k_{1} < \dots < k_{c-1} \\ i_{1} < \dots < i_{b-1}}}
		\varphi_{k_{1} \cdots k_{c-1}}^{\theta} \frac{\partial f_{i_{1} \cdots i_{b-1}}^{\alpha}}{\partial x^{\theta}}
		(\eta^{k_{1}} \wg \cdots \wg \eta^{k_{c-1}}) \wg (\eta^{i_{1}} \wg \cdots \wg \eta^{i_{b-1}}) \partial_{x^{\alpha}}\\
		&\ - (-1)^{(b-1)(c-1)} \sum_{\substack{l_{1} < \dots < l_{c} \\ i_{1} < \dots < i_{b-2}}}
		\psi_{l_{1} \cdots l_{c}}^{\omega} f_{\omega i_{1} \cdots i_{b-2}}^{\alpha}
		(\eta^{l_{1}} \wg \cdots \wg \eta^{l_{c}}) \wg (\eta^{i_{1}} \wg \cdots \wg \eta^{i_{b-2}}) \partial_{x^{\alpha}}\\
		&\ - (-1)^{(b-1)(c-1)} \sum_{\substack{k_{1} < \dots < k_{c-1} \\ j_{1} < \dots < j_{b}}}
		\varphi_{k_{1} \cdots k_{c-1}}^{\theta} \frac{\partial g_{j_{1} \cdots j_{b}}^{\beta}}{\partial x^{\theta}}
		(\eta^{k_{1}} \wg \cdots \wg \eta^{k_{c-1}}) \wg (\eta^{j_{1}} \wg \cdots \wg \eta^{j_{b}}) \partial_{\eta^{\beta}}\\
		&\ - (-1)^{(b-1)(c-1)} \sum_{\substack{l_{1} < \dots < l_{c} \\ j_{1} < \dots < j_{b-1}}}
		\psi_{l_{1} \cdots l_{c}}^{\omega} g_{\omega j_{1} \cdots j_{b-1}}^{\beta}
		(\eta^{l_{1}} \wg \cdots \wg \eta^{l_{c}}) \wg (\eta^{j_{1}} \wg \cdots \wg \eta^{j_{b-1}}) \partial_{\eta^{\beta}}\\ 
		=&\ \sum_{i_{1} < \dots < i_{b+c-2}}
		\Big( 
			\sum_{\sigma \in \Sh(b-1,c-1)} \sgn(\sigma) f_{i_{\sigma(1)} \cdots i_{\sigma(b-1)}}^{\theta} \frac{\partial \varphi_{i_{\sigma(b)} \cdots i_{\sigma(b+c-2)}}^{\alpha}}{\partial x^{\theta}}
			+\sum_{\sigma \in \Sh(b,c-2)} \sgn(\sigma) g_{i_{\sigma(1)} \cdots i_{\sigma(b)}}^{\beta} \varphi_{\beta i_{\sigma(b+1)} \cdots i_{\sigma(b+c-2)}}^{\alpha} \\
			&\ \hspace{16 mm} - (-1)^{(b-1)(c-1)} \sum_{\sigma \in \Sh(c-1,b-1)} \sgn(\sigma) \varphi_{i_{\sigma(1)} \cdots i_{\sigma(c-1)}}^{\theta} \frac{\partial f_{i_{\sigma(c)} \cdots i_{\sigma(b+c-2)}}^{\alpha}}{\partial x^{\theta}}\\
			&\ \hspace{16 mm} - (-1)^{(b-1)(c-1)} \sum_{\sigma \in \Sh(c,b-2)} \sgn(\sigma) \psi_{i_{\sigma(1)} \cdots i_{\sigma(c)}}^{\beta} f_{\beta i_{\sigma(c+1)} \cdots i_{\sigma(b+c-2)}}^{\alpha}
		\Big) \eta^{i_{1}} \wg \cdots \wg \eta^{i_{b+c-2}} \partial_{x^{\alpha}} \\
		&\ + \sum_{j_{1} < \dots < j_{b+c-1}}
		\Big( 
			\sum_{\sigma \in \Sh(b-1,c)} \sgn(\sigma) f_{j_{\sigma(1)} \cdots j_{\sigma(b-1)}}^{\alpha} \frac{\partial \psi_{j_{\sigma(b)} \cdots j_{\sigma(b+c-1)}}^{\omega}}{\partial x^{\alpha}} 
			+\sum_{\sigma \in \Sh(b,c-1)} \sgn(\sigma) g_{j_{\sigma(1)} \cdots j_{\sigma(b)}}^{\beta} \psi_{\beta j_{\sigma(b+1)} \cdots j_{\sigma(b+c-1)}}^{\omega} \\
			&\ \hspace{20 mm} - (-1)^{(b-1)(c-1)} \sum_{\sigma \in \Sh(c-1,b)} \sgn(\sigma) \varphi_{j_{\sigma(1)} \cdots j_{\sigma(c-1)}}^{\alpha} \frac{\partial g_{j_{\sigma(c)} \cdots j_{\sigma(b+c-1)}}^{\omega}}{\partial x^{\alpha}} \\
			&\ \hspace{20 mm} - (-1)^{(b-1)(c-1)} \sum_{\sigma \in \Sh(c,b-1)} \sgn(\sigma) \psi_{j_{\sigma(1)} \cdots j_{\sigma(c)}}^{\beta} g_{\beta j_{\sigma(c+1)} \cdots j_{\sigma(b+c-1)}}^{\omega}
		\Big) \eta^{j_{1}} \wg \cdots \wg \eta^{j_{b+c-1}} \partial_{\eta^{\omega}}. 
	\end{align*}
	By Proposition~\ref{Prop. Equation of B_X}, we have
	\begin{align*}
		&\ B_{[X,Y]}(\varepsilon_{j_{1}}, \dots, \varepsilon_{j_{b+c-1}}) \\
		=&\ (-1)^{b-1+c-1} \Big( 
			\sum_{\sigma \in \Sh(b-1,c)} \sgn(\sigma) f_{j_{\sigma(1)} \cdots j_{\sigma(b-1)}}^{\alpha} \frac{\partial \psi_{j_{\sigma(b)} \cdots j_{\sigma(b+c-1)}}^{\omega}}{\partial x^{\alpha}} 
			+\sum_{\sigma \in \Sh(b,c-1)} \sgn(\sigma) g_{j_{\sigma(1)} \cdots j_{\sigma(b)}}^{\beta} \psi_{\beta j_{\sigma(b+1)} \cdots j_{\sigma(b+c-1)}}^{\omega} \\
			&\ \hspace{20 mm} - (-1)^{(b-1)(c-1)} \sum_{\sigma \in \Sh(c-1,b)} \sgn(\sigma) \varphi_{j_{\sigma(1)} \cdots j_{\sigma(c-1)}}^{\alpha} \frac{\partial g_{j_{\sigma(c)} \cdots j_{\sigma(b+c-1)}}^{\omega}}{\partial x^{\alpha}} \\
			&\ \hspace{20 mm} - (-1)^{(b-1)(c-1)} \sum_{\sigma \in \Sh(c,b-1)} \sgn(\sigma) \psi_{j_{\sigma(1)} \cdots j_{\sigma(c)}}^{\beta} g_{\beta j_{\sigma(c+1)} \cdots j_{\sigma(b+c-1)}}^{\omega}
		\Big) \varepsilon_{\omega}.
	\end{align*}
	On the other hand, by Proposition~\ref{Prop. Equation of B_X} and Proposition~\ref{Prop. B_X as a derivation}, we have
	\begin{align*}
		&\ [B_{X},B_{Y}]_{\RNA}(\varepsilon_{j_{1}}, \dots, \varepsilon_{j_{b+c-1}})\\
		= &\ \sum_{\sigma \in \Sh(c,b-1)} \sgn(\sigma) 
		B_{X}(B_{Y}(\varepsilon_{j_{\sigma(1)}}, \dots, \varepsilon_{j_{\sigma(c)}}), \varepsilon_{j_{\sigma(c+1)}}, \dots, \varepsilon_{j_{\sigma(b+c-1)}})\\
		&\ - (-1)^{(b-1)(c-1)} \sum_{\sigma \in \Sh(b,c-1)} \sgn(\sigma) 
		B_{Y}(B_{X}(\varepsilon_{j_{\sigma(1)}}, \dots, \varepsilon_{j_{\sigma(b)}}), \varepsilon_{j_{\sigma(b+1)}}, \dots, \varepsilon_{j_{\sigma(b+c-1)}})\\ 
		= &\ \sum_{\sigma \in \Sh(c,b-1)} \sgn(\sigma) 
		B_{X}((-1)^{c-1} \psi_{j_{\sigma(1)} \cdots j_{\sigma(c)}}^{\omega} \varepsilon_{\omega}, \varepsilon_{j_{\sigma(c+1)}}, \dots, \varepsilon_{j_{\sigma(b+c-1)}})\\
		&\ - (-1)^{(b-1)(c-1)} \sum_{\sigma \in \Sh(b,c-1)} \sgn(\sigma) 
		B_{Y}((-1)^{b-1}g_{j_{\sigma(1)} \cdots j_{\sigma(b)}}^{\beta} \varepsilon_{\beta}, \varepsilon_{j_{\sigma(b+1)}}, \dots, \varepsilon_{j_{\sigma(b+c-1)}})\\ 
		= &\ \sum_{\sigma \in \Sh(c,b-1)} \sgn(\sigma) 
		\Big( (-1)^{c-1} \psi_{j_{\sigma(1)} \cdots j_{\sigma(c)}}^{\omega} B_{X}(\varepsilon_{\omega}, \varepsilon_{j_{\sigma(c+1)}}, \dots, \varepsilon_{j_{\sigma(b+c-1)}})\\
		&\ \hspace{28 mm}- \langle a_{X}((-1)^{c-1}\psi_{j_{\sigma(1)} \cdots j_{\sigma(c)}}^{\omega}), \varepsilon_{j_{\sigma(c+1)}} \wg \cdots \wg \varepsilon_{j_{\sigma(b+c-1)}} \rangle \varepsilon_{\omega} \Big) \\
		&\ - (-1)^{(b-1)(c-1)} \sum_{\sigma \in \Sh(b,c-1)} \sgn(\sigma) 
		\Big( (-1)^{b-1}g_{j_{\sigma(1)} \cdots j_{\sigma(b)}}^{\beta} B_{Y}(\varepsilon_{\beta}, \varepsilon_{j_{\sigma(b+1)}}, \dots, \varepsilon_{j_{\sigma(b+c-1)}})\\
		&\ \hspace{54 mm}- \langle a_{Y}((-1)^{b-1} g_{j_{\sigma(1)} \cdots j_{\sigma(b)}}^{\beta}), \varepsilon_{j_{\sigma(b+1)}} \wg \cdots \wg \varepsilon_{j_{\sigma(b+c-1)}} \rangle \varepsilon_{\beta} \Big) \\ 
		= &\ \sum_{\sigma \in \Sh(c,b-1)} \sgn(\sigma) 
		(-1)^{c-1} \psi_{j_{\sigma(1)} \cdots j_{\sigma(c)}}^{\beta} (-1)^{b-1} g_{\beta j_{\sigma(c+1)} \cdots j_{\sigma(b+c-1)}}^{\omega} \varepsilon_{\omega}
		- \sum_{\sigma \in \Sh(c,b-1)} \sgn(\sigma) 
		(-1)^{c-1} f_{j_{\sigma(c+1)} \cdots j_{\sigma(b+c-1)}}^{\alpha} \frac{\partial \psi_{j_{\sigma(1)} \cdots j_{\sigma(c)}}^{\omega}}{\partial x^{\alpha}} \varepsilon_{\omega}  \\
		&\ - (-1)^{(b-1)(c-1)} \sum_{\sigma \in \Sh(b,c-1)} \sgn(\sigma) 
		(-1)^{b-1}g_{j_{\sigma(1)} \cdots j_{\sigma(b)}}^{\beta} (-1)^{c-1} \psi_{\beta j_{\sigma(b+1)} \cdots j_{\sigma(b+c-1)}}^{\omega} \varepsilon_{\omega}\\
		&\ + (-1)^{(b-1)(c-1)} \sum_{\sigma \in \Sh(b,c-1)} \sgn(\sigma) 
		(-1)^{b-1} \varphi_{j_{\sigma(b+1)}\cdots j_{\sigma(b+c-1)}} \frac{\partial g_{j_{\sigma(1)} \cdots j_{\sigma(b)}}^{\omega}}{\partial x^{\alpha}} \varepsilon_{\omega}  \\ 
		= &\ (-1)^{b-1+c-1} \Big(\sum_{\sigma \in \Sh(c,b-1)} \sgn(\sigma) 
		\psi_{j_{\sigma(1)} \cdots j_{\sigma(c)}}^{\beta}  g_{\beta j_{\sigma(c+1)} \cdots j_{\sigma(b+c-1)}}^{\omega}
		-(-1)^{(b-1)(c-1)} \sum_{\sigma \in \Sh(b-1,c)} \sgn(\sigma) 
		f_{j_{\sigma(1)} \cdots j_{\sigma(b-1)}}^{\alpha} \frac{\partial \psi_{j_{\sigma(b)} \cdots j_{\sigma(b+c-1)}}^{\omega}}{\partial x^{\alpha}}\\
		&\ \hspace{12 mm}- (-1)^{(b-1)(c-1)} \sum_{\sigma \in \Sh(b,c-1)} \sgn(\sigma) 
		g_{j_{\sigma(1)} \cdots j_{\sigma(b)}}^{\beta} \psi_{\beta j_{\sigma(b+1)} \cdots j_{\sigma(b+c-1)}}^{\omega}
		+ \sum_{\sigma \in \Sh(c-1,b)} \sgn(\sigma) 
		\varphi_{j_{\sigma(1)}\cdots j_{\sigma(c-1)}} \frac{\partial g_{j_{\sigma(c)} \cdots j_{\sigma(b+c-1)}}^{\omega}}{\partial x^{\alpha}} \Big)\varepsilon_{\omega}. 
	\end{align*}
	Then 
	\begin{align*}
		[B_{X},B_{Y}]_{\RNA}(\varepsilon_{j_{1}}, \dots, \varepsilon_{j_{b+c-1}}) 
		&= -(-1)^{(b-1)(c-1)} B_{[X,Y]}(\varepsilon_{j_{1}}, \dots, \varepsilon_{j_{b+c-1}})
		= B_{[Y,X]}(\varepsilon_{j_{1}}, \dots, \varepsilon_{j_{b+c-1}}).
	\end{align*}
	According to Proposition~\ref{Prop. B_X as a derivation}, we obtain that, for any $h \in \ci(M)$,
	\begin{align*}
		&\ B_{[X,Y]}(\varepsilon_{j_{1}}, \dots, \varepsilon_{j_{b+c-2}}, h\varepsilon_{j_{b+c-1}}) \\
		=&\ hB_{[X,Y]}(\varepsilon_{j_{1}}, \dots, \varepsilon_{j_{b+c-2}}, \varepsilon_{j_{b+c-1}}) 
		+ (-1)^{b+c-1} \langle a_{[X,Y]}(h), \varepsilon_{j_{1}} \wg \cdots \wg \varepsilon_{j_{b+c-2}} \rangle \varepsilon_{j_{b+c-1}}\\
		=&\ hB_{[X,Y]}(\varepsilon_{j_{1}}, \dots, \varepsilon_{j_{b+c-2}}, \varepsilon_{j_{b+c-1}}) \\
		&\ + (-1)^{b+c-1}
			\Big( 
			\sum_{\sigma \in \Sh(b-1,c-1)} \sgn(\sigma) f_{j_{\sigma(1)} \cdots j_{\sigma(b-1)}}^{\theta} \frac{\partial \varphi_{j_{\sigma(b)} \cdots j_{\sigma(b+c-2)}}^{\alpha}}{\partial x^{\theta}}
			+\sum_{\sigma \in \Sh(b,c-2)} \sgn(\sigma) g_{j_{\sigma(1)} \cdots j_{\sigma(b)}}^{\beta} \varphi_{\beta j_{\sigma(b+1)} \cdots j_{\sigma(b+c-2)}}^{\alpha} \\
			&\ \hspace{16 mm} - (-1)^{(b-1)(c-1)} \sum_{\sigma \in \Sh(c-1,b-1)} \sgn(\sigma) \varphi_{j_{\sigma(1)} \cdots j_{\sigma(c-1)}}^{\theta} \frac{\partial f_{j_{\sigma(c)} \cdots j_{\sigma(b+c-2)}}^{\alpha}}{\partial x^{\theta}}\\
			&\ \hspace{16 mm} - (-1)^{(b-1)(c-1)} \sum_{\sigma \in \Sh(c,b-2)} \sgn(\sigma) \psi_{j_{\sigma(1)} \cdots j_{\sigma(c)}}^{\beta} f_{\beta j_{\sigma(c+1)} \cdots j_{\sigma(b+c-2)}}^{\alpha}
		\Big) \frac{\partial h}{\partial x^{\alpha}} \varepsilon_{j_{b+c-1}}.
	\end{align*}
	We also have
	\begin{align*}
		&\ [B_{X},B_{Y}]_{\RNA}(\varepsilon_{j_{1}}, \dots, \varepsilon_{j_{b+c-2}}, h\varepsilon_{j_{b+c-1}})\\
		= &\ (-1)^{b-1} \sum_{\sigma \in \Sh(c-1,b-1)} \sgn(\sigma) 
		B_{X}(B_{Y}(\varepsilon_{j_{\sigma(1)}}, \dots, \varepsilon_{j_{\sigma(c-1)}}, h\varepsilon_{j_{b+c-1}}), \varepsilon_{j_{\sigma(c)}}, \dots, \varepsilon_{j_{\sigma(b+c-2)}})\\
		&\ + \sum_{\sigma \in \Sh(c,b-2)} \sgn(\sigma) 
		B_{X}(B_{Y}(\varepsilon_{j_{\sigma(1)}}, \dots, \varepsilon_{j_{\sigma(c)}}), \varepsilon_{j_{\sigma(c+1)}}, \dots, \varepsilon_{j_{\sigma(b+c-2)}}, h\varepsilon_{j_{b+c-1}})\\
		&\ - (-1)^{(b-1)(c-1)} (-1)^{c-1} \sum_{\sigma \in \Sh(b-1,c-1)} \sgn(\sigma) 
		B_{Y}(B_{X}(\varepsilon_{j_{\sigma(1)}}, \dots, \varepsilon_{j_{\sigma(b-1)}}, h\varepsilon_{j_{b+c-1}}), \varepsilon_{j_{\sigma(b)}}, \dots, \varepsilon_{j_{\sigma(b+c-2)}})\\
		&\ - (-1)^{(b-1)(c-1)} \sum_{\sigma \in \Sh(b,c-2)} \sgn(\sigma) 
		B_{Y}(B_{X}(\varepsilon_{j_{\sigma(1)}}, \dots, \varepsilon_{j_{\sigma(b)}}), \varepsilon_{j_{\sigma(b+1)}}, \dots, \varepsilon_{j_{\sigma(b+c-2)}}, h\varepsilon_{j_{b+c-1}})\\ 
		= &\ (-1)^{b-1} \sum_{\sigma \in \Sh(c-1,b-1)} \sgn(\sigma) 
		B_{X}(h B_{Y}(\varepsilon_{j_{\sigma(1)}}, \dots, \varepsilon_{j_{\sigma(c-1)}}, \varepsilon_{j_{b+c-1}}), \varepsilon_{j_{\sigma(c)}}, \dots, \varepsilon_{j_{\sigma(b+c-2)}})\\
		&\ + (-1)^{b-1} \sum_{\sigma \in \Sh(c-1,b-1)} \sgn(\sigma) (-1)^{c}
		B_{X}(\langle a_{Y}(h), \varepsilon_{j_{\sigma(1)}} \wg \cdots \wg \varepsilon_{j_{\sigma(c-1)}} \rangle \varepsilon_{j_{b+c-1}}, \varepsilon_{j_{\sigma(c)}}, \dots, \varepsilon_{j_{\sigma(b+c-2)}})\\
		&\ + \sum_{\sigma \in \Sh(c,b-2)} \sgn(\sigma) 
		h B_{X}(B_{Y}(\varepsilon_{j_{\sigma(1)}}, \dots, \varepsilon_{j_{\sigma(c)}}), \varepsilon_{j_{\sigma(c+1)}}, \dots, \varepsilon_{j_{\sigma(b+c-2)}}, \varepsilon_{j_{b+c-1}})\\
		&\ + \sum_{\sigma \in \Sh(c,b-2)} \sgn(\sigma) (-1)^{b}
		\langle a_{X}(h), B_{Y}(\varepsilon_{j_{\sigma(1)}}, \dots, \varepsilon_{j_{\sigma(c)}}) \wg \varepsilon_{j_{\sigma(c+1)}} \wg \cdots \wg \varepsilon_{j_{\sigma(b+c-2)}}\rangle \varepsilon_{j_{b+c-1}}\\
		&\ - (-1)^{(b-1)(c-1)} (-1)^{c-1} \sum_{\sigma \in \Sh(b-1,c-1)} \sgn(\sigma) 
		B_{Y}(h B_{X}(\varepsilon_{j_{\sigma(1)}}, \dots, \varepsilon_{j_{\sigma(b-1)}}, \varepsilon_{j_{b+c-1}}), \varepsilon_{j_{\sigma(b)}}, \dots, \varepsilon_{j_{\sigma(b+c-2)}})\\
		&\ - (-1)^{(b-1)(c-1)} (-1)^{c-1} \sum_{\sigma \in \Sh(b-1,c-1)} \sgn(\sigma) (-1)^{b}
		B_{Y}( \langle a_{X}(h), \varepsilon_{j_{\sigma(1)}} \wg \cdots \wg \varepsilon_{j_{\sigma(b-1)}} \rangle \varepsilon_{j_{b+c-1}}, \varepsilon_{j_{\sigma(b)}}, \dots, \varepsilon_{j_{\sigma(b+c-2)}})\\
		&\ - (-1)^{(b-1)(c-1)} \sum_{\sigma \in \Sh(b,c-2)} \sgn(\sigma) 
		h B_{Y}(B_{X}(\varepsilon_{j_{\sigma(1)}}, \dots, \varepsilon_{j_{\sigma(b)}}), \varepsilon_{j_{\sigma(b+1)}}, \dots, \varepsilon_{j_{\sigma(b+c-2)}}, \varepsilon_{j_{b+c-1}})\\ 
		&\ - (-1)^{(b-1)(c-1)} \sum_{\sigma \in \Sh(b,c-2)} \sgn(\sigma) (-1)^{c}
		\langle a_{Y}(h), B_{X}(\varepsilon_{j_{\sigma(1)}}, \dots, \varepsilon_{j_{\sigma(b)}}) \wg \varepsilon_{j_{\sigma(b+1)}} \wg \cdots \wg \varepsilon_{j_{\sigma(b+c-2)}} \rangle \varepsilon_{j_{b+c-1}}\\ 
		= &\ (-1)^{b-1} \sum_{\sigma \in \Sh(c-1,b-1)} \sgn(\sigma) 
		h B_{X}(B_{Y}(\varepsilon_{j_{\sigma(1)}}, \dots, \varepsilon_{j_{\sigma(c-1)}}, \varepsilon_{j_{b+c-1}}), \varepsilon_{j_{\sigma(c)}}, \dots, \varepsilon_{j_{\sigma(b+c-2)}})\\
		&\ - (-1)^{b-1} \sum_{\sigma \in \Sh(c-1,b-1)} \sgn(\sigma) 
		\langle a_{X}(h), \varepsilon_{j_{\sigma(c)}} \wg \cdots \wg \varepsilon_{j_{\sigma(b+c-2)}} \rangle B_{Y}(\varepsilon_{j_{\sigma(1)}}, \dots, \varepsilon_{j_{\sigma(c-1)}}, \varepsilon_{j_{b+c-1}})\\
		&\ + (-1)^{b-1} \sum_{\sigma \in \Sh(c-1,b-1)} \sgn(\sigma) (-1)^{c}
		B_{X}( \varphi_{j_{\sigma(1)} \cdots j_{\sigma(c-1)}}^{\theta} \frac{\partial h}{\partial x^{\theta}} \varepsilon_{j_{b+c-1}}, \varepsilon_{j_{\sigma(c)}}, \dots, \varepsilon_{j_{\sigma(b+c-2)}})\\
		&\ + \sum_{\sigma \in \Sh(c,b-2)} \sgn(\sigma) 
		h B_{X}(B_{Y}(\varepsilon_{j_{\sigma(1)}}, \dots, \varepsilon_{j_{\sigma(c)}}), \varepsilon_{j_{\sigma(c+1)}}, \dots, \varepsilon_{j_{\sigma(b+c-2)}}, \varepsilon_{j_{b+c-1}})\\
		&\ + \sum_{\sigma \in \Sh(c,b-2)} \sgn(\sigma) (-1)^{b} (-1)^{c-1}
		f_{\omega j_{\sigma(c+1)} \cdots j_{\sigma(b+c-2)}}^{\alpha} \frac{\partial h}{\partial x^{\alpha}} \psi_{j_{\sigma(1)} \dots j_{\sigma(c)}}^{\omega} \varepsilon_{j_{b+c-1}}\\
		&\ - (-1)^{(b-1)(c-1)} (-1)^{c-1} \sum_{\sigma \in \Sh(b-1,c-1)} \sgn(\sigma) 
		h B_{Y}(B_{X}(\varepsilon_{j_{\sigma(1)}}, \dots, \varepsilon_{j_{\sigma(b-1)}}, \varepsilon_{j_{b+c-1}}), \varepsilon_{j_{\sigma(b)}}, \dots, \varepsilon_{j_{\sigma(b+c-2)}})\\
		&\ + (-1)^{(b-1)(c-1)} (-1)^{c-1} \sum_{\sigma \in \Sh(b-1,c-1)} \sgn(\sigma) 
		\langle a_{Y}(h), \varepsilon_{j_{\sigma(b)}} \wg \cdots \wg \varepsilon_{j_{\sigma(b+c-2)}} \rangle B_{X}(\varepsilon_{j_{\sigma(1)}}, \dots, \varepsilon_{j_{\sigma(b-1)}}, \varepsilon_{j_{b+c-1}})\\
		&\ - (-1)^{(b-1)(c-1)} (-1)^{c-1} \sum_{\sigma \in \Sh(b-1,c-1)} \sgn(\sigma) (-1)^{b}
		B_{Y}( f_{j_{\sigma(1)} \cdots j_{\sigma(b-1)}}^{\alpha} \frac{\partial h}{\partial x^{\alpha}} \varepsilon_{j_{b+c-1}}, \varepsilon_{j_{\sigma(b)}}, \dots, \varepsilon_{j_{\sigma(b+c-2)}})\\
		&\ - (-1)^{(b-1)(c-1)} \sum_{\sigma \in \Sh(b,c-2)} \sgn(\sigma) 
		h B_{Y}(B_{X}(\varepsilon_{j_{\sigma(1)}}, \dots, \varepsilon_{j_{\sigma(b)}}), \varepsilon_{j_{\sigma(b+1)}}, \dots, \varepsilon_{j_{\sigma(b+c-2)}}, \varepsilon_{j_{b+c-1}})\\ 
		&\ - (-1)^{(b-1)(c-1)} \sum_{\sigma \in \Sh(b,c-2)} \sgn(\sigma) (-1)^{c} (-1)^{b-1}
		\varphi_{\beta j_{\sigma(b+1)} \cdots j_{\sigma(b+c-2)}}^{\theta} \frac{\partial h}{\partial x^{\theta}} g_{j_{\sigma(1)} \cdots j_{\sigma(b)}}^{\beta} \varepsilon_{j_{b+c-1}}\\ 
		= &\ h [B_{X},B_{Y}]_{\RNA}(\varepsilon_{j_{1}}, \dots, \varepsilon_{j_{b+c-2}}, \varepsilon_{j_{b+c-1}})\\
		&\ - (-1)^{b-1} \sum_{\sigma \in \Sh(c-1,b-1)} \sgn(\sigma) (-1)^{c-1}
		f_{j_{\sigma(c)} \cdots j_{\sigma(b+c-2)}}^{\alpha} \frac{\partial h}{\partial x^{\alpha}} \psi_{j_{\sigma(1)} \cdots j_{\sigma(c-1)} j_{b+c-1}}^{\omega} \varepsilon_{\omega}\\
		&\ + (-1)^{b-1} \sum_{\sigma \in \Sh(c-1,b-1)} \sgn(\sigma) (-1)^{c}
		\varphi_{j_{\sigma(1)} \cdots j_{\sigma(c-1)}}^{\theta} \frac{\partial h}{\partial x^{\theta}} B_{X}(\varepsilon_{j_{b+c-1}}, \varepsilon_{j_{\sigma(c)}}, \dots, \varepsilon_{j_{\sigma(b+c-2)}})\\
		&\ - (-1)^{b-1} \sum_{\sigma \in \Sh(c-1,b-1)} \sgn(\sigma) (-1)^{c}
		f_{j_{\sigma(c)} \cdots j_{\sigma(b+c-2)}}^{\alpha} \frac{\partial(\varphi_{j_{\sigma(1)} \cdots j_{\sigma(c-1)}}^{\theta} \frac{\partial h}{\partial x^{\theta}})}{\partial x^{\alpha}} \varepsilon_{j_{b+c-1}}\\
		&\ + \sum_{\sigma \in \Sh(c,b-2)} \sgn(\sigma) (-1)^{b} (-1)^{c-1}
		f_{\omega j_{\sigma(c+1)} \cdots j_{\sigma(b+c-2)}}^{\alpha} \frac{\partial h}{\partial x^{\alpha}} \psi_{j_{\sigma(1)} \cdots j_{\sigma(c)}}^{\omega} \varepsilon_{j_{b+c-1}}\\
		&\ + (-1)^{(b-1)(c-1)} (-1)^{c-1} \sum_{\sigma \in \Sh(b-1,c-1)} \sgn(\sigma) (-1)^{b-1}
		\varphi_{j_{\sigma(b)} \cdots j_{\sigma(b+c-2)}}^{\alpha} \frac{\partial h}{\partial x^{\alpha}} g_{j_{\sigma(1)} \cdots j_{\sigma(b-1)} j_{b+c-1}}^{\omega} \varepsilon_{\omega}\\
		&\ - (-1)^{(b-1)(c-1)} (-1)^{c-1} \sum_{\sigma \in \Sh(b-1,c-1)} \sgn(\sigma) (-1)^{b}
		f_{j_{\sigma(1)} \cdots j_{\sigma(b-1)}}^{\alpha} \frac{\partial h}{\partial x^{\alpha}} B_{Y}(\varepsilon_{j_{b+c-1}}, \varepsilon_{j_{\sigma(b)}}, \dots, \varepsilon_{j_{\sigma(b+c-2)}})\\
		&\ + (-1)^{(b-1)(c-1)} (-1)^{c-1} \sum_{\sigma \in \Sh(b-1,c-1)} \sgn(\sigma) (-1)^{b}
		\varphi_{j_{\sigma(b)} \cdots j_{\sigma(b+c-2)}}^{\theta} \frac{\partial(f_{j_{\sigma(1)} \cdots j_{\sigma(b-1)}}^{\alpha} \frac{\partial h}{\partial x^{\alpha}})}{\partial x^{\theta}} \varepsilon_{j_{b+c-1}}\\
		&\ - (-1)^{(b-1)(c-1)} \sum_{\sigma \in \Sh(b,c-2)} \sgn(\sigma) (-1)^{c} (-1)^{b-1}
		\varphi_{\beta j_{\sigma(b+1)} \cdots j_{\sigma(b+c-2)}}^{\theta} \frac{\partial h}{\partial x^{\theta}} g_{j_{\sigma(1)} \cdots j_{\sigma(b)}}^{\beta} \varepsilon_{j_{b+c-1}}\\ 
		= &\ h [B_{X},B_{Y}]_{\RNA}(\varepsilon_{j_{1}}, \dots, \varepsilon_{j_{b+c-2}}, \varepsilon_{j_{b+c-1}})\\
		&\ - (-1)^{b-1} \sum_{\sigma \in \Sh(c-1,b-1)} \sgn(\sigma) (-1)^{c-1}
		f_{j_{\sigma(c)} \cdots j_{\sigma(b+c-2)}}^{\alpha} \frac{\partial h}{\partial x^{\alpha}} \psi_{j_{\sigma(1)} \cdots j_{\sigma(c-1)} j_{b+c-1}}^{\omega} \varepsilon_{\omega}\\
		&\ + (-1)^{b-1} \sum_{\sigma \in \Sh(c-1,b-1)} \sgn(\sigma) (-1)^{c} (-1)^{b-1}
		\varphi_{j_{\sigma(1)} \cdots j_{\sigma(c-1)}}^{\theta} \frac{\partial h}{\partial x^{\theta}} g_{j_{b+c-1} j_{\sigma(c)} \cdots j_{\sigma(b+c-2)}}^{\omega} \varepsilon_{\omega} \\
		&\ - (-1)^{b-1} \sum_{\sigma \in \Sh(c-1,b-1)} \sgn(\sigma) (-1)^{c}
		f_{j_{\sigma(c)} \cdots j_{\sigma(b+c-2)}}^{\alpha} \frac{\partial \varphi_{j_{\sigma(1)} \cdots j_{\sigma(c-1)}}^{\theta}}{\partial x^{\alpha}} \frac{\partial h}{\partial x^{\theta}} \varepsilon_{j_{b+c-1}}\\
		&\ - (-1)^{b-1} \sum_{\sigma \in \Sh(c-1,b-1)} \sgn(\sigma) (-1)^{c}
		f_{j_{\sigma(c)} \cdots j_{\sigma(b+c-2)}}^{\alpha} \varphi_{j_{\sigma(1)} \cdots j_{\sigma(c-1)}}^{\theta} \frac{\partial^{2} h }{\partial x^{\alpha} \partial x^{\theta}} \varepsilon_{j_{b+c-1}}\\
		&\ + \sum_{\sigma \in \Sh(c,b-2)} \sgn(\sigma) (-1)^{b} (-1)^{c-1}
		f_{\omega j_{\sigma(c+1)} \cdots j_{\sigma(b+c-2)}}^{\alpha} \frac{\partial h}{\partial x^{\alpha}} \psi_{j_{\sigma(1)} \cdots j_{\sigma(c)}}^{\omega} \varepsilon_{j_{b+c-1}}\\
		&\ + (-1)^{(b-1)(c-1)} (-1)^{c-1} \sum_{\sigma \in \Sh(b-1,c-1)} \sgn(\sigma) (-1)^{b-1}
		\varphi_{j_{\sigma(b)} \cdots j_{\sigma(b+c-2)}}^{\alpha} \frac{\partial h}{\partial x^{\alpha}} g_{j_{\sigma(1)} \cdots j_{\sigma(b-1)} j_{b+c-1}}^{\omega} \varepsilon_{\omega}\\
		&\ - (-1)^{(b-1)(c-1)} (-1)^{c-1} \sum_{\sigma \in \Sh(b-1,c-1)} \sgn(\sigma) (-1)^{b} (-1)^{c-1}
		f_{j_{\sigma(1)} \cdots j_{\sigma(b-1)}}^{\alpha} \frac{\partial h}{\partial x^{\alpha}} \psi_{j_{b+c-1} j_{\sigma(b)} \cdots j_{\sigma(b+c-2)}}^{\omega} \varepsilon_{\omega}\\
		&\ + (-1)^{(b-1)(c-1)} (-1)^{c-1} \sum_{\sigma \in \Sh(b-1,c-1)} \sgn(\sigma) (-1)^{b} 
		\varphi_{j_{\sigma(b)} \cdots j_{\sigma(b+c-2)}}^{\theta} \frac{\partial f_{j_{\sigma(1)} \cdots j_{\sigma(b-1)}}^{\alpha}}{\partial x^{\theta}} \frac{\partial h}{\partial x^{\alpha}} \varepsilon_{j_{b+c-1}}\\
		&\ + (-1)^{(b-1)(c-1)} (-1)^{c-1} \sum_{\sigma \in \Sh(b-1,c-1)} \sgn(\sigma) (-1)^{b} 
		\varphi_{j_{\sigma(b)} \cdots j_{\sigma(b+c-2)}}^{\theta} f_{j_{\sigma(1)} \cdots j_{\sigma(b-1)}}^{\alpha} \frac{\partial^{2} h}{\partial x^{\theta} \partial x^{\alpha}} \varepsilon_{j_{b+c-1}}\\
		&\ - (-1)^{(b-1)(c-1)} \sum_{\sigma \in \Sh(b,c-2)} \sgn(\sigma) (-1)^{c} (-1)^{b-1}
		\varphi_{\beta j_{\sigma(b+1)} \cdots j_{\sigma(b+c-2)}}^{\theta} \frac{\partial h}{\partial x^{\theta}} g_{j_{\sigma(1)} \cdots j_{\sigma(b)}}^{\beta} \varepsilon_{j_{b+c-1}}\\ 
		= &\ h [B_{X},B_{Y}]_{\RNA}(\varepsilon_{j_{1}}, \dots, \varepsilon_{j_{b+c-2}}, \varepsilon_{j_{b+c-1}})\\
		&\ - (-1)^{bc} \Big(
		\sum_{\sigma \in \Sh(b-1,c-1)} \sgn(\sigma) 
		f_{j_{\sigma(1)} \cdots j_{\sigma(c-1)}}^{\theta} \frac{\partial \varphi_{j_{\sigma(c)} \cdots j_{\sigma(b+c-2)}}^{\alpha}}{\partial x^{\theta}} \frac{\partial h}{\partial x^{\alpha}} \varepsilon_{j_{b+c-1}}\\
		&\ \hspace{16 mm} + \sum_{\sigma \in \Sh(b,c-2)} \sgn(\sigma) 
		g_{j_{\sigma(1)} \cdots j_{\sigma(b)}}^{\beta} \varphi_{\beta j_{\sigma(b+1)} \cdots j_{\sigma(b+c-2)}}^{\alpha} \frac{\partial h}{\partial x^{\alpha}} \varepsilon_{j_{b+c-1}}\\ 
		&\ \hspace{16 mm} - (-1)^{(b-1)(c-1)} \sum_{\sigma \in \Sh(c-1,b-1)} \sgn(\sigma)
		\varphi_{j_{\sigma(1)} \cdots j_{\sigma(c-1)}}^{\theta} \frac{\partial f_{j_{\sigma(c)} \cdots j_{\sigma(b+c-2)}}^{\alpha}}{\partial x^{\theta}} \frac{\partial h}{\partial x^{\alpha}} \varepsilon_{j_{b+c-1}}\\
		&\ \hspace{16 mm} - (-1)^{(b-1)(c-1)} \sum_{\sigma \in \Sh(c,b-2)} \sgn(\sigma) 
		\psi_{j_{\sigma(1)} \cdots j_{\sigma(c)}}^{\omega} f_{\omega j_{\sigma(c+1)} \cdots j_{\sigma(b+c-2)}}^{\alpha} \frac{\partial h}{\partial x^{\alpha}} \varepsilon_{j_{b+c-1}}
		\Big).
	\end{align*}
	Thus 
	\begin{align*}
		B_{[X,Y]}(\varepsilon_{j_{1}}, \dots, \varepsilon_{j_{b+c-2}}, h\varepsilon_{j_{b+c-1}}) 
		&= - (-1)^{(b-1)(c-1)} [B_{X},B_{Y}]_{\RNA}(\varepsilon_{j_{1}}, \dots, \varepsilon_{j_{b+c-2}}, h\varepsilon_{j_{b+c-1}}) \\
		&= [B_{Y},B_{X}]_{\RNA}(\varepsilon_{j_{1}}, \dots, \varepsilon_{j_{b+c-2}}, h\varepsilon_{j_{b+c-1}}).
	\end{align*}
	In conclusion, we obtain that 
	$B_{[X,Y]} = [B_{Y},B_{X}]_{\RNA}$.
\end{proof}

Let $(A, [-,-]_{A}, \rho)$ be a Lie algebroid and $Q \in \frakX(A[1])$ be the corresponding homological vector field on the graded manifold $A[1]$, given by Equation~\eqref{Eq. Lie algebroid structure A implies Q}. Proposition~\ref{Prop. equation of B_Q} states that $\chi(Q) = B_{Q} = [-,-]_{A}$. Therefore, we have the following result.

\begin{prop} \label{Prop. chi is a chain map}
	Let $(A, \mu = [-,-]_{A}, \rho)$ be a Lie algebroid and $Q \in \frakX(A[1])$ be the corresponding homological vector field on the graded manifold $A[1]$. Then 
	\begin{align*}
		\chi: (\frakX(A[1]), \dd_{Q} = -[Q,-]) \rightarrow (\C_{\Lie}^{\bullet}(\Gamma(A)), \delta_{\Lie} = -[-,\mu]_{\RNA}), \ X \mapsto B_{X}
	\end{align*}
	is a chain map.
\end{prop}
\begin{proof}
	According to Proposition~\ref{Prop. definition of chi} and Proposition~\ref{Prop. relation between differential and MC Lie}, we have, for any $X \in \frakX(A[1])$,
	\begin{align*}
		\chi (\dd_{Q} (X)) = B_{\dd_{Q} (X)} = B_{-[Q,X]} = - [B_{X},B_{Q}]_{\RNA} = -[B_{X},\mu]_{\RNA} = \delta_{\Lie}(B_{X}) = \delta_{\Lie}(\chi(X)).
	\end{align*}
\end{proof}

Let $(A,\mu=[-,-]_{A},\rho,P)$ be a Nijenhuis Lie algebroid and $Q \in \frakX(A[1])$ be the corresponding homological vector field on the graded manifold $A[1]$. Proposition~\ref{Prop: Chain map Psi} states that there is a chain map $\Psi: (\C_{\Lie}^{\bullet}(\Gamma(A)), \delta_{\Lie} = -[-,\mu]_{\RNA}) \rightarrow (\C_{\NjO}^{\bullet}(\Gamma(A)), \delta_{\NjO} = [P,-]_{\FNA})$. Combined with Proposition~\ref{Prop. chi is a chain map}, we obtain that 
$$ \Psi \circ \chi: (\frakX(A[1]), \dd_{Q} = - [Q,-]) \rightarrow (\C_{\NjO}^{\bullet}(\Gamma(A)), \delta_{\NjO} = [P,-]_{\FNA})$$
is a chain map.
Moreover, Proposition~\ref{Prop. Phi is a tensor} and Proposition~\ref{Prop. relation between algebra and geometry Nijenhuis} state that the image of $\Psi \circ \chi$ falls in the subcomplex $(\C_{\FN}^{\bullet}(A),\dd_{\FN} = [P,-]_{\FN})$ of $(\C_{\NjO}^{\bullet}(\Gamma(A)), \delta_{\NjO} = [P,-]_{\FNA})$, this induces a chain map $\Phi: (\frakX(A[1]), \dd_{Q} = - [Q,-]) \rightarrow (\C_{\FN}(A), \dd_{\FN} = [P,-]_{\FN}), X \mapsto \Phi(X)$, which is consistent with what we originally gave in Equation~\eqref{Eq. definition of chain map Phi}. We summarize this result as follows.

\begin{thm} \label{Thm. Phi is a chain map}
	Let $(A,[-,-]_{A},\rho,P)$ be a Nijenhuis Lie algebroid and $Q \in \frakX(A[1])$ be the corresponding homological vector field on the graded manifold $A[1]$. Then 
	\begin{align*}
		\Phi: (\frakX(A[1]), \dd_{Q} = - [Q,-]) &\rightarrow (\C_{\FN}^{\bullet}(A), \dd_{\FN} = [P,-]_{\FN}), \
		X \mapsto \Phi(X)
	\end{align*}
	is a chain map.
\end{thm}

Now, we define a cohomology theory of Nijenhuis Lie algebroids. 

\begin{defn}
	Let $(A,[-,-]_{A},\rho,P)$ be a Nijenhuis Lie algebroid and $Q \in \frakX(A[1])$ be the corresponding homological vector field on the graded manifold $A[1]$. 
	We define the cochain complex $(\C^\bullet_{\NjLD}(A), \delta_{\NjLD})$ of Nijenhuis Lie algebroid $(A,[-,-]_{A},\rho,P)$ to be the mapping cone of $\Phi$ shifted by $-1$, that is, $\frakX(A[1]) \oplus \C_{\FN}^{\bullet-1}(A)$,
	the differential 
	$$\delta_{\NjLD}: \C^{\bullet}_{\NjLD}(A) \rightarrow \C^{\bullet + 1}_{\NjLD}(A)$$ 
	is defined as
	\begin{align*}
		\delta_{\NjLD}(X, E)
		=&\ (\dd_{Q}(X), -\Phi(X) -\dd_{\FN}(E)),
	\end{align*}
	for any homogeneous elements $ X \in \frakX(A[1])$ and $E \in \C_{\FN}^{\bullet-1}(A)$.
	
	The cohomology of $(\C^\bullet_{\NjLD}(A), \delta_{\NjLD})$, denoted by $\rmH_{\NjLD}^\bullet(A)$, is called the cohomology of the Nijenhuis Lie algebroid $(A,[-,-]_{A},\rho,P)$.
\end{defn}

\medskip

\section{$L_{\infty}$-algebra structure on the cochain complex of a Nijenhuis Lie algebroid} \label{Section: L infty algebras related with Lie algebroids}

In this section, we will consider the $L_\infty$-algebra controlling deformations of Nijenhuis Lie algebroids. Nijenhuis Lie algebroid structures on a Lie algebroid will be realized as Maurer-Cartan elements in an explicitly constructed $L_\infty$-algebra and it will be seen that the cochain complex of a Nijenhuis Lie algebroid is isomorphic to
the underlying complex of the twisted $L_\infty$-algebra by the Maurer-Cartan element corresponding to the Nijenhuis Lie algebroid structure,
see Proposition~\ref{Prop: NjL is MC element algebroid} and Proposition~\ref{Prop: cohomology complex as the underlying complex of L infinity algebra NjL algebroid}.

\smallskip

\subsection{$L_\infty$-algebra structure on $\mathfrak{C}_{\NjLD}(A)$} \label{subsection: L-infty NjL algebroid} \

Let $A$ be a vector bundle over a manifold $M$. 
We define a cohomological graded space 
$$\mathfrak{C}_{\NjLD}(A) := \frakX(A[1]) \oplus \C_{\FN}^{\bullet}(A),$$
where $\frakX(A[1])$ is the space of vector fields on graded manifold $A[1]$, and $\C_{\FN}^{\bullet}(A)$ is the space of smooth sections of bundle $\wedge^{\bullet}A^{\vee} \ot A$.
We have a degree $0$ map 
\begin{align*}
	\chi: \frakX(A[1]) \longrightarrow \Hom_{\bbR}(\wg^{\bullet} \Gamma(A), \Gamma(A)), \ X \mapsto B_{X},
\end{align*}
and a shift map of degree $-1$:
\begin{align*}
	\Upsilon: \C_{\FN}^{\bullet}(A) \cong \Hom_{\ci(M)}(\wg^{\bullet} \Gamma(A), \Gamma(A)) \longrightarrow \Hom_{\ci(M)}(\wg^{\bullet} \Gamma(A), \Gamma(A[1])), \
	\overline{G} \mapsto G := \overline{G}
\end{align*}
with $|G| := |\overline{G}| - 1$.


Now, let's construct an $L_\infty$-algebra structure on $\mathfrak{C}_{\NjLD}(A)$.
The operators $\{l_b\}_{b \geqslant 1}$, with $|l_b| = 2-b$, on $\mathfrak{C}_{\NjLD}(A)$ are defined as follows:
\begin{enumerate}
		
	\item For homogeneous elements $X, Y \in \frakX(A[1])$, define 
	$$l_2(X\ot Y):= [Y, X] \in \frakX(A[1]), $$
	where the operation $[-, -]$ is the commutator of vector fields on graded manifold $A[1]$.
		
	\item \label{Defn. define l b+1 algebroid}
	Let $b \geqslant 1$. For homogeneous elements $ X \in \frakX(A[1])$ of degree $b-1$ and $\overline{G}_{i} \in \C_{\FN}^{r_{i}}(A)$
	 of degree $r_{i}$, $1 \leqslant i \leqslant b$, define 
	$$l_{b+1}(X \ot \overline{G}_{1}\ot \cdots \ot \overline{G}_{b}) \in \C_{\FN}^{r_{1} + \dots + r_{b}}(A)$$ 
	as :
	\begin{align*}
		\qquad l_{b+1}&(X \ot \overline{G}_{1}\ot \cdots \ot \overline{G}_{b}) := 
		\sum_{\sigma \in \mathbb{S}_{b}} \sum_{k=0}^{b} (-1)^{\eta_{\sigma}} (-1)^{\xi^{\sigma}_{k}} 
		\Upsilon^{-1} 
		G_{\sigma (1)}
			\{ \cdots
				\{
					G_{\sigma(k)}
						\{B_{X}
							\{G_{\sigma(k+1)}, \dots, G_{\sigma(b)}\}
						\}
				\}
			\cdots \}, 
	\end{align*}
	where 
	$(-1)^{\eta_{\sigma}} = \chi(\sigma; \overline{G}_{1}, \dots, \overline{G}_{b}) (-1)^{ \sum\limits_{p=1}^{b-1} \sum\limits_{j=1}^{p} |\overline{G}_{\sigma(j)}|}$,
	$(-1)^{\xi^{\sigma}_{k}} = (-1)^{|X| (\sum\limits_{i=1}^{k}(|\overline{G}_{\sigma(i)}|-1) ) + k}$,
	and the operation $ -\{-\} $ is the shuffle brace operation introduced in Section~\ref{Section: Cohomology and homotopy theory of Lie algebras}. According to  Lemma~\ref{Lem. check tensor property} below, the image of $l_{b+1}$ still belongs to $\C_{\FN}^{\bullet}(A)$.
	
	\item Let $b\geqslant 1$. For homogeneous elements $ X \in \frakX(A[1])$ of degree $b-1$ and $\overline{G}_{1}, \dots, \overline{G}_{b} \in \C_{\FN}^{\bullet}(A) $, for $1 \leqslant k \leqslant b$, define 
	$$l_{b+1}(\overline{G}_1\ot \cdots \ot \overline{G}_k \ot X \ot \overline{G}_{k+1}\ot \cdots\ot \overline{G}_b) \in \C_{\FN}^{\bullet}(A)$$ to be
	\begin{align*}
		l_{b+1}(\overline{G}_1\ot \cdots \ot \overline{G}_k \ot X \ot \overline{G}_{k+1}\ot \cdots\ot \overline{G}_b)
		:=(-1)^{|X|(\sum\limits_{j=1}^k|\overline{G}_j|)+k} l_{b+1}(X \ot \overline{G}_1\ot \cdots \ot \overline{G}_b),
	\end{align*}
	where the RHS has been introduced in \eqref{Defn. define l b+1 algebroid}.	
	
	\item All other components of operators $\{l_n\}_{n\geqslant 1}$ vanish.
\end{enumerate}

\begin{thm} \label{Thm: L-infty of NjL algebroid}
	Given a vector bundle $A$ over a manifold $M$, the graded space ${\mathfrak{C}_{\NjLD}}(A)$ endowed with operations $\{l_b\}_{b\geqslant 1}$ defined above forms an $L_\infty$-algebra.
\end{thm}
\begin{proof}
	It can be checked that the map $\chi$ is a morphism of graded Lie algebras. Then this result can be implied by Definition~\ref{Defn. L-infty algebra as we need} and the following Lemma~\ref{Lem. check tensor property}.
\end{proof}

\begin{lem}\label{Lem. check tensor property}
	Let $b \geqslant 1$. For homogeneous elements $ X \in \frakX(A[1])$ of degree $b-1$ and $\overline{G}_{i} \in \C_{\FN}^{r_{i}}(A)$ 
	of degree $r_{i}$, $1 \leqslant i \leqslant b$, the element $l_{b+1}(X \ot \overline{G}_{1}\ot \cdots \ot \overline{G}_{b})$ defined above belongs to $\C_{\FN}^{r_{1} + \dots + r_{b}}(A)$.
\end{lem}
\begin{proof}
	One only need to check that, for any $E_{1}, E_{2}, \dots, E_{r_{1} + \dots + r_{b}} \in \Gamma(A)$ and $f \in \ci(M)$,
	$$l_{b+1}(X \ot \overline{G}_{1}\ot \cdots \ot \overline{G}_{b})(f E_{1}, E_{2}, \dots, E_{r_{1} + \dots + r_{b}}) = f l_{b+1}(X \ot \overline{G}_{1}\ot \cdots \ot \overline{G}_{b})(E_{1}, E_{2}, \dots, E_{r_{1} + \dots + r_{b}}).$$
	By definition,
	\begin{align*}
		&\ l_{b+1}(X \ot \overline{G}_{1}\ot \cdots \ot \overline{G}_{b})(E_{1}, E_{2}, \dots, E_{r_{1} + \dots + r_{b}}) \\
		=&\ \sum_{\sigma \in \mathbb{S}_{b}} \sum_{k=0}^{b} (-1)^{\eta_{\sigma}} (-1)^{\xi^{\sigma}_{k}} 
		\Upsilon^{-1} 
		G_{\sigma (1)}
			\{ \cdots
				\{
					G_{\sigma(k)}
						\{B_{X}
							\{G_{\sigma(k+1)}, \dots, G_{\sigma(b)}\}
						\}
				\}
			\cdots \} 
		(E_{1}, E_{2}, \dots, E_{r_{1} + \dots + r_{b}})\\
		=&\ \sum_{\sigma \in \mathbb{S}_{b}} \sum_{k=0}^{b} (-1)^{\eta_{\sigma}} (-1)^{\xi^{\sigma}_{k}} 
		\sum_{\tau \in \sh} \sgn(\tau) 
		\sum_{\substack{1 \leqslant i_{1} \leqslant i_{2} \leqslant \dots \leqslant i_{k-1} \leqslant i_{k} \\ \leqslant i_{k+1} < i_{k+2} < \dots < i_{b} }}
		(-1)^{\sum\limits_{q=2}^{k}|G_{\sigma(q)}|(i_{q-1}-1) + \sum\limits_{q=k+1}^{b}|G_{\sigma(q)}|(i_{q}-1)} \Upsilon^{-1}\\
		&
		G_{\sigma(1)}(E_{\tau(1)}, E_{\tau(2)}, \dots, 
			G_{\sigma(2)}(E_{\tau(i_{1})}, E_{\tau(i_{1}+1)}, \cdots
					G_{\sigma(k)}(E_{\tau(i_{k-1})}, E_{\tau(i_{k-1}+1)}, \dots,\\ 
						&\quad B_{X}\Big(E_{\tau(i_{k})}, E_{\tau(i_{k}+1)}, \dots, 
							G_{\sigma(k+1)}(E_{\tau(i_{k+1})}, \dots), \dots,
							G_{\sigma(k+2)}(E_{\tau(i_{k+2})}, \dots), \dots,
							G_{\sigma(b)}(E_{\tau(i_{b})}, \dots), \dots \Big),
							\\
					&\dots)
			\cdots),
		\dots).
	\end{align*}
	Therefore, it can be concluded that
	\begin{align*}
		&\ l_{b+1}(X \ot \overline{G}_{1}\ot \cdots \ot \overline{G}_{b})(fE_{1}, E_{2}, \dots, E_{r_{1} + \dots + r_{b}}) \\ 
		=&\ \sum_{\sigma \in \mathbb{S}_{b}} \sum_{k=1}^{b} (-1)^{\eta_{\sigma}} (-1)^{\xi^{\sigma}_{k}} 
		\sum_{\tau \in \sh} \sgn(\tau) 
		\sum_{\substack{1 < i_{1} \leqslant i_{2} \leqslant \dots \leqslant i_{k-1} \leqslant i_{k} \\ \leqslant i_{k+1} < i_{k+2} < \dots < i_{b} }}
		(-1)^{\sum\limits_{q=2}^{k}|G_{\sigma(q)}|(i_{q-1}-1) + \sum\limits_{q=k+1}^{b}|G_{\sigma(q)}|(i_{q}-1)} \Upsilon^{-1}\\
		&\ 
		G_{\sigma(1)}(f E_{1}, E_{\tau(2)}, \dots, 
			G_{\sigma(2)}(E_{\tau(i_{1})}, E_{\tau(i_{1}+1)}, \cdots
					G_{\sigma(k)}(E_{\tau(i_{k-1})}, E_{\tau(i_{k-1}+1)}, \dots,\\ 
						&\quad B_{X}\Big(E_{\tau(i_{k})}, E_{\tau(i_{k}+1)}, \dots, 
							G_{\sigma(k+1)}(E_{\tau(i_{k+1})}, \dots), \dots,
							G_{\sigma(k+2)}(E_{\tau(i_{k+2})}, \dots), \dots,
							G_{\sigma(b)}(E_{\tau(i_{b})}, \dots), \dots \Big),
							\\
					&\dots)
			\cdots),
		\dots)\\
		&+ \ \sum_{\sigma \in \mathbb{S}_{b}} \sum_{k=1}^{b} (-1)^{\eta_{\sigma}} (-1)^{\xi^{\sigma}_{k}} 
		\sum_{\tau \in \sh} \sgn(\tau) 
		\sum_{\substack{1 = i_{1} < i_{2} \leqslant \dots \leqslant i_{k-1} \leqslant i_{k} \\ \leqslant i_{k+1} < i_{k+2} < \dots < i_{b} }}
		(-1)^{\sum\limits_{q=2}^{k}|G_{\sigma(q)}|(i_{q-1}-1) + \sum\limits_{q=k+1}^{b}|G_{\sigma(q)}|(i_{q}-1)} \Upsilon^{-1}\\
		&\ 
		G_{\sigma(1)}( 
			G_{\sigma(2)}(f E_{1}, E_{\tau(2)}, \cdots
					G_{\sigma(k)}(E_{\tau(i_{k-1})}, E_{\tau(i_{k-1}+1)}, \dots,\\ 
						&\quad B_{X}\Big(E_{\tau(i_{k})}, E_{\tau(i_{k}+1)}, \dots, 
							G_{\sigma(k+1)}(E_{\tau(i_{k+1})}, \dots), \dots,
							G_{\sigma(k+2)}(E_{\tau(i_{k+2})}, \dots), \dots,
							G_{\sigma(b)}(E_{\tau(i_{b})}, \dots), \dots \Big),
							\\
					&\dots)
			\cdots),
		\dots)\\
		&+ \cdots\\
		&+ \ \sum_{\sigma \in \mathbb{S}_{b}} \sum_{k=1}^{b} (-1)^{\eta_{\sigma}} (-1)^{\xi^{\sigma}_{k}} 
		\sum_{\tau \in \sh} \sgn(\tau) 
		\sum_{\substack{1 = i_{1} = i_{2} = \dots = i_{k-1} < i_{k} \\ \leqslant i_{k+1} < i_{k+2} < \dots < i_{b} }}
		(-1)^{\sum\limits_{q=2}^{k}|G_{\sigma(q)}|(i_{q-1}-1) + \sum\limits_{q=k+1}^{b}|G_{\sigma(q)}|(i_{q}-1)} \Upsilon^{-1}\\
		&\ 
		G_{\sigma(1)}( 
			G_{\sigma(2)}( \cdots G_{\sigma(k)}(f E_{1}, E_{\tau(2)}, \dots,\\
						&\quad B_{X}\Big(E_{\tau(i_{k})}, E_{\tau(i_{k}+1)}, \dots, 
							G_{\sigma(k+1)}(E_{\tau(i_{k+1})}, \dots), \dots,
							G_{\sigma(k+2)}(E_{\tau(i_{k+2})}, \dots), \dots,
							G_{\sigma(b)}(E_{\tau(i_{b})}, \dots), \dots \Big),
							\\
					&\dots)
			\cdots),
		\dots)\\
		&+ \ \sum_{\sigma \in \mathbb{S}_{b}} \sum_{k=1}^{b} (-1)^{\eta_{\sigma}} (-1)^{\xi^{\sigma}_{k}} 
		\sum_{\tau \in \sh} \sgn(\tau) 
		\sum_{\substack{1 = i_{1} = i_{2} = \dots = i_{k-1} = i_{k} \\ < i_{k+1} < i_{k+2} < \dots < i_{b} }}
		(-1)^{\sum\limits_{q=2}^{k}|G_{\sigma(q)}|(i_{q-1}-1) + \sum\limits_{q=k+1}^{b}|G_{\sigma(q)}|(i_{q}-1)} \Upsilon^{-1}\\
		&\ 
		G_{\sigma(1)}( 
			G_{\sigma(2)}( \cdots G_{\sigma(k)}(\\
						&\quad B_{X}\Big(f E_{1}, E_{\tau(2)}, \dots, 
							G_{\sigma(k+1)}(E_{\tau(i_{k+1})}, \dots), \dots,
							G_{\sigma(k+2)}(E_{\tau(i_{k+2})}, \dots), \dots,
							G_{\sigma(b)}(E_{\tau(i_{b})}, \dots), \dots \Big),
							\\
					&E_{\tau(k+r_{\sigma(k+1)} + \dots + r_{\sigma(b)} + 1)}, \dots)
			\cdots),
		\dots)\\
		&+ \ \sum_{\sigma \in \mathbb{S}_{b}} \sum_{k=0}^{b-1} (-1)^{\eta_{\sigma}} (-1)^{\xi^{\sigma}_{k}} 
		\sum_{\tau \in \sh} \sgn(\tau) 
		\sum_{\substack{1 = i_{1} = i_{2} = \dots = i_{k-1} = i_{k} \\ = i_{k+1} < i_{k+2} < \dots < i_{b} }}
		(-1)^{\sum\limits_{q=2}^{k}|G_{\sigma(q)}|(i_{q-1}-1) + \sum\limits_{q=k+1}^{b}|G_{\sigma(q)}|(i_{q}-1)} \Upsilon^{-1}\\
		&\ 
		G_{\sigma(1)}( 
			G_{\sigma(2)}( \cdots G_{\sigma(k)}(\\
						&\quad B_{X}\Big(
							G_{\sigma(k+1)}(f E_{1}, E_{\tau(2)}, \dots), E_{\tau(r_{\sigma(k+1)}+1)}, \dots,
							G_{\sigma(k+2)}(E_{\tau(i_{k+2})}, \dots), \dots,
							G_{\sigma(b)}(E_{\tau(i_{b})}, \dots), \dots \Big),
							\\
					&E_{\tau(k+r_{\sigma(k+1)} + \dots + r_{\sigma(b)} + 1)}, \dots)
			\cdots),
		\dots)\\ 
		=&\ f l_{b+1}(X \ot \overline{G}_{1}\ot \cdots \ot \overline{G}_{b})(E_{1}, E_{2}, \dots, E_{r_{1} + \dots + r_{b}})\\
		&- \ \sum_{\sigma \in \mathbb{S}_{b}} \sum_{k=1}^{b} (-1)^{\eta_{\sigma}} (-1)^{\xi^{\sigma}_{k}} 
		\sum_{\tau \in \sh} \sgn(\tau) 
		\sum_{\substack{1 = i_{1} = i_{2} = \dots = i_{k-1} = i_{k} \\ < i_{k+1} < i_{k+2} < \dots < i_{b} }}
		(-1)^{\sum\limits_{q=2}^{k}|G_{\sigma(q)}|(i_{q-1}-1) + \sum\limits_{q=k+1}^{b}|G_{\sigma(q)}|(i_{q}-1)} \Upsilon^{-1}\\
		& \
		\underline{\begin{aligned}
			& \langle a_{X}(f), E_{\tau(2)} \wg E_{\tau(3)} \wg \cdots \wg 
			G_{\sigma(k+1)}(E_{\tau(i_{k+1})}, \dots) \wg \cdots \wg
			G_{\sigma(b)}(E_{\tau(i_{b})}, \dots) \wg \cdots \rangle \\
			&\quad G_{\sigma(1)}( 
			G_{\sigma(2)}( \cdots G_{\sigma(k)}(E_{1},E_{\tau(k+r_{\sigma(k+1)} + \dots + r_{\sigma(b)} + 1)}, \dots)
			\cdots),
		\dots)
		\end{aligned}}_{(\mathrm{I})}\\
		&- \ \sum_{\sigma \in \mathbb{S}_{b}} \sum_{k=0}^{b-1} (-1)^{\eta_{\sigma}} (-1)^{\xi^{\sigma}_{k}} 
		\sum_{\tau \in \sh} \sgn(\tau) 
		\sum_{\substack{1 = i_{1} = i_{2} = \dots = i_{k-1} = i_{k} \\ = i_{k+1} < i_{k+2} < \dots < i_{b} }}
		(-1)^{\sum\limits_{q=2}^{k}|G_{\sigma(q)}|(i_{q-1}-1) + \sum\limits_{q=k+1}^{b}|G_{\sigma(q)}|(i_{q}-1)} \Upsilon^{-1}\\
		&\ 
		\underline{\begin{aligned}
			&\langle a_{X}(f),
							E_{\tau(r_{\sigma(k+1)}+1)} \wg \cdots \wg
							G_{\sigma(k+2)}(E_{\tau(i_{k+2})}, \dots) \wg \cdots \wg
							G_{\sigma(b)}(E_{\tau(i_{b})}, \dots) \wg \cdots \wg E_{\tau(k+r_{\sigma(k+1)} + \dots + r_{\sigma(b)})}\rangle\\
			&\quad G_{\sigma(1)}( 
			G_{\sigma(2)}( \cdots G_{\sigma(k)}(G_{\sigma(k+1)}(E_{1}, E_{\tau(2)}, E_{\tau(3)}, \dots, E_{\tau(r_{\sigma(k+1)})}), E_{\tau(k+r_{\sigma(k+1)} + \dots + r_{\sigma(b)} + 1)}, \dots)
			\cdots),
		\dots).
		\end{aligned}}_{(\mathrm{II})}\\
	\end{align*}

	Notice that, $\xi^{\sigma}_{k+1} - \xi^{\sigma}_{k} = |X|(|\overline{G}_{\sigma(k+1)}|-1)+1 = (b-1)(r_{\sigma(k+1)}-1)+1$. 
	Careful calculation shows that after adjusting the local shuffles $\tau$ with respect to elements 
	$$E_{\tau(2)}, E_{\tau(3)}, \dots, E_{\tau(r_{\sigma(k+1)})}, E_{\tau(r_{\sigma(k+1)}+1)}, \dots, E_{\tau(k+r_{\sigma(k+1)} + \dots + r_{\sigma(b)})}$$
	and recalculating the Koszul sign of $G_{\sigma(k+2)}, \dots, G_{\sigma(b)}$ acting on these elements, $(\mathrm{II})$  can be transformed into $-(\mathrm{I})$, this complete the proof. Specifically,
	\begin{align*}
		(\mathrm{II}) =
		& \ \sum_{\sigma \in \mathbb{S}_{b}} \sum_{k=0}^{b-1} (-1)^{\eta_{\sigma}} (-1)^{\xi^{\sigma}_{k+1}} 
		\sum_{\tau \in \sh} \sgn(\tau) 
		\sum_{\substack{1 = i_{1} = i_{2} = \dots = i_{k-1} = i_{k} \\ = i_{k+1} < i_{k+2} < \dots < i_{b} }}
		(-1)^{\sum\limits_{q=2}^{k+1}|G_{\sigma(q)}|(i_{q-1}-1) + \sum\limits_{q=k+2}^{b}|G_{\sigma(q)}|(i_{q}-1)} \Upsilon^{-1}\\
		&\ \langle a_{X}(f),
			E_{\tau(2)} \wg E_{\tau(3)} \wg \cdots \wg
			G_{\sigma(k+2)}(E_{\tau(i_{k+2})}, \dots) \wg \cdots \wg
			G_{\sigma(b)}(E_{\tau(i_{b})}, \dots) \wg \cdots \rangle\\
		&\quad G_{\sigma(1)}( 
			G_{\sigma(2)}( \cdots G_{\sigma(k)}( 
							G_{\sigma(k+1)}(E_{1}, E_{\tau(k+1+r_{\sigma(k+2)} + \dots + r_{\sigma(b)} + 1)}, \dots), 
							E_{\tau(k+1+r_{\sigma(k+1)} + \dots + r_{\sigma(b)})},
							\dots)
			\cdots),
		\dots)\\
		=&\ \sum_{\sigma \in \mathbb{S}_{b}} \sum_{k=1}^{b} (-1)^{\eta_{\sigma}} (-1)^{\xi^{\sigma}_{k}} 
		\sum_{\tau \in \sh} \sgn(\tau) 
		\sum_{\substack{1 = i_{1} = i_{2} = \dots = i_{k-1} = i_{k} \\ < i_{k+1} < i_{k+2} < \dots < i_{b} }}
		(-1)^{\sum\limits_{q=2}^{k}|G_{\sigma(q)}|(i_{q-1}-1) + \sum\limits_{q=k+1}^{b}|G_{\sigma(q)}|(i_{q}-1)} \Upsilon^{-1}\\
		&\ \langle a_{X}(f),
		E_{\tau(2)} \wg E_{\tau(3)} \wg \cdots \wg
		G_{\sigma(k+1)}(E_{\tau(i_{k+1})}, \dots) \wg \cdots \wg
		G_{\sigma(b)}(E_{\tau(i_{b})}, \dots) \wg \cdots \rangle\\
		&\quad G_{\sigma(1)}( 
			G_{\sigma(2)}( \cdots G_{\sigma(k-1)}(G_{\sigma(k)}(E_{1}, E_{\tau(k+r_{\sigma(k+1)} + \dots + r_{\sigma(b)} + 1)}, \dots),
			E_{\tau(k+r_{\sigma(k)} + \dots + r_{\sigma(b)})},
							\dots)
			\cdots),
		\dots)\\
		=&\ - (\mathrm{I}).
	\end{align*}
\end{proof}

\smallskip

\subsection{Realising Nijenhuis Lie algebroid structures as Maurer-Cartan elements}\

\begin{prop} \label{Prop: NjL is MC element algebroid} 
	Let $A$ be a vector bundle over a manifold $M$. Then a Nijenhuis Lie algebroid structure on $A$ is equivalent to a Maurer-Cartan element in the $L_\infty$-algebra ${\mathfrak{C}_{\NjLD}}(A)$ introduced above.
\end{prop} 

\begin{proof}
	Since $\Gamma(A)$ is concentrated in degree 0, the degree $1$ part of ${\mathfrak{C}_{\NjLD}}(A)$ is $\frakX(A[1])^{1} \oplus \C_{\FN}^{1}(A) $. Take a degree 1 element $\alpha=(Q, \overline{P})$ in ${\mathfrak{C}_{\NjLD}}(A)$,
	then 
	\begin{align*}
			l_2(\alpha\ot \alpha)&=(l_2(Q \ot Q), 0) = ( [Q, Q], 0), \\
			l_3(\alpha^{\ot 3})&=(0, l_3(Q \ot \overline{P} \ot \overline{P} )+l_3(\overline{P} \ot Q \ot \overline{P})+l_3(\overline{P} \ot \overline{P} \ot Q ))\\
			&=(0, 3l_3(Q \ot \overline{P} \ot \overline{P}))\\
			&=\Big(0, -6 \Upsilon^{-1} \big(B_{Q} \{P, P\}- P \{ B_{Q} \{ P\}\} + P\{P \{B_{Q}\}\}\big)\Big), 
	\end{align*}
	and $l_b(\alpha^{\ot b})=0$ for $b \ne 2, 3$.
	By expanding the Maurer-Cartan Equation~(\ref{Eq: mc-equation}), we get
	\begin{eqnarray} \label{Eq: Lie-MC NjL algebroid} 
		[Q, Q]=0, 
	\end{eqnarray}
	\begin{eqnarray} \label{Eq: NjL-MC NjL algebroid} 
		B_{Q} \{P, P\}- P \{ B_{Q} \{ P\}\} + P\{P \{B_{Q}\}\} =0.
	\end{eqnarray}
	By Proposition~\ref{Prop. relation between Lie algebroids and dg manifolds} and Proposition~\ref{Prop. equation of B_Q}, one can find that Equation~(\ref{Eq: Lie-MC NjL algebroid}) is equivalent to saying that $A$ has a Lie algebroid structure with the Lie bracket $[-,-]_{A} = B_{Q}$ on $\Gamma(A)$, and Equation~(\ref{Eq: NjL-MC NjL algebroid}) says exactly that $P$ is a Nijenhuis operator on the Lie algebroid $A$. The proof in the opposite direction is obvious.
\end{proof}

\begin{prop} \label{Prop: cohomology complex as the underlying complex of L infinity algebra NjL algebroid}
	Let $(A, [-,-]_{A}, \rho, P)$ be a Nijenhuis Lie algebroid. Twist the $L_\infty$-algebra $\mathfrak{C}_{\NjLD}(A)$ by the Maurer-Cartan element corresponding to the Nijenhuis Lie algebroid structure $(A, [-,-]_{A}, \rho, P)$, then its underlying complex is isomorphic to $\C^\bullet_{\NjLD}(A)$.
\end{prop}

\begin{proof}
	By Proposition~\ref{Prop: NjL is MC element algebroid}, the Nijenhuis Lie algebra structure $(A, [-,-]_{A}, \rho, P)$ is equivalent to a Maurer-Cartan element $\alpha=(Q, \overline{P})$ in the $L_\infty$-algebra $\mathfrak{C}_{\NjLD}(A)$.
	By Proposition~\ref{Prop: deformed-L-infty}, this Maurer-Cartan element induces a new $L_\infty$-algebra structure $\{l_b^\alpha\}_{b\geqslant 1}$ on the graded space $\mathfrak{C}_{\NjLD}(A)$.
	More precisely, for any $(X, \overline{G}) \in \frakX(A[1])^{b-1} \oplus \C_{\FN}^{b-1}(A) \subset \mathfrak{C}_{\NjLD}(A)^{b-1}$ with $ b \geqslant 1 $, we have  
	\begin{align*}
		l_1^\alpha(X)
		&=l_{1}(X) -l_2(\alpha \ot X) + \sum_{i=2}^{b} (-1)^{\frac{i(i+1)}{2}} \frac{1}{i!} l_{i+1}(\alpha^{\ot i}\ot X)\\
		&=-l_2(Q \ot X) -l_2(\overline{P} \ot X) + \sum_{i=2}^{b} (-1)^{\frac{i(i+1)}{2}} \frac{1}{i!} l_{i+1}(\overline{P}^{\ot i}\ot X)\\
		&=-[X, Q] + (-1)^{\frac{b(b+1)}{2}} \frac{1}{b!} l_{b+1}(\overline{P}^{\ot b}\ot X).
	\end{align*}
	Note that $-[X, Q]$ corresponds to $-(-1)^{b-1} \dd_{Q}(X)$ in $\frakX(A[1])$ and  
	\begin{align*}
	 (-1)^{\frac{b(b+1)}{2}} \frac{1}{b!} l_{b+1}(\overline{P}^{\ot b}\ot X)
		&= (-1)^{\frac{b(b+1)}{2}} (-1)^{b|X|+b} \frac{1}{b!} l_{b+1}( X \ot \overline{P}^{\ot b})\\
		&= \sum_{k=0}^{b} (-1)^{k} \Upsilon^{-1} \underbrace{P \circ \cdots \circ P}_{k} \circ B_{X} \{\underbrace{P, \dots, P}_{b-k}\}, 
	\end{align*}
	which corresponds to $\Phi(X)$. We also have
	\begin{align*}
		l_1^{\alpha}(\overline{G})
		=&\ l_1(\overline{G}) - l_2(\alpha \ot \overline{G}) - \frac{1}{2!} l_3(\alpha \ot \alpha \ot \overline{G})\\
		=&\ -l_3(Q \ot \overline{P} \ot \overline{G})\\
		=&\ \Upsilon^{-1} \big(B_{Q} \{P, G\} - P \{B_{Q}\{G\}\} + (-1)^{b} P \{G\{B_{Q}\}\} 
		+ B_{Q} \{G, P\} - (-1)^{b} G\{B_{Q}\{P\}\} + (-1)^{b} G \{P\{B_{Q}\}\}\big), 
	\end{align*}
	which corresponds to $ (-1)^{b} \dd_{\FN}(\overline{G})$.
	
	In conclusion, 	we have that the underlying complex of the twisted $L_\infty$-algebra $\mathfrak{C}_{\NjLD}(A)$ by Maurer-Cartan element $\alpha$ is isomorphic to the cochain complex $\C^\bullet_{\NjLD}(A)$. 
\end{proof}

Although $\mathfrak{C}_{\NjLD}(A)$ is an $L_\infty$-algebra, the next proposition shows that once the Lie algebroid structure on $A$ is fixed, the graded space $\C_{\FN}^{\bullet}(A)$, which controls deformations of Nijenhuis operators, is a genuine graded Lie algebra. 

\begin{prop} \label{Prop: justifying Cohomology theory of Nijenhuis operator algebroid}
	Let $(A, [-,-]_{A}, \rho)$ be a Lie algebroid. Then we have the following results:
	\begin{enumerate}
		\item The graded space $\C_{\FN}^{\bullet}(A)$ can be endowed with a graded Lie algebra structure, and a Nijenhuis operator on $(A, [-,-]_{A}, \rho)$ is equivalent to a Maurer-Cartan element in this graded Lie algebra.
		\item Given a Nijenhuis operator $P$ on Lie algebroid $(A, [-,-]_{A}, \rho)$, the underlying complex of the twisted dg Lie algebra $\C_{\FN}^{\bullet}(A)$ by the corresponding Maurer-Cartan element is isomorphic to the Fr\"olicher-Nijenhuis cochain complex of Nijenhuis operator $P$, introduced in Theorem~\ref{Thm. geometric version of cohomology of Nijenhuis operator}.
	\end{enumerate}
\end{prop}

\begin{proof}
	(i) The Lie algebroid structure of $A$ corresponds to a degree $1$ vector field $Q$ in $\frakX(A[1])$, which satisfies $[Q, Q]=0$. 
	Then $\alpha=(Q, 0)$ is naturally a Maurer-Cartan element in $L_\infty$-algebra $\mathfrak{C}_{\NjLD}(A)$.
	By the construction of $\{l_b\}_{b\geqslant 1}$ on $\mathfrak{C}_{\NjLD}(A)$, the graded subspace $\C_{\FN}^{\bullet}(A)$ is closed under the action of $\{l_{b}^\alpha\}_{b\geqslant 1}$. Note that $ l_1^\alpha \equiv 0 $ on  $\C_{\FN}^{\bullet}(A)$, and the restriction of $l_n^\alpha$ on  $\C_{\FN}^{\bullet}(A)$ are also $0$ for $n\geqslant 3$. Thus $(\C_{\FN}^{\bullet}(A), l_2^\alpha)$ forms a graded Lie algebra.
	More explicitly, for any $\overline{F} \in \C_{\FN}^{r}(A)$, $\overline{G} \in \C_{\FN}^{k}(A)$, $|\overline{F}| = r$, $|\overline{G}| = k$, we have	
	\begin{align*}
		l_2^\alpha(\overline{F} \ot \overline{G} )
		=&\ l_{2}(\overline{F}, \overline{G})+l_{3}(\alpha, \overline{F}, \overline{G})\\
		=&\ l_{3}(Q, \overline{F}, \overline{G})\\
		=&\ \Upsilon^{-1}\big((-1)^{r}B_{Q}\{F, G\} + F\{B_{Q} \{G\}\} + (-1)^{k}F\{G\{B_{Q}\}\}\\
		&+(-1)^{rk+k+1}B_{Q}\{G, F\}-(-1)^{rk} G\{B_{Q} \{F\} \}+ (-1)^{rk+r+1}G\{F\{B_{Q}\}\}\big),
	\end{align*}
	which corresponds to $(-1)^{rk}[F,G]_{\FN}$.

	Take a degree $1$ element $\overline{P} \in \C_{\FN}^{1}(A)$, then $\overline{P}$
	satisfies the Maurer-Cartan equation
	$$-\frac{1}{2}l_2^{\alpha}(\overline{P} \ot \overline{P})=0 $$
	if and only if
	$$B_{Q}\{P, P\}-P\{B_{Q}\{P\}\}+P\{P\{B_{Q}\}\}=0,$$
	then $B_{Q} = [-,-]_{A}$ implies that $P$ is a Nijenhuis operator on Lie algebroid $(A, [-,-]_{A}, \rho)$.
		
	(ii) Now let $P$ be a Nijenhuis operator on Lie algebroid $(A, [-,-]_{A}, \rho)$. By the first statement, $P$ corresponds to a Maurer-Cartan element $\overline{P}$ in the graded Lie algebra $(\mathfrak{C}_{\NjLD}(A), l_2^\alpha)$.
	For $\overline{F} \in \C_{\FN}^{r}(A) $ of degree $r$, we compute $(l_1^\alpha)^{\overline{P}}(\overline{F})$ as follows:
	$$\begin{aligned}
		(l_1^\alpha)^{\overline{P}}(\overline{F})
		=&\ l_1^\alpha(\overline{F})-l_2^\alpha(\overline{P} \ot \overline{F})\\
		=&\ -l_2^\alpha(\overline{P} \ot \overline{F})\\
		=&\ \Upsilon^{-1} \big(\nu\{P, F\} - P\{\nu\{F\}\} - (-1)^{r} P \{F\{\nu\}\}
		+ \nu\{F, P\} +(-1)^{r} F\{\nu\{P\}\} - (-1)^{r} F\{P\{\nu\}\}\big), 
	\end{aligned}$$
	which corresponds to $ (-1)^{r-1} \dd_{\FN}(\overline{F})$. So
	the underlying complex of the twisted dg Lie algebra $\C_{\FN}^{\bullet}(A)$ by the corresponding Maurer-Cartan element $\overline{P}$ is isomorphic to the Fr\"olicher-Nijenhuis cochain complex of Nijenhuis operator $P$.
\end{proof}

\begin{remark}
	In Theorem~\ref{Thm: L-infty of NjL algebroid}, we constructed an $L_\infty$-algebra associated to a given vector bundle $A$ over $M$, and proved that the Maurer-Cartan elements in this $L_\infty$-algebra corresponds to Nijenhuis Lie algebroid structure on $A$. More generally, if we replace the vector bundle $A$ with a 
	non-positive 
	cohomological graded vector bundle, the same construction in Theorem~\ref{Thm: L-infty of NjL algebroid} still yields an $L_\infty$-algebra. In this case, a Maurer-Cartan element in this $L_\infty$-algebra defines an $L_\infty$-algebroid equipped with a homotopy Nijenhuis structure.
\end{remark}

\medskip

\section{A computational example: Poincar\'e Lemma of a certain class of Nijenhuis operators} \label{Section: An Example of Calculating Frolicher-Nijenhuis Cohomology}

In this section, we will compute the Fr\"olicher-Nijenhuis cohomology of a special case of the Nijenhuis operator, which confirms a conjecture by Bolsinov and Konyaev.

We will continue to use the notations and conventions established in Proposition~\ref{Prop. relation between Lie algebroids and dg manifolds} and its proof. 
Assume that $(A, [-,-]_{A}, \rho, P)$ is a Nijenhuis Lie algebroid over $M$ with Nijenhuis Lie algebroid structure $\rho(\varepsilon_{i}) = \rho_{i}^{\alpha}\partial_{x^{\alpha}}$, $[\varepsilon_{i},\varepsilon_{j}]_{A} = c_{ij}^{k} \varepsilon_{k}$, and $P(\varepsilon_{i}) = P_{i}^{j}\varepsilon_{j}$. Proposition~\ref{Prop. Nijenhuis torsion is a tensor} shows that the Nijenhuis torsion $\calN_{P}$ of $P$ is a tensor of type $(1,2)$, and direct calculation yields coefficients of $\calN_{P}$ are 
\begin{align*}
	{\calN_{P}}_{i j}^{k} 
	=&\ P_{i}^{a}P_{j}^{b}c_{ab}^{k} - P_{b}^{k}P_{i}^{a}c_{aj}^{b} + P_{a}^{k}P_{b}^{a}c_{ij}^{b} - P_{a}^{k}P_{j}^{b}c_{ib}^{a} 
	\\
	&\ 
	+ P_{i}^{a}\rho_{a}^{b} \frac{\partial P_{j}^{k}}{\partial x^{b}} - P_{j}^{a}\rho_{a}^{b} \frac{\partial P_{i}^{k}}{\partial x^{b}} 
	- P_{a}^{k}\rho_{i}^{b} \frac{\partial P_{j}^{a}}{\partial x^{b}} + P_{a}^{k}\rho_{j}^{b} \frac{\partial P_{i}^{a}}{\partial x^{b}}.
\end{align*}
In particular, when $(A, [-,-]_{A}, \rho)$ is taken as the trivial Lie algebroid $(T_{M}, [-,-], \id)$, we have 
\begin{align*}
	{\calN_{P}}_{i j}^{k} 
	=P_{i}^{a}\frac{\partial P_{j}^{k}}{\partial x^{a}} - P_{j}^{a}\frac{\partial P_{i}^{k}}{\partial x^{a}} 
	- P_{a}^{k}\frac{\partial P_{j}^{a}}{\partial x^{i}} + P_{a}^{k} \frac{\partial P_{i}^{a}}{\partial x^{j}}.
\end{align*}
This is consistent with the classical result, see \cite{BKM22}.

 Consider the $(1,1)$-form $P = x^{i} \dd x^{i} \ot \partial_{x^{i}}$ on the manifold $\mathbb{R}^n$, viewed as a $\ci(\mathbb{R}^n)$-linear map
\begin{align*}
	P: \Gamma(T_{\mathbb{R}^n}) \rightarrow \Gamma(T_{\mathbb{R}^n}),\ \partial_{x^{i}} \mapsto x^{i}\partial_{x^{i}},
\end{align*}
i.e., $P=\operatorname{diag} \{x^1, \ldots, x^n \}$. 
It is not hard to see that ${\calN_{P}}_{i j}^{k} = 0$ for any $1 \leqslant i,j,k \leqslant n$. Thus, $P$ is indeed a Nijenhuis operator. In \cite{BKM22},  Bolsinov, Konyaev, and Matveev classified Nijenhuis structures based on the local form of Nijenhuis operators. In their classification, a Nijenhuis operator of the above form is a point in general position within the class of differentially non-degenerate Nijenhuis operators.

In \cite{BK21-17},  Bolsinov and Konyaev have shown that $\rmH_{\FN}^{0}(T_{\mathbb{R}^n}) = \rmH_{\FN}^{1}(T_{\mathbb{R}^n}) = 0$, and they conjectured that  $\rmH_{\FN}^{k}(T_{\mathbb{R}^n})=0$ for all $k>1$. This would imply that we have a kind of Poincar\'e lemma for this kind of Nijenhuis geometry. The following proposition provides an affirmative answer to the conjecture.

\begin{prop}[Poincar\'e Lemma of diagonal Nijenhuis operator]
	All cohomologies 
	$\rmH_{\FN}^{k}(T_{\mathbb{R}^n}) = 0$
	for all $k \geqslant 0$ in the case of $P=\operatorname{diag} \{x^1, \ldots, x^n\}$. In other words, the Fr\"olicher-Nijenhuis cochain complex $ (\C_{\FN}^{\bullet}(T_{\mathbb{R}^n}), \dd_{\FN}=[P, -]_{\FN}) $ of Nijenhuis operator $P$ is exact.
\end{prop}

\begin{proof}
	Construct a homotopy $ h $ on the cochain complex $ \C_{\FN}^{\bullet}(T_{\mathbb{R}^n}) $ as follows: for any $ 1 \leqslant j \leqslant n $ and $ K=K^{\alpha}_{i_{1} \cdots i_{j}} \dd x^{i_{1}} \wg \cdots \wg \dd x^{i_{j}} \ot \partial_{x^{\alpha}} \in \C_{\FN}^{j}(T_{\mathbb{R}^n}) \cong \Hom_{\ci(\mathbb{R}^n)}(\wedge^{j} \Gamma(T_{\mathbb{R}^n}), \Gamma(T_{\mathbb{R}^n})) $ with $ i_{1} < \cdots < i_{j} $, define
	$$ h(K) = \sum_{i_{1}< \cdots < i_{k} < \alpha < i_{k+1} < \cdots < i_{j-1}
	\atop 0 \leqslant k \leqslant j-1 }
	(-1)^{k+1} K^{\alpha}_{i_{1} \cdots i_{k} \alpha i_{k+1} \cdots i_{j-1}}  \dd x^{i_{1}} \wg \cdots \wg \dd x^{i_{j-1}} \ot \partial_{x^{\alpha}}.$$
	In particular, for any $ K \in \C_{\FN}^{0}(T_{\mathbb{R}^n}) \cong \Gamma(T_{\mathbb{R}^n}) $, define $ h(K) = 0 $.

	We claim that $ \dd_{\FN} h + h \dd_{\FN} = \id $, thus the complex $ (\C_{\FN}^{\bullet}(T_{\mathbb{R}^n}), \dd_{\FN}) $ is exact. In fact, for any $ 0 \leqslant j \leqslant n$ and $ K=K^{\alpha}_{i_{1} \cdots i_{j}} \dd x^{i_{1}} \wg \cdots \wg \dd x^{i_{j}} \ot \partial_{x^{\alpha}} \in \C_{\FN}^{j}(T_{\mathbb{R}^n})$, according to Definition~\ref{Defn. definition of FN bracket} and Proposition~\ref{Prop. Some properties of Lie derivative}, we have
	\begin{align*}
		\dd_{\FN}h(K)
		=&\ [P,h(K)]_{\FN}\\
		=&\ \Big[x^{l} \dd x^{l} \ot \partial_{x^{l}}, 
		\sum_{i_{1}< \cdots < i_{k} < \alpha < i_{k+1} < \cdots < i_{j-1}
		\atop 0 \leqslant k \leqslant j-1} (-1)^{k+1}
		K^{\alpha}_{i_{1} \cdots i_{k} \alpha i_{k+1} \cdots i_{j-1}} \dd x^{i_{1}} \wg \cdots \wg \dd x^{i_{j-1}} \ot \partial_{x^{\alpha}} \Big]_{\FN}\\
		=&\ \sum_{i_{1}< \cdots < i_{k} < \alpha < i_{k+1} < \cdots < i_{j-1}
		\atop 0 \leqslant k \leqslant j-1} (-1)^{k+1}
		\dd x^{l} \wg (\dd x^{i_{1}} \wg \cdots \wg \dd x^{i_{j-1}}) \ot [x^{l} \partial_{x^{l}}, K^{\alpha}_{i_{1} \cdots i_{k} \alpha i_{k+1} \cdots i_{j-1}} \partial_{x^{\alpha}}]\\
		&\ + \sum_{i_{1}< \cdots < i_{k} < \alpha < i_{k+1} < \cdots < i_{j-1}
		\atop 0 \leqslant k \leqslant j-1} (-1)^{k+1}
		K^{\alpha}_{i_{1} \cdots i_{k} \alpha i_{k+1} \cdots i_{j-1}} 
		\dd x^{l} \wg (\mathcal{L}_{x^{l} \partial_{x^{l}}} \dd x^{i_{1}} \wg \cdots \wg \dd x^{i_{j-1}}) \ot \partial_{x^{\alpha}}\\
		&\ - \sum_{i_{1}< \cdots < i_{k} < \alpha < i_{k+1} < \cdots < i_{j-1}
		\atop 0 \leqslant k \leqslant j-1} (-1)^{k+1}
		x^{l}
		(\mathcal{L}_{K^{\alpha}_{i_{1} \cdots i_{k} \alpha i_{k+1} \cdots i_{j-1}} \partial_{x^{\alpha}}}\dd x^{l}) \wg (\dd x^{i_{1}} \wg \cdots \wg \dd x^{i_{j-1}}) \ot \partial_{x^{l}}\\
		=&\ \sum_{i_{1}< \cdots < i_{k} < \alpha < i_{k+1} < \cdots < i_{j-1}\atop 0 \leqslant k \leqslant j-1} (-1)^{k+1}
		x^{l} \frac{\partial K^{\alpha}_{i_{1} \cdots i_{k} \alpha i_{k+1} \cdots i_{j-1}}}{\partial x^{l}}
		\dd x^{l} \wg \dd x^{i_{1}} \wg \cdots \wg \dd x^{i_{j-1}} \ot \partial_{x^{\alpha}}\\
		&\ - \sum_{i_{1}< \cdots < i_{k} < \alpha < i_{k+1} < \cdots < i_{j-1}\atop 0 \leqslant k \leqslant j-1} (-1)^{k+1}
		K^{\alpha}_{i_{1} \cdots i_{k} \alpha i_{k+1} \cdots i_{j-1}} \frac{\partial x^{l} }{\partial x^{\alpha}}
		\dd x^{l} \wg \dd x^{i_{1}} \wg \cdots \wg \dd x^{i_{j-1}} \ot \partial_{x^{l}}\\
		&\ + \sum_{i_{1}< \cdots < i_{k} < \alpha < i_{k+1} < \cdots < i_{j-1}\atop 0 \leqslant k \leqslant j-1} (-1)^{k+1}
		K^{\alpha}_{i_{1} \cdots i_{k} \alpha i_{k+1} \cdots i_{j-1}} 
		\dd x^{l} \wg (x^{l}\mathcal{L}_{ \partial_{x^{l}}} \dd x^{i_{1}} \wg \cdots \wg \dd x^{i_{j-1}}) \ot \partial_{x^{\alpha}}\\
		&\ + \sum_{i_{1}< \cdots < i_{k} < \alpha < i_{k+1} < \cdots < i_{j-1}\atop 0 \leqslant k \leqslant j-1} (-1)^{k+1}
		K^{\alpha}_{i_{1} \cdots i_{k} \alpha i_{k+1} \cdots i_{j-1}}
		\dd x^{l} \wg
		\big(
		\dd x^{l} \wg \iota_{ \partial_{x^{l}}} (\dd x^{i_{1}} \wg \cdots \wg \dd x^{i_{j-1}})
		\big) \ot \partial_{x^{\alpha}}\\
		&\ - \sum_{i_{1}< \cdots < i_{k} < \alpha < i_{k+1} < \cdots < i_{j-1}\atop 0 \leqslant k \leqslant j-1} (-1)^{k+1}
		x^{l}
		(K^{\alpha}_{i_{1} \cdots i_{k} \alpha i_{k+1} \cdots i_{j-1}} \mathcal{L}_{ \partial_{x^{\alpha}}}\dd x^{l}) \wg (\dd x^{i_{1}} \wg \cdots \wg \dd x^{i_{j-1}}) \ot \partial_{x^{l}}\\
		&\ - \sum_{i_{1}< \cdots < i_{k} < \alpha < i_{k+1} < \cdots < i_{j-1}\atop 0 \leqslant k \leqslant j-1} (-1)^{k+1}
		x^{l} (\dd K^{\alpha}_{i_{1} \cdots i_{k} \alpha i_{k+1} \cdots i_{j-1}} \wg
		\iota_{\partial_{x^{\alpha}}}\dd x^{l}) \wg (\dd x^{i_{1}} \wg \cdots \wg \dd x^{i_{j-1}}) \ot \partial_{x^{l}}\\
		=&\ \sum_{i_{1}< \cdots < i_{k} < \alpha < i_{k+1} < \cdots < i_{j-1}\atop 0 \leqslant k \leqslant j-1} (-1)^{k+1}
		x^{l} \frac{\partial K^{\alpha}_{i_{1} \cdots i_{k} \alpha i_{k+1} \cdots i_{j-1}}}{\partial x^{l}}
		\dd x^{l} \wg \dd x^{i_{1}} \wg \cdots \wg \dd x^{i_{j-1}} \ot \partial_{x^{\alpha}}\\
		&\ - \sum_{i_{1}< \cdots < i_{k} < \alpha < i_{k+1} < \cdots < i_{j-1}\atop 0 \leqslant k \leqslant j-1} (-1)^{k+1}
		K^{\alpha}_{i_{1} \cdots i_{k} \alpha i_{k+1} \cdots i_{j-1}}
		\dd x^{\alpha} \wg \dd x^{i_{1}} \wg \cdots \wg \dd x^{i_{j-1}}	\ot \partial x^{\alpha}\\
		&\ - \sum_{i_{1}< \cdots < i_{k} < \alpha < i_{k+1} < \cdots < i_{j-1}\atop 0 \leqslant k \leqslant j-1} (-1)^{k+1}
		x^{\alpha} \frac{\partial K^{\alpha}_{i_{1} \cdots i_{k} \alpha i_{k+1} \cdots i_{j-1}}}{\partial x^{m}} 
		\dd x^{m}\wg
		\dd x^{i_{1}} \wg \cdots \wg \dd x^{i_{j-1}} \ot \partial_{x^{\alpha}}\\
		=&\ \sum_{i_{1}< \cdots < i_{k} < \alpha < i_{k+1} < \cdots < i_{j-1}\atop 0 \leqslant k \leqslant j-1} (-1)^{k+1}
		x^{\alpha} \frac{\partial K^{\alpha}_{i_{1} \cdots i_{k} \alpha i_{k+1} \cdots i_{j-1}}}{\partial x^{\alpha}}
		\dd x^{\alpha} \wg \dd x^{i_{1}} \wg \cdots \wg \dd x^{i_{j-1}}	\ot \partial_{x^{\alpha}}\\
		&\ + \sum_{i_{1}< \cdots < i_{k} < \alpha < i_{k+1} < \cdots < i_{j-1}\atop l \neq \alpha, 0 \leqslant k \leqslant j-1} (-1)^{k+1}
		x^{l} \frac{\partial K^{\alpha}_{i_{1} \cdots i_{k} \alpha i_{k+1} \cdots i_{j-1}}}{\partial x^{l}}
		\dd x^{l} \wg \dd x^{i_{1}} \wg \cdots \wg \dd x^{i_{j-1}} \ot \partial_{x^{\alpha}}\\
		&\ - \sum_{i_{1}< \cdots < i_{k} < \alpha < i_{k+1} < \cdots < i_{j-1}\atop 0 \leqslant k \leqslant j-1} (-1)^{k+1}
		K^{\alpha}_{i_{1} \cdots i_{k} \alpha i_{k+1} \cdots i_{j-1}}
		\dd x^{\alpha} \wg \dd x^{i_{1}} \wg \cdots \wg \dd x^{i_{j-1}}	\ot \partial_{x^{\alpha}}\\
		&\ - \sum_{i_{1}< \cdots < i_{k} < \alpha < i_{k+1} < \cdots < i_{j-1}\atop 0 \leqslant k \leqslant j-1} (-1)^{k+1}
		x^{\alpha} \frac{\partial K^{\alpha}_{i_{1} \cdots i_{k} \alpha i_{k+1} \cdots i_{j-1}}}{\partial x^{\alpha}}
		\dd x^{\alpha}\wg
		\dd x^{i_{1}} \wg \cdots \wg \dd x^{i_{j-1}} \ot \partial_{x^{\alpha}}\\
		&\ - \sum_{i_{1}< \cdots < i_{k} < \alpha < i_{k+1} < \cdots < i_{j-1}\atop m \neq \alpha, 0 \leqslant k \leqslant j-1} (-1)^{k+1}
		x^{\alpha} \frac{\partial K^{\alpha}_{i_{1} \cdots i_{k} \alpha i_{k+1} \cdots i_{j-1}}}{\partial x^{m}}
		\dd x^{m}\wg
		\dd x^{i_{1}} \wg \cdots \wg \dd x^{i_{j-1}} \ot \partial_{x^{\alpha}}\\
		=&\ \sum_{i_{1}< \cdots < i_{k} < \alpha < i_{k+1} < \cdots < i_{j-1}\atop l \neq \alpha, 0 \leqslant k \leqslant j-1} (-1)^{k+1}
		x^{l} \frac{\partial K^{\alpha}_{i_{1} \cdots i_{k} \alpha i_{k+1} \cdots i_{j-1}}}{\partial x^{l}}
		\dd x^{l} \wg \dd x^{i_{1}} \wg \cdots \wg \dd x^{i_{j-1}} \ot \partial_{x^{\alpha}}\\
		&\ - \sum_{i_{1}< \cdots < i_{k} < \alpha < i_{k+1} < \cdots < i_{j-1}\atop 0 \leqslant k \leqslant j-1} (-1)^{k+1}
		K^{\alpha}_{i_{1} \cdots i_{k} \alpha i_{k+1} \cdots i_{j-1}}
		\dd x^{\alpha} \wg \dd x^{i_{1}} \wg \cdots \wg \dd x^{i_{j-1}} \ot \partial_{x^{\alpha}}\\
		&\ - \sum_{i_{1}< \cdots < i_{k} < \alpha < i_{k+1} < \cdots < i_{j-1}\atop m \neq \alpha, 0 \leqslant k \leqslant j-1} (-1)^{k+1}
		x^{\alpha} \frac{\partial K^{\alpha}_{i_{1} \cdots i_{k} \alpha i_{k+1} \cdots i_{j-1}}}{\partial x^{m}}
		\dd x^{m}\wg
		\dd x^{i_{1}} \wg \cdots \wg \dd x^{i_{j-1}} \ot \partial_{x^{\alpha}}, 
	\end{align*}
	\begin{align*}
		h \dd_{\FN}(K) 
		=&\ h[ P,K ]_{\FN} \\
		=&\ h[ x^{l} \dd x^{l} \ot \partial_{x^{l}}, K^{\alpha}_{i_{1} \cdots i_{j}} \dd x^{i_{1}} \wg \cdots \wg \dd x^{i_{j}} \ot \partial_{x^{\alpha}} ]_{\FN} \\
		=&\ h \big(
		\dd x^{l} \wg (\dd x^{i_{1}} \wg \cdots \wg \dd x^{i_{j}}) \ot [x^{l} \partial_{x^{l}}, K^{\alpha}_{i_{1} \cdots i_{j}} \partial_{x^{\alpha}}] \big)
		+ h \big(
		K^{\alpha}_{i_{1} \cdots i_{j}} 
		\dd x^{l} \wg (\mathcal{L}_{x^{l} \partial_{x^{l}}} \dd x^{i_{1}} \wg \cdots \wg \dd x^{i_{j}}) \ot \partial_{x^{\alpha}} \big)\\
		&\ - h \big(
		x^{l} 
		(\mathcal{L}_{K^{\alpha}_{i_{1} \cdots i_{j}} \partial_{x^{\alpha}}} \dd x^{l}) \wg (\dd x^{i_{1}} \wg \cdots \wg \dd x^{i_{j}}) \ot \partial_{x^{l}}
		\big)	\\
		=&\ h \big( x^{l} \frac{\partial K^{\alpha}_{i_{1} \cdots i_{j}}}{\partial x^{l}} 
		\dd x^{l} \wg \dd x^{i_{1}} \wg \cdots \wg \dd x^{i_{j}} \ot \partial_{x^{\alpha}}\big)
		- h \big(
		K^{\alpha}_{i_{1} \cdots i_{j}} \frac{\partial{x^{l}}}{\partial x^{\alpha}} 
		\dd x^{l} \wg \dd x^{i_{1}} \wg \cdots \wg \dd x^{i_{j}} \ot \partial_{x^{l}} \big)\\	
		&\ + h \big(
		K^{\alpha}_{i_{1} \cdots i_{j}} 
		\dd x^{l} \wg (x^{l} \mathcal{L}_{ \partial_{x^{l}}} \dd x^{i_{1}} \wg \cdots \wg \dd x^{i_{j}}) \ot \partial_{x^{\alpha}}\big)
		+ h \big(
		K^{\alpha}_{i_{1} \cdots i_{j}} 
		\dd x^{l} \wg (\dd x^{l} \wg \iota_{ \partial_{x^{l}}} (\dd x^{i_{1}} \wg \cdots \wg \dd x^{i_{j}})) \ot \partial_{x^{\alpha}}\big)\\
		&\ - h \big(
		x^{l} 
		(K^{\alpha}_{i_{1} \cdots i_{j}} \mathcal{L}_{ \partial_{x^{\alpha}}} \dd x^{l}) \wg (\dd x^{i_{1}} \wg \cdots \wg \dd x^{i_{j}}) \ot \partial_{x^{l}}
		\big)
		- h \big(
		x^{l} 
		(\dd K^{\alpha}_{i_{1} \cdots i_{j}} \wg \iota_{ \partial_{x^{\alpha}}} \dd x^{l}) \wg (\dd x^{i_{1}} \wg \cdots \wg \dd x^{i_{j}}) \ot \partial_{x^{l}}
		\big)	\\
		=&\ h\big( x^{l} \frac{\partial K^{\alpha}_{i_{1} \cdots i_{j}}}{\partial x^{l}} 
		\dd x^{l} \wg \dd x^{i_{1}} \wg \cdots \wg \dd x^{i_{j}} \ot \partial_{x^{\alpha}} \big)
		- h\big(
		K^{\alpha}_{i_{1} \cdots i_{j}}
		\dd x^{\alpha} \wg \dd x^{i_{1}} \wg \cdots \wg \dd x^{i_{j}} \ot \partial_{x^{\alpha}}\big)\\	
		&\ - h \big(
		x^{\alpha} \frac{\partial K^{\alpha}_{i_{1} \cdots i_{j}}}{\partial x^{m}}
		\dd x^{m}\wg \dd x^{i_{1}} \wg \cdots \wg \dd x^{i_{j}} \ot \partial_{x^{\alpha}}
		\big)	\\
		=&\ - \sum_{\alpha \neq i_{1}, \dots, i_{j}} 
		x^{\alpha} \frac{\partial K^{\alpha}_{i_{1} \cdots i_{j}}}{\partial x^{\alpha}}
		\dd x^{i_{1}} \wg \cdots \wg \dd x^{i_{j}} \ot \partial_{x^{\alpha}}\\	
		&\ 
		+ \sum_{i_{1}< \cdots < i_{k} < \alpha < i_{k+1} < \cdots < i_{j-1}
		\atop l \neq \alpha, 0 \leqslant k \leqslant j-1	} (-1)^{k+2}
		x^{l} \frac{\partial K^{\alpha}_{i_{1} \cdots i_{k} \alpha i_{k+1} \cdots i_{j-1}}}{\partial x^{l}} 
		\dd x^{l} \wg \dd x^{i_{1}} \wg \cdots \wg \dd x^{i_{j-1}} \ot \partial_{x^{\alpha}}\\
		&\ + \sum_{\alpha \neq i_{1}, \dots, i_{j}} K^{\alpha}_{i_{1} \cdots i_{j}}
		\dd x^{i_{1}} \wg \cdots \wg \dd x^{i_{j}} \ot \partial_{x^{\alpha}} 
		+\sum_{\alpha \neq i_{1}, \dots, i_{j}} x^{\alpha} \frac{\partial K^{\alpha}_{i_{1} \cdots i_{j}}}{\partial x^{\alpha}}
		\dd x^{i_{1}} \wg \cdots \wg \dd x^{i_{j}} \ot \partial_{x^{\alpha}}\\
		&\ - \sum_{i_{1}< \cdots < i_{k} < \alpha < i_{k+1} < \cdots < i_{j-1}
		\atop	m \neq \alpha, 0 \leqslant k \leqslant j-1} (-1)^{k+2}
		x^{\alpha} \frac{\partial K^{\alpha}_{i_{1} \cdots i_{k} \alpha i_{k+1} \cdots i_{j-1}}}{\partial x^{m}}
		\dd x^{m}\wg \dd x^{i_{1}} \wg \cdots \wg \dd x^{i_{j-1}} \ot \partial_{x^{\alpha}}\\
		=&\ \sum_{i_{1}< \cdots < i_{k} < \alpha < i_{k+1} < \cdots < i_{j-1}
		\atop l \neq \alpha, 0 \leqslant k \leqslant j-1	} (-1)^{k+2}
		x^{l} \frac{\partial K^{\alpha}_{i_{1} \cdots i_{k} \alpha i_{k+1} \cdots i_{j-1}}}{\partial x^{l}} 
		\dd x^{l} \wg \dd x^{i_{1}} \wg \cdots \wg \dd x^{i_{j-1}} \ot \partial_{x^{\alpha}}
		+ \sum_{\alpha \neq i_{1}, \dots, i_{j}} K^{\alpha}_{i_{1} \cdots i_{j}}
		\dd x^{i_{1}} \wg \cdots \wg \dd x^{i_{j}} \ot \partial_{x^{\alpha}}\\
		&\ - \sum_{i_{1}< \cdots < i_{k} < \alpha < i_{k+1} < \cdots < i_{j-1}
		\atop	m \neq \alpha, 0 \leqslant k \leqslant j-1} (-1)^{k+2}
		x^{\alpha} \frac{\partial K^{\alpha}_{i_{1} \cdots i_{k} \alpha i_{k+1} \cdots i_{j-1}}}{\partial x^{m}}
		\dd x^{m}\wg \dd x^{i_{1}} \wg \cdots \wg \dd x^{i_{j-1}} \ot \partial_{x^{\alpha}}.
	\end{align*}
	So we get
	\begin{align*}
		\dd_{\FN} h(K) + h \dd_{\FN}(K)
		=&\ - \sum_{i_{1}< \cdots < i_{k} < \alpha < i_{k+1} < \cdots < i_{j-1}\atop 0 \leqslant k \leqslant j-1} (-1)^{k+1}
		K^{\alpha}_{i_{1} \cdots i_{k} \alpha i_{k+1} \cdots i_{j-1}}
		\dd x^{\alpha} \wg \dd x^{i_{1}} \wg \cdots \wg \dd x^{i_{j-1}} \ot \partial_{x^{\alpha}}\\
		&\ + \sum_{\alpha \neq i_{1}, \dots, i_{j}} K^{\alpha}_{i_{1} \cdots i_{j}}
		\dd x^{i_{1}} \wg \cdots \wg \dd x^{i_{j}} \ot \partial_{x^{\alpha}}\\
		=&\ \sum_{i_{1}< \cdots < i_{k} < \alpha < i_{k+1} < \cdots < i_{j-1}\atop 0 \leqslant k \leqslant j-1} 
		K^{\alpha}_{i_{1} \cdots i_{k} \alpha i_{k+1} \cdots i_{j-1}}
		\dd x^{i_{1}} \wg \cdots \wg \dd x^{i_{k}} \wg \dd x^{\alpha} \wg \dd x^{i_{k+1}} \wg \cdots \wg \dd x^{i_{j-1}} \ot \partial_{x^{\alpha}}\\
		&\ + \sum_{\alpha \neq i_{1}, \dots, i_{j}} K^{\alpha}_{i_{1} \cdots i_{j}}
		\dd x^{i_{1}} \wg \cdots \wg \dd x^{i_{j}} \ot \partial_{x^{\alpha}}\\
		=&\ K^{\alpha}_{i_{1} \dots i_{j}}
		\dd x^{i_{1}} \wg \cdots \wg \dd x^{i_{j}} \ot \partial_{x^{\alpha}}\\
		=&\ \id(K).
	\end{align*}
	This completes the proof.
\end{proof}

\medskip

\textbf{Acknowledgements:} The first, second and the fourth authors were supported by  the National Key R  $\&$ D Program of China (No. 2024YFA1013803) and by Shanghai Key Laboratory of PMMP (No.22DZ2229014). The third author was supported by the National Natural Science Foundation of China (No. 12101183), and also  by the Postdoctoral Fellowship Program of CPSF (No. GZC20240406).

\bigskip

\end{document}